\documentclass{article}

\newif\ifshownavigationpage
\newif\ifshownotationindex
\newif\ifshowtheoremlinks
\newif\ifshowtheoremtree
\newif\ifshowtheoremlist
\newif\ifshowequationlist

\newif\ifshowcomments
\newif\ifhighlight 
\newif\ifelaborate

\newif\ifshowaddressedcomments
\newif\ifshowrvin
\newif\ifshowrvout

\showtheoremtreetrue

\usepackage{amsthm, amsmath, amssymb}
\usepackage{thmtools}
\usepackage{soul}
\usepackage{authblk}
\usepackage{framed}
\usepackage{mathrsfs}
\usepackage{hyperref}
\hypersetup{
    colorlinks,
    linkcolor={red!50!black},
    citecolor={blue!50!black},
    urlcolor={blue!80!black}
}

\usepackage[dvipsnames]{xcolor}
\usepackage{graphicx, multicol}
\usepackage{marvosym, wasysym}
\usepackage{fancyhdr}
\usepackage{stackengine}
\usepackage{algorithm}
\usepackage{bm}
\usepackage[noend]{algpseudocode}
\usepackage{bbm}
\usepackage{verbatim}
\usepackage{mdframed}
\usepackage{needspace}
\makeatletter
\renewcommand{\ALG@beginalgorithmic}{\scriptsize}
\makeatother
\usepackage{xcolor}
\allowdisplaybreaks

\DeclareFontFamily{U}{mathx}{}
\DeclareFontShape{U}{mathx}{m}{n}{ <-> mathx10 }{}
\DeclareSymbolFont{mathx}{U}{mathx}{m}{n}
\DeclareFontSubstitution{U}{mathx}{m}{n}

\newcommand{\bcdot}{\boldsymbol{\cdot}}

\renewcommand{\dj}[1]{\bm{d}_{J_1}^{_{#1}}}

\newcommand{\dm}[1]{\bm{d}_{M_1}^{_{#1}}}

\newcommand{\dmp}[1]{\bm{d}_{ \mathcal P }^{_{#1}} }

\newcommand{\subInfty}{ _{[0,\infty)} }

\newcommand{\subZeroT}[1]{ _{[0,#1]} }

\DeclareMathOperator*{\argmin}{arg\,min}

\ifhighlight
\else
    
\fi

\newcommand{\rvoutopacity}{20}
\ifshowrvin
    
\else
    
\fi

\ifshowrvout
    \newcommand{\rvout}[1]{{\color{red!\rvoutopacity}{#1}} }
    \newcommand{\rvoutm}[1]{{\color{black!\rvoutopacity}{\ifmmode\text{\sout{\ensuremath{\displaystyle#1}}}\else\sout{#1}\fi}} } 
\else
    \newcommand{\rvout}[1]{}
\fi

\ifshowcomments
    \newcommand{\summ}[1]{{\color{blue}[summary: #1]} } 
    \newcommand{\bz}[1]{{\color{PineGreen}[BZ: #1]} } 
    \newcommand{\rl}[1]{{\color{Periwinkle}[RL: #1]} } 
    \newcommand{\jb}[1]{{\color{Tan}[JB: #1]} } 
    \newcommand{\xw}[1]{{\color{RoyalBlue}[XW: #1]} } 
    \ifshowaddressedcomments
        \newcommand{\bza}[1]{{\color{PineGreen}\sout{[BZ: #1]}} } 
        \newcommand{\rla}[1]{{\color{Periwinkle}\sout{[RL: #1]}} } 
        \newcommand{\jba}[1]{{\color{Tan}\sout{[JB: #1]}} } 
        \newcommand{\xwa}[1]{{\color{RoyalBlue}\sout{[XW: #1]}} } 
    \else
        \newcommand{\bza}[1]{} 
        \newcommand{\rla}[1]{} 
        \newcommand{\jba}[1]{} 
        \newcommand{\xwa}[1]{} 
    \fi
\else
    \newcommand{\summ}[1]{} 
    \newcommand{\bz}[1]{} 
    \newcommand{\rl}[1]{} 
    \newcommand{\jb}[1]{} 
    \newcommand{\bza}[1]{} 
    \newcommand{\rla}[1]{} 
    \newcommand{\jba}[1]{} 
    \newcommand{\xw}[1]{} 
    \newcommand{\xwa}[1]{} 
\fi

\usepackage[shortlabels]{enumitem}

\newlist{thmdependence}{itemize}{10}
\setlist[thmdependence]{nosep,label=-}

\newcommand{\thmtreenode}[4]{\item[#1] {#2}~\ref{#3} {#4}}

\ifshowtheoremlinks
    \newcommand{\linksinthm}[1]{\emph{\linkdest{location, #1}\linktopf{#1} \linktothmtree{location, thm tree #1} }}
    \newcommand{\linksinthmwopf}[1]{\emph{\linkdest{location, #1} \linktothmtree{location, thm tree #1} }}
    \newcommand{\linksinpf}[1]{\linkdest{location, proof of #1}\linktothm{#1} \linktothmtree{location, thm tree #1} }
\else
    \newcommand{\linksinthm}[1]{}
    \newcommand{\linksinthmwopf}[1]{}
    \newcommand{\linksinpf}[1]{}
\fi

\ifshownotationindex
    \newcommand{\notationdef}[2]{\linkdest{location, notation definition of #1}\hyperlink{location, notation index of #1}{#2}}
    
\else
    \newcommand{\notationdef}[2]{#2}
\fi
    
\newcommand{\linktopf}[1]{\hyperlink{location, proof of #1}{\pflinksymbol}}
\newcommand{\linktothm}[1]{\hyperlink{location, #1}{\thmlinksymbol}}
\newcommand{\linktothmtree}[1]{\hyperlink{#1}{\thmtreelinksymbol}}
\newcommand{\thmlinksymbol}{{\tiny [Theorem]}}
\newcommand{\pflinksymbol}{{\tiny [Proof]}}
\newcommand{\thmtreelinksymbol}{{\tiny [ThmTree]}}

\makeatletter
\newcommand{\linkdest}[1]{\Hy@raisedlink{\hypertarget{#1}{}}}
\makeatother

\newcommand{\elaborateopacity}{50}
\newcommand{\elaboratecolor}{RawSienna}
\ifelaborate
    \newcommand{\elaborate}[1]{{\color{\elaboratecolor!\elaborateopacity}{
    \begin{framed}
    \noindent {\footnotesize[Elaboration]}
    #1 
    \end{framed}
    }}\noindent}
\else
    \newcommand{\elaborate}[1]{}
\fi

\newcommand{\stleq}{\underset{ \text{s.t.} }{\leq}}

\newcommand{\lo}{\mathit{o}}
\newcommand{\bo}{\mathcal{O}}

\usepackage[margin=1.2in]{geometry}

\newtheorem{theorem}{Theorem}
\newtheorem{lemma}[theorem]{Lemma}

\newtheorem{proposition}[theorem]{Proposition}

\newtheorem{definition}[theorem]{Definition}

\newtheorem{assumption}{Assumption}

\newtheorem{remark}{Remark}

\theoremstyle{remark}
\newtheorem{example}{Example}

\newtheorem*{theorem-nonumber}{Theorem}
\newtheorem*{condition-nonumber}{Condition}
\newtheorem*{proposition-nonumber}{Proposition}

\usepackage{mathtools}
\DeclarePairedDelimiter{\ceil}{\lceil}{\rceil}
\DeclarePairedDelimiter\floor{\lfloor}{\rfloor}

\newcommand{\D}{\mathbb D}

\newcommand{\cmt}[1]{#1} 
\renewcommand{\cmt}[1]{} 
\renewcommand{\P}{\mathbf{P}}

\newcommand{\E}{\mathbf{E}}
\newcommand{\RV}{\mathcal{RV}}
\newcommand{\MRV}{\mathcal{MHRV}}
\newcommand{\R}{\mathbb{R}}
\newcommand{\Z}{\mathbb{Z}}
\renewcommand{\S}{\mathbb{S}}
\newcommand{\C}{\mathbb{C}}
\newcommand{\M}{\mathbb{M}}

\renewcommand{\complement}{c}

\newcommand{\powersetTilde}[1]{ {\mathcal P}_{#1} }
\newcommand{\barDxepskT}[4]{ \breve{\D}^{ #2  }_{#3; #1}#4  }
\newcommand{\shortbarDepsk}[2]{ \breve{\D}^{#1}_{#2} }

\usepackage[normalem]{ulem}

\usepackage{multirow}
\usepackage{longtable}

\usepackage{array}
\def\delequal{\mathrel{\ensurestackMath{\stackon[1pt]{=}{\scriptscriptstyle\text{def}}}}}
\def\distequal{\mathrel{\ensurestackMath{\stackon[1pt]{=}{\scriptstyle\mathcal{D}}}}}
\newcommand{\norm}[1]{\left\lVert#1\right\rVert}
\usepackage{mathtools}

\algrenewcommand\algorithmicrequire{\textbf{Require:}}
\algrenewcommand\algorithmicensure{\textbf{Postcondition:}}

\usepackage{subfigure}
\usepackage{float}
\usepackage{subfiles}
\title{Sample Path Large Deviations for Multivariate Heavy-Tailed Hawkes Processes and Related L\'evy Processes}

\DeclareMathAccent{\widecheck}{0}{mathx}{"71}

\author[1]{Jose Blanchet} 
\author[2]{Roger J.~A. Laeven}
\author[2]{Xingyu Wang}
\author[3]{Bert Zwart}
\affil[1]{Department of Management Science and Engineering, Stanford University}
\affil[2]{Department of Quantitative Economics, University of Amsterdam}
\affil[3]{Centrum Wiskunde \& Informatica (CWI)}

\begin{document}
\maketitle

\begin{abstract}
\noindent
In this paper, we develop sample path large deviations for multivariate Hawkes processes with heavy-tailed mutual excitation rates. 
Our results address a broad class of rare events in Hawkes processes at the sample path level and, 
via the cluster representation of Hawkes processes and a recent result on the tail asymptotics of the cluster sizes,
unravel the most likely configurations of (multiple) large clusters that could trigger the target events.
Our proof hinges on establishing the asymptotic equivalence, in terms of M-convergence, between a suitably scaled multivariate Hawkes process and a coupled L\'evy process with multivariate hidden regular variation.
Hence, along the way, we derive a sample path large deviations principle for a class of L\'evy processes with multivariate hidden regular variation, which not only plays an auxiliary role in our analysis but is also of independent interest. 
\end{abstract}



\counterwithin{equation}{section}
\counterwithin{lemma}{section}
\counterwithin{corollary}{section}
\counterwithin{theorem}{section}
\counterwithin{definition}{section}
\counterwithin{proposition}{section}
\counterwithin{figure}{section}
\counterwithin{table}{section}

\tableofcontents
\bigskip

\section{Introduction}
\label{sec: intro}

The growing interconnectedness and complexity of modern ecological, economic, engineered, and social systems 
have led to risks and uncertainties that can amplify and cascade across populations, markets, and networks 
through temporal and spatial feedback. 
As a result, understanding and managing the amplification and cascade of risks poses significant challenges across science, engineering, and business.
Multivariate Hawkes processes 
offer a formalized approach to model such dependencies and mutual excitation effects in risks,
and
 have been widely applied in 
finance (\cite{AITSAHALIA2015585,BACRY20132475,Hawkes01022018}),
neuroscience and biology (\cite{Xu2005,LAMBERT20189,10.1214/10-AOS806}), 
seismology (\cite{Ogata01031988,Ikefuji02012022}), 
epidemiology (\cite{chiang2022hawkes}), 
social science (\cite{9378017,doi:10.1073/pnas.0803685105,10.1145/2808797.2814178,rizoiu2017tutorialhawkesprocessesevents}),
queueing systems (\cite{doi:10.1287/stsy.2021.0070,doi:10.1287/stsy.2018.0014,selvamuthu2022infinite}),
and cyber security (\cite{baldwin2017contagion,Bessy-Roland_Boumezoued_Hillairet_2021}).

Specifically, 
this paper focuses on
$\bm N(t) = (N_1(t),N_2(t),\ldots,N_d(t))^\top$,
a $d$-dimensional càdlàg point process 
with initial value $\bm N(0) = \bm 0$,
and its conditional intensities $h^{\bm N}_i$ along each dimension are specified as follows
(here, we write $[d] = \{1,2,\ldots,d\}$).

\begin{definition}[Multivariate Hawkes Processes]
\label{def: hawkes process, conditional intensity}
The  point process $\bm N(t) = (N_i(t))_{i \in [d]}$ satisfies (for each $i \in [d]$ and $t \geq 0$)
\begin{align*}
    \P\big( N_i(t+\Delta) - N_i(t) = 0\ \big|\ \mathcal F_t  \big)
    & = 1 - h^{\bm N}_i(t)\Delta + \lo(\Delta),
    \\ 
    \P\big( N_i(t+\Delta) - N_i(t) = 1\ \big|\ \mathcal F_t  \big)
    & = h^{\bm N}_i(t)\Delta + \lo(\Delta),
    \\
    \P\big( N_i(t+\Delta) - N_i(t) > 1\ \big|\ \mathcal F_t  \big)
    & = \lo(\Delta),
\end{align*}
as $\Delta \downarrow 0$,
with
$\mathcal F_t \delequal \sigma\big\{ \bm N(s):\ s \in [0,t] \big\}$,
and the conditional intensities take the form
\begin{align}
    h^{\bm N}_i(t) = c^{\bm N}_{i}
    +
    \sum_{j \in [d]} \int_{0}^t {\tilde B_{i \leftarrow j}(s)}f^{\bm N}_{i \leftarrow j}(t-s)d N_j(s),
    \qquad\forall i \in [d].
    \label{def: conditional intensity, hawkes process}
\end{align}
Here, 
the $c^{\bm N}_i$'s are positive constants.
For each pair $(i,j) \in [d]^2$,
$f^{\bm N}_{i \leftarrow j}(\cdot)$ is a deterministic, non-negative, and integrable function,
and $\big(\tilde B_{i \leftarrow j}(s)\big)_{s >0}$ are independent copies of some non-negative random variables $\tilde B_{i \leftarrow j}$ (mutually independent across $i,j \in [d]$).
\end{definition}

Intuitively speaking, the constant $c^{\bm N}_i$ represents the base rate at which type-$i$ immigrant events arrive,
and the random  function $\tilde B_{i \leftarrow j}f^{\bm N}_{i \leftarrow j}(\cdot)$
dictates the rate at which any type-$j$ event induces (i.e., gives birth to) type-$i$ events in the future, thus capturing the mutual excitation mechanism of risks.
In Definition~\ref{def: hawkes process, conditional intensity}, the $\tilde B_{i \leftarrow j}$'s are typically referred to as excitation rates, and the $f^{\bm N}_{i \leftarrow j}(\cdot)$'s are called decay functions.
See also \cite{dj1972summary,209288c5-6a29-3263-8490-c192a9031603}, and Chapter 7 of \cite{daley2003introduction} for a detailed treatment of the conditional intensity approach to Hawkes processes.

In this paper, we establish a sample path large deviations principle (LDP) for the multivariate Hawkes process $\bm N(t)$, in the presence of power-law heavy tails in the distribution of the mutual excitation rates $\tilde B_{i\leftarrow j}$.
In particular, our work resolves significant gaps in the existing literature on large deviations of heavy-tailed Hawkes processes.
\begin{itemize}

    \item 
        Existing results are manifestations of \emph{the principle of a single big jump}
        (e.g., \cite{karim2021exact,baeriswyl2023tail,Asmussen_Foss_2018}),
        addressing only a limited class of rare events that are driven by 
        a single point inducing a disproportionately large number of offspring in one generation.
        By contrast, our results capture a much broader class of events caused by multiple big jumps across different components.

    \item 
        A \emph{sample-path level} characterization of large deviations is absent, even in the univariate setting;
        see, for instance, \cite{baeriswyl2023tail}.
        To the best of our knowledge, our work is the first to establish sample path LDP for multivariate heavy-tailed Hawkes processes.
        It is also worth noting that Hawkes processes can be viewed as continuous-time analogs of branching processes with immigration.
        For heavy-tailed branching processes with immigration,
        large deviations in the stationary distribution and partial sums have been studied (e.g., \cite{Basrak02102013,guo2025precise});
        however, sample path large deviations have still not been characterized, even in the univariate case.
        
\end{itemize}

To further explain these gaps and our methodology,
we briefly review the \emph{cluster approach} (also known as the Galton-Watson approach),
which is standard in large deviations analysis of Hawkes processes and exploits the underlying branching structure in $\bm N(t)$.
More precisely, a type-$j$ immigrant gives birth to children along different dimensions and hence induces its own family tree (i.e., cluster).
The cluster size vectors $\bm S_j$ solve the distributional fixed-point equations 
\begin{align}
    \bm S_j \distequal
    \bm e_j + \sum_{i \in [d]}\sum_{ m = 1 }^{ B_{i \leftarrow j}  }\bm S_i^{(m)},
    \qquad j \in [d],
     \label{def: fixed point equation for cluster S i}
\end{align}
where 
$\bm e_j$ is the $j^\text{th}$ unit vector in $\R^d$ (i.e., the vector with its $j^\text{th}$ entry equal to 1 and all other entries equal to 0),
each $\bm S_i^{(m)}$ is an independent copy of $\bm S_i$,
and $B_{i \leftarrow j} \distequal \text{Poisson}\big( \tilde B_{i \leftarrow j} \norm{f^{\bm N}_{i \leftarrow j} }_1  \big)$
with $\tilde B_{i\leftarrow j}$ and $f^{\bm N}_{i \leftarrow j}$ being the excitation rates and decay functions in Definition~\ref{def: hawkes process, conditional intensity}
and $\norm{f}_1 \delequal \int |f(x)|dx$.
For more details about the cluster representation of Hawkes processes,
we refer the reader to \cite{209288c5-6a29-3263-8490-c192a9031603,20.500.11850/151886,daley2003introduction}
(see also Definition~\ref{def: offspring cluster process} in the Appendix).
Here, we note that the canonical representation of $\bm S_j$ in \eqref{def: fixed point equation for cluster S i} is the total progeny of
a multi-type branching process (see, e.g., \cite{JOFFE1967409,haccou2005branching}) across the $d$ dimensions, where $B_{i \leftarrow j}$ is the count of type-$i$ children in one generation from a type-$j$ parent.

Specializing to the large deviations analysis for Hawkes processes,
the {cluster approach} proceeds by first characterizing the behavior of clusters, and then connecting the large deviations of the Hawkes process to those of a compound Poisson
process the increments of which admit the law of the cluster size vectors $\bm S_j$.
In the light-tailed case,
this approach is streamlined by the classical LDP framework (\cite{MR2571413, MR997938, MR2260560, MR758258})
and has been broadly successful in large deviations analyses of Hawkes processes and several extensions, including marked Hawkes processes and compound Hawkes processes;
see, e.g., \cite{Bordenave07112007,stabile2010risk,Karabash03072015,karim2023compound}.
We also refer the reader to \cite{10.1214/14-AAP1003,AIHPB_2014__50_3_845_0}
for large deviations of the non-linear generalization of Hawkes processes (\cite{1e405749-c2fa-3fe8-ae93-09b947fe302a}), where the cluster representation fails.

By contrast,
large deviations analyses of heavy-tailed Hawkes processes are relatively scarce,
with the two aforementioned gaps persisting.
To resolve the associated technical challenges, 
we implement the cluster approach in the multivariate heavy-tailed setting through three steps.
In the first step, we apply the recent progress in \cite{blanchet2025tailasymptoticsclustersizes} to characterize the tail behavior of the cluster size vector $\bm S_i$.
To be more specific,
given some non-empty index set $\bm j \subseteq [d]$, let
$
\R^d(\bm j) \delequal
\{ \sum_{i \in \bm j}w_i\E[\bm S_i]:\ w_i \geq 0\ \forall i \in \bm j  \}$ be the convex cone generated by the vectors $(\E [\bm S_i] )_{i \in \bm j}$.
We show that, over each cone $\R^d(\bm j)$,
Theorem~3.2 of \cite{blanchet2025tailasymptoticsclustersizes} entails that
the measure $\P\big(n^{-1}\bm S_i \in \ \cdot\ \cap \R^d(\bm j)\big)$---characterizing the tail behavior of $\bm S_i$ when restricted over the cone $ \R^d(\bm j)$---is roughly decaying at a power-law rate  $n^{- \alpha(\bm j)}$ as $n \to \infty$.
That is, \emph{hidden regular variation} (e.g., \cite{resnick2002hidden,10.1214/14-PS231}) arises endogenously in the cluster size vector $\bm{S}_i$, which exhibits varying degrees of heavy-tailedness across different directions in Euclidean space.
We review in Section~\ref{subsec: tail of cluster S, statement of results} the precise definitions of the tail indices $\alpha(\bm j)$ and the rigorous statements for tail asymptotics of $\bm S_i$.
Compared to prior results (e.g., \cite{Asmussen_Foss_2018,karim2021exact}),
such characterizations are able to address a much more general class of sets $A$ for the asymptotic analysis of rare events $\{n^{-1}\bm S_i \in A\}$,
revealing that 
\emph{the extremal behaviors of Hawkes process clusters are driven by multiple big jumps}
(i.e., multiple points in $\bm N(t)$ giving birth to a disproportionately large number of offspring),
whose contributions
align with vectors $\big(\E[\bm S_l]\big)_{l \in [d]}$.

In light of the cluster size tail asymptotics,
our second step is to develop a general framework for \emph{sample path large deviations} under increments exhibiting \emph{multivariate hidden regular variation}.
Sample path level characterizations for the principle of a single big jump (e.g.,  \cite{hult2005functional,borovkov_borovkov_2008,denisov2008large, embrechts2013modelling, foss2011introduction})
and the more general multiple big jump principle 
(e.g.,  \cite{borovkov2011large,10.1214/18-AOP1319,bernhard2020heavy})
are available for random walks and L\'evy processes with regularly varying increments.
See also \cite{10.1214/21-AAP1675} for the characterization of the multiple big jump principle in point processes,
which could imply sample path LDPs for heavy-tailed random walks or L\'evy processes through continuous mapping arguments.
However,
studies on sample path large deviations under increments with \emph{multivariate hidden regular variation} are lacking.
For instance, Result~2 in \cite{chen2019efficient} addresses L\'evy processes with independent jumps of different heavy-tailedness along the standard orthogonal basis of $\R^d$.
However, this is not suitable for our study of Hawkes processes, where the cluster size vectors exhibit strong cross-coordinate correlations and can align with arbitrary directions in $\R^d$.  
Meanwhile, Section~5 of \cite{das2023aggregatingheavytailedrandomvectors} explores the tail asymptotics of L\'evy processes with multivariate hidden regular variation but does not provide sample path level results.
 
To overcome this technical challenge, we establish in Theorem~\ref{theorem: LD for levy, MRV} the sample path LDP for L\'evy processes under increments that exhibit a suitable, general form of multivariate hidden regular variation.
The results are stated w.r.t.\ the Skorokhod $J_1$ topology of the càdlàg space $\D[0,T]$.
This development significantly extends existing results:
compared to \cite{10.1214/18-AOP1319,bernhard2020heavy,10.1214/21-AAP1675}, 
our results address multivariate cases with hidden regular variation;
unlike  Theorem~3.3 of \cite{10.1214/18-AOP1319} or  Result~2 in \cite{chen2019efficient}, which assume independent increments along an orthogonal basis, we allow for more general dependence structures of multivariate hidden regular variation in L\'evy processes;
and in contrast to \cite{das2023aggregatingheavytailedrandomvectors}, we provide results at the sample path level.
We present formal statements in Section~\ref{subsec: LD for levy process, MRV} and emphasize the following aspects: 
(i) the proof refines the notion of asymptotic equivalence in $\M$-convergence theory (\cite{10.1214/14-PS231}), distinguishing large jumps by their directions to capture hidden regular variation in the increments of L\'evy processes; 
(ii) the multivariate hidden regular variation is modeled via the $\MRV$ formalism introduced in \cite{blanchet2025tailasymptoticsclustersizes}, which is particularly well-suited to branching processes and Hawkes processes; 
(iii) beyond our application to Hawkes processes, Theorem~\ref{theorem: LD for levy, MRV} is also of independent interest due to the relevance of Lévy models and hidden regular variation in risk management and mathematical finance (e.g., \cite{4fc30f5f-894d-3c6c-9b15-01f8f8d74820, https://doi.org/10.1111/1467-9965.00020, risks10080148, das2023aggregatingheavytailedrandomvectors}).

The third step concerns applying the framework in Theorem~\ref{theorem: LD for levy, MRV} to Hawkes processes. 
Specifically, we show that for large deviations analysis, $\bm N(t)$ is asymptotically equivalent to a Lévy process with cluster size vectors $\bm S_j$ as increments (see \eqref{def: fixed point equation for cluster S i}).
That is, in terms of extremal behaviors, the sample path of the heavy-tailed Hawkes process remains essentially unchanged even if all arrivals in the same cluster are merged into a single jump 
(i.e., if all offspring arrive simultaneously with their immigrant ancestor).
To achieve this, we
(i)
refine the approach in \cite{10.36045/bbms/1170347811} 
to obtain tight bounds on the cluster lifetime (i.e., the time gap between an immigrant and its last offspring),
and 
(ii)
apply again the notion of asymptotic equivalence in terms of $\mathbb M$-convergence.
We outline our proof strategy in Section~\ref{subsec: appendix, proof of main result, Hawkes} and collect technical tools for $\mathbb M$-convergence theory in Section~\ref{subsec: proof, M convergence and asymptotic equivalence}.

Equipped with the tools above,
we
characterize
sample path large deviations for multivariate heavy-tailed Hawkes processes.
More precisely,
{under the presence of regular variation} in the mutual excitation rates of $\bm N(t)$ and under proper tail conditions for the decay functions $f^{\bm N}_{i \leftarrow j}(\cdot)$,
Theorem~\ref{theorem: LD for Hawkes} establishes asymptotics of the form
\begin{align}
    {\mathbf C}^{ \subInfty }_{ \bm k(A) }(A^\circ)
    \leq 
    \liminf_{n \to \infty}
    \frac{
        \P\big(\bar{\bm N}^{_{[0,\infty)}}_n \in A\big)
        }{
            \breve\lambda_{\bm k(A)}(n)
        }
    \leq 
    \limsup_{n \to \infty}
    \frac{
        \P\big(\bar{\bm N}^{_{[0,\infty)}}_n \in A\big)
        }{
            \breve\lambda_{\bm k(A)}(n)
        }
    \leq 
    {\mathbf C}^{ \subInfty }_{ \bm k(A) }(A^-)
    \label{intro, LD for Hawkes}
\end{align}
for general sets $A \subseteq \D[0,\infty)$.
Here, 
$A^\circ$ and $A^-$ are the interior and closure of $A$  w.r.t.\ the product $M_1$ topology of the càdlàg  space $\D[0,\infty)$;
$\bar{\bm N}^{_{[0,\infty)}}_n \delequal \{ \bm N(nt)/n:\ t \geq 0  \}$
is the scaled sample path of $\bm N(t)$ embedded in $\D[0,\infty)$,
the limiting measures $\big({\mathbf C}^{\subInfty}_{ \bm k}(\cdot)\big)_{ \bm k \in \Z^d_+  }$ are supported on $\D[0,\infty)$,
and the rates of decay $\big(\breve \lambda_{\bm k}(n)\big)_{\bm k \in \Z^d_+}$ are regularly varying functions with indices determined by the $\alpha(\bm j)$'s---the power-law indices for tail asymptotics of Hawkes cluster sizes.
In particular, the rate function $\bm k(A) \in \Z^d_+$ solves a discrete optimization problem identifying the \emph{most likely configuration of large clusters} that can drive the Hawkes process $\bar{\bm N}^{_{[0,\infty)}}_n$ into the rare-event set $A$.
Intuitively speaking, 
the nominal behavior of the (scaled) Hawkes process $\bar{\bm N}^{_{[0,\infty)}}_n$ is a linear path with slope $\bm\mu_{\bm N}$ (its expected increments under stationarity),
and recall that, by the tail asymptotics of Hawkes process clusters, the probability of observing a large cluster with size vector lying in the cone $\R^d(\bm j)$ is roughly of order $n^{-\alpha(\bm j)}$.
The optimization problem behind $\bm k(A)$ (see Remark~\ref{remark: interpretation, LD for Hawkes}) then identifies, among all combinations of large clusters that can push the $\bm\mu_{\bm N}$-linear path into the set $A$, 
the configuration that is most likely to occur.
We provide the rigorous statement of Theorem~\ref{theorem: LD for Hawkes} and the precise definitions of the notions involved in Section~\ref{subsec: LD for Hawkes process}.
Here, we highlight:
(i)
the limiting measures ${\mathbf C}^{ \subInfty }_{ \bm k }(\cdot)$ 
can be efficiently computed via Monte Carlo simulation (see Remark~\ref{remark: evaluation of limiting measures, LD for Hawkes});
(ii) 
in the univariate case, Theorem~\ref{theorem: LD for Hawkes} can be uplifted to the $M_1$ topology of $\D[0,T]$
(see Section~\ref{subsec: univariate results, LD for Hawkes, M1 of D0T} of the Appendix);
(iii)
in the multivariate case,
it is generally not possible to strengthen
the characterizations in \eqref{intro, LD for Hawkes} w.r.t.\ the product $M_1$ topology of $\D[0,\infty)$ to those on $\D[0,T]$ or other stronger Skorokhod non-uniform topologies (see Remark~\ref{remark: conditions in LD for Hawkes}).

Regarding the contributions of this work, we emphasize:
(i) The study of large deviations for heavy-tailed Hawkes processes aligns with the importance of power-law heavy tails in fields like epidemiology (\cite{doi:10.1073/pnas.2209234119}), queueing systems (\cite{Asmussen_Foss_2018,ernst2018stability}), and mathematical finance (\cite{Bacry02082016,hardiman2013critical,10.1214/15-AAP1164,horst2024convergenceheavytailedhawkesprocesses});
(ii) This paper addresses major gaps in prior work by developing sample path LDPs for heavy-tailed Hawkes processes that go well beyond the single big jump regime.
For instance, in the context of (unmarked) Hawkes processes,
Proposition~3 of \cite{karim2021exact} and Proposition~8.2 of \cite{baeriswyl2023tail}
(see also \cite{guo2025precise} for closely related results regarding branching processes with immigration)
are comparable to a special case of \eqref{intro, LD for Hawkes} with $A = \{ \xi \in \D[0,\infty):\ \bm c^\top \xi(t) > 1  \}$ with some $\bm c \in \R^d_+$ and $t >0$.
That is, they address only a limited class of rare events, which are characterized by an extreme value of the process at a specific time point rather than by the extremal behavior of the entire sample path, and are driven by a single big jump in the Hawkes process. 
(We also note that Proposition~8.1 of  \cite{baeriswyl2023tail} studies large deviations of the maximum functional regarding the clusters of marked Hawkes processes, and is another manifestation of the principle of a single big jump.)
In contrast, our Theorem~\ref{theorem: LD for Hawkes} covers a much broader class of sample-path-level rare events and reveals the most likely configuration of (multiple) big clusters governing such rare events.
In many applications across finance, machine learning, and operations research, impactful rare events are determined by the entire sample path of the stochastic dynamics rather than just the running maximum or endpoint value,
and are driven by multiple large jumps in the system (see, e.g.,\ \cite{Albrecher_Chen_Vatamidou_Zwart_2020,tankov2003financial,doi:10.1287/moor.1120.0539,wang2022eliminating}).
Our results thus lay the foundation of
precise theoretical insights and efficient rare event simulation for practical systems under clustering and mutually exciting risks.

This paper is structured as follows.
Section~\ref{sec: definitions and notations} reviews the tail asymptotics of Hawkes process clusters. 
Section~\ref{sec: sample path large deviation} presents the main results of this paper.
In the Appendix,
Section~\ref{subsec: proof, M convergence and asymptotic equivalence}
collects useful technical tools for the $\mathbb M$-convergence theory,
Section~\ref{sec: appendix, tail asymptotics of Hawkes process clusters} collects additional details for tail asymptotics of Hawkes process clusters,
Section~\ref{sec: counter examples} provides the details of the counterexamples in Remarks~\ref{remark: conditions in LD for Hawkes} and \ref{remark, tail condition on decay functions},
Sections~\ref{subsubsec: proof of LD with MRV} and \ref{subsec: proof, theorem: LD for Hawkes} contain the proofs for Section~\ref{sec: sample path large deviation},
Section~\ref{sec: appendix, LD for Levy, General Case} provides sample path large deviations for L\'evy Processes with $\MRV$ increments under relaxed conditions, and  
Section~\ref{sec: appendix, theorem tree} provides theorem trees to aid readability of the proofs.

\section{Preliminaries}
\label{sec: definitions and notations}

This section reviews key definitions and results for our subsequent analysis and is structured as follows.
Section~\ref{subsec: MRV} reviews $\MRV$, a notion of multivariate hidden regular variation introduced in \cite{blanchet2025tailasymptoticsclustersizes}.
Section~\ref{subsec: tail of cluster S, statement of results} adapts Theorem~3.2 of \cite{blanchet2025tailasymptoticsclustersizes}
to our setting and characterizes the hidden regular variation in Hawkes process clusters in terms of $\MRV$.
These results are pivotal for the subsequent large-deviation analysis of multivariate Hawkes processes in this paper.

We begin by setting frequently used notations.
Let $\notationdef{notation-R-d-+}{\R^d_+} = [0,\infty)^d$,
 $\mathbb Z$ be the set of integers,
 $\notationdef{notation-non-negative-numbers-and-zero}{\mathbb{Z}_+} = \{0,1,2,\cdots\}$ 
be the set of non-negative integers,
and
 $\notationdef{notation-non-negative-numbers}{\mathbb{N}} = \{1,2,\cdots\}$ be the set of positive integers.
 Given a set $X$ and a countable set $A$, we adopt notations  $\bm x \in X^A$
for vectors of the form $\bm x = (x_i)_{i \in A}$ that are of length $|A|$, with each coordinate $x_i \in X$ indexed by elements of $A$.
 For any $n \in \mathbb Z_+$,
let $\notationdef{set-for-integers-below-n}{[n]} \delequal \{1,2,\cdots,n\}$,
with the convention that $[0] = \emptyset$.
For any $m \in \mathbb N$, 
let $\notationdef{notation-power-set-for-[d]-without-empty-set}{\powersetTilde{m}}$ be the collection of all \emph{non-empty} subsets of $[m]$, i.e., the power set of $\{1,2,\ldots,m\}$ excluding $\emptyset$.
Given some metric space $(\mathbb S,\bm d)$, a set $E \subseteq \mathbb S$, and $r > 0$,
let the closed set
$\notationdef{notation-epsilon-enlargement-of-set-E}{E^r} \delequal 
\{ y \in \mathbb{S}:\ \bm{d}(E,y)\leq r\}$ be the $r$-enlargement of the set $E$,
and the open set
$
\notationdef{notation-epsilon-shrinkage-of-set-E}{E_{r}} \delequal
((E^c)^r)^\complement
=
\{
y \in \mathbb S:\ \bm d(E^c,y) > r
\}
$
be the $r$-shrinkage
of $E$;
besides,
let
$\notationdef{notation-interior-of-set-E}{E^\circ} \delequal \bigcup_{r > 0}E_r$ and $\notationdef{notation-closure-of-set-E}{E^-}\delequal \bigcap_{r > 0}E^r$ be the interior and closure of $E$, respectively.
Throughout this paper,
we consider the $L_1$ norm $\notationdef{notation-L1-norm}{\norm{\bm x}} = \sum_{i \in [d]}|x_i|$ over Euclidean spaces.
For any random element $X$ and Borel measurable set $A$,
we use $\notationdef{notation-law-of-X}{\mathscr{L}(X)}$ to denote the law of $X$, and $\notationdef{notation-law-of-X-cond-on-A}{\mathscr{L}(X|A)}$ for the conditional law of $X$ given the event $A$.
Given real numbers $x$ and $y$,
let $x \wedge y \delequal \min\{x,y\}$, $x\vee y = \max\{x,y\}$ be the minimum and maximum operators, respectively,
and let
$\notationdef{floor-operator}{\floor{x}}\delequal \max\{n \in \mathbb{Z}:\ n \leq x\}$,
$\notationdef{ceil-operator}{\ceil{x}} \delequal \min\{n \in \Z:\ n \geq x\}$
be the floor and ceiling operators, respectively.
Given sequences of non-negative real numbers $(x_n)_{n \geq 1}$ and $(y_n)_{n \geq 1}$, 
we say that $x_n = \bo(y_n)$ (as $n \to \infty$) if there exists some $C \in [0,\infty)$ such that $x_n \leq C y_n\ \forall n\geq 1$, and that $x_n = \lo(y_n)$ if $\lim_{n \rightarrow \infty} x_n/y_n = 0$.

\subsection{Multivariate Hidden Regular Variation}
\label{subsec: MRV}

Recall that a measurable function $\phi:(0,\infty) \to (0,\infty)$ is regularly varying as $x \rightarrow\infty$ with index $\beta \in \R$, denoted as $\phi(x) \in \notationdef{notation-regular-variation}{\RV_\beta}(x)$ as $x \to \infty$, if $\lim_{x \rightarrow \infty}\phi(tx)/\phi(x) = t^\beta$ holds for any $t>0$. 
See, e.g., \cite{bingham1989regular,resnick2007heavy,foss2011introduction} for standard treatments of regular variation.
In this subsection, we review the $\MRV$ formalism, a notion of multivariate hidden regular variation proposed in \cite{blanchet2025tailasymptoticsclustersizes}.
As shown in Section~\ref{subsec: tail of cluster S, statement of results}, $\MRV$
provides the appropriate framework to describe the tail asymptotics of Hawkes process clusters.
In particular, $\MRV$
is characterized through the following key elements.
\begin{itemize}
    \item 
        \emph{Basis}: $\bar{\textbf S} = \{\bar{\bm s}_j \in [0,\infty)^d: j \in [k]\}$ is a collection of $k$ linearly independent vectors;

    \item 
        \emph{Tail indices}: 
        $
        \bm \alpha = \big\{  \alpha(\bm j) \in [0,\infty):\ \bm j \subseteq [k] \big\}
        $
       are strictly monotone w.r.t.\ $\bm j$: that is,
       \begin{align}
           \alpha(\bm j) < \alpha(\bm j^\prime),\qquad
           \forall \bm j \subsetneq \bm j^\prime \subseteq [k],
           \label{def: monoton sequence of tail indices in MHRV}
       \end{align}
       with the convention that $\alpha(\emptyset) = 0$;

    \item 
        \emph{Rate functions}: $\lambda_{\bm j}(n) \in \RV_{-\alpha(\bm j)}(n)$ for each $\bm j \in \powersetTilde{k}$;

    \item 
        \emph{Limiting measures}: $\big(\mathbf C_{\bm j}(\cdot)\big)_{\bm j \in \powersetTilde{k}}$ is a collection of Borel measures over $\R^d_+$.
\end{itemize}
The $\MRV$ formalism describes how the power-law tail behavior of a Borel measure $\nu$ varies across different directions of the Euclidean space.
More precisely, 
let 
$\notationdef{notation-R-d-+-unit-sphere}{\mathfrak N^d_+} \delequal \{\bm x \in \R_+^d:\ \norm{\bm x} = 1\}$ be the unit sphere under the $L_1$ norm, restricted to the positive quadrant,
and let (for each $\bm j \in \powersetTilde{d}$)
\begin{align}
     \notationdef{notation-convex-cone-R-d-i-cluster-size}{\R^d(\bm j;\bar{\textbf S})}
    & \delequal 
    \Bigg\{ \sum_{i \in \bm j}  w_i\bm{\bar s}_i:\ w_i \geq 0\ \forall i \in \bm j \Bigg\};
    \label{def: cone R d index i}
    \\
    \notationdef{notation-enlarged-convex-cone-bar-R-d-i-epsilon}{ \bar \R^d(\bm j, \epsilon;\bar{\textbf S}) }
    & \delequal
    \bigg\{
        w\bm s:\ w \geq 0,\ \bm s \in \mathfrak N^d_+,\ 
            \inf_{  \bm x\in \R^d(\bm j;\bar{\textbf S}) \cap \mathfrak N^d_+   }\norm{ \bm s - \bm x  } \leq \epsilon
    \bigg\},\quad \epsilon > 0;
    \label{def: enlarged cone R d index i epsilon}
    \\ 
     \notationdef{notation-cone-R-d-leq-i-basis-S-index-alpha}{\bar\R^{d}_\leqslant(\bm j,\epsilon; \bar{\textbf S},\bm \alpha)}
            & \delequal
            \bigcup_{
                \bm j^\prime \subseteq [k]:\ 
                \bm j^\prime \neq \bm j,\ \alpha(\bm j^\prime) \leq \alpha(\bm j)
            } \bar\R^d(\bm j^\prime,\epsilon;\bar{\textbf S}),
            \quad \epsilon > 0.
            \label{def: cone R d i basis S index alpha}
\end{align}
When there is no ambiguity about the choice of $\bar{\textbf S}$ and $\bm \alpha$,
we simplify the notations by writing
$
\notationdef{notation-convex-cone-R-d-i-cluster-size-short}{\R^d(\bm j)} \delequal {\R^d(\bm j;\bar{\textbf S})},
$
$
\notationdef{notation-enlarged-convex-cone-bar-R-d-i-epsilon-short}{ \bar \R^d(\bm j, \epsilon) } \delequal \bar \R^d(\bm j, \epsilon;\bar{\textbf S}),
$
and
$
\notationdef{notation-cone-R-d-leq-i-basis-S-index-alpha-short}{\bar\R^{d}_\leqslant(\bm j,\epsilon)} \delequal
{\bar\R^{d}_\leqslant(\bm j,\epsilon; \bar{\textbf S},\bm \alpha)}.
$
We also adopt the convention that $\bar\R^d(\emptyset,\epsilon) = \{\bm 0\}$.
Note that $\R^d(\bm j)$ is the convex cone generated by $\{\bar{\bm s}_i:\ i \in\bm j \}$,
and
$\bar\R^d(\bm j,\epsilon)$ is an enlarged version of $\R^d(\bm j)$ by applying an $\epsilon$-perturbation to the angle of each element under the polar transform.
As shown in Definition~\ref{def: MRV}, the $\MRV$ formalism characterizes the heavy tails in a Borel measure $\nu(\cdot)$ over each cone $\R^d(\bm j)$ by establishing that, under the $\lambda_{\bm j}(n)$-scaling,
the tail behavior of $\nu\big(\ \cdot\ \cap \R^d(\bm j) \big)$ converges to the limiting measure $\mathbf C_{\bm j}$.
In other words, the measure $\nu$ exhibits different degrees of hidden regular variation over each cone $\R^d(\bm j)$,
characterized by the power-law tail index $\alpha(\bm j)$, the rate function $\lambda_{\bm j}(\cdot)$, and the limiting measure $\mathbf C_{\bm j}(\cdot)$. 
Here,
we say that $A \subseteq \R^d_+$ is bounded away from $B \subseteq \R^d_+$ if $\inf_{\bm x \in A,\ \bm y \in B}\norm{\bm x - \bm y} > 0$.

\begin{definition}[$\MRV$]
\label{def: MRV}
Let $\nu(\cdot)$ be a Borel measure on $\R^d_+$ and $\nu_n(\cdot) \delequal \nu( n\ \cdot\ )$ (i.e., $\nu_n(A) = \nu(nA) = \nu\big\{ n\bm x:\ \bm x \in A \big\}$).
The measure
$\nu(\cdot)$ is said to be \textbf{multivariate regularly varying}
with basis $\bar{\textbf S} = \{ \bar s_j:\ j \in [k] \}$,
tail indices $\bm \alpha = \big\{  \alpha(\bm j) \in [0,\infty):\ \bm j \subseteq [k] \big\}$,
rate functions $\lambda_{\bm j}(\cdot)$,
and limiting measures $\mathbf C_{\bm j}(\cdot)$,
denoted as
$
\nu \in \notationdef{notation-MRV}{\MRV\Big(\bar{\textbf S},\bm \alpha, (\lambda_{\bm j})_{ \bm j \in \powersetTilde{k}}, (\mathbf C_{\bm j})_{  \bm j \in \powersetTilde{k} }  \Big)},
$
if the asymptotics
    \begin{align}
            \mathbf C_{\bm j}(A^\circ)
            & \leq 
            \liminf_{n \to \infty}\frac{
                \nu_n(A)
            }{
                \lambda_{\bm j}(n)
            }
            \leq 
            \limsup_{n \to \infty}\frac{
                \nu_n(A)
            }{
                \lambda_{\bm j}(n)
            }
            \leq 
            \mathbf C_{\bm j}(A^-) < \infty
            \label{cond, finite index, def: MRV}
        \end{align}
    hold for any $ \bm j \in \powersetTilde{k}$ and any Borel set $A \subseteq \R^d_+$ that is
        bounded away from $\bar \R^{d}_\leqslant(\bm j,\epsilon)$ under some $\epsilon > 0$.
Furthermore, suppose that for any Borel set $A \subseteq \R^d_+$
        bounded away from $\bar \R^{d}([k],\epsilon)$ under some $\epsilon > 0$,
we also have
\begin{align}
    \nu_n(A) = \lo( n^{-\gamma})\ \text{ as }n \to \infty,
    \qquad \forall \gamma > 0.
    \label{cond, outside of the full cone, finite index, def: MRV}
\end{align}
Then we write 
$
\nu \in \notationdef{notation-MRV}{\MRV^*\Big(\bar{\textbf S},\bm \alpha, (\lambda_{\bm j})_{ \bm j \in \powersetTilde{k}}, (\mathbf C_{\bm j})_{  \bm j \in \powersetTilde{k} }  \Big)}.
$



\end{definition}

We conclude this subsection by briefly noting that: 
(i) Conditions \eqref{cond, finite index, def: MRV} and \eqref{cond, outside of the full cone, finite index, def: MRV} are equivalent to characterizations of heavy tails in terms of the $\mathbb M(\mathbb S\setminus\C)$-convergence (\cite{10.1214/14-PS231}) of polar coordinates;
and
(ii) $\MRV$ provides an appropriate framework for describing heavy tails in contexts such as branching processes and Hawkes processes; in particular, the tail asymptotics stated in Section~\ref{subsec: tail of cluster S, statement of results} would fail under other formalisms of multivariate hidden regular variation (see Remark~7 of \cite{blanchet2025tailasymptoticsclustersizes}).

\subsection{Tail Asymptotics of Hawkes Process Clusters}
\label{subsec: tail of cluster S, statement of results}

The cluster representation of Hawkes processes 
(see, e.g., \cite{209288c5-6a29-3263-8490-c192a9031603,20.500.11850/151886,daley2003introduction}; 
see also Definition~\ref{def: offspring cluster process} in the Appendix)
reveals that the size of the cluster induced by 
a type-$j$ immigrant in the Hawkes process $\bm N(t)$ admits the law of the random vector $\bm S_j$ that solves the distributional fixed-point equation in \eqref{def: fixed point equation for cluster S i}.
The canonical representation of $\bm S_j$ is the total progeny (across the $d$ dimensions) of a multi-type branching process,
where the offspring distributions admit the law
(by Poisson$(X)$ for a non-negative variable $X$, we mean the law of
$\P\big(\text{Poisson}(X) > y\big) = \int_0^\infty \P\big(\text{Poisson}(x) > y\big)\P(X \in dx)$)
\begin{align}
    \notationdef{notation-vector-B-i-number-of-offspring}{(B_{i \leftarrow j})_{i \in [d]}} \distequal \Big(\text{Poisson}(\tilde B_{1 \leftarrow j}\mu^{\bm N}_{1\leftarrow j}),\ldots,\text{Poisson}(\tilde B_{d \leftarrow j}\mu^{\bm N}_{d \leftarrow j})\Big),
    \label{def: vector B i, offspring of type i individual}
\end{align}
where $\big(\tilde B_{i \leftarrow j})_{i \in [d]}$ is introduced in Definition~\ref{def: hawkes process, conditional intensity},
and
\begin{align}
    \notationdef{notation-mu-i-j-in-conditional-intensity}{\mu^{\bm N}_{i \leftarrow j}} \delequal \norm{f^{\bm N}_{i \leftarrow j} }_1=\int_0^\infty f^{\bm N}_{i \leftarrow j}(t)dt < \infty,
    \qquad\forall i,j \in [d],
    \label{def mu i j for fertility function}
\end{align}
with $f^{\bm N}_{i \leftarrow j}(\cdot)$ being the decay functions in Definition~\ref{def: hawkes process, conditional intensity}.

In this subsection,
we specialize the assumptions and results in
\cite{blanchet2025tailasymptoticsclustersizes} to our context of Hawkes processes clusters,
which are pivotal to 
our large-deviation analysis for multivariate heavy-tailed Hawkes processes.
Specifically,
let
\begin{align}
    \notationdef{notation-b-i-j-cluster-size}{\bar b_{i\leftarrow j}} \delequal \E B_{i\leftarrow j}
    =
    \E\big[\tilde B_{i \leftarrow j}\big] \cdot \mu^{\bm N}_{i \leftarrow j},
    \label{def: bar b i j, mean of B i j}
\end{align}
be the expectation of offspring distributions,
where the equality follows from \eqref{def: vector B i, offspring of type i individual}.
Under Assumption~\ref{assumption: subcriticality}, Proposition 1 of \cite{Asmussen_Foss_2018} establishes the existence and uniqueness of the solutions $\bm S_j$ to \eqref{def: fixed point equation for cluster S i} and that $\E \norm{\bm S_{j}} < \infty$ for all $j \in [d]$.

\begin{assumption}[Sub-Criticality]
\label{assumption: subcriticality}
     The spectral radius of the mean offspring matrix $\notationdef{notation-bar-B-matrix-cluster-size}{\bar{\textbf B}} = (\bar b_{i\leftarrow j})_{j,i \in [d]}$
    \xwa{have $\bar{\bm B}$ transposed?}%
    is strictly less than $1$.
\end{assumption}

Assumption~\ref{assumption: heavy tails in B i j} specifies the regularly varying heavy tails in the offspring distribution.
Note that, under the law of the offspring distributions stated in \eqref{def: vector B i, offspring of type i individual},
$B_{j \leftarrow i}$ and $\tilde B_{j \leftarrow i}$ share the same regular variation index $-\alpha_{j \leftarrow i}$.

\begin{assumption}[Heavy Tails in Mutual Excitation Rates]
\label{assumption: heavy tails in B i j}
For any $(i,j) \in [d]^2$,
there exists $\notationdef{notation-alpha-i-j-cluster-size}{\alpha_{j\leftarrow i}} \in (1,\infty)$ such that
\begin{align*}
    \P(\tilde B_{j\leftarrow i} > x) \in \RV_{-{\alpha_{j\leftarrow i}} }(x),\qquad \text{as }x \to \infty.
\end{align*}
\end{assumption}

We also impose the following two regularity conditions.

\begin{assumption}[Full Connectivity]
\label{assumption: regularity condition 1, cluster size, July 2024}
For any $(i,j) \in [d]^2$, $\E S_{j \leftarrow i} > 0$.
\end{assumption}

\begin{assumption}[Distinct Tail Indices]
\label{assumption: regularity condition 2, cluster size, July 2024}
In Assumption~\ref{assumption: heavy tails in B i j},
$\alpha_{j\leftarrow i} \neq \alpha_{j^\prime \leftarrow i^\prime }$ for any $(i,j),(i^\prime,j^\prime) \in [d]^2$ with $(i,j) \neq (i^\prime,j^\prime)$.
\end{assumption}

Theorem~\ref{theorem: main result, cluster size} allows us to characterize
the tail asymptotics of the cluster size vector $\bm S_j$ in terms of $\MRV$ introduced in Definition~\ref{def: MRV}.
To state the results,
we specify the basis, tail indices, rate functions, and limiting measures.
In particular, let the \emph{basis}
$\bar{\textbf S} = \{ \bar{\bm s}_i:\ i \in [d] \}$ be
\begin{align}
    \notationdef{notation-bar-s-i-j-cluster-size}{\bar s_{j \leftarrow i}} \delequal \E S_{j \leftarrow i},
    \qquad
    \notationdef{notation-vector-bar-s-i-cluster-size}{\bm{\bar s}_i} \delequal \E \bm S_i = (\bar s_{1 \leftarrow i},\ \bar s_{2 \leftarrow i},\ldots,\bar s_{d \leftarrow i})^\top.
    \label{def: bar s i, ray, expectation of cluster size}
\end{align}
Assumption~\ref{assumption: subcriticality} ensures that the $\bar{\bm s}_i$'s are linearly independent.
Next, let
\begin{align}
    \notationdef{notation-alpha-*-j-cluster-size}{\alpha^*(j)} \delequal 
        \min_{l \in [d]}\alpha_{j\leftarrow l},
    \qquad 
    \notationdef{notation-l-*-j-cluster-size}{l^*(j)}  \delequal 
        \arg\min_{l \in [d]}\alpha_{j\leftarrow l}.
    \label{def: cluster size, alpha * l * j}
\end{align}
By Assumptions~\ref{assumption: heavy tails in B i j} and \ref{assumption: regularity condition 2, cluster size, July 2024},
for each $j \in [d]$
the argument minimum  $l^*(j)$ uniquely exists  and $\alpha^*(j) > 1$.
Recall that $\powersetTilde{d}$
is the collection of all {non-empty} subsets of $[d]$.
The \emph{tail indices} are defined by
\begin{align}
    \notationdef{notation-cost-function-bm-j-cluster-size}{\alpha(\bm j)} \delequal 1 + \sum_{i \in \bm j} \big(\alpha^*(i) - 1\big),
    \qquad 
        \forall \bm j \in \powersetTilde{d}.
    \label{def: cost function, cone, cluster}
\end{align}
As in Section~\ref{subsec: MRV},
we adopt the convention $\alpha(\emptyset) = 0$.
The \emph{rate functions}
are defined by
\begin{align}
    \notationdef{notation-lambda-j-n-cluster-size}{\lambda_{\bm j}(n)} \delequal
    n^{-1}\prod_{ i \in \bm j  }n\P(B_{i \leftarrow l^*(i)} > n),
    \qquad
    \forall n \geq 1,\ \bm j \in \powersetTilde{d}.
    \label{def: rate function lambda j n, cluster size}
\end{align}
Note that $\lambda_{\bm j}(n) \in \RV_{ -\alpha(\bm j)  }(n)$.
Lastly, for each $i \in [d]$ and $\bm j \in \powersetTilde{d}$,
the limiting measure $\mathbf C^{\bm j}_i(\cdot)$ is supported on the cone $\R^d(\bm j)$ defined in \eqref{def: cone R d index i},
and takes the form
\begin{align}
    \mathbf C^{\bm j}_i(\cdot)
    & \delequal
    \int_{ w_j \geq 0\ \forall j \in \bm j  }
    \mathbbm{I}
    \Bigg\{
       \sum_{ j \in \bm j  }w_{j}\bar{\bm s}_j \in \ \cdot\ 
    \Bigg\}
    \cdot 
    g^{\bm j}_i(\bm w)
    \bigtimes_{ j \in \bm j }\frac{dw_j}{ (w_j)^{  \alpha^*(j) + 1  }   },
    \label{def: measure C i I, cluster, main paper}
\end{align}
where we write $\bm w = (w_j)_{j \in \bm j}$.
The explicit forms of the functions $g^{\bm j}_i(\bm w)$ and the limiting measures $\mathbf C^{\bm j}_i(\cdot)$ are not needed for stating the main results in Section~\ref{sec: sample path large deviation}, so we defer the details to Section~\ref{sec: appendix, tail asymptotics of Hawkes process clusters} in the Appendix. We now state the tail asymptotics of the Hawkes cluster sizes $\bm S_i$ in terms of $\MRV$.

\begin{theorem}[Theorem 3.2 of \cite{blanchet2025tailasymptoticsclustersizes}]
\label{theorem: main result, cluster size}
\linksinthm{theorem: main result, cluster size}
Under Assumptions~\ref{assumption: subcriticality}--\ref{assumption: regularity condition 2, cluster size, July 2024},
it holds for any $i \in [d]$ that
\begin{align}
    \P(\bm S_i \in \ \cdot\ )
    \in \MRV^*
    \Bigg(
        (\bar{\bm s}_j)_{j \in [d]},\ 
        \big(\alpha(\bm j)\big)_{ \bm j \subseteq [d] },\ 
        (\lambda_{\bm j})_{\bm j \in \powersetTilde{d}},\ 
        \big(  \mathbf C^{\bm j}_i \big)_{\bm j \in \powersetTilde{d}}
    \Bigg).
    \nonumber
\end{align}
That is,
given $i \in [d]$ and  $\bm j \subseteq [d]$ with $\bm j \neq \emptyset$, if a Borel measurable set $A \subseteq \R^d_+$ is bounded away from  $\bar{\R}^d_\leqslant(\bm j,\epsilon)$ under some $\epsilon > 0$,
then
 \begin{equation}\label{claim, theorem: main result, cluster size}
     \begin{aligned}
        \mathbf C^{\bm j}_i(A^\circ) 
    & \leq 
        \liminf_{n \to \infty}
        \frac{
        \P(n^{-1}\bm S_i \in A)
        }{
            \lambda_{\bm j}(n)
        }
    \leq 
        \limsup_{n \to \infty}
        \frac{
        \P(n^{-1}\bm S_i \in A)
        }{
            \lambda_{\bm j}(n)
        }
    \leq 
    \mathbf C^{\bm j}_i(A^-) < \infty.
     \end{aligned}
 \end{equation}
Here, $\bar{\R}^d_\leqslant(\bm j,\epsilon)$ is defined in \eqref{def: enlarged cone R d index i epsilon},
the rate functions $\lambda_{\bm j}(\cdot)$ are defined in \eqref{def: rate function lambda j n, cluster size},
and the measures $\mathbf C_i^{\bm j}(\cdot)$ are defined in \eqref{def: measure C i I, cluster, main paper}.
Furthermore, for any Borel measurable set $A \subseteq \R^d_+$  that is bounded away from  $\bar{\R}^d(\{1,2,\ldots,d\},\epsilon)$ under some  $\epsilon > 0$,
\begin{align}
    \lim_{n \to \infty}n^{\gamma}\cdot\P(n^{-1}\bm S_i \in A) = 0,\qquad\forall \gamma > 0.
    \label{claim, 2, theorem: main result, cluster size}
\end{align}
\end{theorem}

To conclude, we note that the rare events $\{n^{-1}\bm S_i \in A\}$ are driven by multiple big jumps (i.e., multiple nodes giving birth to a disproportionately large number of children) that exhibit sophisticated spatiotemporal structures in the underlying branching processes.
In particular, the contributions of these big jumps align with the vectors $(\bar{\bm s}_l)_{l \in [d]}$,
and, given a sufficiently general set $A \subseteq \R^d_+$,
the exact asymptotics of $\P(n^{-1}\bm S_i \in A)$ are dictated by the most likely combination of big jumps that can push $n^{-1}\bm S_i$ into the set $A$, which corresponds to the index set $\bm j \subseteq [d]$ leading to non-trivial bounds in \eqref{claim, theorem: main result, cluster size}. 
For further details, see also Remark~3 of \cite{blanchet2025tailasymptoticsclustersizes}.

\section{Sample Path Large Deviations}
\label{sec: sample path large deviation}

This section presents the main results of this paper and is structured as follows.
Section~\ref{subsec: LD for levy process, MRV} establishes sample path LDPs for L\'evy processes with $\MRV$ increments.
Building upon this framework, Section~\ref{subsec: LD for Hawkes process} then 
establishes sample path LDPs for the multivariate Hawkes process 
$
\bm N(t) = (N_1(t),N_2(t),\ldots,N_d(t))^\top
$
under the presence of power-law heavy tails in the mutual excitation rates.
Section~\ref{subsec: appendix, proof of main result, Hawkes} outlines the proof for sample path LDPs of heavy-tailed Hawkes processes.
We defer the underlying proofs to the Appendix.

\subsection{L\'evy Processes with $\MRV$ Increments}
\label{subsec: LD for levy process, MRV}

We begin by briefly reviewing the law of a $d$-dimensional L\'evy process $\notationdef{notation-levy-process}{\bm{L}} = \{ \bm L(t):\ t \geq 0\}$,
which
is fully characterized by its generating triplet $(\bm c_{\bm L},\bm \Sigma_{\bm L},\nu)$. Here,
\begin{itemize}
    \item 
        $\bm c_{\bm L} \in \mathbb{R}^d$ is the constant drift of the process,

    \item 
        the positive semi-definite matrix $\bm \Sigma_{\bm L} \in \R^{d \times d}$ represents the covariance of the Brownian motion component in $\bm L(t)$, 

    \item 
        and the L\'evy measure $\notationdef{notaiton-levy-measure-nu}{\nu}$ is a Borel measure supported on $\R^d\setminus\{\bm 0\}$ that characterizes the intensity of jumps in $\bm L(t)$, and satisfies $\int (\norm{\bm x}^2\wedge 1) \nu(d\bm x) < \infty$.
\end{itemize}
More precisely, the L\'evy process $\bm L(t)$ admits the L\'evy–Itô decomposition
\begin{align}
    \bm L(t) \distequal \bm c_{\bm L}t + \bm \Sigma_{\bm L}^{1/2} 
    \bm B(t)+ 
    \int_{\norm{\bm x} \leq 1}\bm x\big[ \text{PRM}_{\nu}([0,t]\times d\bm x) - t\nu(d\bm x) \big] 
    + 
    \int_{\norm{\bm x} > 1}\bm x\text{PRM}_{\nu}( [0,t]\times d\bm x ),
    \label{prelim: levy ito decomp}
\end{align}
where $\bm B$ is a standard Brownian motion in $\R^d$, 
and $\text{PRM}_{\nu}$ is a Poisson random measure with intensity measure $\mathcal L_{(0,\infty)}\times \nu$ and is independent of $\bm B$.
Here, $\notationdef{notation-lebesuge-measure-on-R}{\mathcal L_I}$ is the Lebesgue measure restricted on the interval $I \subseteq \R$.
We refer the reader to, e.g., \cite{sato1999levy} for a standard treatment of L\'evy processes.

The goal of this subsection is to establish Theorem~\ref{theorem: LD for levy, MRV} and characterize the sample path LDPs for L\'evy processes exhibiting multivariate hidden regular variation in the increments.
Beyond its application for sample path LDPs of Hawkes processes that will be explored in Section~\ref{subsec: LD for Hawkes process},
 the theoretical framework developed in Theorem~\ref{theorem: LD for levy, MRV} is also of independent interest due to the broad relevance of L\'evy processes and hidden regular variation.
Specifically, we 
focus on the case where the L\'evy measure $\nu$ is supported on $\R^d_+$,
and 
characterize the hidden regular variation in $\nu(\cdot)$ via the notion of $\MRV$ (see Definition~\ref{def: MRV}) by imposing the following assumption.

\begin{assumption}[$\MRV^*$ increments]
\label{assumption: ARMV in levy process}
The L\'evy measure $\nu$ of $\bm L(t)$ is supported on $\R^d_+ = [0,\infty)^d$ and satisfies 
$
\nu \in
    \MRV^*
    \Big(
(\bar{\bm s}_i)_{i \in [d]},
\big(\alpha(\bm j)\big)_{ \bm j \subseteq [d] },
(\lambda_{\bm j})_{ \bm j \in \powersetTilde{d} }, (\mathbf C_{\bm j})_{ \bm j \in \powersetTilde{d} }
\Big)
$
under
\begin{itemize}
    \item a collection of vectors $(\bar{\bm s}_i)_{i \in [d]}$ in $\R^d_+$ that are linearly independent,
    \item a collection of real numbers $(\alpha(\bm j))_{\bm j \subseteq [d]}$ that are strictly monotone w.r.t.\ $\bm j$ (see \eqref{def: monoton sequence of tail indices in MHRV}) with $\alpha(\emptyset) = 0$,
    $\alpha(\bm j) >1\ \forall \bm j \in \powersetTilde{d}$, and $\alpha(\bm j) \neq \alpha(\bm j^\prime)\ \forall \bm j,\bm j^\prime \in \powersetTilde{d}$ with $\bm j\neq \bm j^\prime$,
    \item regularly varying functions $\lambda_{\bm j}(n) \in \RV_{ -\alpha(\bm j) }(n)$ (for each $\bm j \in \powersetTilde{d}$),
    \item Borel measures $(\mathbf C_{\bm j})_{\bm j \subseteq [d]}$ where, for each $\bm j \in \powersetTilde{d}$,
    the measure $\mathbf C_{\bm j}$ is
    supported on $\R^d(\bm j)$,
        such that 
        $
        \mathbf C_{\bm j}(A) < \infty
        $
        holds for any $\epsilon > 0$ and Borel set $A \subseteq \R^d_+$ bounded away from $\bar\R^d_\leqslant(\bm j,\epsilon)$ (see \eqref{def: cone R d i basis S index alpha}).
\end{itemize}
Furthermore,
\begin{align}
    n\lambda_{\bm j}(n) = \prod_{ i \in \bm j }n\lambda_{ \{i\} }(n),
    \qquad
    \forall \bm j \in \powersetTilde{d}.
    \label{cond, additivity of rate functions for MRV, Ld for Levy}
\end{align}
\end{assumption}

We stress that both Assumption~\ref{assumption: ARMV in levy process} and Theorem~\ref{theorem: LD for levy, MRV} are \emph{tailored for the application to heavy-tailed Hawkes processes}.
For instance, the condition~\eqref{cond, additivity of rate functions for MRV, Ld for Levy} on $\lambda_{\bm j}(\cdot)$
is met by \eqref{def: rate function lambda j n, cluster size}---the rate functions for the $\MRV$ tail of Hawkes process clusters.
Similarly,
an implication of \eqref{cond, additivity of rate functions for MRV, Ld for Levy} is that 
\begin{align}
    \alpha(\bm j) - 1 = \sum_{ i \in \bm j  }\big(\alpha(\{i\}) - 1\big),
    \qquad
    \forall \bm j \in \powersetTilde{d},
    \label{implication on tail indices, cond, additivity of rate functions for MRV, Ld for Levy}
\end{align}
due to $\lambda_{\bm j}(n) \in \RV_{ -\alpha(\bm j) }(n)$ for each $\bm j \in \powersetTilde{d}$ in $\MRV$.
Again, this matches \eqref{def: cost function, cone, cluster}---the tail indices for the $\MRV$ tail of Hawkes process clusters.
Moreover, instead of $\MRV$, Assumption~\ref{assumption: ARMV in levy process} imposes the stronger $\MRV^*$ condition (see condition~\eqref{cond, outside of the full cone, finite index, def: MRV} in Definition~\ref{def: MRV}), which agrees with the statements in Theorem~\ref{theorem: main result, cluster size} for the tail asymptotics of Hawkes process clusters.
For L\'evy processes in general contexts,
a relaxed form of Assumption~\ref{assumption: ARMV in levy process} may be considered, which yields a correspondingly weaker version of Theorem~\ref{theorem: LD for levy, MRV}. 
A detailed discussion of this relaxation is provided in Remark~\ref{remark: relaxation of assumption, LD for Levy}.

Throughout the rest of this paper, we use $\notationdef{notation-cadlag-space-D}{\D[0,T]} = \D\big([0,T],\R^d)$
to denote the space of $\R^d$-valued càdlàg functions with compact domain $[0,T]$,
and
$\notationdef{notation-cadlag-space-D-infty}{\D[0,\infty)} = \D\big([0,\infty),\R^d)$ for the space of $\R^d$-valued càdlàg functions with unbounded domain $[0,\infty)$.
For the L\'evy process $\bm L(t)$, let
\begin{align}
        \notationdef{notation-mu-L-expectation-for-L}{\bm \mu_{\bm L}} \delequal \E \bm L(1) =  \bm c_{\bm L} + \int_{\bm x \in \R^d:\ \norm{\bm x} > 1 } \bm x \nu(d \bm x),
        \label{def: expectation of levy process}
    \end{align}
    where $\bm c_{\bm L}\in \R^d$ is the constant drift in the generating triplet of $\bm L(t)$.
Let
\begin{align}
    \bar{\bm L}_n(t) \delequal\bm L(nt)/n,\qquad
    \notationdef{notation-bar-L-n-scaled-levy-process}{\bar{\bm L}^{ _{[0,T]} }_n} \delequal 
\big\{  \bar{\bm L}_n(t):\ t \in [0,T]  \big\}.
\label{def: scale levy process bar L t}
\end{align}
The random element $\bar{\bm L}^{ _{[0,T]} }_n$
is the scaled version of the sample path of $\bm L(t)$ embedded in $\D[0,T]$.
Besides,
recall that we adopt notations $\bm k = (k_{i})_{ i \in \mathcal I } \in A^{ \mathcal I }$
for vectors of length $|\mathcal I|$ with each coordinate taking values in $A$ and indexed by elements in $\mathcal I$.
Informally speaking,
Theorem~\ref{theorem: LD for levy, MRV} establishes LDP of the form (as $n \to \infty$)
\begin{align}
    \P\bigg( \bar{\bm L}_n^{ \subZeroT{T} } \in \ \cdot \ \cap \breve \D^\epsilon_{ \bm{\mathcal{K}}; \bm \mu_{\bm L}  }[0,T]  \bigg)
    \sim 
    \breve{\mathbf C}^{[0,T]}_{ \bm{\mathcal{K}}; \bm \mu_{\bm L}  }(\cdot)\times 
    \prod_{\bm j \in \powersetTilde{d}} 
    \Big( n \lambda_{\bm j}(n) \Big)^{ \mathcal K_{\bm j}   },
    \qquad\forall 
    \bm{\mathcal{K}} = \big(\mathcal K_{\bm j})_{\bm j \in \powersetTilde{d}} \in \Z_+^{\powersetTilde{d}}.
    \label{informal statement, LD for Levy}
\end{align}
Here, the $\lambda_{\bm j}(\cdot)$'s are the rate functions for the $\MRV^*$ condition in Assumption~\ref{assumption: ARMV in levy process},
the vector $\bm{\mathcal{K}} = \big(\mathcal K_{\bm j})_{\bm j \in \powersetTilde{d}}$ corresponds to a \emph{configuration of big jumps} in the L\'evy process,
and the set $\breve \D^\epsilon_{ \bm{\mathcal{K}}; \bm \mu_{\bm L}  }[0,T] $ contains all paths representing \emph{the typical behavior of $\bar{\bm L}_n^{\subZeroT{T}}$ perturbed by big jumps of the $\bm{\mathcal K}$-configuration}.
Therefore, by \eqref{informal statement, LD for Levy}, the probability of observing a rare event in $\bar{\bm L}_n^{ \subZeroT{T} }$ driven by large jumps of the $\bm{\mathcal K}$-configuration is of order 
$
\prod_{\bm j \in \powersetTilde{d}} 
    \big( n \lambda_{\bm j}(n) \big)^{ \mathcal K_{\bm j}   };
$
furthermore, under this scaling, the leading coefficient of the rare-event probabilities is captured by the limiting measure $\breve{\mathbf C}^{[0,T]}_{ \bm{\mathcal{K}}; \bm \mu_{\bm L}  }(\cdot)$.

To formally present Theorem~\ref{theorem: LD for levy, MRV},
we introduce the rigorous definitions of the notions involved in \eqref{informal statement, LD for Levy}.
Given 
$\bm x \in \R^d$, $\epsilon \geq 0$, $\bm{\mathcal K} = (\mathcal K_{\bm j})_{\bm j \in \powersetTilde{d}} \in \mathbb Z_+^{ \powersetTilde{d}}$, and an interval $I$ that is either of the form $I = [0,T]$ or $I = [0,\infty)$,
let  $\bm 1 = \{ \bm 1(t) = t:\ t \in I  \}$ be the linear function with slope $1$,  and define
\begin{equation}    \label{def: j jump path set with drft c, LD for Levy, MRV}
    \begin{aligned}
        \notationdef{notation-LD-cadlag-space-D-j-c}{\barDxepskT{\bm x}{\epsilon}{\bm{\mathcal K} }{(I)}}
    & \delequal 
    \Bigg\{
        \bm x\bm 1 + \sum_{ \bm j \in \powersetTilde{d} }\sum_{k = 1}^{ \mathcal K_{\bm j} }\bm w_{\bm j,k} \mathbbm{I}_{ [t_{\bm j, k},\infty) \cap I }:
        \\
    &
        t_{\bm j,k} \in I \setminus \{0\}\text{ and }\bm w_{\bm j,k} \in \bar\R^d(\bm j, \epsilon)\ 
        \forall \bm j \in \powersetTilde{d},\ k \in [\mathcal K_{\bm j}];
        \ t_{\bm j,k} \neq t_{\bm j^\prime,k^\prime}\ \forall (\bm j,k)\neq (\bm j^\prime,k^\prime)
    \Bigg\}.
    \end{aligned}
\end{equation}
In other words, ${\barDxepskT{\bm x}{\epsilon}{\bm{\mathcal K}}{(I)}}$
is the collection of piece-wise linear functions $\xi$ of the form
\begin{align}
    \xi(t) = \bm x t + \sum_{ \bm j \in \powersetTilde{d} }\sum_{ k = 1 }^{ \mathcal K_{\bm j} }\bm w_{\bm j, k} \mathbbm{I}_{ [t_{\bm j, k},\infty) \cap I }(t),\qquad\forall t \in I,
    \label{expression for xi in breve D set}
\end{align}
where
each $\bm w_{\bm j,k}$ belongs to the cone $\bar{\R}^d(\bm j,\epsilon)$ (see \eqref{def: enlarged cone R d index i epsilon}),
and the jump times $(t_{\bm j,k})_{ \bm j \in \powersetTilde{d},\ k \in [\mathcal K_{\bm j}]  }$ \emph{do not coincide} with one another.
We note that in \eqref{expression for xi in breve D set},
the path $\xi$ vanishes at the origin (i.e., $\xi(0) = \bm 0$) and, for each $\bm j \in \powersetTilde{d}$, the path $\xi$ has at most $\mathcal K_{\bm j}$ jumps (i.e., discontinuities) whose directions lie in the cone $\bar{\R}^d(\bm j,\epsilon)$.
Besides, under  $\bm{\mathcal K} = (0,0,\ldots,0)$, note that ${\barDxepskT{\bm x}{\epsilon}{\bm 0}{(I)}}$ contains only the path $\xi(t) = t \bm x$.
Meanwhile, since we only consider $I = [0,T]$ or $I = [0,\infty)$ in \eqref{def: j jump path set with drft c, LD for Levy, MRV} and \eqref{expression for xi in breve D set} throughout this paper,
we have either  $\mathbbm{I}_{ [t_{\bm j, k},\infty) \cap I } = \mathbbm{I}_{ [t_{\bm j, k},T]}$ or $\mathbbm{I}_{ [t_{\bm j, k},\infty) \cap I } = \mathbbm{I}_{ [t_{\bm j, k},\infty)}$.
Given the tail indices $\big(\alpha(\bm j)\big)_{ \bm j \subseteq [d] }$ in Assumption~\ref{assumption: ARMV in levy process},
let
\begin{align}
    \notationdef{notation-cost-function-LD-for-Levy}{\breve{c}(\bm{\mathcal K})} \delequal 
    \sum_{ \bm j \in \powersetTilde{d} }\mathcal K_{\bm j} \cdot \Big( \alpha(\bm j) - 1 \Big),
    \qquad \forall 
    \bm{\mathcal K} = (\mathcal K_{\bm j})_{ \bm j \in \powersetTilde{d} } \in \mathbb Z_+^{ \powersetTilde{d} }.
    \label{def: cost alpha j, LD for Levy MRV}
\end{align}
Due to $\alpha(\bm j) > 1\ \forall \bm j \in \powersetTilde{d}$ (see Assumption~\ref{assumption: ARMV in levy process}),
we have $\breve c(\bm{\mathcal K}) > 0$ whenever $\bm{\mathcal K} \neq \bm 0$.
Meanwhile, note that in \eqref{informal statement, LD for Levy}, we have
$
\prod_{\bm j \in \powersetTilde{d}} 
    \big( n \lambda_{\bm j}(n) \big)^{ \mathcal K_{\bm j}   } \in \RV_{ -\breve c(\bm{\mathcal K}) }(n).
$
Next, given $\bm x\in \R^d$, $\bm{\mathcal K} = (\mathcal K_{\bm j})_{ \bm j \in \powersetTilde{d} } \in \mathbb Z_+^{\powersetTilde{d} } \setminus \{\bm 0\}$, and an interval $I = [0,T]$ or $I = [0,\infty)$, we define a Borel measure on $\D(I)$ by 
\begin{align}
     \notationdef{notation-measure-C-j-L-LD-for-Levy}{\breve{\mathbf C}^{ I }_{\bm{\mathcal K};\bm x}(\ \cdot\ )}
    & \delequal
    \frac{1}{\prod_{ \bm j \in \powersetTilde{d} } \mathcal K_{\bm j}!} 
    \label{def: measure C type j L, LD for Levy, MRV}        
    \\
    & \quad\cdot
    \int 
    \mathbbm{I}\Bigg\{
        \bm x \bm 1  + \sum_{ \bm j \in \powersetTilde{d} } \sum_{ k \in [\mathcal K_{\bm j}] } \bm w_{\bm j,k}\mathbbm{I}_{ [t_{\bm j, k},\infty)\cap I } \in\ \cdot\ 
    \Bigg\}
    \bigtimes_{\bm j \in \powersetTilde{d}}
    \bigtimes_{ k \in [\mathcal K_{\bm j}] }\bigg( (\mathbf C_{ \bm j } \times  \mathcal L_{I})\Big(d (\bm w_{\bm j, k}, t_{ \bm j,k })\Big)\bigg),
    \nonumber
\end{align}
where the $\mathbf C_{\bm j}$'s are the limiting measures
in the $\MRV^*$ condition of Assumption~\ref{assumption: ARMV in levy process},
$\mathcal L_I$ is the Lebesgue measure restricted on the interval $I$,
and we use $\nu_1 \times \nu_2$ to denote the product measure of $\nu_1$ and $\nu_2$.
Note that, by definitions in \eqref{def: j jump path set with drft c, LD for Levy, MRV}
and the fact that  $\mathbf C_{\bm j}$ is supported on $\R^d(\bm j)$ for each $\bm j \in \powersetTilde{d}$,
the measure
$
{\breve{\mathbf C}^{ I }_{\bm{\mathcal K};\bm x}}
$
is supported on ${\barDxepskT{\bm x}{0}{\bm{\mathcal K}}{(I)}}$.

As will be shown in Theorem~\ref{theorem: LD for levy, MRV},
scenarios corresponding to different $\bm{\mathcal K}$ in \eqref{informal statement, LD for Levy} can be merged together if the rates of decay 
$
\prod_{\bm j \in \powersetTilde{d}} 
    \big( n \lambda_{\bm j}(n) \big)^{ \mathcal K_{\bm j}   } 
$
coincide.
This occurs frequently under the condition~\eqref{cond, additivity of rate functions for MRV, Ld for Levy} in Assumption~\ref{assumption: ARMV in levy process}.
Therefore, this phenomenon is particularly relevant to our analysis of Hawkes processes, as condition~\eqref{cond, additivity of rate functions for MRV, Ld for Levy} is tailored to the $\MRV^*$ tails of Hawkes process clusters (see  \eqref{def: rate function lambda j n, cluster size}).
Specifically,
let
\begin{align}
    c(\bm k) & \delequal \sum_{i \in [d]}k_i \cdot \big(\alpha(\{i\}) - 1\big),
    \qquad\forall \bm k \in \mathbb Z_+^{d},
    \label{def, cost function, k jump set, LD for Levy}
    \\
    \notationdef{notation-scale-function-for-Levy-LD}{\breve \lambda_{\bm k}(n)}
    & \delequal 
    \prod_{ i \in [d] } \Big( n\lambda_{ \{i\} }(n)\Big)^{ k_{i}},
    \qquad\forall \bm k \in \mathbb Z_+^{ d },
    \label{def: scale function for Levy LD MRV}
\end{align}
where $\big(\alpha(\bm j)\big)_{ \bm j \in \powersetTilde{d} }$ and $\big(\lambda_{\bm j}(\cdot)\big)_{\bm j \in \powersetTilde{d}}$ are the tail indices and rate functions for the $\MRV^*$ condition in Assumption~\ref{assumption: ARMV in levy process}.
Note that for each $\bm k \in \mathbb Z_+^{ d }$, we have $\breve \lambda_{\bm k}(n) \in \RV_{-c(\bm k)}(n)$.
Moreover, Definition~\ref{def: assignment, jump configuration k} captures the correspondence between 
$\breve \lambda_{\bm k}(n)$ in \eqref{def: scale function for Levy LD MRV} and
the rates of decay 
$
\prod_{\bm j \in \powersetTilde{d}} 
    \big( n \lambda_{\bm j}(n) \big)^{ \mathcal K_{\bm j}   } 
$
in \eqref{informal statement, LD for Levy}.
Here,
for each non-empty index set $\bm j \in \powersetTilde{d}$, 
let $\bm e(\bm j) = \big(e_{1}(\bm j),\ldots,e_{d}(\bm j)\big)^\top$,
where $e_{l}(\bm j) = \mathbbm{I}\{ l \in \bm j \}$.
That is, in the vector $\bm e({\bm j})$, each coordinate $e_l(\bm j)$ indicates whether $l$ belongs to the index set $\bm j$ or not.

\begin{definition}[Allocation]
\label{def: assignment, jump configuration k}
    Given $\bm k = (k_1,\ldots,k_d)^\top \in \mathbb Z^d_+$,
the vector $\bm{\mathcal K} = (\mathcal K_{\bm j})_{ \bm j \in \powersetTilde{d} } \in \mathbb Z_+^{\powersetTilde{d}}$ is said to be \textbf{an allocation of }$\bm k$ if 
\begin{align}
    \sum_{\bm j \in \powersetTilde{d}}\mathcal K_{\bm j} \bm e({\bm j}) = \bm k.
    \label{def, assignment of k jump set}
\end{align} 
We use $\mathbb A(\bm k)$ to denote {the set of all allocations of $\bm k$}. 
\end{definition}

\begin{remark}
\label{remark: implications, def of allocation}
We note a few important implications of Definition~\ref{def: assignment, jump configuration k}.
First,
for each $\bm k \in \mathbb Z_+^d$, there exist only finitely many allocations (i.e., $|\mathbb A(\bm k)|<\infty$).
Second, by \eqref{implication on tail indices, cond, additivity of rate functions for MRV, Ld for Levy},
\begin{align}
    \breve c(\bm{\mathcal K}) = c(\bm k),
    \qquad\forall \bm k \in \mathbb Z_+^d,\ \bm{\mathcal K} \in \mathbb A(\bm k),
    \label{property, cost c under allocation}
\end{align}
where $c(\cdot)$ and $\breve c(\cdot)$ are defined in \eqref{def, cost function, k jump set, LD for Levy} and \eqref{def: cost alpha j, LD for Levy MRV}, respectively.
Similarly, by \eqref{cond, additivity of rate functions for MRV, Ld for Levy},
\begin{align}
    \breve \lambda_{\bm k}(n)
    =
    \prod_{ \bm j \in \powersetTilde{d} }
    \Big( n\lambda_{ \bm j }(n)\Big)^{ \mathcal K_{\bm j}},
    \qquad\forall \bm k \in \mathbb Z_+^d \setminus\{\bm 0\} ,\ 
    \bm{\mathcal K} = (\mathcal K_{\bm j})_{ \bm j \in \powersetTilde{d} }\in \mathbb A(\bm k).
    \label{property, rate function for assignment mathcal K, LD for Levy}
\end{align}
Furthermore, 
given $\bm x \in \R^d$, $\bm k \in \mathbb Z^d_+$, and an interval $I = [0,T]$ or $I = [0,\infty)$, let 
\begin{align}
    \notationdef{notation-D-k-jump-sets-wrt-basis}{\D_{\bm k;\bm x}(I)}
    & \delequal 
    \Bigg\{
        \bm x \bm 1 + \sum_{j \in [d]}\sum_{ k = 1 }^{k_j}w_{j,k} \bar{\bm s}_j \mathbbm{I}_{ [t_{k,j},\infty)\cap I  }:
        \label{def: k jump path set with drft c, wrt basis, LD for Levy, MRV}
        \\
        &\qquad\qquad\qquad\qquad\qquad
        t_{j,k} \in I\setminus\{0\}\text{ and }w_{j,k} \geq 0\ \forall j \in [d],\ k \in [k_j]
    \Bigg\}
    \nonumber
\end{align}
be the collection of all piece-wise linear functions that have slope $\bm x$ and make jumps along the vectors $(\bar{\bm s}_j)_{j \in [d]}$.
Note that in \eqref{def: k jump path set with drft c, wrt basis, LD for Levy, MRV}, we allow the $t_{j,k}$'s---the arrival times of jumps---to coincide,
which further
allows the paths in $\D_{\bm k;\bm x}(I)$ to exhibit discontinuities that lie in the cones $\big(\R^d(\bm j)\big)_{ \bm j \in \powersetTilde{d}  }$ (see \eqref{def: cone R d index i}).
Therefore,
by definitions in \eqref{def: k jump path set with drft c, wrt basis, LD for Levy, MRV} and \eqref{def: j jump path set with drft c, LD for Levy, MRV},
we have 
$
{\barDxepskT{\bm x}{0}{\bm{\mathcal K} }{(I)}}
\subseteq \D_{ \bm k;\bm x }(I)
$
for any $\bm k \in \mathbb Z_+^d$ and $\bm{\mathcal K} \in \mathbb A(\bm k)$.
As a result, given $\bm k \in \mathbb Z_+^d$ and $\bm{\mathcal K} \in \mathbb A(\bm k)$, the support of 
$
{\breve{\mathbf C}^{ I }_{\bm{\mathcal K};\bm x}}
$
defined in \eqref{def: measure C type j L, LD for Levy, MRV}        
is a subset of $\D_{ \bm k;\bm x }(I)$.
\end{remark}


Now, we are ready to state Theorem~\ref{theorem: LD for levy, MRV}.
Given $\bm x \in \R^d$, $\epsilon \geq 0$, $\bm k \in \mathbb Z_+^{ d } \setminus \{\bm 0\}$, and interval $I$ of the form $I = [0,T]$ or $I = [0,\infty)$,
we define
\begin{align}
    \notationdef{notation-jump-sets-with-smaller-cost-LD-for-levy}{\barDxepskT{\bm x}{\epsilon}{\leqslant \bm k}{(I)}}
    \delequal 
    \bigcup_{
        \substack{
            \bm{\mathcal K}\in  \mathbb Z_+^{ \powersetTilde{d} }:
            \\
            \bm{\mathcal K} \notin \mathbb A(\bm k),\ \breve c(\bm{\mathcal K}) \leq c(\bm k)
        }
    }
    \barDxepskT{\bm x}{\epsilon}{\bm{\mathcal K} }{(I)}.
    \label{def: path with costs less than j jump set, drift x, LD for Levy MRV}
\end{align}
For the càdlàg space $\D[0,T]$, recall the definition of the Skorokhod $J_1$ metric
\begin{align}
    \notationdef{notation-J1-metric}{\dj{[0,T]}(x,y)} \delequal 
    \inf_{\lambda \in \Lambda[0,T] } \sup_{t \in [0,T]}
    |\lambda(t) - t|
\vee \norm{ x(\lambda(t)) - y(t) },
\qquad
\forall x,y \in \D[0,T],
    \label{def: J1 metric on [0,T]}
\end{align}
where
$
\Lambda[0,T]
$
is the set of all homeomorphisms on $[0,T]$.
Theorem~\ref{theorem: LD for levy, MRV} establishes sample path LDPs for L\'evy processes with $\MRV$ increments 
w.r.t.\ the Skorokhod $J_1$ topology of $\D[0,T]$.

\begin{theorem}\label{theorem: LD for levy, MRV}
\linksinthm{theorem: LD for levy, MRV}
    Let Assumption~\ref{assumption: ARMV in levy process} hold.
    Let $T \in (0,\infty)$, $\bm k \in \mathbb Z_+^d \setminus \{\bm 0\}$,
    and let $B$ be a Borel set of $\D[0,T]$ equipped with the Skorokhod $J_1$ topology.
    Suppose that $B$
   is bounded away from $\barDxepskT{\bm{\mu}_{\bm L}}{\epsilon}{\leqslant \bm k}{[0,T]}$ under $\dj{[0,T]}$ for some $\epsilon > 0$.
   Then,
    \begin{equation}\label{claim, finite index, theorem: LD for levy, MRV}
        \begin{aligned}
            \sum_{\bm{\mathcal K} \in \mathbb A(\bm k)  }\breve{\mathbf C}_{\bm{\mathcal K};\bm \mu_{\bm L}}^{ _{[0,T]} }(B^\circ)
        & \leq 
        \liminf_{n \to \infty}
        \frac{
        \P(\bar{\bm L}^{ _{[0,T]} }_n \in B)
        }{
            \breve \lambda_{\bm k}(n)
        }
        \leq 
        \limsup_{n \to \infty}
        \frac{
        \P(\bar{\bm L}^{ _{[0,T]} }_n \in B)
        }{
            \breve \lambda_{\bm k}(n)
        }
        \leq 
         \sum_{ \bm{\mathcal K} \in \mathbb A(\bm k)  }\breve{\mathbf C}_{ \bm{\mathcal K};\bm \mu_{\bm L}}^{ _{[0,T]} }(B^-) < \infty,
        \end{aligned}
    \end{equation}
     where $\mathbb A(\bm k)$ is the set containing all allocations of $\bm k$ (see Definition~\ref{def: assignment, jump configuration k}).
\end{theorem}

We provide the detailed proof of Theorem~\ref{theorem: LD for levy, MRV} in Section~\ref{subsubsec: proof of LD with MRV} of the Appendix,
and briefly note that the limiting measures of the form
$
\breve{\mathbf C}_{ \bm{\mathcal K};\bm x}^{ I }(\cdot)
$
can be efficiently evaluated via Monte Carlo simulation in the context of Hawkes processes
(see Remark~\ref{remark: evaluation of limiting measures, LD for Hawkes} of Section~\ref{subsec: LD for Hawkes process}).
To conclude this subsection, we add a few remarks about 
the interpretation of the LDP~\eqref{claim, finite index, theorem: LD for levy, MRV},
its proof strategy,
and relaxations of assumptions in Theorem~\ref{theorem: LD for levy, MRV}.

\begin{remark}[Interpretations of Asymptotics~\eqref{claim, finite index, theorem: LD for levy, MRV}]
\label{remark: inerpretation, LD for Levy}
Theorem~\ref{theorem: LD for levy, MRV} shows that, given the rare event set $B$, 
the asymptotics of the form \eqref{claim, finite index, theorem: LD for levy, MRV} hold for any $\bm k \in \mathbb Z^d_+ \setminus \{\bm 0\}$
such that $B$ is bounded away from $\barDxepskT{\bm{\mu}_{\bm L}}{\epsilon}{\leqslant \bm k}{[0,T]}$ under $\dj{[0,T]}$ for some $\epsilon > 0$.
However, the choice of $\bm k$ that attains non-trivial bounds in \eqref{claim, finite index, theorem: LD for levy, MRV} is determined by
(with $\D_{ \bm k; \bm \mu_{\bm L} }[0,T]$ defined in \eqref{def: k jump path set with drft c, wrt basis, LD for Levy, MRV})
\begin{align}
    \bm k(B) \delequal 
    \argmin_{
        \substack{
            \bm k \in \mathbb Z^d_+\setminus\{\bm 0\}:\ 
            B \cap \D_{ \bm k; \bm \mu_{\bm L} }[0,T] \neq \emptyset 
        }   
    }
        c(\bm k).
    \label{def: rate function, LD for Levy}
\end{align}
Indeed, as noted in Remark~\ref{remark: implications, def of allocation},
the measure
$
{\breve{\mathbf C}^{ \subZeroT{T} }_{\bm{\mathcal K};\bm \mu_{\bm L}}}(\cdot)
$
is supported on 
$\D_{ \bm k;\bm \mu_{\bm L} }[0,T]$
with $\bm{\mathcal K} \in \mathbb A(\bm k)$.
To attain a strictly positive upper bound in \eqref{claim, finite index, theorem: LD for levy, MRV},
we must have
$
B \cap \D_{ \bm k;\bm \mu_{\bm L} }[0,T] \neq \emptyset.
$
Therefore, 
given any Borel set $B \subset \D[0,T]$ that does not contain the linear path with slope $\bm \mu_{\bm L}$,
the LDP~\eqref{claim, finite index, theorem: LD for levy, MRV} implies
\begin{align*}
    {\mathbf C}_{ \bm k(B) }^{ \subZeroT{T}  }(B^\circ)
    \leq
    \liminf_{n \to \infty}
        \frac{
        \P(\bar{\bm L}^{ _{[0,T]} }_n \in B)
        }{
            \breve \lambda_{\bm k(B)}(n)
        }
    \leq 
    \limsup_{n \to \infty}
        \frac{
        \P(\bar{\bm L}^{ _{[0,T]} }_n \in B)
        }{
            \breve \lambda_{\bm k(B)}(n)
        }
    \leq {\mathbf C}_{ \bm k(B) }^{ \subZeroT{T}  }(B^-)
\end{align*}
with limiting measure
$
{\mathbf C}_{ \bm k }^{ \subZeroT{T}  } = \sum_{ \bm{\mathcal K} \in \mathbb A(\bm k)  }\breve{\mathbf C}_{ \bm{\mathcal K};\bm \mu_{\bm L}}^{ _{[0,T]} },
$
provided that $\bm k(B)$ uniquely exists and $B$ is bounded away from $\barDxepskT{\bm{\mu}_{\bm L}}{\epsilon}{\leqslant \bm k(B)}{[0,T]}$ under $\dj{[0,T]}$ for some $\epsilon > 0$.
From this perspective, it is worth noting that: (i) $\bm k(B)$ plays a role analogous to that of rate functions in the classical LDP; and (ii) 
the power-law rates of decay of rare-event probabilities $\P(\bar{\bm L}^{\subZeroT{T}}_n \in B)$, as well as limiting behavior of $\bar{\bm L}^{\subZeroT{T}}_n$ conditioned on such rare events,
are determined by the discrete optimization problem in \eqref{def: rate function, LD for Levy} regarding \emph{the likelihood of entering} the set $B$.
In particular, note that
the nominal behavior of the scaled path $\bar{\bm L}^{\subZeroT{T}}_n$ is the linear function with slope $\bm \mu_{\bm L}$.
Meanwhile,
by 
Definition~\ref{def: MRV} and
the $\MRV$ condition in Assumption~\ref{assumption: ARMV in levy process},
the probability of observing a big jump in $\bar{\bm L}^{\subZeroT{T}}_n$ that is aligned with the cone $\R^d(\bm j)$
(equivalently, observing a jump in $\bm L(t)$ that is of the size $\bo(n)$ and lies in the cone $\R^d(\bm j)$, over the time horizon $t \in [0,Tn]$)
is of order $T \cdot n\lambda_{\bm j}(n) \in \RV_{ 1-\alpha(\bm j)  }(n)$.
Therefore,
the function $\breve c(\cdot)$ in \eqref{def: cost alpha j, LD for Levy MRV} (and hence $c(\cdot)$ in \eqref{def, cost function, k jump set, LD for Levy} by the equality \eqref{property, cost c under allocation})
corresponds to the ``{rareness}'' of observing a certain configuration of big jumps in the L\'evy process,
and
the solution $\bm k(B)$ in \eqref{def: rate function, LD for Levy}
corresponds to \emph{the most likely configuration of big jumps} that can push the linear path $\bm \mu_{\bm L}\mathbf 1$---the nominal path of $\bar{\bm L}^{\subZeroT{T}}_n$---into the rare event set $B$.
\end{remark}


\begin{remark}[Proof Sketch for Theorem~\ref{theorem: LD for levy, MRV}]
\label{remark: proof sketch, LD for Levy}
Our proof of Theorem~\ref{theorem: LD for levy, MRV} 
relies on the notion of asymptotic equivalence in terms of $\mathbb M$-convergence (\cite{10.1214/14-PS231}).
First,
Proposition~\ref{proposition: asymptotic equivalence, LD for Levy MRV} shows that,
for the purpose of proving \eqref{claim, finite index, theorem: LD for levy, MRV},
it suffices to
study the large-jump approximation $\hat{\bm L}^{>\delta}_n$, which removes any discontinuity in $\bar{\bm L}^{\subZeroT{T}}_n$ with norm less than $\delta$.
The key tool is Lemma~\ref{lemma: concentration of small jump process, Ld for Levy MRV}, which develops concentration inequalities for the L\'evy process $\bm L(t)$ stopped right before the arrival of the first large jump.
Next, Proposition~\ref{proposition: weak convergence, LD for Levy MRV} provides asymptotic analyses of $\hat{\bm L}^{>\delta}_n$.
In particular, 
Lemma~\ref{lemma: asymptotic law for large jumps, LD for Levy MRV} establishes the asymptotic law of the arrival times and sizes of large jumps in $\bm L(t)$.
We achieve this by differentiating the large jumps based on their directions and associating them with one of the cones $\big(\R^d(\bm j)\big)_{\bm j \in \powersetTilde{d}}$,
which enables the application of the $\MRV^*$ tail condition in Assumption~\ref{assumption: ARMV in levy process}.
Then, we obtain Proposition~\ref{proposition: weak convergence, LD for Levy MRV} through a continuous mapping argument.
\end{remark}

\begin{remark}[Relaxation of Assumptions]
\label{remark: relaxation of assumption, LD for Levy}
As noted earlier,
the condition \eqref{cond, additivity of rate functions for MRV, Ld for Levy} on the rate functions and 
the
$\MRV^*$ tail condition (instead of $\MRV$) in Assumption~\ref{assumption: ARMV in levy process} 
are specifically designed for the subsequent analysis of Hawkes processes. In more general settings, one may adopt a more flexible—though less detailed—description of multivariate hidden regular variation in the L\'evy measure $\nu(\cdot)$.
In Section~\ref{sec: appendix, LD for Levy, General Case} of the Appendix, we present an analog of Theorem~\ref{theorem: LD for levy, MRV} for such settings, where condition~\eqref{cond, additivity of rate functions for MRV, Ld for Levy} is dropped, a weaker $\MRV$ condition is used, and the number of vectors in the basis is not fixed at $d$.
Furthermore, our proof strategy for Theorem~\ref{theorem: LD for levy, MRV} could be extended to L\'evy measures supported on $\R^d$ rather than just $\R^d_+$, or to settings where the tail behavior is described using alternative formalisms such as Adapted-MRV (see \cite{das2023aggregatingheavytailedrandomvectors}).
We do not pursue these extensions here, as they are not closely related to our focus on Hawkes processes in this paper.
\end{remark}



\subsection{Multivariate Heavy-Tailed Hawkes Processes}
\label{subsec: LD for Hawkes process}

In this subsection, we 
study sample path LDPs for multivariate heavy-tailed Hawkes processes w.r.t.\ the product $M_1$ topology of $\D[0,\infty) = \D\big([0,\infty),\R^d\big)$.
Here,
the product $M_1$ topology is induced by projecting any path $\xi = (\xi_1,\ldots,\xi_d)^\top \in \D[0,\infty)$
onto the $d$-fold product space of $\D\big([0,\infty),\R\big)$,
and then taking the maximum of the Skorokhod $M_1$ metric along the $d$ dimensions over this product space.

To be more specific, we first review the Skorokhod $M_1$ metric of the càdlàg space with codomain $\R$. 
For any path $\xi \in \D\big([0,T],\R)$, let 
\begin{align*}
    \Gamma_{\xi}
    \delequal 
    \Big\{
        (z,t) \in \R \times [0,T]:\ z \in [\xi(t-) \wedge \xi(t) ,\xi(t-) \vee \xi(t) ]
    \Big\}
\end{align*}
be the connected graph of $\xi$, where we take $\xi(0-) = \xi(0)$.
We define an order over the connected graph $\Gamma_{\xi}$ by saying $(z_1,t_1) \leq (z_2,t_2)$ if 
(i) $t_1 < t_2$ or 
(ii) $t_1 = t_2$ and $|z_1 - \xi(t_1-)| \leq |z_2 - \xi(t_2-)|$.
A mapping $(u(\cdot),s(\cdot))$ is said to be a parametric representation of $\xi$ if 
$t \mapsto (u(t),s(t))$ is a continuous and non-decreasing (over $\Gamma_\xi$) mapping
such that $\{ (u(t),s(t)):\ t \in [0,1]  \} = \Gamma_\xi$.
The Skorokhod $M_1$ metric on $\D\big([0,T],\R\big)$ is defined by (for any $\xi^{(1)},\xi^{(2)} \in \D\big([0,T],\R\big)$)
\begin{align}
    \notationdef{notation-M1-norm-in-R1}{\dm{[0,T]}(\xi^{(1)},\xi^{(2)})}
    \delequal 
    \inf_{ \substack{ (u_i,s_i) \in \Pi(\xi^{(i)})\\ i =1,2 }  }\ 
    \sup_{t \in [0,1]}
    | u_1(t) - u_2(t)| \vee |s_1(t) - s_2(t)|,
    \label{def: M1 metric 0 T in R 1}
\end{align}
where $\Pi(\xi)$ is the set of all parametric representations of $\xi$.
In the multivariate setting,
the product $M_1$ metric of $\D[0,T] = \D\big( [0,T], \R^d \big)$ is defined by
\begin{align}
    \notationdef{notation-M1-norm-product}{\dmp{[0,T]}(\xi^{(1)},\xi^{(2)})}
    \delequal 
    \max_{ j \in [d] }\dm{[0,T]}(\xi^{(1)}_j,\xi^{(2)}_j),
    \qquad
    \forall \xi^{(1)},\xi^{(2)} \in \D[0,T],
    \label{def: product M1 metric}
\end{align}
where we write $\xi^{(i)} = (\xi^{(i)}_1,\xi^{(i)}_2,\ldots,\xi^{(i)}_d)^\top$,
and the subscript $\mathcal P$ indicates the uniform metric on the product space.
Given $t \in (0,\infty)$, we define the projection mapping $\notationdef{notation-projection-mapping-to-D0T}{\phi_t}: \D[0,\infty) \to \D[0,t]$ by 
\begin{align}
    \phi_t(\xi)(s) \delequal \xi(s),\qquad\forall s\in [0,t].
    \label{def: projection mapping from D infty to D T}
\end{align}
This allows us to define the product $M_1$ metric on $\D[0,\infty)$:
\begin{align}
    \notationdef{notation-M1-metric-for-D-infty}{\dmp{[0,\infty)}(\xi^{(1)},\xi^{(2)})}
    \delequal
    \int_0^\infty 
        e^{-t} \cdot \Big[ \dmp{[0,t]}\Big( \phi_t\big(\xi^{(1)}),\ \phi_t\big(\xi^{(2)}\big)\Big)  \wedge 1 \Big]dt,
    \qquad
    \forall \xi^{(1)},\xi^{(2)} \in \D[0,\infty).
    \label{def, projection mapping phi t}
\end{align}
We note that 
the topology induced by the product metric $\dmp{[0,T]}$ agrees with the 
weak $M_1$ topology of the càdlàg space
(see, e.g., Chapter 12 of \cite{whitt2002stochastic}).

Moving on to the large deviations analysis for multivariate Hawkes processes, 
we work with the heavy-tailed setting in Section~\ref{subsec: tail of cluster S, statement of results} 
and impose 
 Assumptions~\ref{assumption: subcriticality}--\ref{assumption: regularity condition 2, cluster size, July 2024}.
Besides, 
we adopt the definitions of $(\bar{\bm s}_{j})_{j \in [d]}$, $\alpha(\cdot)$, $\lambda_{\bm j}(\cdot)$, and 
${\mathbf C_i^{\bm j}(\cdot)}$ in 
\eqref{def: bar s i, ray, expectation of cluster size}--\eqref{def: measure C i I, cluster, main paper},
which specify the basis, tail indices, rate functions, and limiting measures for the $\MRV^*$ tail of Hawkes process clusters established in Theorem~\ref{theorem: main result, cluster size}.
Throughout this subsection, we set
\begin{align}
            \mathbf C_{\bm j}(\cdot) \delequal 
            \sum_{i \in [d]}c^{\bm N}_{i} \cdot \mathbf C^{\bm j}_i(\cdot),
            \qquad
             \forall \bm j \in \powersetTilde{d},
            \label{def: measure C indices j, sample path LD for Hawkes}
\end{align}
where the constants
        $c^{\bm N}_{i}$ are the immigration rates of the Hawkes process (see \eqref{def: conditional intensity, hawkes process}).
Next,
 let
 $\D_{\bm k;\bm x}(I)$,
 $c(\bm k)$, $\breve \lambda_{\bm k}(\cdot)$, 
 $\breve c(\bm{\mathcal K})$, 
 $
 {\breve{\mathbf C}^{ I }_{\bm{\mathcal K};\bm x}(\ \cdot\ )},
 $
 and
 ${\barDxepskT{\bm x}{\epsilon}{\leqslant \bm k}{(I)}}$
be defined as in \eqref{def: j jump path set with drft c, LD for Levy, MRV}--\eqref{def: path with costs less than j jump set, drift x, LD for Levy MRV} of Section~\ref{subsec: LD for levy process, MRV},
with $\mathbf C_{\bm j}$ specified as in \eqref{def: measure C indices j, sample path LD for Hawkes}.
We are ready to state Theorem~\ref{theorem: LD for Hawkes}, the main result of this section.
For each $n \geq 1$, let
\begin{align}
     \bar{\bm N}_n(t) \delequal \bm N(nt)/n,\qquad
    \notationdef{notation-scaled-Hawkess-bar-N-n}{\bar{\bm N}^{_{[0,\infty)}}_n}
    \delequal 
    \big\{
        \bar{\bm N}_n(t):\ t \in [0,\infty)
    \big\},
    \label{def: scaled Hawkes process, LD for Hawkes}
\end{align}
with $\bar{\bm N}^{_{[0,\infty)}}_n$
being the scaled version of the sample path of $\bm N(t)$ embedded in $\D[0,\infty)$.
Besides, let 
    \begin{align}
    \notationdef{notation-mu-N-Hawkess-process}{\bm\mu_{\bm N}} \delequal 
    \sum_{i \in [d]} c^{\bm N}_{i}\bar{\bm s}_i,
    \label{def: approximation, mean of N(t)}
\end{align}
where the $c^{\bm N}_{i}$'s are the immigration rates of the Hawkes process in  \eqref{def: conditional intensity, hawkes process}.
Under the presence of regularly varying tails in the offspring distributions,
Theorem~\ref{theorem: LD for Hawkes} develops sample path LDPs for $\bm N(t)$  w.r.t.\ the product $M_1$ topology of $\D[0,\infty)$.

\begin{theorem}\label{theorem: LD for Hawkes}
\linksinthm{theorem: LD for Hawkes}
Let Assumptions~\ref{assumption: subcriticality}--\ref{assumption: regularity condition 2, cluster size, July 2024} hold.
Let
    $\bm k = (k_{j})_{ j \in [d] } \in \mathbb Z_+^{ d } \setminus \{\bm 0\}$
    and let $B$ be a Borel set in $\D[0,\infty)$ equipped with the product $M_1$ topology.
    Suppose that $B$ is bounded away from $\barDxepskT{\bm\mu_{\bm N}}{\epsilon}{\leqslant \bm k}{[0,\infty)}$ under $\dmp{[0,\infty)}$ for some $\epsilon > 0$, and
\begin{align}
    \int_{x/\log x}^\infty f^{\bm N}_{p \leftarrow q}(t)dt =  \lo\big( \breve \lambda_{\bm k}(x)\big/x \big)
    \ 
    \text{ as }x \to \infty,
    \qquad \forall (p,q)\in [d]^2,
    \label{condtion: tail of fertility function hij}
\end{align}
where the $f^{\bm N}_{p \leftarrow q}$'s are the decay functions in \eqref{def: conditional intensity, hawkes process},
and the $\breve \lambda_{\bm j}(n)$'s are defined in \eqref{def: scale function for Levy LD MRV}.
Besides, suppose that $\alpha(\bm j) \neq \alpha(\bm j^\prime)$ for any $\bm j,\bm j^\prime \in \powersetTilde{d}$ with $\bm j \neq \bm j^\prime$,
where $\alpha(\cdot)$ is defined in \eqref{def: cost function, cone, cluster}.
Then,
\begin{equation} \label{claim, theorem: LD for Hawkes}
    \begin{aligned}
        \sum_{ \bm{\mathcal K} \in \mathbb A(\bm k) }\breve{\mathbf C}^{\subInfty }_{\bm{\mathcal K};\bm \mu_{\bm N}}(B^\circ)
        & \leq 
        \liminf_{n \to \infty}
        \frac{
        \P\big(\bar{\bm N}^{_{[0,\infty)}}_n \in B\big)
        }{
            \breve\lambda_{\bm k}(n)
        }
        \leq 
        \limsup_{n \to \infty}
        \frac{
        \P\big(\bar{\bm N}^{_{[0,\infty)}}_n \in B\big)
        }{
            \breve\lambda_{\bm k}(n)
        }
        \leq 
         \sum_{ \bm{\mathcal K} \in \mathbb A(\bm k) }\breve{\mathbf C}^{ \subInfty  }_{\bm{\mathcal K};\bm\mu_{\bm N}}(B^-) < \infty,
    \end{aligned}
\end{equation}
    where $\mathbb A(\bm k)$ is the set containing all allocations of $\bm k$ (see Definition~\ref{def: assignment, jump configuration k}).
    

\end{theorem}

The intuition of the sample path LDP \eqref{claim, theorem: LD for Hawkes} is that, 
given any Borel set $B \subseteq \D[0,\infty)$ that is bounded away from $\barDxepskT{\bm\mu_{\bm N}}{\epsilon}{\leqslant \bm k}{[0,\infty)}$,
the rare-event probabilities (roughly) admit the asymptotics
$
\P(\bar{\bm N}^{_{[0,\infty)}}_n \in B)
\sim \breve \lambda_{\bm k}(n) \cdot 
\sum_{ \bm{\mathcal K} \in \mathbb A(\bm k) }\breve{\mathbf C}^{ \subInfty  }_{\bm{\mathcal K};\bm\mu_{\bm N}}(B)
$
as $n \to \infty$,
thus exhibiting a power-law rate of decay $\breve \lambda_{\bm k}(n) \in \RV_{ - c(\bm k)  }(n)$
with leading coefficient $\sum_{ \bm{\mathcal K} \in \mathbb A(\bm k) }\breve{\mathbf C}^{ \subInfty  }_{\bm{\mathcal K};\bm\mu_{\bm N}}(B)$.
We elaborate in Remark~\ref{remark: interpretation, LD for Hawkes}
how the rate functions for $\P(\bar{\bm N}^{_{[0,\infty)}}_n \in B)$ are determined via a discrete optimization problem regarding the geometry of $B$ and the tail asymptotics of Hawkes process clusters.

We provide the proof of Theorem~\ref{theorem: LD for Hawkes} in Section~\ref{subsec: appendix, proof of main result, Hawkes}.
To conclude, we explain in Remarks~\ref{remark: evaluation of limiting measures, LD for Hawkes}--\ref{remark, tail condition on decay functions}
the evaluation of the limiting measures in \eqref{claim, theorem: LD for Hawkes} via Monte Carlo simulation,
the topological choices in Theorem~\ref{theorem: LD for Hawkes},
as well as the necessity of condition \eqref{condtion: tail of fertility function hij}.

\begin{remark}[Interpretations of Asymptotics \eqref{claim, theorem: LD for Hawkes}]
\label{remark: interpretation, LD for Hawkes}
Analogous to Remark~\ref{remark: inerpretation, LD for Levy},
the LDP
\eqref{claim, theorem: LD for Hawkes} implies
\begin{align*}
        {\mathbf C}_{ \bm k(B) }^{ \subInfty  }(B^\circ)
    \leq
    \liminf_{n \to \infty}
        \frac{
        \P(\bar{\bm N}^{ \subInfty }_n \in B)
        }{
            \breve \lambda_{\bm k(B)}(n)
        }
    \leq 
    \limsup_{n \to \infty}
        \frac{
        \P(\bar{\bm N}^{ \subInfty }_n \in B)
        }{
            \breve \lambda_{\bm k(B)}(n)
        }
    \leq {\mathbf C}_{ \bm k(B) }^{ \subInfty }(B^-)
\end{align*}
for sufficiently general Borel sets $B$ of $\D[0,\infty)$ equipped with the product $M_1$ topology,
with
$$
\bm k(B) = \argmin_{
        \substack{
            \bm k \in \mathbb Z^d_+\setminus\{\bm 0\}:\ 
            B \cap \D_{ \bm k; \bm \mu_{\bm N} }[0,\infty) \neq \emptyset 
        }   
    }
    c(\bm k),
$$
the function $c(\cdot)$ defined in \eqref{def, cost function, k jump set, LD for Levy},
the set $\D_{ \bm k; \bm \mu_{\bm N} }[0,T]$ defined in \eqref{def: k jump path set with drft c, wrt basis, LD for Levy, MRV},
and 
$
{\mathbf C}_{ \bm k }^{ \subInfty } = \sum_{ \bm{\mathcal K} \in \mathbb A(\bm k)  }\breve{\mathbf C}_{ \bm{\mathcal K};\bm \mu_{\bm N}}^{ \subInfty}.
$
This is made precise by Theorem~\ref{theorem: LD for Hawkes} whenever
$\bm k(B)$ uniquely exists and $B$ is bounded away from $\barDxepskT{\bm{\mu}_{\bm N}}{\epsilon}{\leqslant \bm k(B)}{[0,\infty)}$ under $\dmp{\subInfty}$ for some $\epsilon > 0$.
In other words, 
the power-law rates of decay for $\P( \bar{\bm N}^{\subInfty}_n \in B )$, as well as limiting behavior of $\bar{\bm N}^{\subInfty}_n$ conditioned on such rare events,
are dictated by the discrete optimization problem in $\bm k(B)$.
In particular, by the tail asymptotics of Hawkes process clusters stated in Theorem~\ref{theorem: main result, cluster size},
the probability of observing a large cluster (i.e., of size $\bo(n)$) whose size vector aligns with the cone $\bar \R^d(\bm j)$ is roughly of order $n^{ -\alpha(\bm j) }$ with $\alpha(\cdot)$ defined in \eqref{def: cost function, cone, cluster}.
From this perspective, the function $\breve c(\cdot)$ in \eqref{def: cost alpha j, LD for Levy MRV}
(and hence the function $c(\cdot)$ in \eqref{def, cost function, k jump set, LD for Levy} by the equality \eqref{property, cost c under allocation})
characterizes the ``rareness'' of any configuration of large clusters 
and $\bm k(B)$ identifies the \emph{most likely configuration of large clusters} that can push 
$\bm \mu_{\bm N}\mathbf 1$---the nominal path of the (scaled) Hawkes process $\bar{\bm N}^{\subInfty}_n$---into the rare event set $B$.
\end{remark}

\begin{remark}[Evaluation of Limiting Measures through Monte Carlo Simulation]
\label{remark: evaluation of limiting measures, LD for Hawkes}
We stress that the limiting measures 
$
\breve{\mathbf C}^{ \subInfty  }_{\bm{\mathcal K};\bm\mu_{\bm N}}(\cdot)
$
in \eqref{claim, theorem: LD for Hawkes} can be evaluated efficiently via Monte Carlo simulation.
Specifically, given a Borel set $B \subseteq \D[0,\infty)$ satisfying the conditions in Theorem~\ref{theorem: LD for Hawkes},
we verify in
Lemma~\ref{lemma: B bounded away under J1 infinity} the following claims
under any $\bar\delta >0$ small enough and $T > 0$ large enough:
for any $\bm k \in \Z^d_+ \setminus \{\bm 0\}$,
$\bm{\mathcal K} \in \mathbb A(\bm k)$, and 
$\xi \in B \cap \breve\D^\epsilon_{ \bm{ \mathcal K };\bm \mu_{\bm N}  }[0,\infty)$,
in the expression~\eqref{expression for xi in breve D set} for $\xi$ we have
$
 t_{\bm j, k} \leq T,
$
$
\norm{\bm w_{\bm j, k}} > \bar\delta,
$
and
$
 \bm w_{\bm j,k}\notin \bar{\R}^d_\leqslant(\bm j,\bar\delta)
$
for each $\bm j \in \powersetTilde{d},\ k \in [\mathcal K_{\bm j}]$.
Therefore, as long as we can sample from probability measures
(where the $\mathbf C_{\bm j}$'s are defined in \eqref{def: measure C indices j, sample path LD for Hawkes})
\begin{align}
   \bar{\mathbf C}_{\bm j, \bar\delta}(\cdot) \delequal
    \mathbf C_{\bm j}\Big( \ \cdot\ \cap \bar \R^{d}_>(\bm j,\bar\delta) \Big)
    \Big/
    \mathbf C_{\bm j}\Big( \bar \R^{d}_>(\bm j,\bar\delta) \Big),
    \quad 
    \text{with }
    \bar \R^{d}_>(\bm j,\bar\delta) \delequal 
    \{ \bm x \in \R^d_+ \setminus \bar\R^d_\leqslant(\bm j,\bar\delta) :\ \norm{\bm x} > \bar\delta   \},
    \label{def: measure bar C j bar delta, for monte carlo simulation of limiting measures}
\end{align}
and evaluate the normalization constants $\bar c_{\bm j, \bar\delta} = \mathbf C_{\bm j}\Big( \bar \R^{d}_>(\bm j,\bar\delta) \Big)$,
by \eqref{def: measure C type j L, LD for Levy, MRV} we can
run Monte Carlo simulation to estimate
(under $\bar\delta > 0$ sufficiently small and $T > 0$ sufficiently large) 
\begin{align}
    \breve{\mathbf C}^{ \subInfty  }_{\bm{\mathcal K};\bm\mu_{\bm N}}(B)
    = 
        \frac{
        \prod_{ \bm j \in \powersetTilde{d} } (T \cdot \bar c_{\bm j,\bar\delta}) ^{|\mathcal K_{\bm j}|}
    }{\prod_{ \bm j \in \powersetTilde{d} } \mathcal K_{\bm j}!} 
    \cdot 
    \E 
    \Bigg[
    \mathbbm{I}
    \bigg\{
        \bm \mu_{\bm N}\mathbf{1}
         + \sum_{ \bm j \in \powersetTilde{d} } \sum_{ k \in [\mathcal K_{\bm j}] } 
         \bm Z^{(\bm j,k)}
         \mathbbm{I}_{ [T \cdot  U^{(\bm j, k)},\infty)} \in B
    \bigg\}
    \Bigg],
    \label{details in remark: monte carlo simulation for limiting measure of LD Hawkes}
\end{align}
where the $U^{(\bm j,k)}$'s are i.i.d.\ copies of Unif$(0,1)$, and each $\bm Z^{(\bm j,k)}$ is an independent copy under the law $\bar{\mathbf C}_{\bm j,\bar\delta}$.
Furthermore, the evaluation of constants $\bar c_{\bm j,\bar\delta}$ are detailed in Remark~4 of \cite{blanchet2025tailasymptoticsclustersizes},
and rejection sampling addresses the generation of 
$ \bm Z^{(\bm j,k)} \sim \bar{\mathbf C}_{\bm j, \bar\delta}(\cdot)$,
thanks to the specific form of the limiting measures $\mathbf C^{\bm j}_i(\cdot)$ in \eqref{def: measure C i I, cluster, main paper} for Theorem~\ref{theorem: main result, cluster size}.
We elaborate on the rejection sampling for $\bar{\mathbf C}_{\bm j, \bar\delta}(\cdot)$ in Section~\ref{sec: appendix, tail asymptotics of Hawkes process clusters} of the Appendix, 
where we also provide the explicit expressions for the limiting measures $\mathbf C^{\bm j}_i(\cdot)$ in \eqref{def: measure C i I, cluster, main paper}.

\end{remark}

\begin{remark}[Topological Choices in Theorem~\ref{theorem: LD for Hawkes}]
\label{remark: conditions in LD for Hawkes}
Concerning the choice of Skorokhod non-uniform topologies for càdlàg spaces,
we note that the characterization of asymptotics~\eqref{claim, theorem: LD for Hawkes}
under the product $M_1$ topology of $\D[0,\infty)$ in Theorem~\ref{theorem: LD for Hawkes} is, in general, the tightest one can hope for.
\begin{itemize}
    \item 
        Compared to the $J_1$ topology, the $M_1$ topology would be the right choice for the large deviations analysis of Hawkes processes,  as it allows merging jumps that correspond to distinct but relatively close arrival times of all descendants within the same cluster.

    \item
        Suppose that the size vector of a cluster is $\bm s$.
        At any moment $t$, the current size of the cluster (i.e., containing all descendants born by time $t$) stays within the hypercube $\{ \bm x \in \mathbb N^d:\ \bm x \leq \bm s \}$
        but would generally deviate from the ray $\{ w\bm s:\ w \geq 0 \}$ due to the randomness in the birth times.
        Such arbitrariness and non-linearity in the growth of cluster sizes prevent the LDP~\eqref{claim, theorem: LD for Hawkes} to hold under the strong $M_1$ topology and promotes the use of the product (i.e., weak) $M_1$ topology.

    \item 
        Note that this issue regarding the arbitrary shape of truncated clusters can be safely dismissed for univariate Hawkes processes, whose cluster sizes are non-negative scalar variables.
        In fact, as detailed in Section~\ref{subsec: univariate results, LD for Hawkes, M1 of D0T} of the Appendix,
        Theorem~\ref{theorem: LD for Hawkes} \emph{can be strengthened to the $M_1$ topology of $\D\big([0,T],\R\big)$---the càdlàg space with compact domain---in the univariate setting}.
        Compared to the statements in Theorem~\ref{theorem: LD for Hawkes} w.r.t.\ $\big(\D\big([0,\infty),\R\big),\dm{\subInfty}\big)$ (when applied to univariate cases),
        such extensions to $\big(\D\big([0,T],\R\big),\dm{\subZeroT{T}}\big)$
        allow us to address
        a broader class of events and functionals for heavy-tailed Hawkes processes in the univariate setting.

    \item 
        In comparison,
        as we demonstrate in
        Example~\ref{example: LD for Hawkes, cadlag space with compact domain} 
        in Section~\ref{sec: counter examples} of the Appendix,
        the characterization of LDPs in Theorem~\ref{theorem: LD for Hawkes}  w.r.t.\ \emph{the product $M_1$ topology of $\D[0,\infty)$
        is generally the tightest one can hope for in the multivariate setting}.
        The intuition of the counterexample is that, for any immigrant that arrives almost at the end of the time interval $[0,T]$, 
        its cluster size vector could be truncated arbitrarily at time $T$ due to the randomness in the birth times of the descendants, which prevents \eqref{claim, theorem: LD for Hawkes} to hold w.r.t.\ the product $M_1$ topology of $\D[0,T]$ for rare events set $B$ involving the value of the process at time $T$.
        In particular, we note that Example~\ref{example: LD for Hawkes, cadlag space with compact domain} is valid even if the decay functions $f^{\bm N}_{i \leftarrow j}(\cdot)$ in \eqref{def: conditional intensity, hawkes process} have bounded support or exponentially decaying tails.
        In other words,
         such pathological cases would still arise in  $( \D[0,T],\dmp{\subZeroT{T}})$ even if we strengthen the condition~\eqref{condtion: tail of fertility function hij} or impose exponentially decaying bounds on the tail of birth time distributions of the offspring.
\end{itemize}
\end{remark}

\begin{remark}[Tail Conditions on Decay Functions in Theorem~\ref{theorem: LD for Hawkes}]
\label{remark, tail condition on decay functions}
Even in the univariate setting, 
power-law bounds (like \eqref{condtion: tail of fertility function hij}) on the tail behavior of decay functions are required for Theorem~\ref{theorem: LD for Hawkes} to hold.
In particular, Example~\ref{example: necessity of the tail bound on decay functions} in Section~\ref{sec: counter examples} of the Appendix demonstrates that, as long as the decay function does have a power-law tail, 
one can always identify a simple class of counterexamples regardless of the actual values of the power-law tail indices in the excitation rates and the decay functions of the Hawkes process.
The intuition is that, without sufficiently tight tail bounds on the decay functions, 
some rare events are driven by a large cluster with a long lifetime (i.e., a single big jump spread too widely over the entire timeline), rather than by multiple large clusters.
\end{remark}

\subsection{Proof of Theorem~\ref{theorem: LD for Hawkes}}
\label{subsec: appendix, proof of main result, Hawkes}

Our proof strategy is to establish the asymptotic equivalence between 
$\bar{\bm N}^{ \subInfty }_n$, the scaled sample path of the Hawkes processes defined in \eqref{def: scaled Hawkes process, LD for Hawkes}, and some L\'evy process with $\MRV$ increments.
Specifically,
consider the multivariate compound Poisson process
\begin{align}
    \bm L(t) & \delequal 
    \sum_{ i \in [d] }\sum_{k \geq 0}
    \bm S^{(k)}_i\mathbbm{I}_{ [T^\mathcal{C}_{i;k},\infty) }(t),
    \qquad\forall t \geq 0.
    \label{def: Levy process, LD for Hawkes}
\end{align}
Here, independently
for each $i \in [d]$,
we use
$
0 < T^{\mathcal C}_{i;1} < T^{\mathcal C}_{i;2} < \ldots 
$
to denote a sequence generated by a Poisson process with constant rate $c^{\bm N}_i$ introduced in \eqref{def: conditional intensity, hawkes process}.
Besides, independent from the sequences $(T^{\mathcal C}_{ i;k })_{i\in[d],k \geq 1}$,
each $\bm S^{(k)}_j$ is an i.i.d.\ copy of $\bm S_j$ that solves distributional fixed-point equation \eqref{def: fixed point equation for cluster S i}.
The compound Poisson process $\bm L(t)$ in \eqref{def: Levy process, LD for Hawkes} is intimately related to $\bm N(t)$ through the cluster representation of Hawkes processes,
which we detail in Definition~\ref{def: offspring cluster process} of the Appendix
(see also \cite{209288c5-6a29-3263-8490-c192a9031603,20.500.11850/151886,daley2003introduction}).
In particular, we consider
a coupling between $\bm N(t)$ and $\bm L(t)$
such that $T^{\mathcal C}_{i;k}$ represents the arrival time of the $k^\text{th}$ type-$i$ immigrant in the Hawkes process $\bm N(t)$, and
each $\bm S^{(k)}_i$ is size vector of the cluster induced by the  $k^\text{th}$ type-$i$ immigrant.
From this perspective, the process  $\bm L(t)$ in \eqref{def: Levy process, LD for Hawkes} is obtained by collapsing any cluster in $\bm N(t)$ into a single jump, as if there is no gap between the birth times of the immigrant inducing the cluster and any offspring in this cluster.

More precisely,
our proof of Theorem~\ref{theorem: LD for Hawkes} relies on the following two propositions,
whose detailed proofs are provided in Sections~\ref{subsec: proof for propositions, LD for Hawkes, appendix}--\ref{subsec: proof for propositions, 2, LD for Hawkes, appendix} of the Appendix.
First, by exploiting the $\MRV^*$ tail asymptotics established in Theorem~\ref{theorem: main result, cluster size} for the cluster size vectors $\bm S_i$,
Proposition~\ref{proposition, law of the levy process approximation, LD for Hawkes} shows that 
$\bm L(t)$ is a L\'evy process with $\MRV^*$ increments.

\begin{proposition}\label{proposition, law of the levy process approximation, LD for Hawkes}
\linksinthm{proposition, law of the levy process approximation, LD for Hawkes}
Let Assumptions~\ref{assumption: subcriticality}--\ref{assumption: regularity condition 2, cluster size, July 2024} hold.
The process $\bm L(t)$ defined in \eqref{def: Levy process, LD for Hawkes}
is a L\'evy process with generating triplet $(\bm 0, \textbf 0, \nu)$ (i.e., with no linear drift and no Brownian motion component) and the L\'evy measure $\nu$ takes the form
\begin{align}
    \nu(\cdot) = \sum_{i \in [d]}c^{\bm N}_{i}\cdot\P(\bm S_i \in \ \cdot\ ),
    \label{claim, 1, proposition, law of the levy process approximation, LD for Hawkes}
\end{align}
which implies $\E \bm L(1) = \sum_{i \in [d]}c^{\bm N}_{i} \cdot \E \bm S_i = \bm \mu_{\bm N}$ (see \eqref{def: approximation, mean of N(t)}).
Furthermore,
\begin{align}
    \nu(\cdot) \in \MRV^*\Bigg(
       (\bar{\bm s}_i)_{i \in [d]},\ 
        \big(\alpha(\bm j)\big)_{ \bm j \subseteq [d] },\ 
        (\lambda_{\bm j})_{\bm j \in \powersetTilde{d}},\ 
        ( \mathbf C_{\bm j} )_{ \bm j \in \powersetTilde{d}  }
    \Bigg),
    \label{claim, 2, proposition, law of the levy process approximation, LD for Hawkes}
\end{align}
where $\bar{\bm s}_i =\E\bm S_i$, 
and
$\alpha(\cdot)$,
$\lambda_{\bm j}(\cdot)$,
$\mathbf C_{\bm j}(\cdot)$
are defined in 
\eqref{def: cost function, cone, cluster},
\eqref{def: rate function lambda j n, cluster size},
and \eqref{def: measure C indices j, sample path LD for Hawkes}, respectively.
\end{proposition}

Proposition~\ref{proposition, law of the levy process approximation, LD for Hawkes} allows us to apply Theorem~\ref{theorem: LD for levy, MRV} and obtain sample path LDPs of the form \eqref{claim, theorem: LD for Hawkes} for the scaled paths of $\bm L(t)$, i.e., (under $\bar{\bm L}_n(t) \delequal \bm L(nt)/n$),
\begin{align}
    \bar{\bm L}_n^{ \subZeroT{T}} \delequal
    \big\{
        \bar{\bm L}_n(t):\ t \in [0,T]
    \big\},
    \quad 
    \bar{\bm L}_n^{ \subInfty} \delequal
    \big\{
        \bar{\bm L}_n(t):\ t \in [0,\infty)
    \big\}.
     \label{def: scaled Levy process, LD for Hawkes}
\end{align}
Next, to close the loop for the proof of Theorem~\ref{theorem: LD for Hawkes},
we 
show that
$\bar{\bm N}^{\subInfty}_n$ and $\bar{\bm L}^{\subInfty}_n$ (almost always) stay close under the metric $\dmp{\subInfty}$.
In particular, by refining the approach in \cite{10.36045/bbms/1170347811},
 under condition~\eqref{condtion: tail of fertility function hij} we establish
 power-law tail bounds for the lifetime of Hawkes process clusters, i.e., the gap in the birth times between an immigrant and its last offspring.
 This leads to Proposition~\ref{proposition: asymptotic equivalence between Hawkes and Levy, LD for Hawkes} and verifies that, when collapsing each cluster in $\bar{\bm N}^{\subInfty}_n$ into a single jump and obtaining $\bar{\bm L}^{\subInfty}_n$, the perturbation under $\dmp{\subInfty}$ is inconsequential for the purpose of establishing the LDP~\eqref{claim, theorem: LD for Hawkes}.

\begin{proposition}\label{proposition: asymptotic equivalence between Hawkes and Levy, LD for Hawkes}
\linksinthm{proposition: asymptotic equivalence between Hawkes and Levy, LD for Hawkes}
Let Assumptions~\ref{assumption: subcriticality}--\ref{assumption: regularity condition 2, cluster size, July 2024} hold.
Let
$\bm k \in \mathbb Z_+^{ d } \setminus \{\bm 0\}$.
Under condition \eqref{condtion: tail of fertility function hij},
it holds for any $\Delta > 0$ that 
$$
\lim_{n \to \infty}
        \P\Big(
            \dmp{ \subInfty }\big(\bar{\bm N}^{\subInfty}_n, \bar{\bm L}^{\subInfty}_n\big)
            > \Delta
        \Big)
        \Big/
        \breve \lambda_{\bm k}(n)
        = 0,
$$
where $\breve \lambda_{\bm k}(\cdot)$ is defined in \eqref{def: scale function for Levy LD MRV}.
\end{proposition}

We add two remarks regarding the technical tools involved.
First, for the proof of Theorem~\ref{theorem: LD for Hawkes},
we need to adapt Theorem~\ref{theorem: LD for levy, MRV} to $\D[0,\infty)$.
The results are stated in Theorem~\ref{corollary: LD for Levy with MRV},
and
we detail the steps in Section~\ref{subsec: corollary, LD for Levy, unbounded domain} of the Appendix.
Second, the asymptotic equivalence between 
$\bar{\bm N}^{\subInfty}_n$ and $\bar{\bm L}^{\subInfty}_n$ 
is stated in terms of the $\mathbb M$-convergence (\cite{10.1214/14-PS231}) over the metric space $\big( \D[0,\infty), \dmp{[0,\infty)} \big)$,
and is made precise by our Lemma~\ref{lemma: asymptotic equivalence when bounded away, equivalence of M convergence} in the Appendix.
In particular,
under the choice of 
\begin{align*}
    (\mathbb S,\bm d) = \big( \D[0,\infty), \dmp{\subInfty} \big),\quad 
    \mathbb C = \barDxepskT{\bm \mu_{\bm N}}{\epsilon}{\leqslant \bm k}{[0,\infty)},\quad
    X_n = \bar{\bm N}_n^{\subInfty},\quad 
    Y_n^\delta = \bar{\bm L}_n^{\subInfty}\ \forall \delta > 0,
\end{align*}
and with a dummy marker $V^\delta_n \equiv 1,\ \mathcal V = \{1\}$ in
Lemma~\ref{lemma: asymptotic equivalence when bounded away, equivalence of M convergence},
to prove 
Claim~\eqref{claim, theorem: LD for Hawkes}
it suffices to verify the following two conditions
(given our choice of $Y^\delta_n$, we also write $Y_n = Y^\delta_n = \bar{\bm L}_n^{\subInfty}$):
\begin{enumerate}[(i)]
    \item
        The claim
                $
        \lim_{ n \rightarrow \infty }  {\P\Big( \bm{d}\big(X_n, Y_n\big)
   > \Delta \Big) }\Big/\breve \lambda_{\bm k}(n)= 0
        $
        holds for each $\Delta > 0$ and Borel set $B \subseteq \mathbb S$ that is bounded away from $\mathbb{C}$;


    \item 
         For each Borel set $B \subseteq \mathbb S$ that is bounded away from $\mathbb{C}$,
        we have
        $
         \limsup_{n \to \infty}\P(Y_n \in B)/\breve \lambda_{\bm k}(n) \leq \mu(B^-) < \infty
        $
        and
        $
        \liminf_{n \to \infty}\P(Y_n \in B)/\breve \lambda_{\bm k}(n) \geq \mu(B^\circ),
        $
    where $\mu(\cdot) =  \sum_{ \bm{\mathcal K} \in \mathbb A(\bm k) }\breve{\mathbf C}^{\subInfty }_{\bm{\mathcal K};\bm \mu_{\bm N}}(\cdot)$.
\end{enumerate}

\noindent
To conclude, we provide the---now succinct---proof of Theorem~\ref{theorem: LD for Hawkes}, and refer to Section~\ref{sec: appendix, theorem tree} for the full theorem tree.

\begin{proof}[Proof of Theorem~\ref{theorem: LD for Hawkes}]
\linksinpf{theorem: LD for Hawkes}
As noted above, by Lemma~\ref{lemma: asymptotic equivalence when bounded away, equivalence of M convergence}, it suffices to verify Conditions (i) and (ii).
Proposition~\ref{proposition: asymptotic equivalence between Hawkes and Levy, LD for Hawkes} verifies Condition (i), and for Condition (ii) it is equivalent to  verifying the following claim:
let
    $\bm k = (k_{j})_{ j \in [d] } \in \mathbb Z_+^{ d  } \setminus \{\bm 0\}$,
    and $B$ be a Borel set of $\D[0,\infty)$ equipped with the product $M_1$ topology;
    if $B$ is bounded away from $\barDxepskT{\bm\mu_{\bm N}}{\epsilon}{\leqslant \bm k}{[0,\infty)}$ under $\dmp{[0,\infty)}$
    for some $\epsilon > 0$,
    then,
\begin{align}
    \sum_{ \bm{\mathcal K} \in \mathbb A(\bm k) }\breve{\mathbf C}^{\subInfty }_{\bm{\mathcal K};\bm \mu_{\bm N}}(B^\circ)
        \leq 
        \liminf_{n \to \infty}
        \frac{
        \P\big(\bar{\bm L}^{_{[0,\infty)}}_n \in B\big)
        }{
            \breve\lambda_{\bm k}(n)
        }
        \leq 
        \limsup_{n \to \infty}
        \frac{
        \P\big(\bar{\bm L}^{_{[0,\infty)}}_n \in B\big)
        }{
            \breve\lambda_{\bm k}(n)
        }
        \leq 
        \sum_{ \bm{\mathcal K} \in \mathbb A(\bm k) }\breve{\mathbf C}^{\subInfty }_{\bm{\mathcal K};\bm \mu_{\bm N}}(B^-) < \infty.
    \label{proof, goal, theorem: LD for Hawkes}
\end{align}
This is almost an immediate consequence of Theorem~\ref{theorem: LD for levy, MRV},
except that we need to adapt it to $\D[0,\infty)$---the space of $\R^d$-valued càdlàg functions with unbounded domain.
To this end, we apply
Theorem~\ref{corollary: LD for Levy with MRV}---the $\big( \D[0,\infty), \dj{[0,\infty)} \big)$ counterpart of Theorem~\ref{theorem: LD for levy, MRV}---to obtain sample path LDPs of the form \eqref{proof, goal, theorem: LD for Hawkes} for $\bar{\bm L}^{\subInfty}_n$.
Here, the metric is defined by 
\begin{align*}
     {\dj{[0,\infty)}(\xi^{(1)},\xi^{(2)})}
    \delequal
    \int_0^\infty 
        e^{-t} \cdot \Big[ \dj{[0,t]}\Big( \phi_t\big(\xi^{(1)}),\ \phi_t\big(\xi^{(2)}\big)\Big)  \wedge 1 \Big]dt,
    \qquad
    \forall \xi^{(1)},\xi^{(2)} \in \D[0,\infty),
\end{align*}
with $\phi_t(\xi)(s) = \xi(s)$ being the projection mapping from $\D[0,\infty)$ to $\D[0,t]$.
More specifically, the finite upper bound in \eqref{proof, goal, theorem: LD for Hawkes} is verified by Theorem~\ref{corollary: LD for Levy with MRV}.
Besides,
by Theorem~\ref{corollary: LD for Levy with MRV} 
and Proposition~\ref{proposition, law of the levy process approximation, LD for Hawkes},
Claim~\eqref{proof, goal, theorem: LD for Hawkes} holds if $B$ is also bounded away from $\barDxepskT{\bm \mu_{\bm N}}{\epsilon}{\leqslant \bm k}{[0,\infty)}$ under $\dj{\subInfty}$.
However, due to $\dmp{ \subZeroT{T} }(\xi,\xi^\prime) \leq \dj{\subZeroT{T}}(\xi,\xi^\prime)$ for each $T \in (0,\infty)$ and $\xi,\xi^\prime \in \D[0,T]$ (i.e., the hierarchy of Skorokhod non-uniform topologies),
we have
$
\dmp{\subInfty}(\xi,\xi^\prime) \leq \dj{\subInfty}(\xi,\xi^\prime)
$
for any $\xi,\xi^\prime \in \D[0,\infty)$,
and hence
\begin{align*}
    \dj{\subInfty}\Big( B,\ \barDxepskT{\bm \mu_{\bm N}}{\epsilon}{\leqslant \bm k}{[0,\infty)} \Big)
    \geq
    \dmp{\subInfty}\Big( B,\  \barDxepskT{\bm \mu_{\bm N}}{\epsilon}{\leqslant \bm k}{[0,\infty)} \Big) > 0.
\end{align*}
Here, the strictly positive lower bound holds since, by our running assumption, we take $B$ that is bounded away from $\barDxepskT{\bm \mu_{\bm N}}{\epsilon}{\leqslant \bm k}{[0,\infty)}$ under $\dmp{\subInfty}$.
This concludes the proof.
\end{proof}





\bibliographystyle{abbrv} 
\bibliography{main} 

\begin{thebibliography}{10}

\bibitem{Albrecher_Chen_Vatamidou_Zwart_2020}
H.~Albrecher, B.~Chen, E.~Vatamidou, and B.~Zwart.
\newblock Finite-time ruin probabilities under large-claim reinsurance treaties for heavy-tailed claim sizes.
\newblock {\em Journal of Applied Probability}, 57(2):513--530, 2020.

\bibitem{Asmussen_Foss_2018}
S.~Asmussen and S.~Foss.
\newblock Regular variation in a fixed-point problem for single- and multi-class branching processes and queues.
\newblock {\em Advances in Applied Probability}, 50(A):47–61, 2018.

\bibitem{AITSAHALIA2015585}
Y.~Aït-Sahalia, J.~Cacho-Diaz, and R.~J.~A. Laeven.
\newblock Modeling financial contagion using mutually exciting jump processes.
\newblock {\em Journal of Financial Economics}, 117(3):585--606, 2015.

\bibitem{BACRY20132475}
E.~Bacry, S.~Delattre, M.~Hoffmann, and J.-F. Muzy.
\newblock Some limit theorems for {H}awkes processes and application to financial statistics.
\newblock {\em Stochastic Processes and their Applications}, 123(7):2475--2499, 2013.

\bibitem{Bacry02082016}
E.~Bacry, T.~Jaisson, and J.-F. Muzy.
\newblock Estimation of slowly decreasing {H}awkes kernels: {A}pplication to high-frequency order book dynamics.
\newblock {\em Quantitative Finance}, 16(8):1179--1201, 2016.

\bibitem{baeriswyl2023tail}
F.~Baeriswyl, V.~Chavez-Demoulin, and O.~Wintenberger.
\newblock Tail asymptotics and precise large deviations for some {P}oisson cluster processes.
\newblock {\em Advances in Applied Probability}, page 1–37, 2024.

\bibitem{baldwin2017contagion}
A.~Baldwin, I.~Gheyas, C.~Ioannidis, D.~Pym, and J.~Williams.
\newblock Contagion in cyber security attacks.
\newblock {\em Journal of the Operational Research Society}, 68(7):780--791, 2017.

\bibitem{Basrak02102013}
B.~Basrak, R.~Kulik, and Z.~Palmowski.
\newblock Heavy-tailed branching process with immigration.
\newblock {\em Stochastic Models}, 29(4):413--434, 2013.

\bibitem{bernhard2020heavy}
H.~Bernhard and B.~Das.
\newblock {Heavy-tailed random walks, buffered queues and hidden large deviations}.
\newblock {\em Bernoulli}, 26(1):61--92, 2020.

\bibitem{Bessy-Roland_Boumezoued_Hillairet_2021}
Y.~Bessy-Roland, A.~Boumezoued, and C.~Hillairet.
\newblock Multivariate {H}awkes process for cyber insurance.
\newblock {\em Annals of Actuarial Science}, 15(1):14–39, 2021.

\bibitem{bingham1989regular}
N.~H. Bingham, C.~M. Goldie, and J.~L. Teugels.
\newblock {\em Regular Variation}.
\newblock Number~27. Cambridge University Press, 1989.

\bibitem{blanchet2025tailasymptoticsclustersizes}
J.~Blanchet, R.~J.~A. Laeven, X.~Wang, and B.~Zwart.
\newblock Tail asymptotics of cluster sizes in multivariate heavy-tailed {H}awkes processes, 2025.
\newblock Preprint. Available at arXiv: \href{https://arxiv.org/abs/2503.01004}{2503.01004}.

\bibitem{Bordenave07112007}
C.~Bordenave and G.~L. Torrisi.
\newblock Large deviations of poisson cluster processes.
\newblock {\em Stochastic Models}, 23(4):593--625, 2007.

\bibitem{borovkov2011large}
A.~A. Borovkov.
\newblock Large deviation principles for random walks with regularly varying distributions of jumps.
\newblock {\em Siberian Mathematical Journal}, 52(3):402--410, 2011.

\bibitem{borovkov_borovkov_2008}
A.~A. Borovkov and K.~A. Borovkov.
\newblock {\em Asymptotic Analysis of Random Walks: Heavy-Tailed Distributions}.
\newblock Encyclopedia of Mathematics and its Applications. Cambridge University Press, Cambridge, 2008.

\bibitem{1e405749-c2fa-3fe8-ae93-09b947fe302a}
P.~Brémaud and L.~Massoulié.
\newblock Stability of nonlinear {H}awkes processes.
\newblock {\em The Annals of Probability}, 24(3):1563--1588, 1996.

\bibitem{4fc30f5f-894d-3c6c-9b15-01f8f8d74820}
P.~Carr, H.~Geman, D.~B. Madan, and M.~Yor.
\newblock The fine structure of asset returns: An empirical investigation.
\newblock {\em The Journal of Business}, 75(2):305--332, 2002.

\bibitem{https://doi.org/10.1111/1467-9965.00020}
P.~Carr, H.~Geman, D.~B. Madan, and M.~Yor.
\newblock Stochastic volatility for {L}évy processes.
\newblock {\em Mathematical Finance}, 13(3):345--382, 2003.

\bibitem{chen2019efficient}
B.~Chen, J.~Blanchet, C.-H. Rhee, and B.~Zwart.
\newblock Efficient rare-event simulation for multiple jump events in regularly varying random walks and compound poisson processes.
\newblock {\em Mathematics of Operations Research}, 44(3):919--942, 2019.

\bibitem{10.1214/24-EJP1115}
B.~Chen, C.-H. Rhee, and B.~Zwart.
\newblock {Sample-path large deviations for a class of heavy-tailed Markov-additive processes}.
\newblock {\em Electronic Journal of Probability}, 29(none):1--44, 2024.

\bibitem{doi:10.1287/stsy.2021.0070}
X.~Chen.
\newblock Perfect sampling of {H}awkes processes and queues with {H}awkes arrivals.
\newblock {\em Stochastic Systems}, 11(3):264--283, 2021.

\bibitem{chiang2022hawkes}
W.-H. Chiang, X.~Liu, and G.~Mohler.
\newblock Hawkes process modeling of covid-19 with mobility leading indicators and spatial covariates.
\newblock {\em International Journal of Forecasting}, 38(2):505--520, 2022.

\bibitem{doi:10.1073/pnas.2209234119}
J.~E. Cohen, R.~A. Davis, and G.~Samorodnitsky.
\newblock Covid-19 cases and deaths in the united states follow taylor’s law for heavy-tailed distributions with infinite variance.
\newblock {\em Proceedings of the National Academy of Sciences}, 119(38):e2209234119, 2022.

\bibitem{doi:10.1073/pnas.0803685105}
R.~Crane and D.~Sornette.
\newblock Robust dynamic classes revealed by measuring the response function of a social system.
\newblock {\em Proceedings of the National Academy of Sciences}, 105(41):15649--15653, 2008.

\bibitem{dj1972summary}
D.~J. Daley and D.~Vere-Jones.
\newblock A summary of the theory of point processes.
\newblock {\em Stochastic Point Processes: Statistical Analysis, Theory and Applications, PAW Lewis, ed., Wiley}, pages 299--383, 1972.

\bibitem{daley2003introduction}
D.~J. Daley and D.~Vere-Jones.
\newblock {\em An Introduction to the Theory of Point Processes: Volume I: Elementary Theory and Methods, Volume II: General Theory and Structure}.
\newblock Springer, 2003.

\bibitem{das2023aggregatingheavytailedrandomvectors}
B.~Das and V.~Fasen-Hartmann.
\newblock Aggregating heavy-tailed random vectors: {F}rom finite sums to {L}\'evy processes, 2023.
\newblock Preprint. Available at arXiv: \href{https://arxiv.org/abs/2301.10423}{2301.10423}.

\bibitem{doi:10.1287/stsy.2018.0014}
A.~Daw and J.~Pender.
\newblock Queues driven by {H}awkes processes.
\newblock {\em Stochastic Systems}, 8(3):192--229, 2018.

\bibitem{MR2571413}
A.~Dembo and O.~Zeitouni.
\newblock {\em Large Deviations Techniques and Applications}, volume~38 of {\em Stochastic Modelling and Applied Probability}.
\newblock Springer-Verlag, Berlin, New York, NY, 2010.
\newblock Corrected reprint of the second (1998) edition.

\bibitem{denisov2008large}
D.~Denisov, A.~B. Dieker, and V.~Shneer.
\newblock {Large deviations for random walks under subexponentiality: The big-jump domain}.
\newblock {\em The Annals of Probability}, 36(5):1946--1991, 2008.

\bibitem{MR997938}
J.-D. Deuschel and D.~W. Stroock.
\newblock {\em Large deviations}, volume 137 of {\em Pure and Applied Mathematics}.
\newblock Academic Press, Inc., Boston, MA, 1989.

\bibitem{10.1214/21-AAP1675}
C.~Dombry, C.~Tillier, and O.~Wintenberger.
\newblock {Hidden regular variation for point processes and the single/multiple large point heuristic}.
\newblock {\em The Annals of Applied Probability}, 32(1):191--234, 2022.

\bibitem{embrechts2013modelling}
P.~Embrechts, C.~Kl{\"u}ppelberg, and T.~Mikosch.
\newblock {\em Modelling Extremal Events: for Insurance and Finance}, volume~33.
\newblock Springer Science \& Business Media, New York, NY, 2013.

\bibitem{ernst2018stability}
P.~A. Ernst, S.~Asmussen, and J.~J. Hasenbein.
\newblock Stability and busy periods in a multiclass queue with state-dependent arrival rates.
\newblock {\em Queueing Systems}, 90:207--224, 2018.

\bibitem{MR2260560}
J.~Feng and T.~G. Kurtz.
\newblock {\em Large Deviations for Stochastic Processes}, volume 131 of {\em Mathematical Surveys and Monographs}.
\newblock American Mathematical Society, Providence, RI, 2006.

\bibitem{doi:10.1287/moor.1120.0539}
S.~Foss and D.~Korshunov.
\newblock On large delays in multi-server queues with heavy tails.
\newblock {\em Mathematics of Operations Research}, 37(2):201--218, 2012.

\bibitem{foss2011introduction}
S.~Foss, D.~Korshunov, and S.~Zachary.
\newblock {\em An Introduction to Heavy-tailed and Subexponential Distributions}, volume~6.
\newblock Springer, 2011.

\bibitem{guo2025precise}
J.~Guo and W.~Hong.
\newblock Precise large deviations for the total population of heavy-tailed subcritical branching processes with immigration.
\newblock {\em Journal of Theoretical Probability}, 38(1):1--24, 2025.

\bibitem{haccou2005branching}
P.~Haccou, P.~Jagers, and V.~A. Vatutin.
\newblock {\em Branching Processes: Variation, Growth, and Extinction of Populations}.
\newblock Number~5. Cambridge University Press, 2005.

\bibitem{hardiman2013critical}
S.~J. Hardiman, N.~Bercot, and J.-P. Bouchaud.
\newblock Critical reflexivity in financial markets: {A} {H}awkes process analysis.
\newblock {\em The European Physical Journal B}, 86:1--9, 2013.

\bibitem{Hawkes01022018}
A.~G. Hawkes.
\newblock Hawkes processes and their applications to finance: {A} review.
\newblock {\em Quantitative Finance}, 18(2):193--198, 2018.

\bibitem{209288c5-6a29-3263-8490-c192a9031603}
A.~G. Hawkes and D.~Oakes.
\newblock A cluster process representation of a self-exciting process.
\newblock {\em Journal of Applied Probability}, 11(3):493--503, 1974.

\bibitem{horst2024convergenceheavytailedhawkesprocesses}
U.~Horst, W.~Xu, and R.~Zhang.
\newblock Convergence of heavy-tailed {H}awkes processes and the microstructure of rough volatility, 2024.

\bibitem{hult2005functional}
H.~Hult, F.~Lindskog, T.~Mikosch, and G.~Samorodnitsky.
\newblock {Functional large deviations for multivariate regularly varying random walks}.
\newblock {\em The Annals of Applied Probability}, 15(4):2651--2680, 2005.

\bibitem{Ikefuji02012022}
M.~Ikefuji, R.~J.~A. Laeven, J.~R. Magnus, and Y.~Yue.
\newblock Earthquake risk embedded in property prices: Evidence from five japanese cities.
\newblock {\em Journal of the American Statistical Association}, 117(537):82--93, 2022.

\bibitem{10.1214/15-AAP1164}
T.~Jaisson and M.~Rosenbaum.
\newblock Rough fractional diffusions as scaling limits of nearly unstable heavy tailed {H}awkes processes.
\newblock {\em The Annals of Applied Probability}, 26(5):2860--2882, 2016.

\bibitem{JOFFE1967409}
A.~Joffe and F.~Spitzer.
\newblock On multitype branching processes with $\rho \leq 1$.
\newblock {\em Journal of Mathematical Analysis and Applications}, 19(3):409--430, 1967.

\bibitem{Karabash03072015}
D.~Karabash and L.~Zhu.
\newblock Limit theorems for marked {H}awkes processes with application to a risk model.
\newblock {\em Stochastic Models}, 31(3):433--451, 2015.

\bibitem{karim2021exact}
R.~S. Karim, R.~J.~A. Laeven, and M.~R.~H. Mandjes.
\newblock Exact and asymptotic analysis of general multivariate {H}awkes processes and induced population processes, 2021.
\newblock Preprint. Available at arXiv: \href{https://arxiv.org/abs/2106.03560}{2106.03560}.

\bibitem{karim2023compound}
R.~S. Karim, R.~J.~A. Laeven, and M.~R.~H. Mandjes.
\newblock Compound multivariate {H}awkes processes: Large deviations and rare event simulation.
\newblock {\em Bernoulli}, 2024.
\newblock In press.

\bibitem{KEVEI2021109067}
P.~Kevei and P.~Wiandt.
\newblock Moments of the stationary distribution of subcritical multitype {G}alton-{W}atson processes with immigration.
\newblock {\em Statistics \& Probability Letters}, 173:109067, 2021.

\bibitem{LAMBERT20189}
R.~C. Lambert, C.~Tuleau-Malot, T.~Bessaih, V.~Rivoirard, Y.~Bouret, N.~Leresche, and P.~Reynaud-Bouret.
\newblock Reconstructing the functional connectivity of multiple spike trains using {H}awkes models.
\newblock {\em Journal of Neuroscience Methods}, 297:9--21, 2018.

\bibitem{10.1214/14-PS231}
F.~Lindskog, S.~I. Resnick, and J.~Roy.
\newblock {Regularly varying measures on metric spaces: Hidden regular variation and hidden jumps}.
\newblock {\em Probability Surveys}, 11(none):270--314, 2014.

\bibitem{20.500.11850/151886}
T.~Liniger.
\newblock {\em Multivariate {H}awkes processes}.
\newblock Doctoral thesis, ETH Zurich, Zürich, 2009.
\newblock Diss., Eidgenössische Technische Hochschule ETH Zürich, Nr. 18403, 2009.

\bibitem{Ogata01031988}
Y.~Ogata.
\newblock Statistical models for earthquake occurrences and residual analysis for point processes.
\newblock {\em Journal of the American Statistical Association}, 83(401):9--27, 1988.

\bibitem{9378017}
J.~Olinde and M.~B. Short.
\newblock A self-limiting {H}awkes process: Interpretation, estimation, and use in crime modeling.
\newblock In {\em 2020 IEEE International Conference on Big Data (Big Data)}, pages 3212--3219, 2020.

\bibitem{10.1145/2808797.2814178}
J.~C.~L. Pinto, T.~Chahed, and E.~Altman.
\newblock Trend detection in social networks using {H}awkes processes.
\newblock In {\em Proceedings of the 2015 IEEE/ACM International Conference on Advances in Social Networks Analysis and Mining 2015}, ASONAM '15, page 1441–1448, New York, NY, USA, 2015. Association for Computing Machinery.

\bibitem{protter2005stochastic}
P.~E. Protter.
\newblock {\em Stochastic Integration and Differential Equations}.
\newblock Springer Berlin, 2005.

\bibitem{resnick2002hidden}
S.~Resnick.
\newblock Hidden regular variation, second order regular variation and asymptotic independence.
\newblock {\em Extremes}, 5:303--336, 2002.

\bibitem{resnick2007heavy}
S.~I. Resnick.
\newblock {\em Heavy-Tail Phenomena: Probabilistic and Statistical Modeling}.
\newblock Springer Science \& Business Media, 2007.

\bibitem{10.36045/bbms/1170347811}
P.~Reynaud-Bouret and E.~Roy.
\newblock {Some non asymptotic tail estimates for {H}awkes processes}.
\newblock {\em Bulletin of the Belgian Mathematical Society - Simon Stevin}, 13(5):883--896, 2007.

\bibitem{10.1214/10-AOS806}
P.~Reynaud-Bouret and S.~Schbath.
\newblock {Adaptive estimation for {H}awkes processes; application to genome analysis}.
\newblock {\em The Annals of Statistics}, 38(5):2781--2822, 2010.

\bibitem{10.1214/18-AOP1319}
C.-H. Rhee, J.~Blanchet, and B.~Zwart.
\newblock {Sample path large deviations for Lévy processes and random walks with regularly varying increments}.
\newblock {\em The Annals of Probability}, 47(6):3551--3605, 2019.

\bibitem{rizoiu2017tutorialhawkesprocessesevents}
M.-A. Rizoiu, Y.~Lee, S.~Mishra, and L.~Xie.
\newblock A tutorial on {H}awkes processes for events in social media, 2017.
\newblock Preprint. Available at arXiv: \href{https://arxiv.org/abs/1708.06401}{1708.06401}.

\bibitem{risks10080148}
P.~Sabino.
\newblock Pricing energy derivatives in markets driven by tempered stable and {CGMY} processes of {O}rnstein–{U}hlenbeck type.
\newblock {\em Risks}, 10(8), 2022.

\bibitem{sato1999levy}
K.-I. Sato.
\newblock {\em L{\'e}vy Processes and Infinitely Divisible Distributions}.
\newblock Cambridge University Press, 1999.

\bibitem{selvamuthu2022infinite}
D.~Selvamuthu and P.~Tardelli.
\newblock Infinite-server systems with {H}awkes arrivals and {H}awkes services.
\newblock {\em Queueing Systems}, 101(3):329--351, 2022.

\bibitem{stabile2010risk}
G.~Stabile and G.~L. Torrisi.
\newblock Risk processes with non-stationary {H}awkes claims arrivals.
\newblock {\em Methodology and Computing in Applied Probability}, 12:415--429, 2010.

\bibitem{tankov2003financial}
P.~Tankov.
\newblock {\em Financial Modelling with Jump Processes}.
\newblock Chapman and Hall/CRC, Boca Raton, FL, 2003.

\bibitem{MR758258}
S.~R.~S. Varadhan.
\newblock {\em Large Deviations and Applications}, volume~46 of {\em CBMS-NSF Regional Conference Series in Applied Mathematics}.
\newblock Society for Industrial and Applied Mathematics (SIAM), Philadelphia, PA, 1984.

\bibitem{wang2022eliminating}
X.~Wang, S.~Oh, and C.-H. Rhee.
\newblock Eliminating sharp minima from {SGD} with truncated heavy-tailed noise.
\newblock In {\em International Conference on Learning Representations}, 2022.

\bibitem{wang2023large}
X.~Wang and C.-H. Rhee.
\newblock Large deviations and metastability analysis for heavy-tailed dynamical systems, 2024.
\newblock Preprint. Available at arXiv: \href{https://arxiv.org/abs/2307.03479}{2307.03479}.

\bibitem{whitt2002stochastic}
W.~Whitt.
\newblock {\em Stochastic-Process Limits: {A}n Introduction to Stochastic-Process Limits and Their Application to Queues}.
\newblock Springer, 2002.

\bibitem{Xu2005}
W.~Xu.
\newblock {Diffusion approximations for self-excited systems with applications to general branching processes}.
\newblock {\em The Annals of Applied Probability}, 34(3):2650--2713, 2024.

\bibitem{AIHPB_2014__50_3_845_0}
L.~Zhu.
\newblock Process-level large deviations for nonlinear {H}awkes point processes.
\newblock {\em Annales de l'I.H.P. Probabilit\'es et Statistiques}, 50(3):845--871, 2014.

\bibitem{10.1214/14-AAP1003}
L.~Zhu.
\newblock {Large deviations for {M}arkovian nonlinear {H}awkes processes}.
\newblock {\em The Annals of Applied Probability}, 25(2):548--581, 2015.

\end{thebibliography}

\newpage
\appendix

\section{$\mathbb M(\mathbb S\setminus\mathbb C)$-Convergence and Asymptotic Equivalence}
\label{subsec: proof, M convergence and asymptotic equivalence}

In this appendix, we first review the notion of $\mathbb M(\mathbb S \setminus \mathbb C)$-convergence in \cite{10.1214/14-PS231},
which has been a key tool in large deviations analyses of heavy-tailed stochastic systems \cite{10.1214/14-PS231,10.1214/18-AOP1319,10.1214/24-EJP1115,blanchet2025tailasymptoticsclustersizes}.
Next, we establish an asymptotic equivalence result.
Let $(\mathbb{S},\bm{d})$ be a complete and separable metric space.
Given Borel measurable sets $A,B \subseteq \mathbb S$,
$A$ is said to be bounded away from $B$ (under $\bm d$) if 
$
\bm d(A,B) \delequal \inf_{ x \in A, y \in B }\bm d(x,y) > 0.
$
Let $\mathscr{S}$ be the $\sigma$-algebra of $\mathbb S$.
Given a Borel measurable subset $\mathbb{C} \subseteq \mathbb{S}$,
let $\mathbb{S}\setminus \mathbb{C}$
be the metric subspace of $\mathbb{S}$ in the relative topology with
$\sigma$-algebra
$
\mathscr{S}_{ \mathbb{S}\setminus \mathbb{C}}
\delequal 
\{ A \in \mathscr{S}_{\mathbb{S}}:\ A \subseteq \mathbb{S}\setminus \mathbb{C}\}.
$
The space of measures
\begin{align*}
    \notationdef{notation-M-S-exclude-C}{\mathbb{M}(\S\setminus \mathbb{C})}
    \delequal 
    \{
    \nu(\cdot)\text{ is a Borel measure on }\mathbb{S}\setminus \mathbb{C} :\ \nu(\mathbb{S}\setminus \mathbb{C}^r) < \infty\ \forall r > 0
    \}.
\end{align*}
can be topologized by the sub-basis constructed using sets of the form
$
\{
\nu \in \mathbb{M}(\mathbb{S}\setminus \mathbb{C}):\ \nu(f) \in G
\},
$
where $G \subseteq [0,\infty)$ is open, $f \in \mathcal{C}({ \mathbb{S}\setminus \mathbb{C} })$,
and
$
\notationdef{notation-mathcal-C-S-exclude-C}{\mathcal{C}({ \mathbb{S}\setminus \mathbb{C} })}
$
is the set of all real-valued, non-negative, bounded and continuous functions with support bounded away from $\mathbb{C}$ (i.e., there exists $r >0$ such that $f(x) = 0\ \forall x \in \mathbb{C}^r$).
Now, we recall the concept of $\mathbb M(\mathbb S \setminus \mathbb C)$-convergence.

\begin{definition}[$\mathbb M(\mathbb S \setminus \mathbb C)$-convergence]
\label{def: M convergence}
Given a sequence $\mu_n \in \M(\S\setminus\C)$ and some $\mu \in \M(\S\setminus\C)$,
we say that \textbf{$\mu_n$ converges to $\mu$ in $\M(\S\setminus\C)$} as $n \to \infty$
if
\begin{align*}
    \lim_{n \to \infty}|\mu_n(f) - \mu(f)| = 0,
    \qquad
    \forall f \in \mathcal C(\mathbb S \setminus \mathbb C).
\end{align*}
\end{definition}
When there is no ambiguity about the choice of $\mathbb S$ and $\mathbb C$, we simply refer to the convergence mode in Definition~\ref{def: M convergence} as $\mathbb M$-convergence.
Next, we review the Portmanteau Theorem for $\mathbb M$-convergence.

\begin{theorem}[Theorem 2.1 of \cite{10.1214/14-PS231}]
\label{portmanteau theorem M convergence}
 Let  $\mu_n, \mu \in \M(\S\setminus\C)$.
 The following statements are equivalent.
\begin{enumerate}[(i)]
    \item
        $\mu_n \to \mu$ in $\M(\S\setminus\C)$ as $n \to \infty$.

    \item
        $\int f d\mu_n \rightarrow \int f d\mu$ for any $f \in \mathcal C(\mathbb S\setminus\mathbb C)$ that is also uniformly continuous on $\mathbb S$.

    \item
        For any closed set $F$ and open set $G$ that are bounded away from $\mathbb C$,
        \begin{align*}
    \limsup_{n \to \infty}\mu_n(F) \leq \mu(F),
    \qquad
    \liminf_{n \to \infty}\mu_n(G) \geq \mu(G).
\end{align*}
\end{enumerate}
\end{theorem}

The verification of $\mathbb M$-convergence
is often facilitated by
the asymptotic equivalence between two families of random objects. 
In the proof of this paper, we work with the version of  asymptotic equivalence in Lemma~\ref{lemma: asymptotic equivalence when bounded away, equivalence of M convergence}, which is in the same spirit as 
Lemma~4.2 of \cite{blanchet2025tailasymptoticsclustersizes} with only two key differences.
First, 
Lemma~\ref{lemma: asymptotic equivalence when bounded away, equivalence of M convergence} addresses abstract metric spaces, whereas Lemma~4.2 of \cite{blanchet2025tailasymptoticsclustersizes} specifically applies to the space of polar coordinates.
Second, we require condition (ii) to hold only for \emph{almost} all $\delta > 0$ sufficiently close to $0$.

\begin{lemma}
\label{lemma: asymptotic equivalence when bounded away, equivalence of M convergence}
\linksinthm{lemma: asymptotic equivalence when bounded away, equivalence of M convergence}
Let $X_n$ and $Y^\delta_n$ be random elements taking values in a complete and separable metric space $(\mathbb{S},\bm{d})$,
and $V_n^\delta$ be random elements taking values in a countable set $\mathbb V$.
Furthermore, 
given $n \geq 1$, $X_n$, $Y_n^\delta$, and $V^\delta_n$ (for any $\delta > 0$) are supported on the same probability space.
Let $\epsilon_n$ be a sequence of positive real numbers with $\lim_{n \to \infty}\epsilon_n = 0$.
Let $\mathbb C \subseteq \mathbb S$,
let
$\mathcal V \subset \mathbb V$ be a set containing only finitely many elements,
and  let $\mu_v \in \mathbb M(\mathbb S \setminus \mathbb C)$ for each $v \in \mathcal V$.
Suppose that 
\begin{enumerate}[(i)]
    \item
        (\textbf{Asymptotic equivalence})
        Given $\Delta > 0$ and $B \in \mathscr{S}_\mathbb{S}$ that is bounded away from $\mathbb{C}$,
        \begin{align*}
    \lim_{\delta \downarrow 0}\lim_{ n \rightarrow \infty } \epsilon^{-1}_n {\P\Big( \bm{d}\big(X_n, Y^\delta_n\big)\mathbbm{I}\big( X_n \in B\text{ or }Y^\delta_n \in B \big) > \Delta \Big) }= 0;
\end{align*}
        

    \item
        (\textbf{$\mathbb M(\mathbb S\setminus\mathbb C)$-Convergence})
        Given $B \in \mathscr{S}_\mathbb{S}$ that is bounded away from $\mathbb{C}$
        and $\Delta \in \big(0,\bm d(B,\mathbb C)\big)$,
        there exists some $\delta_0 = \delta_0(B,\Delta) > 0$
        such that for all but countably many $\delta \in (0,\delta_0)$,
    \begin{align*}
       \limsup_{n \to \infty}\epsilon^{-1}_n\P(Y^\delta_n \in B, V^\delta_n = v) & \leq \mu_v(B^{\Delta}),
        \ \ 
        \liminf_{n \to \infty}\epsilon^{-1}_n\P(Y^\delta_n \in B, V^\delta_n = v) \geq \mu_v(B_\Delta),
        \quad\forall v \in \mathcal V,
        \\
       & \lim_{n \to \infty}\epsilon^{-1}_n\P(Y^\delta_n \in B,\ V^\delta_n \notin \mathcal V) = 0.
    \end{align*}
\end{enumerate}
Then
$\epsilon^{-1}_n \P(X_n \in\ \cdot\ ) \rightarrow \sum_{v \in \mathcal V}\mu_v(\cdot)$ in $\mathbb{M}(\mathbb{S}\setminus \mathbb{C})$.
\end{lemma}

\begin{proof}[Proof of Lemma~\ref{lemma: asymptotic equivalence when bounded away, equivalence of M convergence}]
\linksinpf{lemma: asymptotic equivalence when bounded away, equivalence of M convergence}
Take any Borel measurable $B \subseteq \mathbb S$ that is bounded away from $\mathbb C$.
W.l.o.g., in this proof we only consider $\Delta > 0$ small enough
that $\bm d(B,\mathbb C) > 2\Delta$, and hence $B^{2\Delta}$ is still bounded away from $\mathbb C$.

First, for any $n \geq 1$ and $\delta > 0$,
\begin{align*}
    & \P(X_n \in B)
    \\
    & = 
    \P\big(
        X_n \in B;\ \bm d( X_n, Y^\delta_n ) \leq \Delta
    \big)
    +
    \P\big(
        X_n \in B;\ \bm d( X_n, Y^\delta_n ) > \Delta
    \big)
    \\ 
    & \leq 
    \P\big(
        Y^\delta_n \in B^\Delta
    \big)
    +
    \P\big(
        X_n\in B\text{ or }Y^\delta_n \in B;\ \bm d( X_n, Y^\delta_n ) > \Delta
    \big)
    \\ 
    & = 
    \sum_{v \in \mathcal V}\P\big(
        Y^\delta_n \in B^\Delta,\ V^\delta_n = v
    \big)
    +
    \P\big(
        Y^\delta_n \in B^\Delta,\ V^\delta_n \notin \mathcal V
    \big)
    +
    \P\big(
        X_n\in B\text{ or }Y^\delta_n \in B;\ \bm d( X_n, Y^\delta_n ) > \Delta
    \big).
\end{align*}
Then, by Condition (ii),
there exists some $\delta_0 > 0$ such that for all but countably many $\delta \in (0,\delta_0)$,
\begin{align}
     & \limsup_{n \to \infty}
    \epsilon^{-1}_n
    \P(X_n \in B)
    \leq
    \sum_{v \in \mathcal V}
    \mu_v(B^{2\Delta}) + 
    \limsup_{n \to \infty} \epsilon^{-1}_n
    \P\Big( \bm{d}\big(X_n, Y^\delta_n\big)\mathbbm{I}\big( X_n \in B\text{ or }Y^\delta_n \in B \big) > \Delta \Big).
    \nonumber
\end{align}
This allows us to pick a sequence $\delta_k \downarrow 0$ such that 
Condition (ii) of Lemma~\ref{lemma: asymptotic equivalence when bounded away, equivalence of M convergence} holds under each $\delta = \delta_k$, which implies
\begin{align}
    & \limsup_{n \to \infty}
    \epsilon^{-1}_n
    \P(X_n \in B)
    \nonumber
    \\ 
    & \leq 
    \sum_{v \in \mathcal V}
    \mu_v(B^{2\Delta}) + 
    \limsup_{k \to \infty}
    \limsup_{n \to \infty} \epsilon^{-1}_n
    \P\Big( \bm{d}\big(X_n, Y^{\delta_k}_n\big)\mathbbm{I}\big( X_n \in B\text{ or }Y^{\delta_k}_n \in B \big) > \Delta \Big)
    \nonumber
    \\ 
    & = 
    \sum_{v \in \mathcal V}
    \mu_v(B^{2\Delta}) 
    \qquad
    \text{ by Condition (i).}
    \label{proof: upper bound, lemma: asymptotic equivalence when bounded away, equivalence of M convergence}
\end{align}
Analogously, observe the lower bound (for each $n \geq 1$ and $\delta > 0$)
\begin{align*}
    \P(X_n \in B)
    & \geq 
    \P\big(X_n \in B;\ \bm d( X_n, Y^\delta_n ) \leq \Delta\big)
    \\ 
    & \geq
    \P\big( Y^\delta_n \in B_\Delta; \ \bm d( X_n, Y^\delta_n ) \leq \Delta\big)
    \\ 
    & \geq 
    \P( Y^\delta_n \in B_\Delta)
    -
    \P\big( Y^\delta_n \in B_\Delta; \ \bm d( X_n, Y^\delta_n ) > \Delta\big)
    \\ 
    & \geq 
    \P( Y^\delta_n \in B_\Delta)
    -
    \P\big( Y^\delta_n \in B\text{ or }X_n \in B; \ \bm d( X_n, Y^\delta_n ) > \Delta\big)
    \\ 
    & \geq
    \sum_{v \in \mathcal V}\P( Y^\delta_n \in B_\Delta,\ V^\delta_n = v)
    -
    \P\big( Y^\delta_n \in B\text{ or }X_n \in B; \ \bm d( X_n, Y^\delta_n ) > \Delta\big).
\end{align*}
Again, by Condition (ii),
one can pick a sequence $\delta_k \downarrow 0$ such that 
\begin{align}
    & \liminf_{n \to \infty}
    \epsilon^{-1}_n
    \P(X_n \in B)
    \nonumber
    \\ 
    & \geq 
    \liminf_{k \to \infty}\liminf_{n \to \infty}
    \sum_{v \in \mathcal V}
    \epsilon^{-1}_n
    \P( Y^{\delta_k}_n \in B_\Delta,\ V^{\delta_k}_n = v)
    \nonumber
    \\ 
    &\qquad
    -
    \limsup_{k \to \infty}\limsup_{n \to \infty} \epsilon^{-1}_n
    \P\Big( \bm{d}\big(X_n, Y^{\delta_k}_n\big)\mathbbm{I}\big( X_n \in B\text{ or }Y^{\delta_k}_n \in B \big) > \Delta \Big)
    \nonumber
    \\ 
    & \geq 
    \sum_{v \in \mathcal V}\mu_v(B_{2\Delta})
     \qquad
    \text{ by Conditions (i) and (ii)}.
    \label{proof: lower bound, lemma: asymptotic equivalence when bounded away, equivalence of M convergence}
\end{align}
Since 
$|\mathcal V| < \infty$,
$\mu_v \in \mathbb M(\mathbb S\setminus\mathbb C)$ for each $v \in \mathcal V$,
and
$B^{2\Delta}$ is bounded away from $\mathbb C$,
we have $\sum_{v \in \mathcal V}\mu_v(B^{2\Delta}) < \infty$.
By sending $\Delta \downarrow 0$ in \eqref{proof: upper bound, lemma: asymptotic equivalence when bounded away, equivalence of M convergence} and \eqref{proof: lower bound, lemma: asymptotic equivalence when bounded away, equivalence of M convergence},
it then follows from the continuity of measures that 
\begin{align*}
    \sum_{v \in \mathcal V}\mu_v(B^\circ) \leq 
    \liminf_{n \to \infty}\epsilon_n^{-1}\P(X_n \in B)
    \leq 
    \limsup_{n \to \infty}\epsilon_n^{-1}\P(X_n \in B)
    \leq \sum_{v \in \mathcal V}\mu_v(B^-).
\end{align*}
By the arbitrariness of our choice of $B$,
we conclude the proof using Theorem~\ref{portmanteau theorem M convergence}.
\end{proof}

\section{Details for Tail Asymptotics of Hawkes Process Clusters}
\label{sec: appendix, tail asymptotics of Hawkes process clusters}

To precisely define the limiting measures $\mathbf C^{\bm j}_i$ in \eqref{def: measure C i I, cluster, main paper},
 we review a few definitions in \cite{blanchet2025tailasymptoticsclustersizes}.
\begin{definition}[Type]
\label{def: type of clusters}
We say that ${\bm I} = (I_{k,j})_{k \geq 1,\ j \in [d]}$ is a \textbf{type} if 
\begin{itemize}

    \item $I_{k,j}\in\{0,1\}$ for each $k \geq 1$ and $j \in [d]$;
 
    \item there exists $\notationdef{notation-K-I-for-type-I-cluster-size}{\mathcal K^{\bm I}} \in \mathbb Z_+$ such that $\sum_{j \in [d]}I_{k,j} = 0\ \forall k > \mathcal K^{\bm I}$ and $\sum_{j \in [d]}I_{k,j} \geq 1 \ \forall 1 \leq k \leq \mathcal K^{\bm I}$;

    \item for each $j \in [d]$, we have $\sum_{k \geq 1}I_{k,j} \leq 1$;
    
      \item for $k = 1$, the set $\{ j \in [d]:\ I_{1,j} = 1 \}$ is either empty or contains exactly one element.
\end{itemize}
Let $\mathscr I$ be the collection of all types.
For each $\bm I \in \mathscr I$, let
\begin{align}
    \notationdef{notation-active-indices-of-type-I-cluster-size}{\bm j^{\bm I}} \delequal 
    \big\{ j \in [d]:\ I_{k,j} = 1\text{ for some }k \geq 1  \big\},
    \qquad
    \notationdef{notation-active-indices-at-depth-k-of-type-I-cluster-size}{\bm j^{\bm I}_k} \delequal 
    \big\{
        j \in [d]:\ I_{k,j} = 1
    \big\}
    \ \ 
    \forall k \geq 1.
\end{align}
\end{definition}

Intuitively speaking, 
$\mathcal K^{\bm I}$ is the \emph{depth} of $\bm I$,
$\bm j^{\bm I}$ is the set of \emph{active indices} in $\bm I$, and $\bm j^{\bm I}_k$ is the set of active indices of $\bm I$ at depth $k$.
For any $\bm I \in \mathscr I$, note that:
(i) $\bm j^{\bm I} = \emptyset$ (and hence $\mathcal K^{\bm I} = 0$) if and only if $I_{k,j}\equiv 0$ for any $k,j$;
and 
(ii) when $\mathcal K^{\bm I} \geq 1$, there uniquely exists some $j^{\bm I}_1 \in [d]$ such that $\bm j^{\bm I}_1 = \{j^{\bm I}_1\}$.
Next,
given $\beta > 0$, we define a Borel measure on $(0,\infty)$ by
\begin{align}
    \notationdef{notation-power-law-measure-nu-beta}{\nu_\beta(dw)} \delequal \frac{ \beta dw}{ w^{\beta + 1} }\mathbbm{I}\{w > 0\}.
    \label{def, power law measure nu beta}
\end{align}
Given non-empty sets $\mathcal I \subseteq [d]$ and $\mathcal J \subseteq [d]$,
we say that 
$
\{\mathcal J(i):\ i \in \mathcal I\}
$
is an \notationdef{notation-assignment-from-J-to-I}{\textbf{assignment of $\mathcal J$ to $\mathcal I$}} if 
\begin{align}
    \mathcal J(i) \subseteq \mathcal J\quad \forall i \in \mathcal I;
    \qquad 
    \bigcup_{i \in \mathcal I}\mathcal{J}(i) = \mathcal J;
    \qquad 
    \mathcal J(i) \cap \mathcal J(i^\prime) = \emptyset\quad \forall i \neq i^\prime.
    \label{def: assignment from mathcal J to mathcal I}
\end{align}
We use $\notationdef{notation-assignment-from-mathcal-J-to-mathcal-I}{\mathbb T_{ \mathcal I \leftarrow \mathcal J }}$ to denote the set containing  all assignments of $\mathcal{J}$ to $\mathcal I$.
Given non-empty $\mathcal I \subseteq [d]$ and $\mathcal J \subseteq [d]$,
define the mapping
\begin{align}
    \notationdef{notation-function-g-mathcal-I-mathcal-J-cluster-size}{g_{\mathcal I \leftarrow \mathcal J}(\bm w)}
    \delequal 
    \sum_{
        \{\mathcal J(i):\ i \in \mathcal I\} \in \mathbb T_{ \mathcal I \leftarrow \mathcal J }
    }
    \prod_{i \in\mathcal I}
        \prod_{j \in \mathcal J(i)}
        w_i\bar s_{l^*(j) \leftarrow i},
    \quad
    \forall \bm w = (w_i)_{i \in \mathcal I} \in [0,\infty)^{|\mathcal I|}.
    \label{def: function g mathcal I mathcal J, for measure C i bm I, cluster size}
\end{align}
Here, recall that $\bar{\bm s}_j = \E \bm S_j$, $\bar s_{j \leftarrow i} = \E S_{j \leftarrow i}$ (see \eqref{def: bar s i, ray, expectation of cluster size}),
and $l^*(j)$ is defined in \eqref{def: cluster size, alpha * l * j}.

Given  a type $\bm I \in \mathscr I$ whose active index set $\bm j^{\bm I}$ is non-empty,
recall that when $\bm j^{\bm I} \neq \emptyset$,
there uniquely exists some $j^{\bm I}_1 \in [d]$ such that $\bm j^{\bm I}_1 = \{j^{\bm I}_1\}$.
Let
\begin{align}
    \notationdef{notation-measure-nu-for-type-I-cluster-size}{\nu^{\bm I}(d \bm w)}
    & \delequal
    \bigtimes_{k = 1}^{ \mathcal K^{\bm I} }
    \bigg(
        \bigtimes_{ j \in \bm j^{\bm I}_k }\nu_{\alpha^*(j)  }(d w_{k,j})
    \bigg),
    \label{def, measure nu type I, cluster size, appendix}
    \\
    \notationdef{notation-measure-C-i-type-I-cluster-size}{\mathbf C_i^{\bm I}(\cdot)}
        & \delequal
    \int \mathbbm{I}
    \Bigg\{
        \sum_{ k = 1 }^{ \mathcal K^{\bm I} }\sum_{ j \in \bm j^{\bm I}_k  }w_{k,j}\bar{\bm s}_j \in \ \cdot\ 
    \Bigg\}
     \Bigg(
     \bar s_{l^*(j^{\bm I}_1) \leftarrow i}
    \prod_{ k = 1 }^{ \mathcal K^{\bm I} - 1}
        g_{ \bm j^{\bm I}_k \leftarrow \bm j^{\bm I}_{k+1} }(\bm w_k)
    \Bigg)
    \nu^{\bm I}(d \bm w),
    \label{def: measure C i I, cluster, appendix}
\end{align}
where we adopt the notations $\bm w_k = (w_{k,j})_{j \in \bm j^{\bm I}_k}$
and $\bm w = (\bm w_k)_{k \in [\mathcal K^I]}$.
Besides, note that $\mathbf C^{\bm I}_i(\cdot)$ is supported on the cone $\R^d(\bm j^{\bm I})$ (see \eqref{def: cone R d index i}).
The measures $\mathbf C^{\bm j}_i(\cdot)$ in \eqref{def: measure C i I, cluster, main paper} admit the form
\begin{align}
    \mathbf C^{\bm j}_i = \sum_{ \bm I \in \mathscr I:\ \bm j^{\bm I} = \bm j }\mathbf C^{\bm I}_i.
    \label{def: measure C j i, for tail asymptotics}
\end{align}

Next, continuing the discussions in Remark~\ref{remark: evaluation of limiting measures, LD for Hawkes},
we note that the evaluation of the limiting measures 
$
\breve{\mathbf C}^{ \subInfty  }_{\bm{\mathcal K};\bm\mu_{\bm N}}(\cdot)
$
in \eqref{claim, theorem: LD for Hawkes} can be addressed by rejection sampling for $\mathbf C^{\bm I}_i(\cdot)$ in \eqref{def: measure C i I, cluster, appendix}.
Indeed, 
as noted in Remark~\ref{remark: evaluation of limiting measures, LD for Hawkes}, the key task is to sample under the law $\bar{\mathbf C}_{\bm j, \bar\delta}(\cdot)$ in \eqref{def: measure bar C j bar delta, for monte carlo simulation of limiting measures}.
Furthermore, the evaluation of $\mathbf C^{\bm I}_i\big( \bar \R^{d}_>(\bm j,\bar\delta)   \big)$ is detailed in Remark~4 of \cite{blanchet2025tailasymptoticsclustersizes}.
Then,
by \eqref{def: measure C indices j, sample path LD for Hawkes} and \eqref{def: measure C j i, for tail asymptotics},
it suffices to check how to sample from probability measures
$$\hat{\mathbf C}^{\bm I}_{i,\bar\delta}(\cdot) \delequal
    \mathbf C^{\bm I}_i\Big( \ \cdot\ \cap \bar \R^{d}_>(\bm j,\bar\delta)   \Big)\Big/
     \mathbf C^{\bm I}_i\Big( \bar \R^{d}_>(\bm j,\bar\delta)   \Big),
$$
where 
$\bm I \in \mathscr I$ is such that $\bm j^{\bm I} = \bm j$ (see Definition~\ref{def: type of clusters}).
Furthermore,
the proof of Lemma~4.10 (b) of \cite{blanchet2025tailasymptoticsclustersizes}
implies the existence of some constant $M^\prime < \infty$ such that (we write $\bm w_k = (w_{k,j})_{j \in \bm j^{\bm I}_k}$)
$$
\Bigg[ \prod_{ k = 1  }^{ \mathcal K^{\bm I} - 1 }g_{ \bm j^{\bm I}_k \leftarrow \bm j^{\bm I}_{k+1} }(\bm w_k)\Bigg] \cdot 
    \prod_{k = 1}^{\mathcal K^{\bm I}}\prod_{j \in \bm j^{\bm I}_k}{  w_{k,j}^{ 
-\alpha^*(j) - 1 } }
\leq M^\prime \cdot   \prod_{k = 1}^{\mathcal K^{\bm I}}\prod_{j \in \bm j^{\bm I}_k}{  w_{k,j}^{ 
-\alpha^*(j) }  }
$$
whenever
$
\sum_{ k = 1 }^{ \mathcal K^{\bm I} }\sum_{ j \in \bm j^{\bm I}_k  }w_{k,j}\bar{\bm s}_j \in \bar \R^{d}_>(\bm j,\bar\delta).
$
Here,
the mapping $g_{\mathcal I \leftarrow \mathcal J}(\cdot)$ is defined in \eqref{def: function g mathcal I mathcal J, for measure C i bm I, cluster size},
$\alpha^*(\cdot)$ is defined in \eqref{def: cluster size, alpha * l * j},
and
the $\mathcal K^{\bm I}$ and $\bm j^{\bm I}_k$'s are specified in Definition~\ref{def: type of clusters}.
Then, by definitions in \eqref{def: measure C i I, cluster, appendix},
to sample from $\hat{\mathbf C}^{\bm I}_{i,\bar\delta}(\cdot)$
it suffices to 
set
\begin{align}
    \hat g^{\bm I}(\bm w) \delequal 
    \Bigg[ \prod_{ k = 1  }^{ \mathcal K^{\bm I} - 1 }g_{ \bm j^{\bm I}_k \leftarrow \bm j^{\bm I}_{k+1} }(\bm w_k) \Bigg]\cdot 
     \prod_{ k = 1  }^{ \mathcal K^{\bm I}  }
    \prod_{j \in \bm j^{\bm I}_k}{  \bigg(\frac{\delta}{w_{k,j}}\bigg)^{ 
\alpha^*(j) + 1 }  },
\quad
\hat f^{\bm I}(\bm w) \delequal 
 \prod_{ k = 1  }^{ \mathcal K^{\bm I}  }
    \prod_{j \in \bm j^{\bm I}_k}{  \bigg(\frac{\delta}{w_{k,j}}\bigg)^{ 
\alpha^*(j) }  },
    \nonumber
\end{align}
pick $\delta > 0$ small enough and $M > 0$ large enough, and run rejection sampling as follows.
\begin{itemize}
    \item 
        Independently
        for each $k \in [\mathcal K^{\bm I}]$ and $j \in \bm j^{\bm I}_k$,
        generate a Pareto (i.e., exact power-law) random variable $W^{(\delta)}_{k,j}$ with lower bound $\delta$ and power-law index $\alpha^*(j) - 1$;
        Write $\bm W^{(\delta)} = \big(W^{(\delta)}_{k,j}\big)_{ k \in [\mathcal K^{\bm I}], j \in \bm j^{\bm I}_k }$.

    \item 
        Generate $U \sim \text{Unif}(0,1)$.

    \item 
        If $ \sum_{k \in [\mathcal K^{\bm I}]}\sum_{ j \in \bm j }W^{(\delta)}_{k,j}\bar{\bm s}_j \in \bar \R^{d}_>(\bm j,\bar\delta)$ and 
        $
        U < \hat g^{\bm I}( \bm W^{(\delta)}) \big/\big[
            M \cdot \hat f^{\bm I}( \bm W^{(\delta)}) 
        \big],
        $
        return $$\sum_{k \in [\mathcal K^{\bm I}]}\sum_{ j \in \bm j }W^{(\delta)}_{k,j}\bar{\bm s}_j.$$
        Otherwise, rerun this procedure.
\end{itemize}

\section{Counterexamples of Topology and Tail Behavior}
\label{sec: counter examples}

\begin{example}
\label{example: LD for Hawkes, cadlag space with compact domain}
In this two-dimensional example, we demonstrate that the LDP \eqref{claim, theorem: LD for Hawkes} stated in Theorem~\ref{theorem: LD for Hawkes} would generally fail w.r.t.\ the product $M_1$ topology of $\D[0,T]$, the càdlàg space with compact domain.
The intuition behind the counterexample is that, for any immigrant arriving near the end of the time interval $[0, T]$, its cluster size vector may be arbitrarily truncated at time $T$ due to randomness in the birth times of its descendants.
As a result, the LDP in \eqref{claim, theorem: LD for Hawkes} generally fails under the product $M_1$ topology on $\D[0, T]$ for rare event sets $B$ that involve the value of the process at time $T$.

Without loss of generality, we fix $T = 1$, and lighten the notations by writing 
\begin{align*}
    & \D = \D\big([0,1],\R^d\big),\qquad 
    { \shortbarDepsk{\epsilon}{\bm{\mathcal K} }  }
    =
    \barDxepskT{\bm \mu_{\bm N}}{\epsilon}{\bm{\mathcal K}}{[0,1]},
    \qquad 
    { \shortbarDepsk{\epsilon}{\leqslant \bm k} }
    =
    \barDxepskT{\bm \mu_{\bm N}}{\epsilon}{\leqslant \bm k}{[0,1]},
    \\ 
    &
    {\breve{\mathbf C}_{\bm{\mathcal K}} }
    =
    \breve{\mathbf C}^{ _{[0,1]} }_{\bm{\mathcal K};\bm \mu_{\bm N}},
    \qquad 
    {\bar{\bm N}_n}
    =
    \bar{\bm N}_n^{ _{[0,1]} },
    \qquad 
    {\dmp{}}=\dmp{[0,1]}.
\end{align*}
We focus on the case with $d = 2$ and adopt all assumptions imposed in Theorem~\ref{theorem: LD for Hawkes}.
Besides, we impose the following conditions on the fertility functions (see \eqref{def: conditional intensity, hawkes process}):
\begin{align}
    f^{\bm N}_{1 \leftarrow 1}(t) = 0,& \qquad \forall t > 1,
    \label{example, cond fertility function 1, LD for Hawkes, cadlag space with compact domain}
    \\
    f^{\bm N}_{2 \leftarrow 1}(t) = 0,& \qquad \forall t \in [0,2].
    \label{example, cond fertility function 2, LD for Hawkes, cadlag space with compact domain}
\end{align}
Here, condition \eqref{example, cond fertility function 1, LD for Hawkes, cadlag space with compact domain} implies that the time a type-1 parent waits to give birth to a type-1 child would, with probability 1, take values over $[0,1]$;
condition~\eqref{example, cond fertility function 2, LD for Hawkes, cadlag space with compact domain} implies that the time a type-1 parent waits to give birth to a type-2 child is always strictly larger than $2$.
We note that:
(i) it is easy to identify
concrete cases that are compatible with both \eqref{example, cond fertility function 1, LD for Hawkes, cadlag space with compact domain}--\eqref{example, cond fertility function 2, LD for Hawkes, cadlag space with compact domain}
and the
tail condition \eqref{condtion: tail of fertility function hij} in Theorem~\ref{theorem: LD for Hawkes} under any $\bm k$;
for instance,
one can take $f^{\bm N}_{1 \leftarrow 1}(t) = \mathbbm{I}_{(0,1)}(t)$, which induces a uniform distribution over $(0,1)$,
and $f^{\bm N}_{2 \leftarrow 1}(t) = \mathbbm{I}_{(2,\infty)}(t)\cdot \exp(-(t-2))$,
which induces an exponential distribution with $+2$ offset;
(ii)
the exact form of conditions \eqref{example, cond fertility function 1, LD for Hawkes, cadlag space with compact domain}--\eqref{example, cond fertility function 2, LD for Hawkes, cadlag space with compact domain} are not pivotal to the analysis below and are imposed only for simplicity;
for instance, the arguments below would still hold after minor modifications if we set $f^{\bm N}_{j \leftarrow i}(t) = A_{j\leftarrow i} \exp(- b_{j \leftarrow i}\cdot t)$
(i.e., with exact exponential tails in decay functions, which make the Hawkes process Markovian).

Regarding the tail indices for the clusters, we assume that
\begin{align}
    \alpha^*(1) = \alpha_{1 \leftarrow 1} > 1,\qquad \alpha_{2 \leftarrow 1} > \alpha_{1\leftarrow 1}, \qquad \alpha^*(2) > \alpha^*(1) + 1,
    \label{example, cond tail index alpha *, LD for Hawkes, cadlag space with compact domain}
\end{align}
where
$\alpha_{i \leftarrow j}$ is the regular variation index for $B_{i \leftarrow j}$ (see Assumption~\ref{assumption: heavy tails in B i j} and the discussion right above) and
$\alpha^*(\cdot)$ is defined in \eqref{def: cluster size, alpha * l * j}.

We are interested in studying the asymptotics (as $n \to \infty$)
\begin{align*}
    \P\big( n^{-1}\bm N(n) \in A\big)
    =
    \P\big( \bar{\bm N}_n \in E  \big),
\end{align*}
where $\bar{\bm N}_n = \{ \bm N(nt)/n:\ t \in [0,1]  \}$ is the scaled sample path of the Hawkes process $\bm N(t)$ embedded in $\D$, and 
\begin{align}
     A \delequal
    \big\{
        (x_1,x_2)^\text{T} \in \R^2_+:\ x_1 > 1, \ c x_1 > x_2
    \big\},
    \qquad
     E \delequal \big\{ \xi \in \D:\ \xi(1) \in A  \big\}.
    \label{example, def set A B, LD for Hawkes, cadlag space with compact domain}
\end{align}
To specify the choice of $c$,
we note that
by Assumption~\ref{assumption: regularity condition 1, cluster size, July 2024},
in $\bar{\bm s}_i = (\bar s_{1 \leftarrow i},\bar s_{2\leftarrow i})^\top$ for each $i \in \{1,2\}$ (see \eqref{def: bar s i, ray, expectation of cluster size})
and  $\bm \mu_{\bm N} = (\mu_{\bm N,1},\mu_{\bm N, 2})^\top$ (see \eqref{def: approximation, mean of N(t)}),
all coordinates are strictly positive.
Therefore, 
we can fix $c > 0$ small enough such that
the set
$
\{ \bm x + \bm \mu_{\bm N}:\ \bm x \in  \bar\R^2(\{1,2\},\epsilon)  \}
$
is bounded away from $A$ for some $\epsilon > 0$;
see \eqref{def: enlarged cone R d index i epsilon} for the definition of cones $\bar\R^d(\bm j,\epsilon)$.
More generally, given $a \in (0,\infty)$, there exists $ M =  M(c,a) \in (0,\infty)$ such that
\begin{align}
    \begin{pmatrix}
        0 \\ 0
    \end{pmatrix}
    \leq 
    \begin{pmatrix}
        w_1 \\ w_2
    \end{pmatrix}
    \leq 
    \begin{pmatrix}
        a \\ a
    \end{pmatrix},
    \ 
    y > M
    \qquad\Longrightarrow\qquad
    \begin{pmatrix}
        w_1 \\ w_2
    \end{pmatrix}
    +
    \begin{pmatrix}
        y \\ 0
    \end{pmatrix}
    \in A.
    \label{example, choice of M, LD for Hawkes, cadlag space with compact domain}
\end{align}

We start by making a few observations.
Recall that $\powersetTilde{m}$ is the collection of non-empty subsets of $[m]$,
and hence $\powersetTilde{2} = \{ \{1\}, \{2\}, \{1,2\} \}$.
Also, in the context of Theorem~\ref{theorem: LD for Hawkes}, recall the definitions of 
\begin{align}
     \shortbarDepsk{\epsilon}{\leqslant \bm k} = 
 \bigcup_{
        \substack{
            \bm{\mathcal K} \in  \mathbb Z_+^{ \powersetTilde{2} }:
            \\
            \bm{\mathcal K} \notin \mathbb A(\bm k),\ \breve c(\bm{\mathcal K}) \leq c(\bm k)
        }
    }
    \shortbarDepsk{\epsilon}{ \bm{\mathcal K}},
    \qquad 
    \breve c(\bm{\mathcal K})
    =
    \sum_{ \bm j \in \powersetTilde{2} }\mathcal K_{\bm j} \cdot \big( \bm \alpha(\bm j) - 1\big),
    \qquad
    c(\bm k) = \sum_{i = 1,2}k_i \cdot  \big(\alpha(\{i\}) - 1\big),
    \label{example, def of k jump sets and index breve c, LD for Hawkes, cadlag space with compact domain}
\end{align}
where $\bm \alpha(\bm j ) = 1 + \sum_{i \in \bm j}\big(\alpha^*(i) - 1 \big)$ is defined in \eqref{def: cost function, cone, cluster}.
In particular, note that $\alpha(\{1\}) = \alpha^*(1)$, $\alpha(\{2\}) = \alpha^*(2)$, and $\alpha(\{1,2\}) = \alpha^*(1) + \alpha^*(2) - 1$.
Also,
we introduce notations $\bm{\mathcal K}^{\ \bm i} = (\mathcal K^{\ \bm i}_{\bm j})_{\bm j \in \powersetTilde{2}} \in \mathbb Z_+^{ \powersetTilde{2} }$
by setting $\mathcal K^{\ \bm i}_{\bm j} = 1$ if $\bm j = \bm i$, and $\mathcal K^{\ \bm i}_{\bm j} = 0$ otherwise.
For instance, in $\bm{\mathcal K}^{\{1,2\}} = (\mathcal K^{\{1,2\}}_{\bm j})_{\bm j \in \powersetTilde{2}}$, we have $\mathcal K^{\{1,2\}}_{\{1,2\}} = 1$ and $\mathcal K^{\{1,2\}}_{\{1\}} = \mathcal K^{\{1,2\}}_{\{2\}} = 0$.
By  \eqref{example, cond tail index alpha *, LD for Hawkes, cadlag space with compact domain} and \eqref{example, def of k jump sets and index breve c, LD for Hawkes, cadlag space with compact domain},
we have
\begin{align*}
    \shortbarDepsk{\epsilon}{\leqslant (0,1)} \subseteq 
    \bigcup_{ \bm{\mathcal K} \in \mathbb Z_+^{ \powersetTilde{2} } :\ \mathcal K_{\{2\}} = \mathcal K_{ \{1,2\} } = 0  }\shortbarDepsk{\epsilon}{\bm{\mathcal K}}.
\end{align*}
Furthermore, for any $\bm{\mathcal K} \in \mathbb Z_+^{ \powersetTilde{2} }$ such that $\mathcal K_{\{2\}} = \mathcal K_{ \{1,2\} } = 0 $
and any $\xi \in \shortbarDepsk{\epsilon}{\bm{\mathcal K}}$,
we have
\begin{align*}
    \xi(1) \in \bar\R^2(\{1\},\epsilon) + \bm \mu_{\bm N}.
\end{align*}
Meanwhile,
recall the definitions of
the sets $A$ and $E$ in \eqref{example, def set A B, LD for Hawkes, cadlag space with compact domain}.
It has been noted above that
$
\bar\R^2(\{1,2\},\epsilon) + \bm \mu_{\bm N}
$
is bounded away from $A$, 
which implies that the set $E$ is bounded away from $\shortbarDepsk{\epsilon}{\leqslant (0,1) }$ under $\dmp{}$.
Now, suppose that the asymptotics \eqref{claim, theorem: LD for Hawkes} stated in Theorem~\ref{theorem: LD for Hawkes} hold w.r.t.\ the product $M_1$ topology of $\D$ (i.e., on the metric space $(\D,\dmp{})$).
Then,
we are led to believe that
\begin{align}
    \P\big(\bar{\bm N}_n \in E\big) = \bo\Big( \breve\lambda_{(0,1) }(n)\Big),
    \qquad\text{as }n \to \infty.
    \label{example, wrong claim, LD for Hawkes, cadlag space with compact domain}
\end{align}
However, our analysis below disproves \eqref{example, wrong claim, LD for Hawkes, cadlag space with compact domain} by verifying that
\begin{align}
    \liminf_{n \to \infty}
    \P\big(\bar{\bm N}_n \in E\big)
    \big/ \P(B_{1 \leftarrow 1} > n) > 0.
    \label{example, goal, LD for Hawkes, cadlag space with compact domain}
\end{align}
Indeed, by Assumption~\ref{assumption: heavy tails in B i j},
we have 
$
 \P(B_{1 \leftarrow 1} > n) \in \RV_{  -\alpha_{1 \leftarrow 1} }(n);
$
by 
 \eqref{example, cond tail index alpha *, LD for Hawkes, cadlag space with compact domain} and \eqref{example, def of k jump sets and index breve c, LD for Hawkes, cadlag space with compact domain},
we have
$
\breve \lambda_{(0,1)}(n) \in \RV_{ -c((0,1)) }(n)
$
with $c\big( (0,1)\big) = \alpha^*(2) - 1 > \alpha_{1\leftarrow 1}$.
Then, by
\begin{align*}
    \liminf_{n \to \infty}
    \frac{\P\big(\bar{\bm N}_n \in E\big)}{ \breve \lambda_{(0,1)}(n)}
    \geq 
    \liminf_{n \to \infty}
    \frac{\P\big(\bar{\bm N}_n \in E\big)}{ \P(B_{1 \leftarrow 1} > n) }
    \cdot 
    \liminf_{n \to \infty}
     \frac{ \P(B_{1 \leftarrow 1} > n) }{ \breve \lambda_{(0,1)}(n) },
\end{align*}
we end up with 
$
\liminf_{n \to \infty}
    {\P\big(\bar{\bm N}_n \in E\big)}\big/{ \breve\lambda_{ (0,1) } (n) } = \infty,
$
which is a clear contradiction to \eqref{example, wrong claim, LD for Hawkes, cadlag space with compact domain}.
We thus confirm that it is generally not possible to strengthen
the characterizations in \eqref{claim, theorem: LD for Hawkes} stated in Theorem~\ref{theorem: LD for Hawkes} w.r.t.\ the product $M_1$ topology of $\D[0,\infty)$ to those on $\D[0,T]$.

Below, we verify Claim \eqref{example, goal, LD for Hawkes, cadlag space with compact domain}.
To proceed, we adopt the notations in \eqref{def: center process, Poisson cluster process, notation}--\eqref{def: branching process approach, N as a point process, notation section}
for the cluster representation of Hawkes processes.
Specifically, for each $j \in \{1,2\}$, 
we use the sequence $\big(T_{j;k}^\mathcal{C}\big)_{k \geq 1}$ to denote the arrival times of type-$j$ immigrants, which are generated by a Poisson process $G^{\bm N}_j(t)$ with rate $c^{\bm N}_{j}$ 
(see \eqref{def: conditional intensity, hawkes process}).
For  $\bm N^{\mathcal O}_{ (T^\mathcal{C}_{j;k},j) }$, the cluster process induced by the  $k^\text{th}$ type-$j$ immigrant, 
we denote its size by
\begin{align*}
    S^{(k)}_{i \leftarrow j} & \delequal 
    \sum_{m \geq 0}\mathbbm{I}\Big\{ A^\mathcal{O}_{j;k}(m) = i \Big\},
    \qquad
    \bm S^{(k)}_j \delequal \Big( S^{(k)}_{1 \leftarrow j}, S^{(k)}_{2 \leftarrow j}, \ldots, S^{(k)}_{d \leftarrow j}\Big)^\text{T}. 
\end{align*}
Next, we define the event (with $\bm \mu_{\bm N}$ defined in \eqref{def: approximation, mean of N(t)})
\begin{align}
    \text{(I)}
    =
    \Bigg\{
        n^{-1}
        \sum_{j \in \{1,2\}}\  \sum_{ \substack{k \geq 1 
        :
        T_{j;k}^\mathcal{C} \leq n - 2 } }
        \bm S^{(k)}_j
        \leq 2\bm \mu_{\bm N}
    \Bigg\},
    \nonumber
\end{align}
which represents the case where the accumulated size of all clusters arrived by time $n-2$ is upper bounded by the vector $2n\bm \mu_{\bm N}$.
Also, recall that we use $G^{\bm N}_j(t)$ to denote the Poisson process generating the arrival times $T^\mathcal{C}_{j;k}$ for the immigrants.
Let
\begin{align*}
    \text{(II)}
    =
    \Big\{
        G^{\bm N}_1(n - 1) - G^{\bm N}_1(n - 2) = 1;\ G^{\bm N}_1(n) - G^{\bm N}_1(n - 1) = 0;\ 
        G^{\bm N}_2(n) - G^{\bm N}_2(n - 2) = 0
    \Big\}.
\end{align*}
That is, on event (II), there is only one type-$1$ immigrant who arrived during the time window $(n-2,n]$, with the arrival time lying in $(n-2,n-1]$, and no type-$2$ immigrant arrived during $(n-2,n]$.
On this event, for the only type-$1$ immigrant that arrived during $(n-2,n]$, we use $B_{1 \leftarrow 1}$ to specifically denote the number of its type-$1$ children (in one generation).
Let
\begin{align*}
    \text{(III)} = \big\{ B_{1 \leftarrow 1} > nM  \big\}.
\end{align*}
Here, by \eqref{example, choice of M, LD for Hawkes, cadlag space with compact domain}, we fix some $M$ large enough such that 
\begin{align}
    \begin{pmatrix}
        0 \\ 0
    \end{pmatrix}
    \leq 
    \begin{pmatrix}
        w_1 \\ w_2
    \end{pmatrix}
    \leq 
    2\bm \mu_{\bm N},
    \ 
    y > M
    \qquad\Longrightarrow\qquad
    \begin{pmatrix}
        w_1 \\ w_2
    \end{pmatrix}
    +
    \begin{pmatrix}
        y \\ 0
    \end{pmatrix}
    \in A.
    \label{example, choice of M, 2, LD for Hawkes, cadlag space with compact domain}
\end{align}
Now, we make some observations.
First, on the event $\text{(II)}$,
by \eqref{example, cond fertility function 1, LD for Hawkes, cadlag space with compact domain}
we know that all the $B_{1 \leftarrow 1}$ children of the type-$1$ ancestor must have arrived by time $n$,
and they may further give birth to more type-$1$ offspring by time $n$.
On the other hand, by \eqref{example, cond fertility function 2, LD for Hawkes, cadlag space with compact domain},
we know that any type-$2$ offspring in this cluster, induced by the type-$1$ immigrant who arrived during $(n-2,n]$, will have to arrive after time $n$.
Then, by \eqref{example, choice of M, 2, LD for Hawkes, cadlag space with compact domain}, on event $\text{(I)} \cap \text{(II)} \cap \text{(III)}$ we must have $n^{-1}\bm N(n) \in A$.
This leads to
\begin{align}
    \liminf_{n \to \infty}
    \frac{ \P\big(\bar{\bm N}_n \in E\big) }{ \P(B_{1\leftarrow 1} > n) }
    =
     \liminf_{n \to \infty}
    \frac{ \P\big(n^{-1}\bm N(n) \in A \big) }{ \P(B_{1\leftarrow 1} > n) }
    \geq 
    \lim_{n \to \infty }\P\big( \text{(I)} \big)
    \cdot 
    \P\big( \text{(II)}\big)
    \cdot 
    \lim_{n \to \infty} \frac{ \P(B_{1 \leftarrow 1} > nM) }{ \P(B_{1\leftarrow 1} > n) },
    \nonumber
\end{align}
where the inequality follows from the independent and stationary increments in Poisson processes $G^{\bm N}_j(t)$.
By the law of large numbers, we get 
$
\lim_{n \to \infty }\P\big( \text{(I)} \big) = 1.
$
By the law of Poisson processes, we have
\begin{align*}
    \P\big(\text{(II)}\big)
    =
    \P\Big(
        G^{\bm N}_1(1) = 1;\ G^{\bm N}_1(2) - G^{\bm N}_1(1) = 0;\ 
        G^{\bm N}_2(2)  = 0
    \Big) > 0.
\end{align*}
Lastly, by the regularly varying law of $B_{1 \leftarrow 1}$, we get
$
\lim_{n \to \infty} \frac{ \P(B_{1 \leftarrow 1} > nM) }{ \P(B_{1 \leftarrow 1} > n) } = M^{ -\alpha_{1\leftarrow 1} } > 0.
$
This concludes the proof of Claim~\eqref{example, goal, LD for Hawkes, cadlag space with compact domain}.
\end{example}

\begin{example}
\label{example: necessity of the tail bound on decay functions}
    This example demonstrates that, even in the univariate setting, we need sufficiently tight power-law bounds on the tail behavior of the decay functions (like in \eqref{condtion: tail of fertility function hij}) for the large deviations asymptotics stated in Theorem~\ref{theorem: LD for Hawkes} to hold.
    The intuition of this is that, without sufficiently tight tail bounds on the decay functions, some rare events may be driven by a single large cluster with a long lifetime—effectively a big jump spread broadly across the entire timeline—rather than by multiple large clusters.

    Specifically, we consider a univariate Hawkes process $N(t)$ with initial value $N(0) = 0$ and conditional intensity 
    $
    h^{N}(t) = c^{N} + \int_0^t \tilde B(s) f^{N}(t-s)dN(s),
    $
    where the constant $c^{N} > 0$ is the  immigration rate, 
    the decay function $f^{N}(\cdot)$ takes the form
    \begin{align}
        f^N(t) = \frac{\theta}{ (1+t)^{\theta+1}  },\qquad\forall t \geq 0,
        \label{def: decay function, example: necessity of the tail bound on decay functions}
    \end{align}
    for some $\theta > 1$ 
    (i.e., $\int_t^\infty f^N(s)ds = (1+t)^{-\theta}$, and $\norm{f^N}_1 = \int_0^\infty f^N(t)dt = 1$),
    and the excitation rates $\tilde B(s)$ are independent copies of some non-negative random variable $\tilde B$
    with law
    \begin{align}
        \P(\tilde B > x) = \frac{1}{ \max\{1, x/\tilde c  \}^\alpha },\qquad\forall x \in \R.
        \label{def: law of excitation rates, example: necessity of the tail bound on decay functions}
    \end{align}
    Here, we have $\alpha > 1$ for the power-law tail index; 
    the constant $\tilde c > 0$ is small enough such that $\E[\tilde B] \cdot \norm{f^N}_1 = \E[\tilde B] < 1$ (so the sub-critical condition in Assumption~\ref{assumption: subcriticality} also holds).
    By \eqref{def: approximation, mean of N(t)}, in this univariate setting we have
    $
    \mu_N = \frac{c^N}{ 1 -  \E[\tilde B] \cdot \norm{f^N}_1  } = \frac{c^N}{ 1 -  \E[\tilde B]  }.
    $

Recall that
$\bar{N}_n(t) = N(nt)/n$,
and that
$\bar{\bm N}^{\subInfty}_n \delequal \{ \bar N_n(t):\ t \geq 0  \}$ is the scaled sample path of the Hawkes process $N(t)$.
In this example, we study events of the form
$
\{ \bar{\bm N}^{\subInfty}_n \in A(k) \}
$
with
\begin{align}
    A(k) \delequal \big\{
        \xi \in \D\big([0,\infty),\R\big):\
        \xi(t_l + u) - \xi(t_l) > \mu_N\ \forall l = 0,1,\ldots,k - 1
    \big\},
\end{align}
where 
\begin{itemize}
    \item 
        $k$ is some positive integer;

    \item 
        the $t_l$'s are such that
        $0 = t_0 < t_1 < t_2 < \ldots < t_{k-1} < 1$;

    \item 
        Furthermore, $u \in (0,1)$ is small enough such that, across $l = 0,1,\ldots,k - 1$,
        the intervals $I_l = [t_l,t_l + u]$ are mutually disjoint and are all contained within $[0,1]$.
\end{itemize}
In other words, for any path $\xi$ that belongs to the set $A(k)$, the increment of $\xi$ over each of the mutually disjoint intervals $I_l$ is bounded from below by $\mu_N$. 
In particular, note that 
$A(k)$ is bounded away from (by adapting \eqref{def: j jump path set with drft c, LD for Levy, MRV} and  \eqref{def: path with costs less than j jump set, drift x, LD for Levy MRV} to the univariate setting)
\begin{align*}
    \breve \D^\epsilon_{ \leqslant k;\mu_N}[0,\infty)
    =
    \bigcup_{ k^\prime = 0,1,\ldots,k - 1  }\breve \D^\epsilon_{ k^\prime;\mu_N}[0,\infty)
    =\breve \D^\epsilon_{ k - 1;\mu_N}[0,\infty)
\end{align*}
under $\dm{[0,\infty)}$
for any $\epsilon > 0$ small enough:
indeed,
a linear path of slope $\mu_N$ would only increase by $u \cdot \mu_N < \mu_N$ over each interval $I_l$,
meaning that at least one jump needs to be added to the linear path  over each $(I_l)_{l = 0,1,\ldots,k - 1}$ for it to enter the set $A(k)$.
Also, when there is no ambiguity about the choice of $k$, we simply write $A = A(k)$.

We note that, regardless of the value of the tail indices $\theta$ in \eqref{def: decay function, example: necessity of the tail bound on decay functions} for the decay function and  $\alpha$ in \eqref{def: law of excitation rates, example: necessity of the tail bound on decay functions} for the law of the excitation rates,
the tail condition \eqref{condtion: tail of fertility function hij} would be violated eventually for all $k$ large enough.
Indeed, by \eqref{def: rate function lambda j n, cluster size} and \eqref{property, rate function for assignment mathcal K, LD for Levy},
in this univariate setting we have
\begin{align*}
    \breve \lambda_k(x) = 
    \Big(
        x \P(\tilde B > x)
    \Big)^k
    =
\bo(x^{ -k(\alpha - 1) }).
\end{align*}
Therefore, condition \eqref{condtion: tail of fertility function hij} fails if $\theta < k(\alpha - 1) + 1$.
Suppose that Theorem~\ref{theorem: LD for Hawkes} holds even after dropping the tail condition \eqref{condtion: tail of fertility function hij} on the decay functions,
then we are led to believe that
\begin{align}
    \P\big( \bar{\bm N}^{\subInfty}_n \in A \big) =
    \bo\big(\breve \lambda_k(n)\big) = 
    \bo\big( n^{ -k(\alpha - 1)  } \big),
    \quad 
    \text{as }n \to \infty.
    \label{false claim, example: necessity of the tail bound on decay functions}
\end{align}
However, our analysis below derives the lower bound 
\begin{align}
    \liminf_{ n \to \infty }\P\big( \bar{\bm N}^{\subInfty}_n \in A \big) \cdot n^{ \alpha(\theta+1) - 1  } > 0.
    \label{goal, lower bound, example: necessity of the tail bound on decay functions}
\end{align}
This clearly contradicts \eqref{false claim, example: necessity of the tail bound on decay functions} for all $k$ large enough such that $k(\alpha - 1) > \alpha(\theta + 1) - 1$.
In summary, even in the univariate setting,
it is generally not possible to remove
tail bounds on the decay functions from Theorem~\ref{theorem: LD for Hawkes}.

The rest of this example is devoted to establishing the lower bound \eqref{goal, lower bound, example: necessity of the tail bound on decay functions},
and
we proceed by defining a few events.
Let $0 < T_1 < T_2 < \ldots$ be a sequence generated by a Poisson process with rate $c^N$.
This sequence denotes the arrival time sequence of the immigrants in $N(t)$.
Also, let $(B_i)_{i \geq 1}$ be independent copies of 
$
B \distequal \text{Poisson}( \tilde B \cdot \norm{f^N}_1 ) \distequal \text{Poisson}( \tilde B )
$
(recall that we have $\norm{f^N}_1 = 1$).
Note that $\P(B > x) \sim \P(\tilde B > x)$ as $x \to \infty$.
In the cluster representation of $N(t)$, each $B_i$ denotes the count of first-generation offspring of the immigrant arriving at time $T_i$.
We fix $C \in(0,\infty)$ large enough such that
\begin{align}
    \frac{C\theta u}{4} > \mu_N,
    \label{choice of C, example: necessity of the tail bound on decay functions}
\end{align}
and
define the event
\begin{align}
    (A:*) \delequal 
    \bigg\{
        \exists i \text{ such that } T_i \leq \frac{nu}{2},\  B_i > C n^{\theta + 1}
    \bigg\}.
    \label{def, event A *, example: necessity of the tail bound on decay functions}
\end{align}
That is, on the event $(A:*)$, an immigrant arrives before time $\frac{nu}{2}$ and plans to give birth to more than $n^{\theta + 1}$ descendants in one generation.
Next, on this event, we use $i^*$ to denote the smallest $i$ such that the conditions in \eqref{def, event A *, example: necessity of the tail bound on decay functions} hold. 
Let the sequence $(T^{(i^*)}_{ k  })_{k \geq 1}$
be independent copies with density $f^N(\cdot)/\norm{f^N}_1 = f^N(\cdot)$;
this sequence denotes the gap between the arrival times of the $i^*$-th immigrant and its first-generation offspring.
Using these notations, we define the event
\begin{align}
    (A:0) \delequal 
    \bigg\{
        \#\big\{ k \leq \ceil{C n^{\theta + 1}}:\ T^{(i^*)}_k \leq nu/2  \big\} > \mu_N \cdot n
    \bigg\}.
\end{align}
On the event $(A:*)\cap (A:0)$, note that in the Hawkes process $N(t)$
we must have more than $\mu_{\bm N} \cdot n$ points arriving by time $nu$.
Indeed, on this event, there is an immigrant who arrives by time $nu/2$ and gives birth to more than $n \cdot \mu_N$ of its first-generation descendants within $nu/2$ time after its arrival.
Then, under the $\bo(n)$ scaling for both time and space in $\bar{\bm N}^{\subInfty}_n$,
we have $\bar{ N}_n(u) - \bar{ N}_n(0) > \mu_N$.
Besides, for each $l \in [k - 1]$ we define 
\begin{align}
    (A:l) \delequal 
    \Bigg\{
        \#\bigg\{ k \leq \ceil{C  n^{\theta + 1}}:\ T^{(i^*)}_k \in \bigg[ nt_l,\  nt_l + \frac{nu}{2} \bigg]  \bigg\}
        > \mu_N \cdot n
    \Bigg\}.
\end{align}
Similarly, on the event $(A:*)\cap (A:l)$ we must have more than $\mu_N \cdot n$ points arriving during $[nt_l,n(t_l + u)]$.
Then, under the $\bo(n)$ scaling for both time and space in $\bar{\bm N}^{\subInfty}_n$,
we have $\bar{ N}_n(t_l + u) - \bar{N}_n(t_l) > \mu_N$.
In summary,
on the event $(A:*) \cap \big( \bigcup_{l = 0,1,\ldots,k}(A:l) \big)$,
we have 
$
\bar{\bm N}^{\subInfty}_n \in A.
$
Furthermore, if we can verify that
\begin{align}
    \liminf_{n \to \infty}\P\big( (A:*) \big) \cdot n^{ \alpha(\theta+1) - 1  } & > 0,
    \label{claim, event A *, example: necessity of the tail bound on decay functions}
    \\ 
    \lim_{n \to\infty}\P\big( (A:l)\ \big|\ (A:*)  \big) & = 1,\qquad\forall l = 0,1,\ldots,k - 1,
    \label{claim, event A l, example: necessity of the tail bound on decay functions}
\end{align}
then, the lower bound \eqref{goal, lower bound, example: necessity of the tail bound on decay functions} follows immediately.

\medskip
\noindent
\textbf{Verification of Claim~\eqref{claim, event A *, example: necessity of the tail bound on decay functions}}.
By the law of a compound Poisson process, 
\begin{align*}
    \P\big( (A:*) \big)
    & = 
    \P\bigg(
        \text{Poisson}\Big( \frac{un}{2} \cdot c^N \cdot \P(B > C n^{\theta + 1}) \Big)
        \geq 1
    \bigg)
    = 
    1 - \exp
    \bigg(
        -  \frac{c^Nu}{2} \cdot n\P(B >C n^{\theta + 1})
    \bigg).
\end{align*}
Then, by  $\P(B > x) \sim \P(\tilde B > x)$ as $x \to \infty$ and \eqref{def: law of excitation rates, example: necessity of the tail bound on decay functions},
\begin{align*}
    \lim_{n \to \infty}
    \P\big( (A:*) \big)\cdot n^{ \alpha(\theta+1) - 1 }
    & = 
    \frac{c^Nu}{2} \cdot
    \lim_{n \to \infty}
     n^{ \alpha(\theta+1) - 1 } \cdot
    {
        n\P(B > C n^{\theta + 1})
    }
    = 
    \frac{c^Nu}{2} \cdot \big(C/\tilde c\big)^{-\alpha} > 0.
\end{align*}

\medskip
\noindent
\textbf{Verification of Claim~\eqref{claim, event A l, example: necessity of the tail bound on decay functions}, ($l = 0$)}.
We first verify the case where $l = 0$.
Since the $T^{(i^*)}_k$'s are independent copies with density $f^N(\cdot)$, we have
\begin{align*}
    & \P\Big( \big((A:0)\big)^\complement\ \Big|\ (A:*) \Big)
    \\ 
    & = 
    \P\bigg(
        \text{Binomial}\Big( \ceil{Cn^{\theta + 1}}, \int_0^{ nu/2  }f^N(t)dt  \Big) \leq \mu_N n
    \bigg)
    \\ 
    & \leq 
    \P\bigg(
        \text{Binomial}\Big( \ceil{Cn^{\theta + 1}}, 1/2  \Big) \leq \mu_N n
    \bigg)
    \quad 
    \text{for any $n$ large enough due to }\norm{f^N}_1 = 1
    \\ 
    & \leq 
    \P\bigg(
        \text{Binomial}\Big( \ceil{Cn^{\theta + 1}}, 1/2  \Big) \leq \ceil{C n^{\theta + 1}}/4
    \bigg)
    \quad
    \text{for any $n$ large enough due to $\theta > 0$}.
\end{align*}
Then, by the law of large numbers,
\begin{align*}
    \lim_{n \to \infty}
    \P\Big( \big((A:0)\big)^\complement\ \Big|\ (A:*) \Big)
    & \leq 
    \lim_{n \to \infty}
    \P\bigg(
        \text{Binomial}\Big( \ceil{Cn^{\theta + 1}}, 1/2  \Big) < \ceil{Cn^{\theta + 1}}/4
    \bigg)
    = 0.
\end{align*}
This verifies that 
$
\lim_{n \to \infty}
    \P\big( (A:0)\ \big|\ (A:*) \big) = 1.
$

\medskip
\noindent
\textbf{Verification of Claim~\eqref{claim, event A l, example: necessity of the tail bound on decay functions}, ($l = 1,2,\ldots,k-1$)}.
Next, we take any $l \in [k-1]$. Note that
\begin{align*}
    \P\Big( \big((A:l)\big)^\complement\ \Big|\ (A:*) \Big)
    & = 
    \P\bigg(
        \text{Binomial}\Big( \ceil{Cn^{\theta + 1}}, \underbrace{ \int_{nt_l}^{ n(t_l + \frac{u}{2})  }f^N(t)dt}_{ \delequal F_l }  \Big) \leq \mu_N n
    \bigg).
\end{align*}
By \eqref{def: decay function, example: necessity of the tail bound on decay functions}
and that $t_l + u < 1$,
\begin{align*}
    F_l = (nt_l)^{- \theta} - \bigg(nt_l + \frac{nu}{2}\bigg)^{- \theta}
    \geq
    \frac{\theta}{ (1+n)^{\theta + 1}  } \cdot \frac{nu}{2}.
\end{align*}
As a result, 
for any $n$ large enough we have
 $F_l \geq \frac{\theta u}{4}  \cdot n^{-\theta}$, and hence
\begin{align*}
    \P\Big( \big((A:l)\big)^\complement\ \Big|\ (A:*) \Big)
    \leq 
    \P\bigg(
        \text{Binomial}
        \Big( \ceil{Cn^{\theta + 1}}, 
            \frac{\theta u}{4}  \cdot n^{-\theta}
        \Big) \leq \mu_N n
    \bigg).
\end{align*}
 By Chernoff bounds, for any $w > 0$ we have
 \begin{align*}
     \P\Big( \big((A:l)\big)^\complement\ \Big|\ (A:*) \Big)
     \leq 
     \exp(w\mu_N n)
     \cdot 
     \bigg( 1 - \frac{\theta u}{4 n^\theta} +   \frac{\theta u e^{-w}}{4 n^\theta}  \bigg)^{ C n^{\theta + 1} }. 
 \end{align*}
Then, due to $(1 + x_n)^n \sim \exp( n x_n)$ when $\lim_{n \to \infty }x_n = 0$,
we get (for any $w > 0$)
\begin{align*}
    \limsup_{n \to \infty}
    \P\Big( \big((A:l)\big)^\complement\ \Big|\ (A:*) \Big)
    & \leq 
    \exp\bigg( - C n^{\theta + 1}\cdot (1 - e^{-w})\cdot \frac{\theta u}{4 n^\theta} + w\mu_N n \bigg)
    \\ 
    & = 
    \exp\bigg(
    -n \cdot \Big( \frac{C\theta u(1 - e^{-w}) }{4} - w\mu_N   \Big)
    \bigg).
\end{align*}
By our choice of $C$ in \eqref{choice of C, example: necessity of the tail bound on decay functions},
it holds for any $w > 0$ small enough that 
$
\frac{C\theta u(1 - e^{-w}) }{4} - w\mu_N  > 0,
$
which implies 
$
\limsup_{n \to \infty}
    \P\Big( \big((A:l)\big)^\complement\ \Big|\ (A:*) \Big) = 0
$
and hence
$
\lim_{n \to \infty}
    \P\big( (A:l)\ \big|\ (A:*) \big) = 1.
$
\end{example}

\section{Proof for Large Deviations of L\'evy Processes with $\MRV$ Increments}
\label{subsubsec: proof of LD with MRV}

Without loss of generality, we prove Theorem~\ref{theorem: LD for levy, MRV} for $T = 1$.
Besides, considering the arbitrariness of the drift constant $\bm c_{\bm L}$ in \eqref{def: expectation of levy process},
we can w.l.o.g.\ impose the following assumption and focus on centered L\'evy processes for the proof of Theorem~\ref{theorem: LD for levy, MRV}.
\begin{assumption}[WLOG Assumption for Theorem~\ref{theorem: LD for levy, MRV}]
\label{assumption: WLOG, LD for Levy MRV}
$\bm \mu_{\bm L} = \bm 0$ in \eqref{def: expectation of levy process}, and $T = 1$.
\end{assumption}
Henceforth in Section~\ref{subsubsec: proof of LD with MRV}, we also lighten notations in the proofs by writing
\begin{align}
    \notationdef{notation-LD-cadlag-space-D-j-c-shorthand}{ \shortbarDepsk{\epsilon}{ \bm{\mathcal K} }  }
    \delequal
    \barDxepskT{\bm 0}{\epsilon}{\bm{\mathcal K} }{[0,1]},
    \quad
    \notationdef{notation-jump-sets-with-smaller-cost-LD-for-levy-shorthand}{ \shortbarDepsk{\epsilon}{\leqslant \bm k} }
    \delequal
    \barDxepskT{\bm 0}{\epsilon}{\leqslant \bm k}{[0,1]},
    \quad
    \notationdef{notation-measure-C-j-L-LD-for-Levy-shorthand}{\breve{\mathbf C}_{\bm{\mathcal K} } }
    \delequal
    \breve{\mathbf C}^{ _{[0,1]} }_{\bm{\mathcal K} ;\bm 0},
    \quad 
    \notationdef{notation-bar-L-n-scaled-levy-process-shorthand}{\bar{\bm L}_n}
    \delequal
    \bar{\bm L}_n^{ _{[0,1]} },
    \quad 
    \notationdef{notation-J1-metric-shorthand}{\dj{}}\delequal\dj{[0,1]}.
    \label{proof, def: shorthand, LD for Levy}
\end{align}

\subsection{Proof of Theorem~\ref{theorem: LD for levy, MRV}}
\label{subsec: proof, LD of Levy, proof of main result}

We start by identifying the \emph{large jumps} in $\bar{\bm L}_n$.
Given $n \in \mathbb N$ and $\delta > 0$, let
\begin{align}
    \notationdef{notation-ith-large-jump-tau-delta-n-i}{\tau^{>\delta}_n(k)} & \delequal
    \inf\bigg\{
        t > \tau^{>\delta}_n(k - 1):\ 
            \norm{ \Delta \bar{\bm L}_n(t) } > \delta
    \bigg\},\ 
    \ \forall k \geq 1,
    \qquad
    \tau^{>\delta}_n(0) \delequal 0;
    \label{def: large jump time tau for Levy process L}
    \\
     \notationdef{notation-ith-large-jump-W-delta-n-i}{\bm W^{>\delta}_n(k)} & \delequal \Delta \bar{\bm L}_n\Big( \tau^{>\delta}_n(k) \Big),
    \qquad \forall k \geq 1.\label{def: large jump size tau for Levy process L}
\end{align}
In \eqref{def: large jump time tau for Levy process L}, note that $\Delta \bar{\bm L}_n(t) = \Delta \bm{L}(nt)/n$.
Intuitively speaking, the sequence $\big(\tau^{>\delta}_n(k)\big)_{k \geq 1}$ marks the arrival times of jumps in $\bar{\bm L}_n(t)$ with size (in terms of $L_1$ norm) larger than $\delta$, which correspond to jumps of size larger than $n\delta$ in $\bm L_n(t)$;
the sequence $\big({\bm W^{>\delta}_n(k)}\big)_{k \geq 1}$ represents the sizes of the large jumps.

Building upon the definitions in \eqref{def: large jump time tau for Levy process L}--\eqref{def: large jump size tau for Levy process L},
we introduce a \emph{large-jump approximation} for $\bar{\bm L}_n(t)$.
Specifically, given $\delta > 0$ and $n \geq 1$, let 
\begin{align}
    \notationdef{notation-large-jump-approximation-LD-for-Levy-MRV}{\hat{\bm L}^{>\delta}_n(t)}
    \delequal 
    \sum_{ k \geq 1 }\bm W^{>\delta}_n(k)\mathbbm I_{ [\tau^{>\delta}_n(k), 1] }(t),
    \qquad\forall t \in [0,1],
    \label{def: large jump approximation, LD for Levy MRV}
\end{align}
and 
$
\hat{\bm L}^{>\delta}_n \delequal
\{\hat{\bm L}^{>\delta}_n(t):\ t \in [0,1]\}.
$
Clearly, ${\hat{\bm L}^{>\delta}_n(t)}$ is a step function (i.e., piece-wise constant) in $\D = \D\big([0,1],\R^d\big)$ that vanishes at the origin and approximates $\bar{\bm L}_n(t)$ by only keeping the large jumps (under threshold $\delta$).

Our proof of Theorem~\ref{theorem: LD for levy, MRV}
hinges on
 Propositions~\ref{proposition: asymptotic equivalence, LD for Levy MRV} and \ref{proposition: weak convergence, LD for Levy MRV}
that outline the following key steps:
(i) first, we establish the asymptotic equivalence between $\bar{\bm L}_n$ and $\hat{\bm L}^{>\delta}_n$ as in Lemma~\ref{lemma: asymptotic equivalence when bounded away, equivalence of M convergence};
(ii) next, we provide a detailed asymptotic analysis for the law of large jumps.
To carry out these two steps, we further classify the large jumps $\bm W^{>\delta}_n(k)$ based on their locations
w.r.t.\ some cones $\bar\R^{>\delta}(\bm j)$.
Specifically, let
\begin{align}
    \notationdef{notation-tilde-R-delta-j-cone}{\bar\R^{>\delta}(\bm j)}
    \delequal
    \big\{
        \bm x \in \R^d_+:
        \norm{\bm x} > \delta,\ \bm x \in \bar{\R}^d(\bm j,\delta)\setminus \bar{\R}^d_\leqslant(\bm j,\delta)
    \big\},
    \qquad
    \forall 
    \delta > 0,\ \bm j \in \powersetTilde{d}.
    \label{def: tilde R geq delta, LD for Levy}
\end{align}
We highlight several useful properties regarding ${\bar\R^{>\delta}(\bm j)}$.
\begin{itemize}
    \item 
        First, note that $\bigcup_{ \bm j \in \powersetTilde{d}  }\bar\R^{>\delta}(\bm j) = \{ \bm x \in \R^d_+:\ \norm{\bm x} > \delta,\ \bm x \in \bar{\R}([d],\delta)   \}$.
Indeed, 
it is obvious that 
$\bigcup_{ \bm j \in \powersetTilde{d}  }\bar\R^{>\delta}(\bm j) \subseteq \{ \bm x \in \R^d_+:\ \norm{\bm x} > \delta,\ \bm x \in \bar{\R}([d],\delta)   \}$.
On the other hand, given any $\delta > 0$ and 
$\bm x \in \bar\R^d([d],\delta)$ with $\norm{\bm x} > \delta$,
the argument minimum
\begin{align}
    \bm j (\bm x) \delequal \underset{ \bm j \in \powersetTilde{d}:\ \bm x \in \bar\R^d(\bm j,\delta)  }{\arg\min}
    \alpha(\bm j)
    \nonumber
\end{align}
uniquely exists under 
the condition $\alpha(\bm j) \neq \alpha(\bm j^\prime)\ \forall \bm j,\bm j^\prime \in \powersetTilde{d}$
in Assumption~\ref{assumption: ARMV in levy process}.
Then,
we must have $\bm x \in \bar\R^{>\delta}(\bm j(\bm x))$ due to the definition of $\bar\R^d_\leqslant(\bm j,\delta) = 
\bigcup_{
                \bm j^\prime \subseteq [k]:\ 
                \bm j^\prime \neq \bm j,\ \alpha(\bm j^\prime) \leq \alpha(\bm j)
            } \bar\R^d(\bm j^\prime,\delta).
$
This confirms that 
$\bigcup_{ \bm j \in \powersetTilde{d}  }\bar\R^{>\delta}(\bm j) \supseteq \{ \bm x \in \R^d_+:\ \norm{\bm x} > \delta,\ \bm x \in \bar{\R}([d],\delta)   \}$.

    \item 
        Next, we note that $\bar \R^{>\delta}(\bm j) \cap \bar \R^{>\delta}(\bm j^\prime) = \emptyset$ holds for any $\bm j,\bm j^\prime \in \powersetTilde{d}$ with $\bm j \neq \bm j^\prime$.
To see why, we assume
w.l.o.g.\ that $\alpha(\bm j^\prime) \leq \alpha(\bm j)$.
Then, due to $\bm j^\prime \neq \bm j$ and $\alpha(\bm j^\prime) \leq \alpha(\bm j)$,
we  have $\bar\R^{>\delta}(\bm j^\prime) \subseteq \bar\R^d_\leqslant(\bm j,\delta)$.
Therefore, for any $\bm x \in \bar \R^{>\delta}(\bm j^\prime)$, we must have
$\bm x \notin \bar\R^d(\bm j,\delta)\setminus \bar\R^d_\leqslant(\bm j,\delta)$.
Repeating this argument for all $\bm j, \bm j^\prime \in \powersetTilde{d}$,
we confirm that $\bar \R^{>\delta}(\bm j) \cap \bar \R^{>\delta}(\bm j^\prime) = \emptyset$.

    \item 
        In summary, 
the collection of sets $\bar \R^{>\delta}(\bm j)$ provide a partition of $\{ \bm x\in \R^d_+:\ \norm{\bm x} > \delta,\ \bm x\in \bar\R^d([d],\delta)  \}$.
Meanwhile, note that
\begin{align}
    \sum_{ \bm j \in \powersetTilde{d} }\mathbf C_{\bm j}\Big(
        \partial{\bar{\R}^{>\delta}(\bm j)}
        \Big) = 0,
    \qquad
    \text{for all but countably many }\delta > 0,
    \label{proof: choice of c, bounded away set from cone i, Ld for Levy MRV}
\end{align}
where we use $\notationdef{notation-boundary-of-a-set}{\partial E} \delequal E^- \setminus E^\circ$ to denote the boundary of the set $E$.
This claim is a  consequence of the definition of $\MRV$,
and we state the proof below.
\end{itemize}

\begin{lemma}\label{lemma: zero mass for boundary sets, LD for Levy}
\linksinthm{lemma: zero mass for boundary sets, LD for Levy}
Let Assumption~\ref{assumption: ARMV in levy process} hold.
Let $\bar\R^{>\delta}(\bm j)$ be defined as in \eqref{def: tilde R geq delta, LD for Levy}.
For each $\bm j \in \powersetTilde{d}$,
\begin{align}
    \mathbf C_{\bm j}\Big(
        \partial{\bar{\R}^{>\delta}(\bm j)}
        \Big) = 0,
    \qquad
    \text{for all but countably many }\delta > 0.
    \label{claim, lemma: zero mass for boundary sets, LD for Levy}
\end{align}
\end{lemma}

\begin{proof}
\linksinpf{lemma: zero mass for boundary sets, LD for Levy}
Take any $\delta_0 > 0$.
For each $\delta > \delta_0$, note that
\begin{align*}
    \partial \bar \R^{>\delta}(\bm j)
    & \subseteq
    \underbrace{ 
        \big\{ \bm x \in \R^d_+ \setminus \bar\R^d_\leqslant(\bm j,\delta_0):\ \norm{\bm x} = \delta \big\}
    }_{ \delequal E_1(\delta) }
    \cup 
    \underbrace{ \big\{
        \bm x \in \R^d_+ \setminus \bar\R^d_\leqslant(\bm j,\delta_0):\ 
        \bm x \in \partial \bar\R^d(\bm j,\delta),\ \norm{\bm x} > \delta_0 
    \big\}
    }_{ \delequal E_2(\delta) }
    \\ 
    & \quad 
    \cup 
    \underbrace{ 
       \big\{ \bm x \in \R^d_+ \setminus \bar\R^d_\leqslant(\bm j,\delta_0):\   \bm x \in \partial \bar\R^d_\leqslant(\bm j,\delta),\ \norm{\bm x} > \delta_0   \big\}
    }_{ \delequal E_3(\delta) }.
\end{align*}
We first analyze the set $E_1(\delta)$.
Note that $E_1(\delta)\cap E_1(\delta^\prime) = \emptyset$ for any $\delta_0 < \delta < \delta^\prime$, and
$
\bigcup_{ \delta > \delta_0 }E_1(\delta) \subseteq \R^d_+ \setminus \bar\R^d_\leqslant(\bm j,\delta_0).
$
Meanwhile, 
under the $\MRV$ condition in Assumption~\ref{assumption: ARMV in levy process},
we have $$\mathbf C_{\bm j}\Big(  \big(\bar\R^d_\leqslant(\bm j,\delta_0)\big)^\complement \Big) < \infty,\qquad,\delta > 0;$$ see Definition~\ref{def: MRV}.
As a result, given any $\epsilon > 0$, there exists at most finitely many $\delta \in (\delta_0,\infty)$ such that $\mathbf C_{\bm j}\big(E_1(\delta)\big) > \epsilon$.
Sending $\epsilon \downarrow 0$, we confirm that
$
\mathbf C_{\bm j}\big(E_1(\delta)\big) = 0
$
for all but countably many $\delta \in (\delta_0, \infty)$.

Similarly, note that 
 $
        \bigcup_{ \delta > \delta_0 }E_2(\delta) \subseteq \R^d_+ \setminus \bar\R^d_\leqslant(\bm j,\delta_0),
        $
and that 
$
        E_2(\delta) \cap E_2(\delta^\prime) = \emptyset
        $
for any $\delta^\prime > \delta > \delta_0$.
The same arguments above confirm that 
$
\mathbf C_{\bm j}\big(E_2(\delta)\big) = 0
$
for all but countably many $\delta \in (\delta_0, \infty)$.
Next, 
by the definition of $\bar\R^d_\leqslant(\bm j,\delta)$,
we have $E_3(\delta) \subseteq \bigcup_{ \bm j^\prime \in \powersetTilde{d}:\ \bm j^\prime \neq \bm j,\ \alpha(\bm j^\prime) \leq \alpha(\bm j) }E^{\bm j^\prime}_3(\delta)$, where
\begin{align*}
    E^{\bm j^\prime}_3(\delta)
    \delequal 
    \big\{ \bm x \in \R^d_+ \setminus \bar\R^d_\leqslant(\bm j,\delta_0):\   \bm x \in \partial \bar\R^d(\bm j^\prime,\delta),\ \norm{\bm x} > \delta_0   \big\}.
\end{align*}
For each $\bm j^\prime \in \powersetTilde{d}$ with $\bm j^\prime \neq \bm j,\ \alpha(\bm j^\prime) \leq \alpha(\bm j)$,
one can see that 
 $
        \bigcup_{ \delta \in (\delta_0,\infty) }E_3^{\bm j^\prime}(\delta) \subseteq \R^d_+ \setminus \bar\R^d_\leqslant(\bm j,\delta_0),
        $
and 
$
        E_3^{\bm j^\prime}(\delta) \cap E_3^{\bm j^\prime}(\delta^\prime) = \emptyset
        $
for any $\delta^\prime > \delta > \delta_0$.
Again, we get
$$
\mathbf C_{\bm j}\big(E_3(\delta)\big)
\leq 
\sum_{ \bm j^\prime \in \powersetTilde{d}:\ \bm j^\prime \neq \bm j,\ \alpha(\bm j^\prime) \leq \alpha(\bm j) }
\mathbf C_{\bm j}\big(E_3^{\bm j^\prime}(\delta)\big) = 0
$$
for all but countably many $\delta \in (\delta_0, \infty)$.

In summary, we have verified that the claim
$
\mathbf C_{\bm j}\big( \partial \bar\R^{>\delta}(\bm j) \big) = 0
$
holds
for all but countably many $\delta \in (\delta_0, \infty)$.
We conclude the proof by sending $\delta_0 \downarrow 0$.
\end{proof}

To proceed, given $\delta  > 0$, $\bm j \in \powersetTilde{d}$, and $n \in \mathbb N$, let (for each $k\geq 1$)
\begin{align}
    \notationdef{notation-kth-jump-time-in-cone-i-LD-for-MRV}{\tau^{>\delta}_{ n }(k; \bm j)}
    & \delequal
    \inf\Big\{ 
        t >  \tau^{>\delta}_{ n }(k - 1; \bm j):\ 
         \Delta \bar{\bm L}_n(t) \in \bar\R^{>\delta}(\bm j)
    \Big\},
    \label{def: large jump in cone i, tau > delta n k i, LD for Levy MRV}
    \\ 
    \notationdef{notation-kth-jump-size-in-cone-i-LD-for-MRV}{\bm W^{>\delta}_{ n }(k; \bm j)}
    & \delequal 
    \Delta \bar{\bm L}_n\Big( \tau^{>\delta}_{ n }(k; \bm j) \Big),
     \label{def: large jump size in cone i, tau > delta n k i, LD for Levy MRV}
\end{align}
and we adopt the convention $\tau^{>\delta}_{ n }(0; \bm j) = 0$.
Analogous to the definitions in \eqref{def: large jump time tau for Levy process L}--\eqref{def: large jump size tau for Levy process L}, the ${\tau^{>\delta}_{ n }(k; \bm j)}$'s and ${\bm W^{>\delta}_{ n }(k; \bm j)}$'s are the arrival times and sizes of the $k^\text{th}$ large jump that belongs to $\bar{\R}^{>\delta}(\bm j)$.
To provide the individual count for different types of large jumps, we define
\begin{align}
    \notationdef{notation-count-for-large-jump-in-cone-j}{K^{>\delta}_n(\bm j)} \delequal \max\{ k \geq 0:\ \tau^{>\delta}_n(k;\bm j) \leq 1 \},
    \qquad
    \forall \bm j \in \powersetTilde{d},
    \label{def: K delta n type j, count scalar, LD for Levy}
\end{align}
and keep track of all counts by defining
\begin{align}
    \notationdef{notation-count-for-large-jump-in-cones}{\bm K^{>\delta}_n}
    \delequal
    \Big( {K^{>\delta}_n(\bm j)} \Big)_{ \bm j \in \powersetTilde{d} } \in \mathbb Z_+^{ \powersetTilde{d} }.
    \label{def: K delta n type j, count vector, LD for Levy}
\end{align}
Similarly, we define
\begin{align}
    \widetilde \tau^{>\delta}_n(k)
    &\delequal 
    \inf\big\{
        t > \tau^{>\delta}_n(k-1):\  \Delta \bar{\bm L}_n(t) \notin \bar\R^d([d],\delta)
    \big\},
    \qquad
    \tilde \tau^{>\delta}_n(0) = 0,
    \label{def: tilde tau n delta k, big jump outside of full cone}
    \\ 
    \widetilde K^{>\delta}_n 
    & \delequal 
    \max\{ k \geq 0:\ \widetilde\tau^{>\delta}_n(k) \leq 1 \},
    \label{def: tilde K n delta k, big jump outside of full cone}
\end{align}
for big jumps that are not aligned with any direction of the cone $\bar\R^d([d],\delta)$.
Given $\bm{\mathcal K} = (\mathcal K_{\bm j})_{ \bm j \in \powersetTilde{d} } \in \mathbb Z_+^{\powersetTilde{d}}$, 
note that on the event $\{ {\bm K^{>\delta}_n} = \bm{\mathcal K}  \} \cap \{ \widetilde K^{>\delta}_n = 0  \}$,
the process $\hat{\bm L}^{>\delta}_n(t)$ admits the form
\begin{align}
    {\hat{\bm L}^{>\delta}_n(t)}
    =
    \sum_{ \bm j \in \powersetTilde{d} } \sum_{  k \in [\mathcal K_{\bm j}] }
        \bm W^{>\delta}_n(k;\bm j)\mathbbm{I}_{ [ \tau^{>\delta}_n(k;\bm j) ,1 ]  }(t),
        \qquad
        \forall t \in [0,1]
        \label{property: approximation hat L n when type = j, LD for Levy MRV}
\end{align}
with the ${\tau^{>\delta}_{ n }(k; \bm j)}$'s and ${\bm W^{>\delta}_{ n }(k; \bm j)}$'s defined in 
\eqref{def: large jump in cone i, tau > delta n k i, LD for Levy MRV}--\eqref{def: large jump size in cone i, tau > delta n k i, LD for Levy MRV}.
In particular, on the event $\{ {\bm K^{>\delta}_n} = \bm{\mathcal K}  \} \cap \{ \widetilde K^{>\delta}_n = 0  \}$, 
we have $\hat{\bm L}^{>\delta}_n \in \shortbarDepsk{\delta}{ \bm{\mathcal K} } = 
 \barDxepskT{\bm 0}{\delta}{ \bm{\mathcal K} }{[0,1]},$; 
see \eqref{def: j jump path set with drft c, LD for Levy, MRV}.

We are now ready to state  Propositions~\ref{proposition: asymptotic equivalence, LD for Levy MRV} and \ref{proposition: weak convergence, LD for Levy MRV}.
Specifically, let 
\begin{align}
    \bar{\bm K}^{>\delta}_n
    \delequal
    \big({\bm K}^{>\delta}_n, \widetilde K^\delta_n\big)
\end{align}
be the concatenation of $\bm K^{>\delta}_n$ and $\widetilde K^{>\delta}_n$ defined in 
\eqref{def: K delta n type j, count scalar, LD for Levy}--\eqref{def: tilde K n delta k, big jump outside of full cone}.
Let 
\begin{align}
    \bar{\mathbb A}(\bm k)
    \delequal 
    \big\{
        (\bm{\mathcal K}, 0):\ \bm{\mathcal K} \in {\mathbb A}(\bm k)
    \big\},
    \qquad\forall
    \bm k \in \mathbb Z^d_+,
\end{align}
where $\mathbb A(\bm k)$ is introduced in Definition~\ref{def: assignment, jump configuration k}.
That is, we augment each allocation $ \bm{\mathcal K} \in  \mathbb A(\bm k)$ with the same constant $0$.

\begin{proposition}\label{proposition: asymptotic equivalence, LD for Levy MRV}
\linksinthm{proposition: asymptotic equivalence, LD for Levy MRV}
Let Assumptions~\ref{assumption: ARMV in levy process} and \ref{assumption: WLOG, LD for Levy MRV} hold.
For each $\bm k = (k_{j})_{ j \in [d] } \in \mathbb Z_+^{ d } \setminus \{\bm 0\}$ and $\Delta > 0$,
\begin{align}
    \lim_{n \to \infty}
    \P\Big(
        \dj{}\big( \hat{\bm L}_n^{>\delta}, \bar{\bm L}_n \big)
        > \Delta 
    \Big) 
    \Big/ \breve \lambda_{\bm k}(n)
    = 0,
    \qquad
    \forall \delta > 0\text{ small enough,}
    \label{proof: goal 1, AE, theorem: LD for levy, MRV}
\end{align}
where $\breve \lambda_{\bm k}(n)$ is defined in \eqref{def: scale function for Levy LD MRV}.
\end{proposition}

\begin{proposition}\label{proposition: weak convergence, LD for Levy MRV}
\linksinthm{proposition: weak convergence, LD for Levy MRV}
Let Assumptions~\ref{assumption: ARMV in levy process} and \ref{assumption: WLOG, LD for Levy MRV} hold.
Let $f: \D \to [0,\infty)$ be bounded (i.e., $\norm{f} \delequal \sup_{\xi \in \D}|f(\xi)| < \infty$) and continuous (w.r.t.\ the $J_1$ topology of $\D$).
Let $\bm k = (k_{j})_{ j \in [d] } \in \mathbb Z_+^{ d } \setminus \{\bm 0\}$ and $\epsilon > 0$.
Suppose 
that
$B = \text{supp}(f)$
is bounded away from $\shortbarDepsk{\epsilon}{\leqslant \bm k}$
    under $\dj{}$.
Then, 
\begin{enumerate}[(a)]
    \item
        for any $\delta \in (0,\epsilon)$, 
        \begin{align}
            \lim_{n \to \infty}
    \frac{
        \E\Big[ f(\hat{\bm L}^{>\delta}_n)\mathbbm{I}\big\{ \bar{\bm K}^{>\delta}_n \notin \bar{\mathbb A}(\bm k) \big\} \Big]
    }{
        \breve \lambda_{\bm k}(n)
    }
    = 0;
            \nonumber
        \end{align}

    \item
        there exists $\delta_0 > 0$ such that
        for all but countably many $\delta \in (0,\delta_0)$,
        \begin{align*}
            \lim_{n \to \infty}
            \frac{
                \E\Big[ f(\hat{\bm L}^{>\delta}_n)\mathbbm{I}\big\{ 
                \bar{\bm K}^{>\delta}_n = (\bm{\mathcal K},0)
                \big\} \Big]
         }{
             \breve \lambda_{\bm k}(n)
            }
            =
            \breve{\mathbf C}_{\bm{\mathcal K}}(f) < \infty,
            \qquad\forall \bm{\mathcal K} \in \mathbb A(\bm k),
        \end{align*}
        where the measure $\breve{\mathbf C}_{\bm{\mathcal K}}(\cdot) = \breve{\mathbf C}^{ \subZeroT{1} }_{\bm{\mathcal K};\bm 0}(\cdot)$ is defined in \eqref{def: measure C type j L, LD for Levy, MRV}.
\end{enumerate}
\end{proposition}

We defer their proofs to Section~\ref{subsec: proof, LD of Levy, proof of technical results},
and conclude this subsection by applying Propositions~\ref{proposition: asymptotic equivalence, LD for Levy MRV} and \ref{proposition: weak convergence, LD for Levy MRV} and establishing Theorem~\ref{theorem: LD for levy, MRV}.

\begin{proof}[Proof of Theorem~\ref{theorem: LD for levy, MRV}]
\linksinpf{theorem: LD for levy, MRV}
Without loss of generality, we prove Theorem~\ref{theorem: LD for levy, MRV} under Assumption~\ref{assumption: WLOG, LD for Levy MRV}.
For any $\epsilon > 0$ and any Borel set $B \subseteq \D$ that is bounded away from $\shortbarDepsk{\epsilon}{\leqslant \bm k}$,
note that $B^\Delta$ is also bounded away from $\shortbarDepsk{\epsilon}{\leqslant \bm k}$ under any $\Delta > 0$ sufficiently small.
By Urysohn's Lemma, one can identify some bounded and continuous $f:\D \to [0,1]$ such that $\mathbbm{I}_{ B } \leq f \leq \mathbbm{I}_{B^{\Delta}}$.
Then, part (b) of Proposition~\ref{proposition: weak convergence, LD for Levy MRV} confirms that
$\sum_{\bm{\mathcal K} \in \mathbb A(\bm k)}\breve{\mathbf C}_{\bm{\mathcal K}}(B)
\leq 
\sum_{\bm{\mathcal K} \in \mathbb A(\bm k)}\breve{\mathbf C}_{\bm{\mathcal K}}(f)
< \infty$.
This establishes the finite upper bound in Claim~\eqref{claim, finite index, theorem: LD for levy, MRV} and verifies that $\sum_{\bm{\mathcal K}\in \mathbb A(\bm k)}\breve{\mathbf C}_{\bm{\mathcal K}} \in \mathbb M\big( \D \setminus  \shortbarDepsk{\epsilon}{\leqslant \bm k} \big)$, thus allowing us to apply Lemma~\ref{lemma: asymptotic equivalence when bounded away, equivalence of M convergence}.
Next, given $\epsilon > 0$ and
under the choice of 
\begin{align}
    (\mathbb S,\bm d) = (\D,\dj{}),\quad 
    \mathbb C = \shortbarDepsk{\epsilon}{\leqslant \bm k},\quad 
    \epsilon_n = \breve \lambda_{\bm k}(n),
    \quad 
    X_n = \bar{\bm L}_n,\quad 
    Y^\delta_n = \hat{\bm L}^{>\delta}_n,
    \quad 
    V^\delta_n = \bar{\bm K}^{>\delta}_n,
    \quad
    \mathcal V = \bar{\mathbb A}(\bm k),
    \label{proof, choice of abstract notations, theorem: LD for levy, MRV}
\end{align}
Condition (i) of Lemma~\ref{lemma: asymptotic equivalence when bounded away, equivalence of M convergence},
is verified by Proposition~\ref{proposition: asymptotic equivalence, LD for Levy MRV},
and Condition (ii) is verified by Proposition~\ref{proposition: weak convergence, LD for Levy MRV}.
In particular, 
in Condition (i) of  Lemma~\ref{lemma: asymptotic equivalence when bounded away, equivalence of M convergence}, note that 
$$
\P\big( \bm{d}(X_n, Y^\delta_n)\mathbbm{I}( X_n \in B\text{ or }Y^\delta_n \in B ) > \Delta \big)
\leq 
\P\big( \bm{d}(X_n, Y^\delta_n) > \Delta \big),
$$
and the claim
$
\P\big( \bm{d}(X_n, Y^\delta_n) > \Delta \big)= \lo( \epsilon_n)
$
as $n \to \infty$, under the choices in \eqref{proof, choice of abstract notations, theorem: LD for levy, MRV} with $\delta > 0$ small enough, is exactly the content of Proposition~\ref{proposition: asymptotic equivalence, LD for Levy MRV}.
As for Condition (ii), note that the claims are trivial when $B = \emptyset$.
In case that $B \neq \emptyset$, since $B$ is bounded away from $\shortbarDepsk{\epsilon}{\leqslant \bm k}$, we pick $\Delta > 0$ small enough $B^\Delta$ is still bounded away from $\mathbb C$.
By Urysohn's Lemma, one can find some bounded and continuous $f:\D \to [0,1]$ such that $\mathbbm{I}_{ B } \leq f \leq \mathbbm{I}_{B^{\Delta}}$.
Then, by applying part (b) of Proposition~\ref{proposition: weak convergence, LD for Levy MRV} onto such $f$,
one can identify some $\delta_0 = \delta_0(B,\Delta) > 0$ such that for all but countably many $\delta \in (0,\delta_0)$,
\begin{align*}
    & \limsup_{n \to \infty}
    \frac{
        \P\big(\hat{\bm L}^{>\delta}_n \in B,\ \bar{\bm K}^{>\delta}_n = (\bm{\mathcal K},0) \big)
        }{
           \breve \lambda_{\bm k}(n)
        }
        \\ 
        &
    \leq 
    \lim_{n \to \infty}
    \frac{
        \E\big[ f(\hat{\bm L}^{>\delta}_n)\mathbbm{I}\{ \bar{\bm K}^{>\delta}_n = (\bm{\mathcal K},0) \} \big]
        }{
           \breve \lambda_{\bm k}(n)
        }
    =
    \breve{\mathbf C}_{ \bm{\mathcal K} }(f)
    \leq 
    \breve{\mathbf C}_{\bm{\mathcal K}}( B^{\Delta}),
    \qquad \forall \bm{\mathcal K} \in \mathbb A(\bm k).
\end{align*}
By a similar argument using Urysohn's Lemma, one can identify some bounded and continuous $f$ with
$\mathbbm{I}_{ B_\Delta } \leq f \leq \mathbbm{I}_{B}$.
Using part (b) of Proposition~\ref{proposition: weak convergence, LD for Levy MRV} again (and by picking a smaller $\delta_0 > 0$ if necessary), it follows for all but countably many $\delta \in (0,\delta_0)$ that
\begin{align*}
    \liminf_{n \to \infty}
    \frac{
        \P\big(\hat{\bm L}^{>\delta}_n \in B, \bar{\bm K}^{>\delta}_n = (\bm{\mathcal K},0) \big)
        }{
           \breve \lambda_{\bm k}(n)
        }
    \geq
    \breve{\mathbf C}_{\bm{\mathcal K}}( B_{\Delta}),
    \qquad \forall \bm{\mathcal K} \in \mathbb A(\bm k).
\end{align*}
Likewise, by part (a) of of Proposition~\ref{proposition: weak convergence, LD for Levy MRV},
for the continuous and bounded $f$ with $\mathbbm{I}_{ B } \leq f \leq \mathbbm{I}_{B^{\Delta}}$ identified above,
it holds for any $\delta > 0$ small enough that 
\begin{align*}
    \limsup_{n \to \infty}
    \frac{
        \P\big(\hat{\bm L}^{>\delta}_n \in B,\ \bar{\bm K}^{>\delta}_n \notin \bar{\mathbb A}(\bm k)\big)
        }{
           \breve \lambda_{\bm k}(n)
        }
    \leq 
    \lim_{n \to \infty}
    \frac{
        \E\big[ f(\hat{\bm L}^{>\delta}_n)\mathbbm{I}\{ \bar{\bm K}^{>\delta}_n \notin \bar{\mathbb A}(\bm k)\} \big]
        }{
           \breve \lambda_{\bm k}(n)
        }
        = 0.
\end{align*}
These calculations verify Condition (ii) of Lemma~\ref{lemma: asymptotic equivalence when bounded away, equivalence of M convergence} under the choices in \eqref{proof, choice of abstract notations, theorem: LD for levy, MRV}.
Using Lemma~\ref{lemma: asymptotic equivalence when bounded away, equivalence of M convergence}, we conclude the proof of Theorem~\ref{theorem: LD for levy, MRV}.
\end{proof}

\subsection{Proofs of  Propositions~\ref{proposition: asymptotic equivalence, LD for Levy MRV} and \ref{proposition: weak convergence, LD for Levy MRV}}
\label{subsec: proof, LD of Levy, proof of technical results}

This subsection collects the proofs of Propositions~\ref{proposition: asymptotic equivalence, LD for Levy MRV} and \ref{proposition: weak convergence, LD for Levy MRV}.
To this end, we prepare a few technical lemmas.
First, Lemma~\ref{lemma: concentration of small jump process, Ld for Levy MRV} establishes a concentration inequality for $\bar{\bm L}_n(t)$ before the arrival of any large jump.

\begin{lemma}\label{lemma: concentration of small jump process, Ld for Levy MRV}
\linksinthm{lemma: concentration of small jump process, Ld for Levy MRV}
Let Assumption~\ref{assumption: ARMV in levy process} hold.
Given any $\epsilon,\beta,T > 0$, there exists $\delta_0 = \delta_0(\epsilon,\beta)$ such that
\begin{align*}
    \lim_{n \to \infty}n^{-\beta}\cdot\P\Bigg(
        \sup_{ t \in [0,T]:\ t < \tau^{>\delta}_n(1)  }
        \norm{
            \bar{\bm L}_n(t) - \bm\mu_{\bm L}t
        }
        > \epsilon
    \Bigg)
    = 0
    \qquad \forall \delta \in (0,\delta_0).
\end{align*}
\end{lemma}

\begin{proof}\linksinpf{lemma: concentration of small jump process, Ld for Levy MRV}
Without loss of generality, we work with the assumption that $\mu_{\bm L} = \bm 0$ and $T = 1$.
Recall the L\'evy–Itô decomposition stated in \eqref{prelim: levy ito decomp}:
\begin{align*}
    \bm L(t)
    & \distequal
    \underbrace{ \bm \Sigma_{\bm L}^{1/2} \bm B(t) }_{ \delequal \bm L_1(t)  }
    +
    \underbrace{ \int_{\norm{\bm x} \leq 1}\bm x\big[ \text{PRM}_{\nu}([0,t]\times d\bm x) - t\nu(dx) \big]  }_{ \delequal \bm L_2(t) }
    +
    \underbrace{ \bm c_{\bm L}t + \int_{\norm{\bm x} > 1}\bm x\text{PRM}_{\nu}( [0,t]\times d\bm x ) }_{ \delequal \bm L_3(t) },
\end{align*}
where 
$\nu$ is the L\'evy measure supported on $\R^d_+ \setminus \{\bm 0\}$ 
satisfying $\int_{ \bm x \in \R^d_+ \setminus\{\bm 0\}  } (\norm{\bm x}^2\wedge 1) \nu(d\bm x) < \infty$.
By definitions in \eqref{def: large jump time tau for Levy process L}, $\tau^{>\delta}_n(1)$ is the arrival time of the first discontinuity in $\bar{\bm L}_n(t)$ with norm larger than $\delta$.
Therefore, after the $1/n$ time-scaling, $n\tau^{>\delta}_n(1)$ is the arrival time of the first discontinuity in $\bm L(t)$ larger than $n\delta$. 
Then, it suffices to show that (as $n \to \infty$)
\begin{align}
    \P\Big( \sup_{ t \in [0,n]  }\norm{ \bm L_1(t)  } > n\epsilon/3  \Big) & = \lo(n^{-\beta}),
    \label{proof, goal 1, lemma: concentration of small jump process, Ld for Levy MRV}
    \\
    \P\Big( \sup_{ t \in [0,n]  }\norm{ \bm L_2(t)  } > n\epsilon/3  \Big) & = \lo(n^{-\beta}),
    \label{proof, goal 2, lemma: concentration of small jump process, Ld for Levy MRV}
    \\ 
    \P\bigg( \sup_{ t \in [0,n]:\ t < n\tau^{>\delta}_n(1)  }\norm{ \bm L_3(t)  } > n\epsilon/3  \bigg) & = \lo(n^{-\beta}),
    \qquad
    \forall \delta > 0\text{ small enough}.
    \label{proof, goal 3, lemma: concentration of small jump process, Ld for Levy MRV}
\end{align}

\medskip
\noindent
\textbf{Proof of Claim \eqref{proof, goal 1, lemma: concentration of small jump process, Ld for Levy MRV}}.
We use $\sigma_{i,j}^{1/2}$ to denote the element in the $i^\text{th}$ row and $j^\text{th}$ column of the  positive semi-definite matrix $\bm \Sigma_{\bm L}^{1/2} \in \R^{d \times d}$,
and write $\bm B(t) = \big(B_1(t),\ldots,B_d(t)\big)^\text{T}$.
Let $C \delequal \max_{i,j \in [d]}| \sigma_{i,j}^{1/2} |$.
Suppose that 
$
\sup_{t \in [0,n]}|B_j(t)| \leq \frac{n\epsilon}{3Cd^2}
$
holds for each $j \in [d]$.
Then, under the $L_1$ norm,
\begin{align*}
    \sup_{ t \in [0,n]  }\norm{ \bm L_1(t)  } 
    & \leq 
    \sum_{ i \in [d]}\sum_{j \in [d]}|\sigma^{1/2}_{i,j}| \cdot \sup_{ t\in [0,n] }|B_j(t)|
    \leq 
    \sum_{j \in [d]}\sup_{ t\in [0,n] }|B_j(t)| \cdot Cd
    \leq d \cdot \frac{n\epsilon}{3Cd^2} \cdot Cd \leq \epsilon/3.
\end{align*}
Therefore, it suffices to show that
$
\P\big( \sup_{t \in [0,n]}|B(t)| > \frac{ n\epsilon  }{3Cd^2} \big) =  \lo(n^{-\beta})
$
holds for a standard Brownian motion $B(t)$ in $\R^1$.
By Doob's maximal inequality and the MGF of $B(t)$, we get
\begin{align*}
    \P\bigg( \sup_{t \in [0,n]}|B(t)| > \frac{ n\epsilon  }{3Cd^2} \bigg)
    & \leq 
    2 \cdot \exp\bigg[ -\frac{1}{2n} \cdot \bigg( \frac{n\epsilon}{3Cd^2}\bigg)^2 \bigg]
    =
    2 \cdot \exp\bigg(
        - n \cdot \frac{\epsilon^2}{18C^2d^4}
    \bigg)
    =  \lo(n^{-\beta}),
\end{align*}
and conclude the proof of Claim \eqref{proof, goal 1, lemma: concentration of small jump process, Ld for Levy MRV}.

\medskip
\noindent
\textbf{Proof of Claim \eqref{proof, goal 2, lemma: concentration of small jump process, Ld for Levy MRV}}.
Note that each coordinate of $\bm L_2(t) = (L_{2,1}(t),\ldots,L_{2,d}(t))^\top$ is a L\'evy process with bounded jumps and also a martingale.
Therefore, each $L_{2,j}(t)$ has finite moments of any order;
see Theorem 34 in Chapter I of \cite{protter2005stochastic}.
Meanwhile, by 
Burkholder-Davis-Gundy inequalities (see Theorem 48 in Chapter IV of \cite{protter2005stochastic}),
for each $p \in [1,\infty)$, there exists some constant $c_p \in (0,\infty)$ (whose value does not vary with the law of $\bm L$) such that 
 \begin{align*}
     \E \bigg[ \sup_{t \in [0,n]}|L_{2,j}(t)|^p \bigg]
     \leq 
     c_p \E\bigg[  \Big( [L_{2,j},L_{2,j}](n) \Big)^{p/2}  \bigg],
     \qquad
     \forall j \in [d],\ n \geq 1,
 \end{align*}
 where $V_j(t) \delequal [L_{2,j},L_{2,j}](t)$ is the quadratic variation process of $L_{2,j}(t)$.
Due to the independent and stationary increments of L\'evy processes $L_{2,j}(t)$,
the quadratic variation $V_j(t) = [L_{2,j},L_{2,j}](t)$ also has stationary and independent increments. 
Then, by the Minkowski inequality,
for each $p \in [1,\infty)$ we have
\begin{align*}
     \Bigg( \E \bigg[ \sup_{t \in [0,n]}|L_{2,j}(t)|^p \bigg] \Bigg)^{2/p}
     & \leq 
     (c_p)^{2/p} \cdot 
     \Bigg( \E\bigg[  \Big( [L_{2,j},L_{2,j}](n) \Big)^{p/2}  \bigg] \Bigg)^{2/p}
     \\
     &
     \leq 
     (c_p)^{2/p} \cdot n \cdot \Big(\underbrace{ \E \big[ \big(V_j(1)\big)^{p/2}\big]}_{ \delequal v_{p,j} < \infty }\Big)^{2/p}.
\end{align*}
As a result, for each $j \in [d]$, $p \in [1,\infty)$, and $n \geq 1$,
\begin{align*}
    \P\bigg( \sup_{t \in [0,n]}|L_{2,j}(t)| > \frac{ n\epsilon  }{3d} \bigg)
    & \leq 
    \bigg( \frac{
        3d
    }{
        \epsilon
    }\bigg)^p \cdot 
    \frac{
         \E \big[ \sup_{t \in [0,n]}|L_{2,j}(t)|^p \big]
    }{ n^p}
    \\ 
    &
    \leq 
    \bigg( \frac{
        3d
    }{
        \epsilon
    }\bigg)^p \cdot 
    \frac{
        n^{p/2} \cdot c_p v_{p,j}
    }{
        n^p
    }
    =
    \bigg( \frac{
        3d
    }{
        \epsilon
    }\bigg)^p \cdot 
    \frac{
        c_p v_{p,j}
    }{
        n^{p/2}
    }.
\end{align*}
Picking $p > 2\beta$,
we conclude the proof of Claim~\eqref{proof, goal 2, lemma: concentration of small jump process, Ld for Levy MRV}.

\medskip
\noindent
\textbf{Proof of Claim \eqref{proof, goal 3, lemma: concentration of small jump process, Ld for Levy MRV}}.
Note that $\bm L_3(t)$ is a compound Poisson process with drift, i.e., 
$
\bm L_3(t) =
    \bm c_{\bm L} t + \sum_{ i  = 1  }^{N(t)}\bm Z_n,
$
where $N(t)$ is a Poisson process with rate $\lambda = \nu\big( \{ \norm{\bm x} \in \R^d_+:\ \norm{\bm x} > 1   \} \big)$,
and the 
$\bm Z_n$'s are i.i.d.\ copies under the law
\begin{align*}
    \P(\bm Z \in \ \cdot \ ) = \frac{
        \nu\big( \ \cdot \ \cap \{ \norm{\bm x} \in \R^d_+:\ \norm{\bm x} > 1   \} \big)
    }{
         \nu\{ \norm{\bm x} \in \R^d_+:\ \norm{\bm x} > 1   \}
    }.
\end{align*}
In particular, under our running assumption $\mu_{\bm L} = \bm 0$,
by \eqref{def: expectation of levy process}
we  have $\bm c_{\bm L} = -\lambda \E\bm Z$ for the linear drift, and hence
\begin{align}
    \bm L_3(t) =
    \sum_{ i  = 1  }^{N(t)}\bm Z_n - t\cdot \lambda\E\bm Z.
    \nonumber
\end{align}
Also, by Assumption~\ref{assumption: ARMV in levy process} and the $\MRV$ condition in Definition~\ref{def: MRV},
$\P(\norm{\bm Z} > n) \in \RV_{ -\alpha(\{i^*\}) }(n)$ with $\alpha(\{i^*\}) > 1$,
where $i^* \delequal \arg\min_{ i \in [d]  }\alpha(\{i\})$.
This confirms that $\norm{\E\bm Z} < \infty$ and $\norm{\bm Z}$ has regularly varying law with tail index larger than 1.

To proceed, let $j^* = j^*(n,\delta)$ be the index of the first $\bm Z_j$ with $\norm{\bm Z_j} > n\delta$,
and $\tau^* = \tau^*(n,\delta)$ be its arrival time in the compound Poisson process $\sum_{ i  = 1  }^{N(t)}\bm Z_n$.
Given $\hat \epsilon > 0$, on the event 
$
\big\{
        \sup_{ t \in [0,n] }| N(t) - \lambda t | \leq n\hat\epsilon
    \big\},
$
observe that 
\begin{align*}
    \sup_{ t \in [0,n]:\ t < n\tau^{>\delta}_n(1)  }\norm{ \bm L_3(t)  }
    & = 
     \sup_{ t \in [0,n]:\ t < \tau^*  }
     \norm{
            \sum_{ i  = 1  }^{N(t)}\bm Z_n - t\cdot \lambda\E\bm Z
     }
     \\ 
     & \leq 
     \sup_{ t \in [0,n]:\ t < \tau^*  }
     \norm{
            \sum_{ i  = 1  }^{N(t)}\bm Z_n - N(t)\cdot \E\bm Z
     }
     +
     \sup_{ t \in [0,n]:\ t < \tau^*  }| N(t) - \lambda t  | \cdot \norm{\E\bm Z}
     \\ 
     & \leq 
     \max_{ j \leq (\lambda + \hat\epsilon)\cdot n:\ j < j^*  }
     \norm{
        \sum_{i = 1}^j \bm Z_i - i \cdot \E\bm Z
     }
     + n\hat\epsilon \cdot \norm{\E\bm Z}
     \\ 
     & \leq 
     \max_{j \leq (\lambda + \hat\epsilon)\cdot n}
     \norm{
        \sum_{i = 1}^j \bm Z_i \mathbbm{I}\{ \norm{\bm Z_i} \leq n\delta  \} - i \cdot \E\bm Z
     }
      + n\hat\epsilon \cdot \norm{\E\bm Z}.
\end{align*}
In particular, by picking $\hat\epsilon > 0$ small enough, we get $\hat\epsilon \cdot \norm{\E\bm Z} < \epsilon/6$.
Therefore, to prove Claim~\eqref{proof, goal 3, lemma: concentration of small jump process, Ld for Levy MRV},
it suffices to show that 
\begin{align}
    \P\bigg(
             \sup_{ t \in [0,n] }| N(t) - \lambda t | > n\hat\epsilon
        \bigg) & = \lo(n^{-\beta}),
        \label{proof, goal 3, subgoal 1, lemma: concentration of small jump process, Ld for Levy MRV}
        \\ 
    \P\Bigg(
         \max_{j \leq (\lambda + \hat\epsilon)\cdot n}
     \norm{
        \sum_{i = 1}^j \bm Z_i \mathbbm{I}\{ \norm{\bm Z_i} \leq n\delta  \} - i \cdot \E\bm Z
     }
      > n\epsilon/6
    \Bigg)
        & = \lo(n^{-\beta}),
        \label{proof, goal 3, subgoal 2, lemma: concentration of small jump process, Ld for Levy MRV}
    \quad\forall \delta > 0\text{ small enough}.
\end{align}
However, the bound \eqref{proof, goal 3, subgoal 1, lemma: concentration of small jump process, Ld for Levy MRV} follows from Doob's maximal inequality and Cramer's theorem,
and the bound \eqref{proof, goal 3, subgoal 2, lemma: concentration of small jump process, Ld for Levy MRV} 
follows from Lemma~3.1 in \cite{wang2023large} (i.e., concentration inequalities for truncated regularly varying vectors).
This concludes the proof of Claim~\eqref{proof, goal 3, lemma: concentration of small jump process, Ld for Levy MRV}.
\end{proof}

The next result justifies the use of $\hat{\bm L}^{>\delta}_n$ in \eqref{def: large jump approximation, LD for Levy MRV} as an approximator to $\bar{\bm L}_n$.

\begin{lemma}\label{lemma: large jump approximation, LD for Levy MRV}
\linksinthm{lemma: large jump approximation, LD for Levy MRV}
Let $k \in \mathbb N$ and $\epsilon > 0$.
For any $n \geq 1$ and $\delta > 0$,
it holds on event
\begin{align*}
   &  \big\{ \tau^{>\delta}_n(k) \leq 1 < \tau^{>\delta}_n(k+ 1)  \big\}
   \\ 
   &\qquad
    \cap 
    \Bigg(
        \bigcap_{ m = 1 }^{k+1}
        \underbrace{
        \bigg\{
            \sup_{ t \in [\tau^{>\delta}_n(m-1), 1]:\ t < \tau^{>\delta}_n(m)}
            \norm{
            \bar{\bm L}_n(t) - \bar{\bm L}_n\Big( \tau^{>\delta}_n(m-1) \Big)
            } \leq \frac{\epsilon}{k+1}
        \bigg\}
        }_{ \delequal A_n(m;\delta,k,\epsilon) }
    \Bigg)
\end{align*}
that
\begin{align*}
    \sup_{ t \in [0,1] }\norm{
        \hat{\bm L}^{>\delta}_n(t) - \bar{\bm L}_n(t)
    } \leq \epsilon.
\end{align*}
\end{lemma}

\begin{proof}
\linksinpf{lemma: large jump approximation, LD for Levy MRV}
For any $\xi,\tilde \xi \in \D$ and any $0\leq u< v \leq 1$, observe the elementary bound
\begin{align*}
    & \sup_{t \in [0,v]}\norm{\xi(t) - \tilde \xi(t)}
    \\
    & \leq 
    \sup_{ t \in [0,u]}
    \norm{\xi(t) - \tilde \xi(t)}
    +
    \sup_{ t \in [u,v) }
    \norm{
    \Big(\xi(t) - \xi(u) \Big) - \Big(\tilde \xi(t) - \tilde \xi(u)\Big)
    }
    +
    \norm{
    \Delta \xi(v) - \Delta \tilde \xi(v)
    }.
\end{align*}
Furthermore, applying this bound inductively, it holds for any $0 < t_1 < t_2 < \ldots < t_k \leq 1$ that (under the convention that $t_0 = 0$ and $t_{k+1} = 1$)
\begin{align*}
    & \sup_{t \in [0,1]}\norm{\xi(t) - \tilde \xi(t)}
    \\ 
    & \leq 
    \sum_{m = 1  }^{k+1}
    \sup_{ t \in [t_{m-1},t_m) }
    \norm{
    \Big(\xi(t) - \xi(t_{m-1}) \Big) - \Big(\tilde \xi(t) - \tilde \xi(t_{m-1})\Big)
    }
    +
    \norm{
    \Delta \xi(t_m) - \Delta \tilde \xi(t_{m})
    }.
\end{align*}
Now, on event 
$
\big\{ \tau^{>\delta}_n(k) \leq 1 < \tau^{>\delta}_n(k+ 1)  \big\}
\cap 
\big(
    \bigcap_{m = 1}^{k+1}A_n(m;\delta,k,\epsilon)
\big),
$
we apply this bound with $\xi = \bar{\bm L}_n$, $\tilde \xi = \hat{\bm L}^{>\delta}_n$, and 
$
t_m = \tau^{>\delta}_n(m).
$
In particular, by \eqref{def: large jump approximation, LD for Levy MRV},
we have
$
\sup_{ t \in [t_{m-1},t_m) }
    \norm{
    \tilde \xi(t) - \tilde \xi(t_{m-1})
    } = 0
$
and
$
\norm{
    \Delta \xi(t_m) - \Delta \tilde \xi(t_{m})
    } = 0.
$
The claims then follow directly from the condition in the event $A_n(m;\delta,k,\epsilon)$.
\end{proof}


In Lemma~\ref{lemma: asymptotic law for large jumps, LD for Levy MRV}, we develop useful asymptotics for ${\tau^{>\delta}_{ n }(k; \bm j)}$ and ${\bm W^{>\delta}_{ n }(k; \bm j)}$.
Specifically, given $\delta > 0$, $n \geq 1$, 
we define the event
\begin{align}
    \widetilde E^{>\delta}_n \delequal \big\{ \widetilde K^{>\delta}_n = 0  \big\}.
\end{align}
Furthermore, given
$\mathcal K = (\mathcal K_{\bm j})_{ \bm j \in \powersetTilde{d} } \in \mathbb Z_+^{\powersetTilde{d}}$, we define
\begin{align}
    \notationdef{notation-event-E->delta-n-c-j}{E^{>\delta}_n(\mathcal K)} \delequal 
    \big\{
        \bm K^{>\delta}_n = \mathcal K
    \big\}
    \cap \widetilde E^{>\delta}_n
    =
    \Bigg( \bigcap_{ \bm j \in \powersetTilde{d}  }
    \underbrace{ \Big\{
        K^{>\delta}_n(\bm j) = \mathcal K_{\bm j}
    \Big\}
    }_{ \delequal E^{>\delta}_{n;\bm j}(\mathcal K_{\bm j})   }
    \Bigg) \cap 
    \big\{ \widetilde K^{>\delta}_n = 0  \big\}.
    \label{def: event-E->delta-n-c-j}
\end{align}
Given $\bm j \in \powersetTilde{d}$ and $c > 0$,
let $\big(\bm W_{*}^{(c)}(k;\bm j)\big)_{k \geq 1}$
be sequences that are independent across $\bm j \in \powersetTilde{d}$,
where each $\bm W_{*}^{(c)}(k;\bm j)$
is an i.i.d.\ copy of
$
\notationdef{notation-W-*-c-type-i-LD-for-Levy-MRV}{\bm W_{*}^{(c)}(\bm j)}
$
with law
\begin{align}
    \P\Big(\bm W_{*}^{(c)}(\bm j) \in \ \cdot\ \Big)
    \delequal 
    \mathbf C_{\bm j}\Big(\ \cdot\ \cap
        \bar \R^{>c}(\bm j)
    \Big)\Big/
     \mathbf C_{\bm j}\Big( \bar \R^{>c}(\bm j) \Big),
     \label{def: limiting law of large jumps, LD for levy, MRV}
\end{align}
where the $\mathbf C_{\bm j}$'s are the limiting measures in the $\MRV$ condition of Assumption~\ref{assumption: ARMV in levy process},
and the $\bar \R^{>c}(\bm j)$'s are defined in \eqref{def: tilde R geq delta, LD for Levy}.
Besides, let the $U_{\bm j, k}$'s be i.i.d.\ copies of Unif$(0,1)$,
and, for each $k \geq 1$ and $ \bm j \in \powersetTilde{d}$,
let 
$
U_{\bm j,(1:k)} \leq U_{\bm j,(2:k)} \leq \ldots \leq U_{\bm j,(k:k)}
$
be the order statistics of $(U_{\bm j,q})_{q \in [k]}$.

\begin{lemma}\label{lemma: asymptotic law for large jumps, LD for Levy MRV}
\linksinthm{lemma: asymptotic law for large jumps, LD for Levy MRV}
    Let Assumptions~\ref{assumption: ARMV in levy process} and \ref{assumption: WLOG, LD for Levy MRV} hold.
    Let 
    $\mathcal K = (\mathcal K_{\bm j})_{ \bm j \in \powersetTilde{d} } \in \mathbb Z_+^{\powersetTilde{d}}$,
    and let
    $\bm k \in \mathbb Z^d_+$ satisfy \eqref{def, assignment of k jump set} (i.e., $\mathcal K$ is an allocation of $\bm k$).
    First,
    \begin{align}
        \lim_{n \to \infty}
        n^{\gamma}\cdot 
        \P\Big( \big( \widetilde E^{>\delta}_n \big)^\complement\Big) = 0,
        \qquad\forall \gamma > 0,\ \delta > 0.
        \label{claim, probability of event tilde E, lemma: asymptotic law for large jumps, LD for Levy MRV}
    \end{align}
    Next,
    for any $\delta > 0$
    if Claim \eqref{proof: choice of c, bounded away set from cone i, Ld for Levy MRV} holds, then
    \begin{align}
        \lim_{n \to \infty}
        \frac{
           \P\big( E^{>\delta}_n(\mathcal K) \big)
        }{
            \breve \lambda_{\bm k}(n)
        }
        =
        \prod_{ \bm j \in \powersetTilde{d} }
            \frac{1}{ \mathcal K_{\bm j} !}\cdot 
            \Big( 
                \mathbf C_{\bm j}\big( \bar \R^{>\delta}(\bm j) \big)
            \Big)^{ \mathcal K_{\bm j}},
            \label{claim, probability of event E, lemma: asymptotic law for large jumps, LD for Levy MRV}
    \end{align}
    where
    $
    \breve \lambda_{\bm j}(\cdot)
    $
    is defined in \eqref{def: scale function for Levy LD MRV}.
    Furthermore,
    if $\mathcal K \neq \bm 0$, then
    \begin{align}
        & \mathscr L
        \Bigg(
        \Big(
            \tau^{>\delta}_n(k;\bm j),\bm W ^{>\delta}_n(k;\bm j)
            \Big)_{\bm j \in \powersetTilde{d},\ k \in [\mathcal K_{\bm j}]}
        \Bigg|\  E^{>\delta}_n(\mathcal K)
        \Bigg)
        \rightarrow
        \mathscr L
        \Bigg(
        \Big(
            U_{\bm j,(k:\mathcal K_{\bm j})}, \bm W^{(\delta)}_*(k;\bm j)
        \Big)_{\bm j \in \powersetTilde{d},\ k \in [\mathcal K_{\bm j}]}
        \Bigg)
         \label{claim, weak convergence of conditional law, lemma: asymptotic law for large jumps, LD for Levy MRV}
    \end{align}
    as $n \to \infty$
    in terms of weak convergence.
\end{lemma}

\begin{proof}\linksinpf{lemma: asymptotic law for large jumps, LD for Levy MRV}
Recall the definition of $\nu_n(A) = \nu\{n \bm x:\ \bm x \in A\}$.
First, by the law of the (scaled) L\'evy process $\bar{\bm L}_n(t) = \frac{1}{n}\bm L(nt)$ in \eqref{prelim: levy ito decomp} and the definitions in \eqref{def: tilde tau n delta k, big jump outside of full cone}--\eqref{def: tilde K n delta k, big jump outside of full cone},
\begin{align*}
     \P\Big( \big( \widetilde E^{>\delta}_n \big)^\complement\Big)
     & = 
     \P\Bigg(
        \text{Poisson}\Big(
            n\cdot \nu_n\big( \R^d_+ \setminus \bar \R^{d}([d],\delta) \big)
        \Big) \geq 1
    \Bigg)
    \\ 
    & \leq 
    \E\Bigg[
        \text{Poisson}\Big(
            n\cdot \nu_n\big( \R^d_+ \setminus \bar \R^{d}([d],\delta) \big)
        \Big)
    \Bigg]
    =
    n\cdot \nu_n\bigg( \Big( \bar \R^{d}([d],\delta) \Big)^\complement \bigg)
    \ 
    \text{by Markov's inequality}.
\end{align*}
By the $\MRV$ condition in Assumption~\ref{assumption: ARMV in levy process} (in particular, Claim~\eqref{cond, outside of the full cone, finite index, def: MRV} in Definition~\ref{def: MRV}),
we verify Claim~\eqref{claim, probability of event tilde E, lemma: asymptotic law for large jumps, LD for Levy MRV}.

Next,
by the independence when splitting the Poisson random measure $\text{PRM}_{\nu}$ in \eqref{prelim: levy ito decomp},
$$
\P\Big( E^{>\delta}_n(\mathcal K) \Big)
=
\Bigg( \prod_{ \bm j \in \powersetTilde{d} }\P\Big( E^{>\delta}_{n;\bm j}(\mathcal K_{\bm j}) \Big) \Bigg)
\cdot \P\big(\widetilde E^{>\delta}_n\big).
$$
On the one hand, Claim~\eqref{claim, probability of event tilde E, lemma: asymptotic law for large jumps, LD for Levy MRV} implies that 
$
\lim_{n \to \infty}\P\big(\widetilde E^{>\delta}_n\big) = 1.
$
On the other hand, in light of property \eqref{property, rate function for assignment mathcal K, LD for Levy}, to prove \eqref{claim, probability of event E, lemma: asymptotic law for large jumps, LD for Levy MRV} it suffices to show that for any $\bm j \in \powersetTilde{d}$ and $k \geq 0$,
\begin{align}
    \lim_{n \to \infty}
    \frac{
        \P\big( E^{>\delta}_{n;\bm j}(k) \big)
    }{
        \big( n \lambda_{\bm j}(n)\big)^{k}
    }
    =
    \frac{1}{k!} \cdot \Big(\mathbf C_{\bm j}\big( \bar \R^{>\delta}(\bm j) \big)\Big)^{k}.
    \label{proof, goal for prob of event E, lemma: asymptotic law for large jumps, LD for Levy MRV}
\end{align}
By the law of the (scaled) L\'evy process $\bar{\bm L}_n(t) $ in \eqref{prelim: levy ito decomp} and the definitions in \eqref{def: large jump in cone i, tau > delta n k i, LD for Levy MRV}--\eqref{def: large jump size in cone i, tau > delta n k i, LD for Levy MRV},
\begin{align}
    \P\Big( E^{>\delta}_{n;\bm j}(k) \Big)
    & = 
    \P\Bigg(
        \text{Poisson}\Big(
            n\cdot \underbrace{ \nu_n\big( \bar\R^{>\delta}(\bm j) \big) }_{ \delequal \theta_n(\bm j,\delta)  }
        \Big) = k
    \Bigg)
    = 
    \exp\big( - n \theta_n(\bm j,\delta) \big)
    \cdot 
    \frac{
        \big(n \theta_n(\bm j,\delta)\big)^{k}
    }{
        k !
    }.
    \label{proof, calculation 1, prob of event E, lemma: asymptotic law for large jumps, LD for Levy MRV}
\end{align}
By the $\MRV$ condition in Assumption~\ref{assumption: ARMV in levy process},
\begin{align*}
    & \lim_{n \to \infty}
    \frac{
        \theta_n(\bm j,\delta)
    }{
        \lambda_{\bm j}(n)
    }
    =
    \mathbf C_{\bm j}\Big( \bar \R^{>\delta}(\bm j) \Big);
    \qquad
    \text{see \eqref{cond, finite index, def: MRV} and our choice of $\delta$ in \eqref{proof: choice of c, bounded away set from cone i, Ld for Levy MRV}}.
\end{align*}
This also implies $\theta_n(\bm j,\delta) = \bo\big( \lambda_{\bm j}(n) \big) = \lo(n)$ (due to $\lambda_{\bm j}(n) \in \RV_{ -\alpha(\bm j) }(n)$ with $\alpha(\bm j) > 1$),
and hence
$
\lim_{n \to \infty} \exp\big( - n \theta_n(\bm j,\delta) \big) = 1.
$
Plugging these limits back into \eqref{proof, calculation 1, prob of event E, lemma: asymptotic law for large jumps, LD for Levy MRV}, we conclude the proof of Claim~\eqref{proof, goal for prob of event E, lemma: asymptotic law for large jumps, LD for Levy MRV}.

Next, we prove the weak convergence stated in \eqref{claim, weak convergence of conditional law, lemma: asymptotic law for large jumps, LD for Levy MRV}.
Again, by the independence of the Poisson splitting,
it suffices to show that for any $\bm j \in \powersetTilde{d}$ and $K \geq 1$,
\begin{align}
& \mathscr L
        \bigg(
        \Big(
            \tau^{>\delta}_n(k;\bm j),\bm W ^{>\delta}_n(k;\bm j)
            \Big)_{k \in [K]}
        \bigg|\ 
            E^{>\delta}_{n;\bm j}(K)
        \bigg)
        \rightarrow
        \mathscr L
        \bigg(
            \Big(
            U_{\bm j,(k:K)}, \bm W^{(\delta)}_*(k;\bm j)
        \Big)_{ k \in [K]}
        \bigg).
    \label{proof: goal, weak convergence claim, lemma: asymptotic law for large jumps, LD for Levy MRV}
\end{align}
Next, by the law of a compound Poisson process,
the arrival times $\tau^{>\delta}_n(k;\bm j)$ are independent of the jump sizes  $\bm W^{>\delta}_n(k;\bm j)$,
which implies that the law of the jump size sequence $\big(\bm W^{>\delta}_n(k;\bm j)\big)_{k \geq 1}$  is independent of the event $E^{>\delta}_{n;\bm j}(K)$.
Furthermore,
the conditional law of
the $\tau^{>\delta}_n(k;\bm j)$'s---the sequence of arrival times of jumps in a compound Poisson process conditioning on the number of jumps---admits the form
\begin{align*}
    \mathscr L
    \bigg(
        \tau^{>\delta}_n(1;\bm j), \tau^{>\delta}_n(2;\bm j), \ldots, \tau^{>\delta}_n(K;\bm j)
        \ \bigg|\ E^{>\delta}_{n;\bm j}(K)
    \bigg)
    =
    \mathscr L
    \bigg(
        U_{\bm j,(1:K)},U_{\bm j,(2:K)},\ldots,U_{\bm j,(K:K)}
    \bigg).
\end{align*}
Therefore, it only remains to study the marginal law for $\bm W^{>\delta}_n(1;\bm j)$ (since all $\bm W^{>\delta}_n(k;\bm j)$'s admit the same law) and show that
\begin{align*}
    \P\Big( 
        \bm W^{>\delta}_n(1;\bm j)\in \ \cdot\ 
    \Big)
    \rightarrow
    \P\Big(\bm W^{(\delta)}_*(\bm j) \in \ \cdot \ \Big)
\end{align*}
in terms of weak convergence,
where the law of $\bm W^{(\delta)}_*(\bm j)$ is defined in \eqref{def: limiting law of large jumps, LD for levy, MRV}.
To this end, 
note that by the definitions in \eqref{def: large jump in cone i, tau > delta n k i, LD for Levy MRV}--\eqref{def: large jump size in cone i, tau > delta n k i, LD for Levy MRV},
\begin{align}
    \P\Big(
        \bm W^{>\delta}_n(1; \bm j) \in\ \cdot\ 
    \Big)
    = \frac{
        \nu_n\big( \bar \R^{>\delta}(\bm j) \cap\ \cdot\   \big)
    }{
        \nu_n\big( \bar \R^{>\delta}(\bm j) \big)
    },
    \qquad
    \forall \bm j \in \powersetTilde{d}.
    \label{proof: law of large jump in cone i, lemma: asymptotic law for large jumps, LD for Levy MRV}
\end{align}
Then, given a Borel set $A \subseteq \R^d_+$,
\begin{align*}
    & \limsup_{n \to \infty}
    \P\Big( 
        \bm W^{>\delta}_n(1;\bm j)\in A
    \Big)
    \\ 
    & = 
    \limsup_{n \to \infty}
    \frac{
        \nu_n\Big( A \cap \bar \R^{>\delta}(\bm j)\Big)
        \Big/ \lambda_{\bm j}(n)
    }{
        \nu_n\Big( \bar \R^{>\delta}(\bm j)\Big)
        \Big/ \lambda_{\bm j}(n)
    }
    \qquad
    \text{by \eqref{proof: law of large jump in cone i, lemma: asymptotic law for large jumps, LD for Levy MRV}}
    \\
    & \leq 
    \mathbf C_{\bm j}\bigg( \Big( A \cap \bar\R^{>\delta}(\bm j) \Big)^-  \bigg)
    \bigg/ 
    \mathbf C_{\bm j}\bigg( \Big( \bar \R^{>\delta}(\bm j)\Big)^\circ  \bigg)
    \qquad
    \text{by the $\MRV$ condition and \eqref{cond, finite index, def: MRV}}
    \\ 
    & \leq 
    \mathbf C_{\bm j}\Big( A^- \cap  \bar \R^{>\delta}(\bm j)  \Big)
    \Big/ 
    \mathbf C_{\bm j}\Big(  \bar \R^{>\delta}(\bm j) \Big)
    \qquad
    \text{because of $(A\cap B)^- \subseteq A^- \cap B^-$ and \eqref{proof: choice of c, bounded away set from cone i, Ld for Levy MRV} 
    }
    \\
    & = \P\Big(\bm W^{(\delta)}_*(\bm j) \in A^- \Big);
    \qquad\text{see \eqref{def: limiting law of large jumps, LD for levy, MRV}.}
\end{align*}
Using the Portmanteau Theorem, we conclude the proof.
\end{proof}

Given $\bm{\mathcal K} = (\mathcal K_{\bm j})_{\bm j \in \powersetTilde{d}} \in \mathbb Z_+^{ \powersetTilde{d}}$ and $\epsilon \geq 0$,
recall that any path $\xi \in \shortbarDepsk{\epsilon}{ \bm{\mathcal K}}$ admits the form
\begin{align}
    \xi(t) = \sum_{\bm j \in \powersetTilde{d} }\sum_{k = 1}^{ \mathcal K_{\bm j} }\bm w_{\bm j,k} \mathbbm{I}_{ [t_{\bm j,k},1] }(t),\qquad \forall t \in [0,1],
    \label{proof: path expression for elements in set D j, LD for levy MRV}
\end{align}
where $\bm w_{\bm j,k} \in \bar{\R}^d(\bm j, \epsilon)$ and $t_{\bm j,k} \in (0,1]$ for each $\bm j \in \powersetTilde{d}$ and $k \in [\mathcal K_{\bm j}]$, and any two elements in  the sequence $(t_{\bm j,k})_{ \bm j \in \powersetTilde{d},\ k \in [\mathcal K_{\bm j}] }$ will not coincide;
see \eqref{def: j jump path set with drft c, LD for Levy, MRV} and \eqref{expression for xi in breve D set}.
Besides, recall that in Definition~\ref{def: assignment, jump configuration k},
we use $\mathbb A(\bm k)$ to denote the set of all allocations of $\bm k \in \mathbb Z^d_+$.
We prepare the next two lemmas to collect useful properties of sets $ \shortbarDepsk{\epsilon}{\mathcal K}$ and measures $\breve{\mathbf C}_{\mathcal K}(\cdot)$ (see \eqref{def: measure C type j L, LD for Levy, MRV}).

\begin{lemma}\label{lemma: choice of bar epsilon and bar delta, LD for Levy MRV}
\linksinthm{lemma: choice of bar epsilon and bar delta, LD for Levy MRV}
    Let Assumption~\ref{assumption: ARMV in levy process} hold.
    Let $\bm k = (k_{j})_{j \in [d]} \in \mathbb Z_+^{d}$ with $\bm k \neq (0,\ldots,0)$, $\epsilon > 0$,
    and let $B$ be a Borel set of $\D$ equipped with Skorokhod $J_1$ topology.
    Suppose that $B$ is bounded away from $\shortbarDepsk{\epsilon}{\leqslant \bm k}$ under $\dj{}$ for some $\epsilon > 0$.
    Then, there exist $\bar\epsilon> 0$ and $\bar\delta > 0$ such that the following claims hold:
    \begin{enumerate}[(a)]
        \item
            $\dj{}\big( B^{\bar\epsilon}, \shortbarDepsk{\bar\epsilon}{\leqslant \bm k}\big)> \bar\epsilon$;
            
        \item
            Given any $\bm{\mathcal K} = (\mathcal K_{\bm j})_{ \bm j \in \powersetTilde{d} } \in \mathbb A(\bm k)$ and
            $\xi \in B^{\bar\epsilon} \cap \shortbarDepsk{\bar\epsilon}{ \bm{\mathcal K}}$,
            in the expression \eqref{proof: path expression for elements in set D j, LD for levy MRV} for $\xi$ we have
            \begin{align}
                \norm{\bm w_{\bm j, k}} > \bar\delta\text{ and }
                 \bm w_{\bm j,k}\notin \bar{\R}^d_\leqslant(\bm j,\bar\delta),
                \qquad \forall \bm j \in \powersetTilde{d},\ k \in [\mathcal K_{\bm j}],
                \label{claim, condition, part a, lemma: choice of bar epsilon and bar delta, LD for Levy MRV}
            \end{align}
            where the set $\bar\R^d_\leqslant(\bm j,\epsilon) = \bar\R^d_{\leqslant}(\bm j,\epsilon;\bar{\textbf S},\bm \alpha)$
            is defined in \eqref{def: cone R d i basis S index alpha}.

    \end{enumerate}
\end{lemma}

\begin{proof}\linksinpf{lemma: choice of bar epsilon and bar delta, LD for Levy MRV}
Part (a) follows directly from the fact that
$B$ is bounded away from $\shortbarDepsk{\epsilon}{\leqslant \bm k}$ under $\dj{}$ for some  $\epsilon > 0$, as well as the monotonicity of $\shortbarDepsk{\epsilon^\prime}{\bm{\mathcal K} } \subseteq \shortbarDepsk{\epsilon}{\bm{\mathcal K} }$
for any $\epsilon \geq \epsilon^\prime \geq 0$ and $\bm{\mathcal K} \in \mathbb Z^{\powersetTilde{d}}_+$.

Moving on to part (b),
since there are only finitely elements in $\mathbb A(\bm k)$,
if suffices to fix some $\bm{\mathcal K} = (\mathcal K_{\bm j})_{\bm j \in \powersetTilde{d}} \in \mathbb A(\bm k)$
and prove the claim \eqref{claim, condition, part a, lemma: choice of bar epsilon and bar delta, LD for Levy MRV}.
Specifically,
we take any $\bar\delta \in (0,\bar\epsilon)$
and proceed with a proof by contradiction.
Suppose that, for some $\xi \in B^{\bar\epsilon} \cap \shortbarDepsk{\bar\epsilon}{\bm{\mathcal K}}$, 
there are $\bm j^* \in \powersetTilde{d}$ and $k^* \in [\mathcal K_{\bm j^*}]$ (which, of course, requires that $\mathcal K_{\bm j^*} \geq 1$)
such that
$\norm{\bm w_{\bm j^*,k^*}} \leq \bar\delta$ 
in the expression for $\xi$  in \eqref{proof: path expression for elements in set D j, LD for levy MRV}.
Then, with
\begin{align}
    \xi^*(t) 
    \delequal 
     \sum_{\bm j \in \powersetTilde{d}\setminus\{\bm j^*\} }\sum_{k = 1}^{ \mathcal K_{\bm j} }\bm w_{\bm j,k} \mathbbm{I}_{ [t_{\bm j,k},1] }(t)
     +
     \sum_{ k \in [\mathcal K_{\bm j^*}]:\ k \neq k^* }\bm w_{\bm j^*,k} \mathbbm{I}_{ [t_{\bm j^*,k},1] }(t),
     \qquad \forall t \in [0,1],
    \nonumber
\end{align}
we construct a path $\xi^*$ by removing the jump $\bm w_{\bm j^*,k^*}$ from $\xi$.
By defining $\bm{\mathcal K}^\prime = (\mathcal K^\prime_{\bm j})_{ \bm j \in \powersetTilde{d}  }$ as
\begin{align*}
    \mathcal K^\prime_{\bm j} \delequal
    \begin{cases}
        \mathcal K_{\bm j}&\text{ if }\bm j \neq \bm j^*
        \\ 
        \mathcal K_{\bm j^*} - 1 & \text{ if }\bm j = \bm j^*
    \end{cases},
\end{align*}
we have $\xi^* \in  \shortbarDepsk{\bar\epsilon}{\bm{\mathcal K}^\prime}$
and $\breve{c}(\bm{\mathcal K}^\prime) < \breve{c}(\bm{\mathcal K})$ (see \eqref{def: cost alpha j, LD for Levy MRV} and \eqref{property, cost c under allocation}),
thus implying 
$
\xi^* \in \shortbarDepsk{\bar\epsilon}{\leqslant\bm k}.
$
However, due to
\begin{align}
    \dj{}(\xi,\xi^*) \leq \sup_{t \in [0,1]}\norm{\xi(t) - \xi^*(t)} = \norm{\bm w_{\bm j^*,k^*}} \leq \bar\delta < \bar\epsilon
    \nonumber
\end{align}
and $\xi \in B^{\bar\epsilon}$, 
 we arrive at the contradiction that $\dj{}(B^{\bar\epsilon}, \shortbarDepsk{\bar\epsilon}{\leqslant \bm k}) \leq \bar\epsilon$.
In summary, we have shown that
$
\norm{\bm w_{\bm j, k}} > \bar\delta
$
for any $\bm j \in \powersetTilde{d},\ k \in [\mathcal K_{\bm j}]$ in the expression \eqref{proof: path expression for elements in set D j, LD for levy MRV} for $\xi$.

Next,
suppose that, for some $\xi \in B^{\bar\epsilon} \cap \shortbarDepsk{\bar\epsilon}{\mathcal K}$, 
there exist $\bm j^* \in \powersetTilde{d}$ and $k^* \in [\mathcal K_{\bm j^*}]$ (which, again, requires that $\mathcal K_{\bm j^*} \geq 1$) such that $\bm w_{\bm j^*,k^*} \in \bar\R^d_\leqslant(\bm j^*,\bar\delta)$. 
Then, by the definition of $\bar\R^d_\leqslant(\bm j^*,\bar\delta)$ in \eqref{def: cone R d i basis S index alpha}
and our choice of $\bar\delta \in (0,\bar\epsilon)$,
there exists some $\bm j^\prime \in \powersetTilde{d}$ with 
$\bm j^\prime \neq \bm j^*$, $\alpha(\bm j^\prime) \leq \alpha(\bm j^*)$ such that $\bm w_{\bm j^*,k^*} \in \bar\R^d(\bm j^\prime,\bar\delta) \subseteq \bar\R^d(\bm j^\prime,\bar\epsilon)$.
Now, define
$\bm{\mathcal K}^\prime = (\mathcal K^\prime_{\bm j})_{ \bm j \in \powersetTilde{d}  }$ by
\begin{align*}
    \mathcal K^\prime_{\bm j}
    \delequal
    \begin{cases}
        \mathcal K_{\bm j^*} - 1 & \text{ if }\bm j = \bm j^*
        \\ 
        \mathcal K_{\bm j^\prime} + 1 & \text{ if }\bm j = \bm j^\prime 
        \\ 
        \mathcal K_{\bm j} & \text{ otherwise}
    \end{cases},
\end{align*}
and note that $\xi \in \shortbarDepsk{\bar\epsilon}{\bm{\mathcal K}^\prime}$ due to
$\bm w_{\bm j^*,k^*} \in \bar\R^d(\bm j^\prime,\bar\epsilon)$.
Due to $\alpha(\bm j^\prime) \leq \alpha(\bm j^*)$,
we have
$\breve c(\bm{\mathcal K}^\prime) \leq \breve c(\bm{\mathcal K}) = c(\bm k)$ (see \eqref{def: cost alpha j, LD for Levy MRV}),
and hence $\xi \in \shortbarDepsk{\bar\epsilon}{\leqslant\bm k}$ (see \eqref{def: path with costs less than j jump set, drift x, LD for Levy MRV}).
In light of the running assumption $\xi \in B^{\bar\epsilon}$,
we arrive at the contradiction that 
$B^{\bar\epsilon}\cap \shortbarDepsk{\bar\epsilon}{\leqslant \bm k} \neq \emptyset$.
This concludes the proof of part (b).
\end{proof}

\begin{lemma}\label{lemma: finiteness of measure breve C k}
\linksinthm{lemma: finiteness of measure breve C k}
Let Assumption~\ref{assumption: ARMV in levy process} hold.
Let $\bm k = (k_{j})_{j \in [d]} \in \mathbb Z_+^{d}$ such that $\bm k \neq \bm 0$.
Let $T,\epsilon > 0$ and $\bm x \in \R^d$.
For any Borel set $B$ of $(\D[0,T],\dj{\subZeroT{T}})$
that is bounded away from $\shortbarDepsk{\epsilon}{\leqslant \bm k;\bm x}[0,T]$ under $\dj{\subZeroT{T}}$,
\begin{align*}
    \sum_{\bm{\mathcal K} \in \mathbb A(\bm k)}\breve{\mathbf C}^{ \subZeroT{T} }_{\bm{\mathcal K};\bm x}(B) < \infty. 
\end{align*}
\end{lemma}

\begin{proof}\linksinpf{lemma: finiteness of measure breve C k}
Without loss of generality, we impose Assumption~\ref{assumption: WLOG, LD for Levy MRV} (i.e., we focus on the case of $T = 1$ and $\bm x = \bm 0$), and adopt the notations in \eqref{proof, def: shorthand, LD for Levy}.
Also,
since there are only finitely many elements in $\mathbb A(\bm k)$,
it suffices to fix some $\bm{\mathcal K} = (\mathcal K_{\bm j})_{\bm j \in \powersetTilde{d}} \in \mathbb A(\bm k)$ and
 some Borel set $B$ of $(\D,\dj{})$ that is bounded away from $\shortbarDepsk{\epsilon}{\leqslant \bm k}$, and show that
$
\breve{
    \mathbf C
    }_{\bm{\mathcal K}}(B)<\infty.
$
Since the measure $\breve{\mathbf C}_{\bm{\mathcal K}}(\cdot)$ is supported on 
    $\shortbarDepsk{0}{\bm{\mathcal K}}$ (see  \eqref{def: measure C type j L, LD for Levy, MRV} and the remarks right after),
    it suffices to show that 
\begin{align}
    \breve{
    \mathbf C
    }_{\bm{\mathcal K}}\big( B \cap \shortbarDepsk{0}{\bm{\mathcal K}} \big)
    < \infty.
    \nonumber
\end{align}
Since $B$ is bounded away from $\shortbarDepsk{\epsilon}{\leqslant \bm k}$,
Lemma~\ref{lemma: choice of bar epsilon and bar delta, LD for Levy MRV} allows us to fix some $\bar\delta > 0$ such that the following claim holds:
given any $\xi \in B  \cap \shortbarDepsk{0}{\bm{\mathcal K}}$, 
in the expression \eqref{proof: path expression for elements in set D j, LD for levy MRV} for $\xi$ we have
\begin{align}
    \norm{\bm w_{\bm j, k}} > \bar\delta\text{ and }
                 \bm w_{\bm j,k}\notin \bar{\R}^d_\leqslant(\bm j,\bar\delta),
                \qquad \forall \bm j \in \powersetTilde{d},\ k \in [\mathcal K_{\bm j}].
                \nonumber
\end{align}
Then, by the definition of $\breve{
    \mathbf C
    }_{\bm{\mathcal K}}$ in \eqref{def: measure C type j L, LD for Levy, MRV},
\begin{align}
    \breve{
    \mathbf C
    }_{\bm{\mathcal K}}\big( B \cap \shortbarDepsk{0}{\bm{\mathcal K}} \big)
    & \leq 
    \frac{1}{\prod_{ \bm j \in \powersetTilde{d} } \mathcal K_{\bm j}!} 
    \cdot 
    \prod_{ \bm j \in \powersetTilde{d} }
    \bigg(
        \underbrace{ \mathbf C_{\bm j}\Big(
            \big\{
                \bm w \in \R^d_+:\ \norm{\bm w} > \bar\delta,\ \bm w \notin \bar\R^d_\leqslant(\bm j,\bar\delta)
            \big\}
        \Big) }_{ \delequal c_{\bm j}(\bar\delta) }
    \bigg)^{ \mathcal K_{\bm j} }.
    \label{proof: inequality, lemma: finiteness of measure breve C k}
\end{align}
In particular, for any $\delta \in (0,\bar\delta)$, the set 
$
 \big\{
                \bm w \in \R^d_+:\ \norm{\bm w} > \bar\delta,\ \bm w \notin \bar\R^d_\leqslant(\bm j,\bar\delta)
            \big\}
$
is bounded away from $\bar\R^d_\leqslant(\bm j,\delta)$.
Then, by the $\MRV$ condition in Assumption~\ref{assumption: ARMV in levy process}
(in particular, Claim~\eqref{cond, finite index, def: MRV} in Definition~\ref{def: MRV})
we get
$c_{\bm j}(\bar\delta) < \infty$  for each $\bm j \in \powersetTilde{d}$.
Plugging these bounds back into \eqref{proof: inequality, lemma: finiteness of measure breve C k}, we conclude the proof.
\end{proof}

Now, we are ready to state the proof of Propositions~\ref{proposition: asymptotic equivalence, LD for Levy MRV} and \ref{proposition: weak convergence, LD for Levy MRV}.

\begin{proof}[Proof of Proposition~\ref{proposition: asymptotic equivalence, LD for Levy MRV}]
\linksinpf{proposition: asymptotic equivalence, LD for Levy MRV}
Recall the definition of 
$c:\mathbb Z_+^d \to [0,\infty)$ in \eqref{def, cost function, k jump set, LD for Levy} and
$\breve c: \mathbb Z_+^{\powersetTilde{d}} \to [0,\infty)$ in \eqref{def: cost alpha j, LD for Levy MRV}.
Define events
\begin{align}
     B_0 & \delequal
    \Big\{
        \dj{}\big( \hat{\bm L}_n^{>\delta}, \bar{\bm L}_n \big)
        > \Delta 
    \Big\},
    \quad 
    B_1 \delequal 
    \Big\{
        \widetilde K^{>\delta}_n = 0
    \Big\},
    \quad
    \notationdef{notation-event-B1-in-proof-MRV}{B_2} \delequal \Big\{
        \breve c(\bm{K}^{>\delta}_n) \leq c({\bm k})
        \Big\}.
    \label{def: events B1 and B2; proposition: asymptotic equivalence, LD for Levy MRV}
\end{align}
To establish Claim~\eqref{proof: goal 1, AE, theorem: LD for levy, MRV}, it suffices to show that for any $\delta > 0$ small enough,
\begin{align}
    \lim_{n \to \infty}
    \P\big( B_0 \setminus B_1 \big) 
    \big/ \breve \lambda_{\bm k}(n) 
    & = 0,
    \label{proof: asymptotic equivalence, goal 1, theorem: LD for levy, MRV}
    \\
    \lim_{n \to \infty}
    \P\big( (B_0 \cap B_1) \setminus B_2 \big) 
    \big/ \breve \lambda_{\bm k}(n) 
    & = 0,
    \label{proof: asymptotic equivalence, goal 2, theorem: LD for levy, MRV}
    \\
    \lim_{n \to \infty}
    \P\big( B_0 \cap B_1 \cap B_2 \big) 
    \big/ \breve \lambda_{\bm k}(n) 
    & = 0.
    \label{proof: asymptotic equivalence, goal 3, theorem: LD for levy, MRV}
\end{align}

\medskip
\noindent
\textbf{Proof of Claim~\eqref{proof: asymptotic equivalence, goal 1, theorem: LD for levy, MRV}}.
By Claim~\eqref{claim, probability of event tilde E, lemma: asymptotic law for large jumps, LD for Levy MRV} of Lemma~\ref{lemma: asymptotic law for large jumps, LD for Levy MRV},
we get
$
\P\big((B_1)^\complement\big) = \lo\big(\breve \lambda_{\bm k}(n)\big)
$
as $n \to \infty$ under any $\delta > 0$.

\medskip
\noindent
\textbf{Proof of Claim~\eqref{proof: asymptotic equivalence, goal 2, theorem: LD for levy, MRV}}.
We prove this claim for any $\delta > 0$.
Let 
$$\mathbb K_{>} \delequal \Big\{ \bm{\mathcal K} \in \mathbb Z_+^{\powersetTilde{d}} :\ \breve c(\bm{\mathcal K}) >  c({\bm k})   \Big\}.$$
Note that 
$
(B_0 \cap B_1) \setminus B_2 \subseteq (B_2)^\complement
=
\{
\breve c(\bm{K}^{>\delta}_n) > c({\bm k})
 \}
= \{
    \bm{K}^{>\delta}_n \in \mathbb K_>
\}.
$
Meanwhile, by the definition of the linear function $\breve c(\cdot)$ in \eqref{def: cost alpha j, LD for Levy MRV}
with coefficients $\alpha(\bm j) - 1 > 0$ for any $\bm j \in \powersetTilde{d}$ (see Assumption~\ref{assumption: ARMV in levy process}),
there exists a subset $\mathbb K^*_{>} \subseteq  \mathbb K_{>}$ 
such that $|\mathbb K^*_{>}| < \infty$ (i.e., containing only finitely many elements) and
\begin{align}
    \breve c(\bm{\mathcal K}^\prime) > c({\bm k})
    \quad \Longrightarrow \quad 
    \bm{\mathcal K}^{\prime} \geq \bm{\mathcal K}\text{ for some }\bm{\mathcal K} \in \mathbb K^*_{>}.
    \nonumber
\end{align}
Here, 
the ordering $\bm x \geq \bm y$ between two real vectors is equivalent to $x_i \geq y_i$ for all $i$.
Then,
$$
\big\{
\breve c(\bm{K}^{>\delta}_n) >  c({\bm k})
 \big\}
\subseteq
\bigcup_{  \bm{\mathcal K} \in \mathbb K^*_>  }\big\{ \bm{K}^{>\delta}_n \geq \bm{\mathcal K}  \big\}.
$$
Therefore, to prove Claim~\eqref{proof: asymptotic equivalence, goal 2, theorem: LD for levy, MRV},
it suffices to fix some $\bm{\mathcal K} \in \mathbb K^*_>$
(or, more generally, some 
$\bm{\mathcal K} = (\mathcal K_{\bm j})_{ \bm j \in \powersetTilde{d} } \in \mathbb Z_+^{\powersetTilde{d}}$ such that $\breve c(\bm{\mathcal K}) > c({\bm k})$)
and show that 
\begin{align}
    \lim_{n \to \infty}
    \P\big(
        \bm{{K}}^{>\delta}_n \geq \bm{\mathcal K}
    \big)
    \big/ \breve \lambda_{\bm k}(n) & = 0.
    \nonumber
\end{align}
By the definition of $\bm K^{>\delta}_n$ in \eqref{def: K delta n type j, count vector, LD for Levy} and the law of $\bm L(t)$ in \eqref{prelim: levy ito decomp}, 
\begin{align}
   &  \P\big(
        \bm{{K}}^{>\delta}_n \geq \bm{\mathcal K}
    \big)
    \label{proof: display 1; asymptotic equivalence, goal 1, theorem: LD for levy, MRV}
    \\
    & = 
    \prod_{\bm j \in \powersetTilde{d}}
    \P\bigg(
        K^{>\delta}_n(\bm j) \geq  \mathcal K_{\bm j}
    \bigg)
    \quad
    \text{due to the independence of Poisson splitting}
    \nonumber
    \\ 
    & = 
    \prod_{\bm j \in \powersetTilde{d}}
    \P\bigg(
        \text{Poisson}\Big( n\underbrace{ \nu_n\big( \bar \R^{>\delta}(\bm j)\big) }_{ \delequal \theta_n(\bm j,\delta)  } \Big) \geq \mathcal K_{\bm j}
        \bigg)
    \qquad
    \text{by \eqref{def: large jump in cone i, tau > delta n k i, LD for Levy MRV}--\eqref{def: K delta n type j, count scalar, LD for Levy}, where }\nu_n(A)  = \nu\{n\bm x:\ \bm x \in A\}
    \nonumber
    \\ 
    & \leq 
    \prod_{\bm j \in \powersetTilde{d} }
    \Big( n \theta_n(\bm j,\delta) \Big)^{\mathcal K_{\bm j}}
    \qquad
    \text{due to }
    \P\big( \text{Poisson}(\lambda) \geq k \big) \leq \lambda^k\ \forall \lambda \geq 0.
    \nonumber
\end{align}
Under the $\MRV$ condition in Assumption~\ref{assumption: ARMV in levy process}
(in particular, by Claim~\eqref{cond, finite index, def: MRV} in Definition~\ref{def: MRV}),
we have
$
\theta_n(\bm j,\delta) = \bo\big( \lambda_{\bm j}(n)  \big)
$
for each $\bm j \in \powersetTilde{d}$.
Plugging these bounds into display \eqref{proof: display 1; asymptotic equivalence, goal 1, theorem: LD for levy, MRV}, we get
\begin{align*}
    \P\big(
        \bm{{K}}^{>\delta}_n \geq \bm{\mathcal K}
    \big)
=
\bo\Bigg( 
    \prod_{ \bm j \in \powersetTilde{d}  } \big( n\lambda_{\bm j}(n)\big)^{ \mathcal K_{\bm j}  }
\Bigg),
\qquad\text{as }n \to \infty.
\end{align*}
By  the definition of $\breve c(\cdot)$ in \eqref{def: cost alpha j, LD for Levy MRV}
and $\lambda_{\bm j}(n)\in \RV_{-\alpha(\bm j) }(n)$,
we have 
$
\prod_{ \bm j \in \powersetTilde{d}  } \big( n\lambda_{\bm j}(n)\big)^{ \mathcal K_{\bm j}  }
\in \RV_{ -\breve c(\bm{\mathcal K}) }(n).
$
Lastly, due to $\breve \lambda_{\bm k}(n) \in \RV_{ -c(\bm k) }(n)$ (see \eqref{def: scale function for Levy LD MRV}) and $\breve c(\bm{\mathcal K}) > c(\bm k)$,
we get
$
\P\big(
        \bm{{K}}^{>\delta}_n \geq \bm{\mathcal K}
    \big)
= \lo\big(  \breve \lambda_{\bm k}(n) \big).
$
This concludes the proof of Claim~\eqref{proof: asymptotic equivalence, goal 2, theorem: LD for levy, MRV}.


\medskip
\noindent
\textbf{Proof of Claim~\eqref{proof: asymptotic equivalence, goal 3, theorem: LD for levy, MRV}}.
We first note that the set
\begin{align*}
    \mathbb K_{\leqslant}
    \delequal 
    \big\{
        \bm{\mathcal K} \in \mathbb Z_+^{\powersetTilde{d}}:\ 
        \bm{\mathcal K} \notin \mathbb A(\bm k),\ \breve c(\bm{\mathcal K}) \leq c(\bm k)
    \big\}
\end{align*}
contains only finitely many elements.
Therefore, 
it suffices to fix some $\bm{\mathcal K}\in \mathbb K_{\leqslant}$ and show that 
\begin{align}
    \lim_{n \to \infty}
    \P\Big(
    \big\{
        \dj{}\big(\bar{\bm L}_n,\ \hat{\bm L}^{>\delta}_n\big) > \Delta
    \big\}
    \cap 
    \big\{
        \bm K^{>\delta}_n = \bm{\mathcal K},\ \widetilde K^{>\delta}_n = 0
    \big\}
    \Big)
    \Big/ \breve \lambda_{\bm k}(n) = 0,
    \qquad
    \forall \delta > 0\text{ small enough.}
    \nonumber
\end{align}
Let events $A_n(m;\delta,k,\Delta)$ be defined as in Lemma~\ref{lemma: large jump approximation, LD for Levy MRV},
and let $|\bm{\mathcal K}|= \sum_{\bm j \in \powersetTilde{d}} \mathcal K_{\bm j}$.
By Lemma~\ref{lemma: large jump approximation, LD for Levy MRV},
\begin{align}
    \dj{}\big(\bar{\bm L}_n,\ \hat{\bm L}^{>\delta}_n\big) \leq \Delta
    \text{ holds on the event }
    \big\{
        \bm K^{>\delta}_n = \bm{\mathcal K},\ \widetilde K^{>\delta}_n = 0
    \big\} \cap \big(\cap_{m = 1}^{|\bm{\mathcal K}| + 1}A_n(m;\delta,|\bm{\mathcal K}|,\Delta)\big).
    \label{proof: bar L n close to hat L n; asymptotic equivalence, goal 2, theorem: LD for levy, MRV}
\end{align}
Then, it only remains to show that
$
\lim_{n \to \infty}\P
    \Big(
        \big(\bigcap_{m = 1}^{|\bm{\mathcal K}| + 1}A_n(m;\delta,|\bm{\mathcal K}|,\Delta)\big)^\complement
    \Big)
    \Big/\breve \lambda_{\bm k}(n) = 0.
$
By the strong Markov property of the L\'evy process $\bar{\bm L}_n$ at stopping times $\tau^{>\delta}_n(m)$
defined in \eqref{def: large jump time tau for Levy process L}, we have
\begin{align*}
    \P
    \Bigg(
        \Big(\bigcap_{m = 1}^{|\bm{\mathcal K}|+ 1}A_n(m;\delta,|\bm{\mathcal K}|,\Delta)\Big)^\complement
    \Bigg)
    & \leq 
    (|\bm{\mathcal K}|+1) \cdot 
    \P\Bigg(
        \sup_{ t \in [0,1]:\ t < \tau^{>\delta}_n(1)  }
        \norm{
            \bar{\bm L}_n(t)
        }
        > \frac{\Delta}{|\bm{\mathcal K}| + 1}
    \Bigg)
    \\ 
    & = 
    \lo\Big( 
        \breve \lambda_{\bm k}(n)
    \Big)
    \qquad
    \text{ as $n\to\infty$ for any $\delta > 0$ small enough.}
\end{align*}
The last line follows from Lemma~\ref{lemma: concentration of small jump process, Ld for Levy MRV}.
This concludes the proof of 
Claim~\eqref{proof: asymptotic equivalence, goal 3, theorem: LD for levy, MRV}.
\end{proof}

\begin{proof}[Proof of Proposition~\ref{proposition: weak convergence, LD for Levy MRV}]
\linksinpf{proposition: weak convergence, LD for Levy MRV}
{(a)} Note that
\begin{align*}
    & f(\hat{\bm L}^{>\delta}_n)\mathbbm{I}\big\{ \bar{\bm K}^{>\delta}_n \notin \bar{\mathbb A}(\bm k) \big\}
    \\
    & \leq
    \norm{f}\cdot
    \mathbbm{I}\big\{ \hat{\bm L}^{>\delta}_n \in B  \}\mathbbm{I}\big\{ \bar{\bm K}^{>\delta}_n \notin \bar{\mathbb A}(\bm k) \big\}
    \qquad
    \text{due to $B = \text{supp}(f)$}
    \\ 
    & \leq 
    \norm{f}\cdot\underbrace{ \mathbbm{I}\big\{ \hat{\bm L}^{>\delta}_n \in B  \}\mathbbm{I}\big\{ 
        \widetilde K^{>\delta}_n \geq 1
    \big\} }_{  \delequal I_1(n,\delta) }
    +
    \norm{f}\cdot\underbrace{ \mathbbm{I}\big\{ \hat{\bm L}^{>\delta}_n \in B  \}\mathbbm{I}\big\{
        \breve c(\bm K^{>\delta}_n) > c(\bm k)
        \big\} }_{  \delequal I_2(n,\delta) }
    \\ 
    &\quad
    +
    \norm{f}\cdot\underbrace{
        \mathbbm{I}\big\{ \hat{\bm L}^{>\delta}_n \in B  \}
    \mathbbm{I}\big\{ 
        \bm K^{>\delta}_n \notin \mathbb A(\bm k),\ 
        \breve c(\bm K^{>\delta}_n) \leq c(\bm k),\ \widetilde K^{>\delta}_n = 0
        \big\}
    }_{ \delequal I_3(n,\delta) }.
\end{align*}
Repeating the proof of Claims~\eqref{proof: asymptotic equivalence, goal 1, theorem: LD for levy, MRV} and \eqref{proof: asymptotic equivalence, goal 2, theorem: LD for levy, MRV}
for Proposition~\ref{proposition: asymptotic equivalence, LD for Levy MRV},
we get
$
\E\big[ I_1(n,\delta) \big]= \lo\big(  \breve \lambda_{\bm k}(n) \big)
$
and
$
\E\big[ I_2(n,\delta) \big]= \lo\big(  \breve \lambda_{\bm k}(n) \big)
$
(as $n \to \infty$) for any $\delta > 0$.
Next, let
$
\mathbb K_< \delequal
\big\{
    \bm{\mathcal K} \in \mathbb Z_+^{ \powersetTilde{d} }:\ 
    \bm{\mathcal K} \notin \mathbb A(\bm k),\ \breve c(\bm{\mathcal K}) \leq c(\bm k)
\big\}.
$
By the definition of 
$\shortbarDepsk{\epsilon}{ \leqslant \bm k  } = \shortbarDepsk{\epsilon}{ \leqslant \bm k;\bm 0  }[0,1]$ in \eqref{claim, finite index, theorem: LD for levy, MRV},
for any $\bm{\mathcal K} \in \mathbb K_<$ we have $\shortbarDepsk{\delta}{ \bm{\mathcal K}} \subset \shortbarDepsk{\delta}{\leqslant\bm k}$.
Therefore,
given $\bm{\mathcal K} \in \mathbb K_<$,
it holds on the event $\{ \bm K^{>\delta}_n = \bm{\mathcal K},\ \widetilde K^{>\delta}_n = 0  \}$ that
\begin{align*}
    \hat{\bm L}_n^{>\delta} \in \shortbarDepsk{\delta}{\mathcal K} \subseteq \shortbarDepsk{\epsilon}{ \leqslant \bm k  },
    \qquad
    \text{due to property~\eqref{property: approximation hat L n when type = j, LD for Levy MRV} and our choice of $\delta \in (0,\epsilon)$.}
\end{align*}
Since $B$ is bounded away from $\shortbarDepsk{\epsilon}{\leqslant \bm k}$ under $\dj{}$ 
we must have $\big\{ \hat{\bm L}^{>\delta}_n \in B  \} \cap \big\{ 
        \bm K^{>\delta}_n \notin \mathbb A(\bm k),\ 
        \breve c(\bm K^{>\delta}_n) \leq c(\bm k) 
        \big\} = \emptyset$.
In summary, we have shown that $I_3(n,\delta) = 0$ for each $\delta \in (0,\epsilon)$ and $n \geq 1$.
This concludes the proof of part (a).

\medskip
\noindent
(b)
Recall that there are only finitely many assignments for $\bm k$ (i.e., $|\mathbb A(\bm k)| < \infty$).
Therefore, it suffices to prove part (b) for some fixed $\bm{\mathcal K} = (\mathcal K_{\bm j})_{\bm j \in \powersetTilde{d}} \in \mathbb A(\bm k)$.
By the definition of $B =\text{supp}(f)$, we have
$
\breve{\mathbf C}_{\bm{\mathcal K}}(f) \leq \norm{f} \cdot \breve{\mathbf C}_{\bm{\mathcal K}}(B).
$
Using Lemma~\ref{lemma: finiteness of measure breve C k}, we confirm that 
$
\breve{\mathbf C}_{\bm{\mathcal K}}(f) < \infty.
$
Next,
applying Lemma~\ref{lemma: choice of bar epsilon and bar delta, LD for Levy MRV} onto $B=\text{supp}(f)$, one can fix some $\bar\epsilon,\ \bar\delta > 0$ such that 
$
\dj{}(B^{\bar\epsilon},\shortbarDepsk{\bar\epsilon}{\leqslant \bm k}) > \bar\epsilon
$
and, given $\xi \in B^{\bar\epsilon} \cap \shortbarDepsk{\bar\epsilon}{\bm{\mathcal K}}$,
in the expression \eqref{proof: path expression for elements in set D j, LD for levy MRV} for $\xi$ we have
\begin{align}
    \norm{\bm w_{\bm j, k}} > \bar\delta\text{ and }
                \bm w_{\bm j, k} \notin \bar{\R}^d_\leqslant(\bm j,\bar\delta),
                \qquad \forall \bm j \in \powersetTilde{d},\ k \in [\mathcal K_{\bm j}].
    \label{proof: choice of bar delta; goal 1, AE, theorem: LD for levy, MRV}
\end{align}
To proceed, let
\begin{align*}
    I_{\bm{\mathcal{K}}}(n,\delta) \delequal f(\hat{\bm L}^{>\delta}_n)\mathbbm{I}\big\{ \bm K^{>\delta}_n = \bm{\mathcal K},\ \widetilde K^{>\delta}_n = 0  \big\}.
\end{align*}
Note that
$\{ \bm K^{>\delta}_n = \bm{\mathcal K},\ \widetilde K^{>\delta}_n = 0  \big\} = E^{>\delta}_n(\bm{\mathcal K})$;
see \eqref{def: event-E->delta-n-c-j}.
Also, note that Claim~\eqref{proof: choice of c, bounded away set from cone i, Ld for Levy MRV} holds for all but countably many $\delta > 0$, which is verified by Lemma~\ref{lemma: zero mass for boundary sets, LD for Levy}.
Henceforth in this proof, we only consider $\delta \in (0,\bar\delta)$ such that Claim~\eqref{proof: choice of c, bounded away set from cone i, Ld for Levy MRV} holds.
Next,
let $|\bm{\mathcal K}| \delequal  \sum_{\bm j \in \powersetTilde{d}}\mathcal K_{\bm j}$,
and define the mapping $h(\cdot)$ as follows:
given
$
\textbf{W} = (\bm w_1,\ldots,\bm w_{|\bm{\mathcal K}|}) \in (\R^{d}_+)^{|\bm{\mathcal K}|}
$
and 
$
\bm t = (t_1,\ldots,t_{|\bm{\mathcal K}|}) \in (0,1]^{|\bm{\mathcal K}|},
$
let 
\begin{align}
    h(\textbf{W},\bm t) \delequal 
    \sum_{ k = 1,2,\ldots, |\bm{\mathcal K}| }\bm w_k \mathbbm{I}_{ [t_k,1] }.
    \label{proof, def mapping h, proposition: weak convergence, LD for Levy MRV}
\end{align}
One can see that $h:(\R^{d}_+)^{|\bm{\mathcal K}|} \times (0,1]^{|\bm{\mathcal K}|} \to \D$ is continuous (w.r.t.\ the $J_1$ topology of $\D$) when restricted on the domain
$
(\R^{d}_+)^{|\bm{\mathcal K}|} \times (0,1)^{|\bm{\mathcal K}|, *},
$
with
$
A^{k,*} \delequal
    \big\{
        (t_1,\ldots,t_k) \in A^k:\ t_i \neq t_j \ \forall i,j \in [d] \text{ with }i\neq j
    \big\}.
$
This allows us to apply the continuous mapping theorem and study the asymptotic law of $f(\hat{\bm L}^{>\delta}_n)$ conditioned on the event 
$
E^{>\delta}_n(\bm{\mathcal K}).
$
Specifically,
on $\{ \bm K^{>\delta}_n = \bm{\mathcal K},\ \widetilde K^{>\delta}_n = 0 \}$, we write
\begin{align*}
    \textbf{W}^{>\delta}_n 
    & =
    \Big(
        \bm W^{>\delta}_n(k;\bm j)
    \Big)_{\bm j \in \powersetTilde{d},\  k \in [\mathcal K_{\bm j}] },
    \qquad
    \textbf{T}^{>\delta}_n 
    = 
    \Big(
        \tau^{>\delta}_n(k;\bm j)
    \Big)_{\bm j \in \powersetTilde{d},\  k \in [\mathcal K_{\bm j}] }.
\end{align*}
By \eqref{property: approximation hat L n when type = j, LD for Levy MRV}, on the event
$
E^{>\delta}_n(\bm{\mathcal K})
$
we have
$
\hat{\bm L}^{>\delta}_n = h\big(\textbf{W}^{>\delta}_n, \textbf{T}^{>\delta}_n \big).
$
As a result,
\begin{align*}
    \frac{ \E \big[ I_{\bm{\mathcal{K}}}(n,\delta) \big] }{
        \breve \lambda_{\bm k}(n)
    }
    =
    \E\Big[
        f\Big(
            h\big(\textbf{W}^{>\delta}_n, \textbf{T}^{>\delta}_n \big)
        \Big)
        \
        \Big| 
        \ 
           E^{>\delta}_n(\bm{\mathcal K})
    \Big]
    \cdot 
    \frac{
        \P\big( E^{>\delta}_n(\bm{\mathcal K}) \big)
    }{
        \breve \lambda_{\bm k}(n)
    }.
\end{align*}
Then, by Lemma~\ref{lemma: asymptotic law for large jumps, LD for Levy MRV} and the continuity of $f\big(h(\cdot)\big)$ on $(\R^{d}_+)^{ |\mathcal K| } \times (0,1)^{ |\mathcal K|, *}$, we obtain the following results
(note that $f(\xi) = 0\ \forall \xi \notin B$), where we use the $U_{\bm j, k}$'s to denote i.i.d.\ copies of Unif$(0,1)$,
$
U_{\bm j,(1:k)} \leq U_{\bm j,(2:k)} \leq \ldots \leq U_{\bm j,(k:k)}
$
for the order statistics of $(U_{\bm j,q})_{q \in [k]}$,
$\mathcal L_I$ for the Lebesgue measure restricted on interval $I$,
and $\mathbf C_{\bm j}(\cdot)$ for the location measure in the $\MRV$ condition of Assumption~\ref{assumption: ARMV in levy process} supported on $\R^d(\bm j)$ (see Definition~\ref{def: MRV}):
\begin{align}
    & \lim_{n \to \infty}
    \E\bigg[
        f\Big(
            h\big(\textbf{W}^{>\delta}_n, \textbf{T}^{>\delta}_n \big)
        \Big)
        \
        \bigg| 
        \ 
            E^{>\delta}_n(\bm{\mathcal K})
    \bigg]
    \nonumber
    \\
    & = 
    \int 
    f\Bigg(
        \sum_{\bm j \in \powersetTilde{d}}\sum_{ k \in [\mathcal K_{\bm j}] }
        \bm w_{\bm j, k}\mathbbm{I}_{[t_{\bm j,k},1]}
    \Bigg)
    \nonumber
    \\
    &\qquad
    \cdot 
    \mathbbm{I}\Bigg\{ 
        \sum_{\bm j \in \powersetTilde{d}}\sum_{ k \in [\mathcal K_{\bm j}] }
        \bm w_{\bm j, k}\mathbbm{I}_{[t_{\bm j,k},1]} \in B 
        \Bigg\}
        \cdot
    \mathbbm{I}\Big\{ 
        \bm w_{\bm j,k} \in \bar\R^{>\delta}(\bm j)\ 
        \forall \bm j \in \powersetTilde{d},\ k \in [\mathcal K_{\bm j}]
        \Big\}
    \nonumber
    \\
    &\qquad
    \cdot
    \Bigg(
        \bigtimes_{ \bm j \in \powersetTilde{d} }\bigtimes_{k \in [\mathcal K_{\bm j}]}
        \frac{
            \mathbf C_{\bm j}(d\bm w_{\bm j,k})
        }{
            \mathbf C_{\bm j}\big( \bar\R^{>\delta}(\bm j) \big)
        }
    \Bigg)
    \Bigg(
        \bigtimes_{\bm j \in \powersetTilde{d}}
        \P\bigg(
        \Big( U_{\bm j,(1:\mathcal K_{\bm j})}, \ldots,U_{\bm j,(\mathcal K_{\bm j}:\mathcal K_{\bm j})}\Big)
        \in d (t_{\bm j,1},\ldots, t_{\bm j, \mathcal K_{\bm j}})
        \bigg)
    \Bigg)
    \nonumber
    \\
    & \stackrel{(*)}{=} 
    \int 
    f\Bigg(
        \sum_{\bm j \in \powersetTilde{d}}\sum_{ k \in [\mathcal K_{\bm j}] }
        \bm w_{\bm j, k}\mathbbm{I}_{[t_{\bm j,k},1]}
    \Bigg)
    \nonumber
    \\
    &\qquad
    \cdot 
    \mathbbm{I}\Bigg\{ 
        \sum_{\bm j \in \powersetTilde{d}}\sum_{ k \in [\mathcal K_{\bm j}] }
        \bm w_{\bm j, k}\mathbbm{I}_{[t_{\bm j,k},1]} \in B 
        \Bigg\}
        \cdot
    \mathbbm{I}\Big\{ 
        \bm w_{\bm j,k} \notin \bar\R^d_\leqslant(\bm j,\delta)\ 
        \forall \bm j \in \powersetTilde{d},\ k \in [\mathcal K_{\bm j}]
        \Big\}
    \nonumber
    \\
    &\qquad
    \cdot
    \Bigg(
        \bigtimes_{ \bm j \in \powersetTilde{d} }\bigtimes_{k \in [\mathcal K_{\bm j}]}
        \frac{
            \mathbf C_{\bm j}(d\bm w_{\bm j,k})
        }{
            \mathbf C_{\bm j}\big( \bar \R^{>\delta}(\bm j) \big)
        }
    \Bigg)
    \Bigg(
        \bigtimes_{\bm j \in \powersetTilde{d}}
        \P\bigg(
        \Big( U_{\bm j,(1:\mathcal K_{\bm j})}, \ldots,U_{\bm j,(\mathcal K_{\bm j}:\mathcal K_{\bm j})}\Big)
        \in d (t_{\bm j,1},\ldots, t_{\bm j, \mathcal K_{\bm j}})
        \bigg)
    \Bigg)
    \nonumber
    \\ 
    & \stackrel{(\triangle)}{=}
    \int
    f\Bigg(
        \sum_{\bm j \in \powersetTilde{d}}\sum_{ k \in [\mathcal K_{\bm j}] }
        \bm w_{\bm j, k}\mathbbm{I}_{[t_{\bm j,k},1]}
    \Bigg)
    \nonumber
    \\ 
    &
    \qquad
    \cdot
    \Bigg(
        \bigtimes_{ \bm j \in \powersetTilde{d} }\bigtimes_{k \in [\mathcal K_{\bm j}]}
        \frac{
            \mathbf C_{\bm j}(d\bm w_{\bm j,k})
        }{
            \mathbf C_{\bm j}\big( \bar \R^{>\delta}(\bm j) \big)
        }
    \Bigg)
    \Bigg(
        \bigtimes_{\bm j \in \powersetTilde{d}}
        \P\bigg(
        \Big( U_{\bm j,(1:\mathcal K_{\bm j})}, \ldots,U_{\bm j,(\mathcal K_{\bm j}:\mathcal K_{\bm j})}\Big)
        \in d (t_{\bm j,1},\ldots, t_{\bm j, \mathcal K_{\bm j}})
        \bigg)
    \Bigg).
    \nonumber
\end{align}
In the display above, the step $(*)$ holds since $\bar\R^{>\delta}(\bm j) = \bar\R^d(\bm j,\delta)\setminus \bar \R^d_\leqslant(\bm j,\delta)$, and $\mathbf C_{\bm j}(\cdot)$ is supported on $\R^d(\bm j) \subseteq \bar\R^d(\bm j,\delta)$;
the step $(\triangle)$ follows from  the definition of $B = \text{supp}(f)$,
our choice of $\bar\delta$ in \eqref{proof: choice of bar delta; goal 1, AE, theorem: LD for levy, MRV} and $\delta \in (0,\bar\delta)$.
To proceed, we make a few observations.
First, using $U_{(1:k)} < U_{(2:k)} < \ldots < U_{(k:k)}$ to denote the order statistics of $k$ copies of Unif$(0,1)$, we have
\begin{align*}
    \P\Big(
        (U_{(1:k)},U_{(2:k)},\ldots,U_{(k:k)}) \in d(t_1,\ldots,t_k)
    \Big)
    =
    k! \mathbbm{I}\{ 0< t_1 < t_2 < \ldots < t_k < 1\}dt_1dt_2\ldots dt_k.
\end{align*}
Second, the mapping $h$ defined in \eqref{proof, def mapping h, proposition: weak convergence, LD for Levy MRV} is invariant under permutation of its arguments.
That is, given $\textbf{W} = (\bm w_1,\ldots,\bm w_{ |\bm{\mathcal K}| })$ and $\bm t = (t_1,\ldots,t_{|\bm{\mathcal K}|})$,
we have
$$
h(\textbf W,\bm t) =
h\Big( (\bm w_{\sigma(i)})_{ i =1,\ldots,|\bm{\mathcal K}| },\ (t_{\sigma(i)})_{ i =1,\ldots,|\bm{\mathcal K}| }  \Big)
$$
for any permutation $(\sigma(i))_{ i =1,\ldots,|\bm{\mathcal K}| }$ of $\{1,2,\ldots,|\bm{\mathcal K}|\}$.
These two properties allow us to re-evaluate step $(\triangle)$ as an integral over the domain 
$
\{ (t_{\bm j,k})_{\bm j, k}:\ t_{\bm j,k} \in (0,1) \forall \bm j,k  \}
$
instead of imposing the ordering constraint 
$
t_{\bm j, 1} < t_{\bm j, 2} < \ldots < t_{\bm j, \mathcal K_{\bm j}}
$
for each $\bm j$.
Specifically,
\begin{align}
    & \lim_{n \to \infty}
    \E\bigg[
        f\Big(
            h\big(\textbf{W}^{>\delta}_n, \textbf{T}^{>\delta}_n \big)
        \Big)
        \
        \bigg| 
        \ 
            E^{>\delta}_n(\bm{\mathcal K})
    \bigg]
    \label{proof: limit 1, proposition: weak convergence, LD for Levy MRV}
    \\ 
    & = 
   \Bigg[\prod_{ \bm j \in \powersetTilde{d} } \bigg( 
        \mathbf C_{\bm j}\big(\bar\R^{>\delta}(\bm j)\big)
    \bigg)^{- \mathcal K_{\bm j} } \Bigg]
    \nonumber
    \\ 
    &\qquad\qquad
    \cdot 
    \int 
    f\Bigg(
        \sum_{\bm j \in \powersetTilde{d}}\sum_{ k \in [\mathcal K_{\bm j}] }
        \bm w_{\bm j, k}\mathbbm{I}_{[t_{\bm j,k},1]}
    \Bigg)
    \bigtimes_{ \bm j \in \powersetTilde{d} }\bigtimes_{k \in [\mathcal K_{\bm j}]}
    \bigg(\mathbf C_{\bm j} \times \mathcal L_{(0,1)}\Big(d (\bm w_{\bm j, k}, t_{ \bm j,k })\Big)\bigg)
    \nonumber
    \\ 
    & = 
   \Bigg[\prod_{ \bm j \in \powersetTilde{d} } \bigg( 
        \mathbf C_{\bm j}\big(\bar\R^{>\delta}(\bm j)\big)
    \bigg)^{- \mathcal K_{\bm j} } \Bigg]
    \cdot 
    \Bigg( {\prod_{ \bm j \in \powersetTilde{d} } \mathcal K_{\bm j}!}  \Bigg)
    \cdot
    \breve{\mathbf C}_{\mathcal K}(f)
    \qquad
    \text{ by definitions in \eqref{def: measure C type j L, LD for Levy, MRV} and \eqref{proof, def: shorthand, LD for Levy}}.
    \nonumber
\end{align}
On the other hand, applying the claim~\eqref{claim, probability of event E, lemma: asymptotic law for large jumps, LD for Levy MRV} in Lemma~\ref{lemma: asymptotic law for large jumps, LD for Levy MRV},
we get
\begin{align}
    \lim_{n \to \infty}
    \frac{
        \P\big( E^{>\delta}_n(\mathcal K) \big)
    }{
        \breve \lambda_{\bm k}(n)
    }
    =
    \prod_{ \bm j \in \powersetTilde{d} }
            \frac{1}{\mathcal K_{\bm j} !}\cdot 
            \Big( 
                \mathbf C_{\bm j}\big( \bar\R^{>\delta}(\bm j) \big)
            \Big)^{\mathcal K_{\bm j}}.
    \label{proof: limit 2, proposition: weak convergence, LD for Levy MRV}
\end{align}
Combining \eqref{proof: limit 1, proposition: weak convergence, LD for Levy MRV} and \eqref{proof: limit 2, proposition: weak convergence, LD for Levy MRV}, we conclude the proof.
\end{proof}


\section{Proof for Large Deviations of Multivariate Heavy-Tailed Hawkes Processes}
\label{subsec: proof, theorem: LD for Hawkes}

\subsection{Adapting Theorem~\ref{theorem: LD for levy, MRV} to $\D[0,\infty)$}
\label{subsec: corollary, LD for Levy, unbounded domain}

We first establish sample path large deviations for L\'evy processes with $\MRV$ increments (i.e., a suitable modification of Theorem~\ref{theorem: LD for levy, MRV}) w.r.t.\ the $J_1$ topology on $\D[0,\infty)$.
Specifically, 
recall that the projection mapping
${\phi_t}: \D[0,\infty) \to \D[0,t]$ is defined by $\phi_t(\xi)(s) = \xi(s)$ for any $s \in [0,t]$,
and
let
\begin{align}
    \notationdef{notation-J1-metric-for-D-infty}{\dj{[0,\infty)}(\xi^{(1)},\xi^{(2)})}
    \delequal
    \int_0^\infty 
        e^{-t} \cdot \Big[ \dj{[0,t]}\Big( \phi_t\big(\xi^{(1)}),\ \phi_t\big(\xi^{(2)}\big)\Big)  \wedge 1 \Big]dt,
    \qquad
    \forall \xi^{(1)},\xi^{(2)} \in \D[0,\infty).
    \label{def: J1 metric on [0,infty]}
\end{align}
We need the following two lemmas.
First, Lemma~\ref{lemma: B bounded away under J1 infinity} adapts Lemmas~\ref{lemma: choice of bar epsilon and bar delta, LD for Levy MRV} and \ref{lemma: finiteness of measure breve C k} to $\D[0,\infty)$.

\begin{lemma}\label{lemma: B bounded away under J1 infinity}
\linksinthm{lemma: B bounded away under J1 infinity}
Let Assumption~\ref{assumption: ARMV in levy process} hold.
Let
    $\bm k = (k_{j})_{ j \in [d] } \in \mathbb Z_+^{ d } \setminus \{\bm 0\}$,
    and $\epsilon > 0$.
Given any Borel set $B$ of $(\D[0,\infty),\dj{\subInfty})$ that is bounded away from $\barDxepskT{\bm\mu_{\bm L}}{\epsilon}{\leqslant \bm k}{[0,\infty)}$ under $\dj{[0,\infty)}$,
\begin{enumerate}[(i)]
    \item
        there exist $T \in (0,\infty)$ and $\bar\delta \in (0,\infty)$ such that the following claim holds:
        for any 
        $\bm{\mathcal K} = (\mathcal K_{\bm j})_{\bm j \in \powersetTilde{d}} \in \mathbb A(\bm k)$
        and
        $\xi \in B \cap \shortbarDepsk{\epsilon}{\bm{\mathcal K}; \bm \mu_{\bm L}}[0,\infty)$, in the expression \eqref{expression for xi in breve D set} for $\xi$ (i.e., under $I = [0,\infty)$ and $\bm x = \bm \mu_{\bm L}$) , we have
        \begin{align}
            t_{\bm j, k} \leq T,\quad
            \norm{\bm w_{\bm j, k}} > \bar\delta,\text{ and }
                 \bm w_{\bm j,k}\notin \bar{\R}^d_\leqslant(\bm j,\bar\delta),
                  \qquad\qquad \forall \bm j \in \powersetTilde{d},\ k \in [\mathcal K_{\bm j}];
                  \label{claim, part i, lemma: B bounded away under J1 infinity}
        \end{align}

    \item
        $
            \sum_{ \bm{\mathcal K} \in \mathbb A(\bm k) }\breve{ \mathbf C }^{ \subInfty }_{ \bm{\mathcal K};\bm \mu_{\bm L} }(B) < \infty.
        $
\end{enumerate}
\end{lemma}

\begin{proof}\linksinpf{lemma: B bounded away under J1 infinity}
Fix some $\Delta>0$ such that 
$
\dj{\subInfty}\big(B, \barDxepskT{\bm\mu_{\bm L}}{\epsilon}{\leqslant \bm k}{[0,\infty)} \big) > \Delta.
$

\medskip
\noindent
$(i)$
Since $|\mathbb A(\bm k)| < \infty$,
it suffices to prove part $(i)$ for some fixed $\bm{\mathcal K} = (\mathcal K_{\bm j})_{\bm j \in \powersetTilde{d}} \in \mathbb A(\bm k)$.
We first show that Claim~\eqref{claim, part i, lemma: B bounded away under J1 infinity} holds for any $T > 0$ large enough such that 
\begin{align}
    e^{-T} < \Delta.
    \label{proof, part i, choice of T, lemma: B bounded away under J1 infinity}
\end{align}
We consider a proof by contradiction.
Suppose that for some $\xi \in B \cap \shortbarDepsk{\epsilon}{\bm{\mathcal K}; \bm \mu_{\bm L}}[0,\infty)$, in expression \eqref{expression for xi in breve D set}
there exist some $\bm j^* \in \powersetTilde{d}$ and $k^* \in [\mathcal K_{\bm j^*}]$ such that $t_{\bm j^*, k^*} > T$.
By defining
\begin{align}
    \xi^\prime(t) \delequal 
    \bm \mu_{\bm L}t
    + 
    \sum_{ \bm j \in \powersetTilde{d} }\sum_{ k = 1 }^{ \mathcal K_{\bm j} }\bm w_{\bm j, k} 
    \mathbbm{I}\{ t_{\bm j,k} \leq T   \}\cdot 
    \mathbbm{I}_{ [t_{\bm j, k},\infty) }(t),
    \qquad\forall t \geq 0,
    \nonumber
\end{align}
we construct a path $\xi^\prime$ by removing any jump in $\xi$ that arrives after time $T$.
On the one hand, the existence of $t_{\bm j^*, k^*} > T$ implies that we are removing at least one jump when defining $\xi^\prime$,
and hence
$
\xi^\prime \in \shortbarDepsk{\epsilon}{\leqslant \bm k;\bm \mu_{\bm L}}[0,\infty)
$
by definitions in \eqref{def: path with costs less than j jump set, drift x, LD for Levy MRV}.
On the other hand, since $\xi(t) = \xi^\prime(t)$ for each $t \in [0,T]$, we have 
$
\dj{\subInfty}(\xi,\xi^\prime) \leq \int_{t > T}e^{-t}dt = e^{-T} < \Delta. 
$
Due to $\xi \in B$, we arrive at the contradiction that $
\dj{\subInfty}\big(B, \barDxepskT{\bm \mu_{\bm L}}{\epsilon}{\leqslant \bm k}{[0,\infty)} \big) < \Delta,
$
which verifies claim~\eqref{claim, part i, lemma: B bounded away under J1 infinity} under our choice of $T$ in \eqref{proof, part i, choice of T, lemma: B bounded away under J1 infinity}.

Next, 
through a proof by contradiction,
we show that  Claim~\eqref{claim, part i, lemma: B bounded away under J1 infinity} holds under 
\begin{align}
    \bar\delta \in (0, \epsilon \wedge \Delta).
    \label{proof, part i, choice of bar delta, lemma: B bounded away under J1 infinity}
\end{align}
First, suppose that for some $\xi \in B \cap \shortbarDepsk{\epsilon}{\bm{\mathcal K}; \bm \mu_{\bm L}}[0,\infty)$, 
in expression \eqref{expression for xi in breve D set}
there exist some $\bm j^* \in \powersetTilde{d}$ and $k^* \in [\mathcal K_{\bm j^*}]$ such that $\norm{\bm w_{ \bm j^*,k^* }} \leq \bar\delta$.
By defining 
\begin{align}
    \xi^\prime(t) \delequal 
    \bm \mu_{\bm L}t
    + 
    \sum_{ \bm j \in \powersetTilde{d} }\sum_{ k = 1 }^{ \mathcal K_{\bm j} }\bm w_{\bm j, k} 
    \mathbbm{I}\Big\{ (\bm j,k) \neq (\bm j^*,k^*)  \Big\}\cdot 
    \mathbbm{I}_{ [t_{\bm j, k},\infty) }(t),
    \qquad\forall t \geq 0,
    \nonumber
\end{align}
we construct a path $\xi^\prime$ by removing the jump $\bm w_{ \bm j^*,k^* }$---arriving at time $t_{ \bm j^*,k^* }$---from $\xi$.
Again, we have 
$
\xi^\prime \in \shortbarDepsk{\epsilon}{\leqslant \bm k;\bm \mu_{\bm L}}[0,\infty).
$
On the other hand, due to $\norm{\bm w_{ \bm j^*,k^* }} \leq \bar\delta$,
we have $\sup_{t \geq 0}\norm{ \xi(t) - \xi^\prime(t) } \leq \bar\delta $, and hence
$
\dj{\subInfty}(\xi,\xi^\prime) \leq \int_0^\infty e^{-t }\cdot \bar\delta dt = \bar\delta < \Delta;
$
see \eqref{proof, part i, choice of bar delta, lemma: B bounded away under J1 infinity}.
We then arrive at the contradiction that 
$
\dj{\subInfty}\big(B, \barDxepskT{\bm \mu_{\bm L}}{\epsilon}{\leqslant \bm k}{[0,\infty)} \big) < \Delta.
$
Second,
suppose that for some $\xi \in B \cap \shortbarDepsk{\epsilon}{\bm{\mathcal K}; \bm\mu_{\bm L}}[0,\infty)$, in expression \eqref{expression for xi in breve D set}
there exist some $\bm j^* \in \powersetTilde{d}$ and $k^* \in [\mathcal K_{\bm j^*}]$ such that ${\bm w_{ \bm j^*,k^* }} \in \bar{\R}^d_\leqslant(\bm j^*,\bar\delta)
\subseteq
\bar{\R}^d_\leqslant(\bm j^*,\epsilon)
$ 
(due to our choice of $\bar\delta < \epsilon$ in \eqref{proof, part i, choice of bar delta, lemma: B bounded away under J1 infinity}).
Then, by definitions in \eqref{def: cone R d i basis S index alpha},
there exists
some $\bm j^\prime \in \powersetTilde{d}$ with 
$\bm j^\prime \neq \bm j^*$, $\alpha(\bm j^\prime) \leq \alpha(\bm j^*)$ such that $\bm w_{\bm j^*,k^*} \in \bar\R^d(\bm j^\prime,\epsilon)$.
Now, define
$\bm{\mathcal K}^\prime = (\mathcal K^\prime_{\bm j})_{ \bm j \in \powersetTilde{d}  }$ by
\begin{align*}
    \mathcal K^\prime_{\bm j}
    \delequal
    \begin{cases}
        \mathcal K_{\bm j^*} - 1 & \text{ if }\bm j = \bm j^*
        \\ 
        \mathcal K_{\bm j^\prime} + 1 & \text{ if }\bm j = \bm j^\prime 
        \\ 
        \mathcal K_{\bm j} & \text{ otherwise}
    \end{cases}.
\end{align*}
First, $\bm w_{\bm j^*,k^*} \in \bar\R^d(\bm j^\prime,\epsilon)$ implies that
$
\xi \in \shortbarDepsk{\epsilon}{\bm{\mathcal K}^\prime;\bm \mu_{\bm L}}[0,\infty).
$
Moreover, due to $\alpha(\bm j^\prime) \leq \alpha(\bm j^*)$,
we have $\breve c(\bm{\mathcal K}^\prime) \leq \breve c(\bm{\mathcal K}) = c(\bm k)$,
and hence
$
\shortbarDepsk{\epsilon}{\bm{\mathcal K}^\prime;\bm\mu_{\bm L}}[0,\infty)
\subseteq
\shortbarDepsk{\epsilon}{\leqslant \bm k;\bm\mu_{\bm L}}[0,\infty).
$
By $\xi \in B$,
we now arrive at 
$
B \cap \shortbarDepsk{\epsilon}{\leqslant \bm k;\bm \mu_{\bm L}}[0,\infty) \neq\emptyset,
$
which clearly contradicts $
\dj{\subInfty}\big(B, \barDxepskT{\bm \mu_{\bm L}}{\epsilon}{\leqslant \bm k}{[0,\infty)} \big) > \Delta.
$
This concludes the proof of claim~\eqref{claim, part i, lemma: B bounded away under J1 infinity} under our choice of $\bar\delta$ in \eqref{proof, part i, choice of bar delta, lemma: B bounded away under J1 infinity}.

\medskip
\noindent
$(ii)$
Again, it suffices to fix some $\bm{\mathcal K} = (\mathcal K_{\bm j})_{\bm j \in \powersetTilde{d}} \in \mathbb A(\bm k)$
and show that 
$
\breve{ \mathbf C }^{ \subInfty }_{ \bm{\mathcal K};\bm\mu_{\bm L} }(B) < \infty.
$
Let $T$ and $\bar\delta$ be characterized as in part $(i)$.
By the definition of $\breve{ \mathbf C }^{ \subInfty }_{ \bm{\mathcal K};\bm\mu_{\bm L} }$ in \eqref{def: measure C type j L, LD for Levy, MRV},
\begin{align*}
    \breve{ \mathbf C }^{ \subInfty }_{ \bm{\mathcal K};\bm\mu_{\bm L} }(B)
    & \leq
    \frac{1}{\prod_{ \bm j \in \powersetTilde{d} } \mathcal K_{\bm j}!} 
    \cdot
    \int 
    \mathbbm{I}\Bigg\{
        t_{\bm j, k} \leq T,\
            \norm{\bm w_{\bm j, k}} > \bar\delta,\text{ and }
                 \bm w_{\bm j,k}\notin \bar{\R}^d_\leqslant(\bm j,\bar\delta),
                  \quad \forall \bm j \in \powersetTilde{d},\ k \in [\mathcal K_{\bm j}]
    \Bigg\}
    \\
    &\qquad\qquad\qquad\qquad\qquad\qquad\qquad\qquad\qquad
    \bigtimes_{\bm j \in \powersetTilde{d}}
    \bigtimes_{ k \in [\mathcal K_{\bm j}] }\bigg( (\mathbf C_{ \bm j } \times  \mathcal L_{(0,\infty)})\Big(d (\bm w_{\bm j, k}, t_{ \bm j,k })\Big)\bigg)
    \\ 
    & \leq 
    \frac{1}{\prod_{ \bm j \in \powersetTilde{d} } \mathcal K_{\bm j}!} 
    \cdot 
    \prod_{ \bm j \in \powersetTilde{d} }
    \bigg(
        T \cdot 
        \underbrace{ \mathbf C_{\bm j}\Big(
            \big\{
                \bm w \in \R^d_+:\ \norm{\bm w} > \bar\delta,\ \bm w \notin \bar\R^d_\leqslant(\bm j,\bar\delta)
            \big\}
        \Big) }_{ \delequal c_{\bm j}(\bar\delta) }
    \bigg)^{ \mathcal K_{\bm j} }.
\end{align*}
Furthermore, it has been shown in the proof of Lemma~\ref{lemma: finiteness of measure breve C k} that $c_{\bm j}(\bar\delta) < \infty$.
This concludes the proof of part $(ii)$.
\end{proof}

Recall the definition of the projection mapping
${\phi_t}: \D[0,\infty) \to \D[0,t]$ in \eqref{def: projection mapping from D infty to D T},
i.e., $\phi_t(\xi)(s) = \xi(s)$ for any $s \in [0,t]$.
Lemma~\ref{lemma: bounded away condition under J1, from infty to T} connects the bounded-away-from-$\barDxepskT{\bm x}{\epsilon}{\leqslant \bm k}{[0,\infty)}$ condition (under $\dj{\subInfty}$) to that of 
$\barDxepskT{\bm x}{\epsilon}{\leqslant \bm k}{[0,T]}$ and $\dj{\subZeroT{T}}$.

\begin{lemma}\label{lemma: bounded away condition under J1, from infty to T}
\linksinthm{lemma: bounded away condition under J1, from infty to T}
Let
    $\bm k \in \mathbb Z_+^{ d } \setminus \{\bm 0\}$,
    $\epsilon > 0$, and $\bm x\in \R^d$.
Given a Borel set $B$ of $(\D[0,\infty),\dj{\subInfty})$ that is bounded away from $\barDxepskT{\bm x}{\epsilon}{\leqslant \bm k}{[0,\infty)}$ under $\dj{[0,\infty)}$,
it holds for any $T > 0$ large enough that 
$
\phi_T(B)
$
is bounded away from 
$\barDxepskT{\bm x}{\epsilon}{\leqslant \bm k}{[0,T]}$ under $\dj{[0,T]}$.
\end{lemma}

\begin{proof}\linksinpf{lemma: bounded away condition under J1, from infty to T}
Fix some $\Delta>0$ such that 
$
\dj{\subInfty}\big(B, \barDxepskT{\bm x}{\epsilon}{\leqslant \bm k}{[0,\infty)} \big) > \Delta.
$
Using a proof by contradiction,
we show that the claim holds for any $T > 0$ large enough such that 
\begin{align}
    e^{-T} < \Delta/2.
    \label{proof: choice of T, lemma: bounded away condition under J1, from infty to T}
\end{align}
Specifically,
suppose there are sequences $ \xi_n \in \phi_T(B)$, $\xi^\prime_n \in \barDxepskT{\bm x}{\epsilon}{\leqslant \bm k}{[0,T]}$
such that 
\begin{align}
    \lim_{n \to \infty}\dj{[0,T]}(\xi_n,\xi^\prime_n) = 0.
    \label{proof, implication of proof by contradicton, lemma: bounded away condition under J1, from infty to T}
\end{align}
Then, 
considering the definition of $\barDxepskT{\bm x}{\epsilon}{\leqslant \bm k}{[0,T]}$ in  \eqref{def: path with costs less than j jump set, drift x, LD for Levy MRV}
and
by picking a sub-sequence if needed,
we can w.l.o.g.\ assume the existence of $\bm{\mathcal K} = (\mathcal K_{\bm j})_{\bm j \in \powersetTilde{d}} \in \mathbb Z_+^{ \powersetTilde{d} }$ such that
$\bm{\mathcal K} \in \mathbb A(\bm k),\ \breve c(\bm{\mathcal K}) \leq c(\bm k)$,
and
$
\xi^\prime_n \in \barDxepskT{\bm x}{\epsilon}{\bm{\mathcal K}}{[0,T]}
$
for each $n \geq 1$.
In particular, due to $\bm k \neq \bm 0$, we must have $\bm{\mathcal K} \neq \bm 0$.
Note that
$
|\bm{\mathcal K}|\delequal  \sum_{ \bm j \in \powersetTilde{d}  }\mathcal K_{\bm j}
$
is the number of jumps for any path in $\barDxepskT{\bm x}{\epsilon}{ \bm{\mathcal K} }{[0,T]}$.
Since $\xi^\prime_n \in  \barDxepskT{\bm x}{\epsilon}{\bm{\mathcal K} }{[0,T]}$ for each $n$,
by \eqref{expression for xi in breve D set} we know that $\xi^\prime_n$ admits the form
\begin{align}
    \xi^\prime_n(t) 
    = t \bm x
    + \sum_{ \bm j \in \powersetTilde{d} }\sum_{ k = 1 }^{ \mathcal K_{\bm j} }\bm w^{(n)}_{\bm j, k} \mathbbm{I}_{ [t^{(n)}_{\bm j, k},T]}(t),
    \qquad\forall t \in [0,T],
    \nonumber
\end{align}
where 
each $\bm w^{(n)}_{\bm j,k}$ belongs to the cone $\bar{\R}^d(\bm j,\epsilon)$.
Next, we fix some 
\begin{align}
    \delta \in \bigg(0, \frac{\Delta}{
     2 + 4\norm{\bm x} + 4|\bm{\mathcal K}|
    }\bigg).
    \label{proof, choice of delta, lemma: bounded away condition under J1, from infty to T}
\end{align}
By \eqref{proof, implication of proof by contradicton, lemma: bounded away condition under J1, from infty to T},
it holds for any $n$ large enough that we have
$
\dj{[0,T]}(\xi_n,\xi^\prime_n) < \delta. 
$
In other words, for each $n$ large enough there exists $\lambda_n$---a homeomorphism on $[0,T]$---such that 
\begin{align}
    \sup_{t \in [0,T]}| \lambda_n(t) -t   | \vee \norm{ \xi^\prime_n\big(\lambda_n(t)\big) - \xi_n(t)  } < \delta.
    \label{proof, J1 norm on compact set, lemma: bounded away condition under J1, from infty to T}
\end{align}
Under such large $n$,
we consider $t \in (0,T)$ such that
\begin{align}
    t \notin [t^{(n)}_{\bm j,k} - \delta,t^{(n)}_{\bm j,k} + \delta],
    \qquad \forall \bm j \in \powersetTilde{d},\ k \in [\mathcal K_{\bm j}].
    \label{proof, interval for regular t, lemma: bounded away condition under J1, from infty to T}
\end{align}
By \eqref{proof, interval for regular t, lemma: bounded away condition under J1, from infty to T}, $\xi^\prime_n(\cdot)$ does note make any jump over the interval $[t - \delta, t + \delta] \cap [0,T]$,
meaning that, over this interval, the path $\xi^\prime_n(\cdot)$ is a linear function with slope $\bm x$.
Then, by \eqref{proof, J1 norm on compact set, lemma: bounded away condition under J1, from infty to T}, we get
\begin{align}
    \sup_{ t_1,t_2 \in [t-\delta,t+\delta] \cap [0,T]  }\norm{ \xi_n(t_1) - \xi^\prime_n(t_2)   } \leq \delta + \norm{\bm x} \cdot 2\delta.
    \nonumber
\end{align}
Therefore, by modifying $\lambda_n(\cdot)$ only over $[t-\delta,t]$,
one can obtain $\hat \lambda^{(t)}_n$---a homeomorphism over $[0,t]$---such that
\begin{align}
    \sup_{u \in [0,t]}| \hat\lambda^{(t)}_n(u) -u   | \vee \norm{ \xi^\prime_n\big(\hat\lambda^{(t)}_n(u)\big) - \xi_n(u)  } < \delta\cdot(1 + 2\norm{\bm x}).
    \nonumber
\end{align}
Applying this bound on $\dj{[0,t]}( \phi_t(\xi_n),\phi_t(\xi^\prime_n)) \wedge 1$ for any $t \in (0,T)$ satisfying \eqref{proof, interval for regular t, lemma: bounded away condition under J1, from infty to T},
and applying the trivial upper bound $1$ for any $t$ where \eqref{proof, interval for regular t, lemma: bounded away condition under J1, from infty to T} doest not hold,
we get (for all $n$ large enough)
\begin{align}
    & \int_0^T e^{-t}
    \cdot \Big[ \dj{[0,t]}\Big( \phi_t\big(\xi_n),\ \phi_t\big(\xi^\prime_n\big)\Big)  \wedge 1 \Big]dt
    \label{proof, ineq 1, lemma: bounded away condition under J1, from infty to T}
    \\ 
    & \leq 
    \int_{ t \in [0,T]:\  t \notin [t^{(n)}_{\bm j,k} - \delta,t^{(n)}_{\bm j,k} + \delta]
    \ \forall \bm j \in \powersetTilde{d},\ k \in [\mathcal K_{\bm j}]  }
    e^{-t} \cdot 
    \delta\cdot(1 + 2\norm{\bm x})dt 
    +
    |\bm{\mathcal K}| \cdot 2\delta
    \nonumber
    \\ 
    & \leq 
    \delta \cdot \big(
        1 + 2\norm{\bm x} + 2|\bm{\mathcal K}|
    \big)
    < \Delta/2 
    \qquad
    \text{by \eqref{proof, choice of delta, lemma: bounded away condition under J1, from infty to T}}.
    \nonumber
\end{align}
Furthermore, since $\xi_n \in \phi_t(B)$, for each $n \geq 1$ there is $\tilde \xi_n \in B$ such that $\tilde \xi_n(t) = \xi_n(t)\ \forall t \in [0,T]$.
Also, by extending the path $\xi^\prime_n$ linearly with slope $\bm x$ over $(T,\infty)$,
we obtain $\tilde \xi^\prime_n \in \barDxepskT{\bm x}{\epsilon}{\bm{\mathcal K}}{[0,\infty)}$
such that $\tilde \xi^\prime_n(t) = \xi^\prime_n(t)\ \forall t \in [0,T]$.
Following from the inequality in display \eqref{proof, ineq 1, lemma: bounded away condition under J1, from infty to T},
\begin{align*}
    \dj{[0,\infty)}(\tilde\xi_n,\tilde \xi^\prime_n)
    & < \frac{\Delta}{2} + 
    \int_T^\infty e^{-t}
    \cdot \Big[ \dj{[0,t]}\Big( \phi_t\big(\tilde\xi_n),\ \phi_t\big(\tilde\xi^\prime_n\big)\Big)  \wedge 1 \Big]dt
    \\ 
    & < \Delta 
    \qquad
    \text{by \eqref{proof: choice of T, lemma: bounded away condition under J1, from infty to T}},
\end{align*}
which clearly contradicts $
\dj{\subInfty}\big(B, \barDxepskT{\bm x}{\epsilon}{\leqslant \bm k}{[0,\infty)} \big) > \Delta.
$
This concludes the proof.
\end{proof}

Now, we are ready to prove Theorem~\ref{corollary: LD for Levy with MRV},
which adapts Theorem~\ref{theorem: LD for levy, MRV} to $\big( \D[0,\infty), \dj{[0,\infty)} \big)$.

\begin{theorem}\label{corollary: LD for Levy with MRV}
\linksinthm{corollary: LD for Levy with MRV}
    Let Assumption~\ref{assumption: ARMV in levy process} hold.
    Let
    $\bm k = (k_{j})_{j \in [d] } \in \mathbb Z_+^{ d} \setminus \{\bm 0\}$,
    and let $B$ be a Borel set of $\D[0,\infty)$ equipped with $J_1$ topology (i.e., induced by $\dj{\subInfty}$).
    Suppose that $B$  is bounded away from $\barDxepskT{\bm\mu_{\bm L}}{\epsilon}{\leqslant \bm k}{[0,\infty)}$ under $\dj{[0,\infty)}$ for some $\epsilon > 0$.
    Then,
    \begin{align}
        \sum_{ \bm{\mathcal K} \in \mathbb A(\bm k) }\breve{\mathbf C}_{\bm{\mathcal K};\bm \mu_{\bm L}}^{ _{[0,\infty)} }(B^\circ)
        \leq 
        \liminf_{n \to \infty}
        \frac{
        \P\big(\bar{\bm L}^{ _{[0,\infty)} }_n \in B\big)
        }{
            \breve \lambda_{\bm k}(n)
        }
        \leq 
        \limsup_{n \to \infty}
        \frac{
        \P\big(\bar{\bm L}^{ _{[0,\infty)} }_n \in B\big)
        }{
            \breve \lambda_{\bm k}(n)
        }
        \leq 
        \sum_{ \bm{\mathcal K} \in \mathbb A(\bm k) }\breve{\mathbf C}_{\bm{\mathcal K};\bm \mu_{\bm L}}^{ _{[0,\infty)} }(B^-) < \infty,
        \label{claim, finite index, corollary: LD for levy, MRV}
    \end{align}
    where
    ${\breve \lambda_{\bm k}(n)}$,
    $\bm \mu_{\bm L}$, 
    $\breve{\mathbf C}_{\bm{\mathcal K};\bm \mu_{\bm L}}^{ _{[0,\infty)} }$,
    and $\barDxepskT{\bm\mu_{\bm L}}{\epsilon}{\leqslant \bm k}{[0,\infty)}$
    are defined in 
    \eqref{def: scale function for Levy LD MRV},
    \eqref{def: expectation of levy process}, 
    \eqref{def: measure C type j L, LD for Levy, MRV},
    and \eqref{def: path with costs less than j jump set, drift x, LD for Levy MRV}, 
     respectively.
\end{theorem}

\begin{proof}\linksinpf{corollary: LD for Levy with MRV}
Part $(ii)$ of Lemma~\ref{lemma: B bounded away under J1 infinity} confirms that 
$\sum_{ \bm{\mathcal K} \in \mathbb A(\bm k) }\breve{\mathbf C}_{ \bm{\mathcal K};\bm \mu_{\bm L}}^{ _{[0,\infty)} } \in \mathbb M\big( \D[0,\infty) \setminus \barDxepskT{\bm \mu_{\bm L}}{\epsilon}{\leqslant \bm k}{[0,\infty)} \big)$
(under $\dj{[0,\infty)}$) for any $\epsilon > 0$,
which verifies the finite upper bound in \eqref{claim, finite index, corollary: LD for levy, MRV}.
Then,
by Theorem~\ref{portmanteau theorem M convergence} (the Portmanteau Theorem for $\mathbb M$-convergence),
it suffices to prove the following claim:
for any $f:\D[0,\infty) \to [0,\infty)$ \emph{uniformly continuous} (w.r.t.\ $(\D[0,\infty),\dj{[0,\infty)})$) and bounded (i.e., $\norm{f} = \sup_{\xi \in \D[0,\infty)} f(\xi) < \infty$)
such that
$B \delequal \text{supp}(f)$ is bounded away from $\barDxepskT{\bm \mu_{\bm L}}{\epsilon}{\leqslant \bm k}{[0,\infty)}$
under $\dj{\subInfty}$ for some $\epsilon > 0$, we have
\begin{align}
    \lim_{n \to \infty}
        \frac{ 
            \E f(\bar{\bm L}_n^{ \subInfty })
        }{\breve \lambda_{\bm k}(n)}
        =
        \sum_{ \bm{\mathcal K} \in \mathbb A(\bm k) }\breve{\mathbf C}_{\bm{\mathcal K};\bm \mu_{\bm L}}^{ _{[0,\infty)} }(f).
        \label{proof, goal 1, corollary: LD for Levy with MRV}
\end{align}
To proceed, for each $t \in [0,\infty)$ we define the mapping $f_t:\D[0,t] \to \R$ by 
(note that in the display~\eqref{proof, def phi t inverse, corollary: LD for Levy with MRV} below, the path $\xi$ is an element of $\D[0,t]$)
\begin{align}
    f_t(\xi) \delequal f\big( \phi^\text{inv}_t(\xi)\big),
    \quad \text{where }
    \phi^\text{inv}_t(\xi)(s)\delequal \xi(s\wedge t) + [0 \vee (s - t)]\cdot \bm\mu_{\bm L} \ \forall s \geq 0.
    \label{proof, def phi t inverse, corollary: LD for Levy with MRV}
\end{align}
That is, $\phi^\text{inv}_t$ extends the path $\xi\in\D[0,t]$ to $\D[0,\infty)$ linearly with slope $\bm\mu_{\bm L}$  over $(t,\infty)$, which, in a way, ``inverts'' the projection mapping $\phi_t$ in \eqref{def: projection mapping from D infty to D T}.
Now, we make a few observations.
\begin{itemize}
    \item 
        Define $\hat \phi_t: \D[0,\infty) \to \D[0,\infty)$ by 
\begin{align*}
    \hat \phi_t(\xi)(s) \delequal \xi(t\wedge s) + [0 \vee (s - t)]\cdot \bm\mu_{\bm L},
    \qquad
    \forall \xi \in \D[0,\infty),\ s \geq 0.
\end{align*}
By the definition of $\dj{[0,\infty)}$ in \eqref{def: J1 metric on [0,infty]},
for any $t > 0$ and $\xi \in \D[0,\infty)$, we have
$
\dj{[0,\infty)}
    \big(
        \hat \phi_t(\xi),\xi
    \big)
    \leq \int_{s > t  } e^{-s}ds = e^{-t}.
$
Then, by the uniform continuity of the function $f$ fixed in \eqref{proof, goal 1, corollary: LD for Levy with MRV}, given $\epsilon > 0$, there exists $T = T(\epsilon) > 0$ such that
\begin{align}
    \big|
        f(\xi) - f\big( \hat\phi_t(\xi) \big)
    \big|
    < \epsilon,
    \qquad\forall t \geq T,\ \xi \in \D[0,\infty).
    \label{proof: ineq 1, corollary: LD for Levy with MRV}
\end{align}

    \item 
        Recall that $B = \text{supp}(f)$, and note that
        \begin{align}
            f(\xi) = 0\text{ and }f\big(\hat \phi_t(\xi)\big) = 0,
            \qquad \forall t > 0,\ \xi \in \D[0,\infty)\text{ such that }\phi_t(\xi) \notin \phi_t(B).
            \label{proof: property 2, corollary: LD for Levy with MRV}
        \end{align}

    \item 
        Combining \eqref{proof: ineq 1, corollary: LD for Levy with MRV} and \eqref{proof: property 2, corollary: LD for Levy with MRV},
        we know that given $\epsilon > 0$, there is $T = T(\epsilon) \in (0,\infty)$ such that
        \begin{align}
            \big|
                f(\xi) - f\big( \hat\phi_t(\xi) \big)
            \big|
            \leq \epsilon \cdot \mathbbm{I}\big\{ \phi_t(\xi) \in \phi_t(B)  \big\},
            \qquad\forall t \geq T,\ \xi \in \D[0,\infty).
            \label{proof, observation about phi t, implication, corollary: LD for Levy with MRV}
        \end{align}

    \item 
        Recall that in \eqref{proof, def phi t inverse, corollary: LD for Levy with MRV}, we use $\xi$ to denote an element of $\D[0,t]$, which differs from our convention henceforth that $\xi \in \D[0,\infty)$.
        By definitions in \eqref{proof, def phi t inverse, corollary: LD for Levy with MRV},
        it holds for any $t > 0$ and $\xi \in \D[0,\infty)$ that
        $
        f_t\big(\phi_t(\xi)\big) = f\big(\hat\phi_t(\xi)\big).
        $
        Then, by \eqref{proof, observation about phi t, implication, corollary: LD for Levy with MRV}, given $\epsilon > 0$ there exists $T = T(\epsilon) > 0$ such that 
        for any $t \geq T$,
        \begin{align}
        \limsup_{n \to \infty}\frac{ 
            \E f(\bar{\bm L}_n^{ \subInfty })
        }{\breve \lambda_{\bm k}(n)}
        & \leq 
        \lim_{n \to \infty} \frac{\E f_t(\bar{\bm L}_n^{ \subZeroT{t} })
        }{\breve \lambda_{\bm k}(n)}
        +
        \epsilon \cdot 
        \limsup_{n \to \infty} \frac{\P(\bar{\bm L}_n^{ \subZeroT{t}} \in \phi_t(B) )
        }{\breve \lambda_{\bm k}(n)},
        \label{proof, UB for goal 1, corollary: LD for Levy with MRV}
        \\
        \liminf_{n \to \infty}\frac{ 
            \E f(\bar{\bm L}_n^{ \subInfty })
        }{\breve \lambda_{\bm k}(n)}
        & \geq 
        \lim_{n \to \infty} \frac{\E f_t(\bar{\bm L}_n^{ \subZeroT{t} })
        }{\breve \lambda_{\bm k}(n)}
        -
        \epsilon \cdot 
        \limsup_{n \to \infty} \frac{\P(\bar{\bm L}_n^{ \subZeroT{t}} \in \phi_t(B) )
        }{\breve \lambda_{\bm k}(n)}.
        \label{proof, LB for goal 1, corollary: LD for Levy with MRV}
        \end{align}
\end{itemize}
Suppose we can show that
\begin{align}
    & \lim_{t \to \infty}\lim_{n \to \infty} \frac{\E f_t(\bar{\bm L}_n^{ \subZeroT{t} })
        }{\breve \lambda_{\bm k}(n)} 
        = 
        \sum_{ \bm{\mathcal K} \in \mathbb A(\bm k) }\breve{\mathbf C}_{\bm{\mathcal K};\bm \mu_{\bm L}}^{ _{[0,\infty)} }(f),
        \label{proof, intermediate goal 1,  corollary: LD for Levy with MRV}
    \\ 
        & \limsup_{t \to \infty}
    \limsup_{n \to \infty} \frac{\P(\bar{\bm L}_n^{ \subZeroT{t}} \in \phi_t(B) )
        }{\breve \lambda_{\bm k}(n)} < \infty.
        \label{proof, intermediate goal 2,  corollary: LD for Levy with MRV}
\end{align}
Then, by plugging these results into  \eqref{proof, UB for goal 1, corollary: LD for Levy with MRV}--\eqref{proof, LB for goal 1, corollary: LD for Levy with MRV} and sending $t \to \infty$, $\epsilon\downarrow 0$,
we conclude the proof of claim~\eqref{proof, goal 1, corollary: LD for Levy with MRV}.
Now, it only remains to prove claims~\eqref{proof, intermediate goal 1,  corollary: LD for Levy with MRV} and \eqref{proof, intermediate goal 2,  corollary: LD for Levy with MRV}.

\medskip
\noindent
\textbf{Proof of Claim \eqref{proof, intermediate goal 1,  corollary: LD for Levy with MRV}}.
For any $t > 0$, the function $f_t$ is continuous and bounded.
Besides, note that 
$f_t(\xi) = 0$ whenever $\xi \notin \phi_t(B)$,
where $B = \text{supp}(f)$.
On the other hand,
 Lemma~\ref{lemma: bounded away condition under J1, from infty to T} shows that $\phi_t(B)$ is bounded away from 
$\barDxepskT{\bm\mu_{\bm L}}{\epsilon}{\leqslant \bm k}{[0,t]}$ under $\dj{[0,t]}$ for any $t > 0$ large enough.
This allows us to apply Theorem~\ref{theorem: LD for levy, MRV} and Theorem~\ref{portmanteau theorem M convergence}, the Portmanteau Theorem for $\mathbb M$-convergence, to show that
\begin{align}
    \lim_{n \to \infty} \frac{\E f_t(\bar{\bm L}_n^{ \subZeroT{t} })
        }{\breve \lambda_{\bm k}(n)}
        =
        \sum_{ \bm{\mathcal K} \in \mathbb A(\bm k) }\breve{\mathbf C}_{\bm{\mathcal K};\bm\mu_{\bm L}}^{ _{[0,t]} }(f_t),
        \qquad\forall t > 0 \text{ sufficiently large.}
    \nonumber
\end{align}
Also, note that $|\mathbb A(\bm k)| < \infty$ (i.e., the set contains only finitely many elements).
Therefore, to prove Claim~\eqref{proof, intermediate goal 1,  corollary: LD for Levy with MRV},
it suffices to fix some $\bm{\mathcal K} = (\mathcal K_{\bm j})_{\bm j \in \powersetTilde{d}} \in \mathbb A(\bm k)$
and show that
$
\breve{\mathbf C}_{\bm{\mathcal K};\bm\mu_{\bm L}}^{ _{[0,t]} }(f_t)
    =
    \breve{\mathbf C}_{\bm{\mathcal K};\bm\mu_{\bm L}}^{ _{[0,\infty)} }(f)
$
for any $t > 0$ sufficiently large.
Now,
by the definition of $f_t$ in \eqref{proof, def phi t inverse, corollary: LD for Levy with MRV} and 
$
\breve{ \mathbf C }_{ \bm{\mathcal K};\bm \mu_{\bm L} }^{ \subZeroT{t} }
$
in \eqref{def: measure C type j L, LD for Levy, MRV},
\begin{align}
    \breve{ \mathbf C }_{ \bm{\mathcal K};\bm\mu_{\bm L} }^{ \subZeroT{t} }(f_t)
    & = 
    \frac{1}{\prod_{ \bm j \in \powersetTilde{d} } \mathcal K_{\bm j}!} 
    \cdot
    \int 
    f\Bigg(\phi^\text{inv}_t\bigg(
        \bm\mu_{\bm L}\bm 1  + \sum_{ \bm j \in \powersetTilde{d} } \sum_{ k \in [\mathcal K_{\bm j}] } \bm w_{\bm j,k}\mathbbm{I}_{ [t_{\bm j, k},t] } 
    \bigg)\Bigg)
    \label{proof, expression for measure breve C mathcal K, corollary: LD for Levy with MRV}
    \\ 
    & \qquad\qquad\qquad\qquad\qquad\qquad\qquad
    \bigtimes_{\bm j \in \powersetTilde{d}}
    \bigtimes_{ k \in [\mathcal K_{\bm j}] }\bigg( (\mathbf C_{ \bm j } \times  \mathcal L_{(0,t)})\Big(d (\bm w_{\bm j, k}, t_{ \bm j,k })\Big)\bigg)
    \nonumber
    \\
    & = 
    \frac{1}{\prod_{ \bm j \in \powersetTilde{d} } \mathcal K_{\bm j}!} 
    \cdot
    \int 
    f\Bigg(\phi^\text{inv}_t\bigg(
        \bm \mu_{\bm L}\bm 1  + \sum_{ \bm j \in \powersetTilde{d} } \sum_{ k \in [\mathcal K_{\bm j}] } \bm w_{\bm j,k}\mathbbm{I}_{ [t_{\bm j, k},t] } 
    \bigg)\Bigg)
    \nonumber
    \\ 
    & \qquad
    \cdot \mathbbm{I}\Big\{
        t_{\bm j,k} <t\ \forall \bm j \in \powersetTilde{d},\ k \in [\mathcal K_{\bm j}]
    \Big\}
    \bigtimes_{\bm j \in \powersetTilde{d}}
    \bigtimes_{ k \in [\mathcal K_{\bm j}] }\bigg( (\mathbf C_{ \bm j } \times  \mathcal L_{(0,\infty)})\Big(d (\bm w_{\bm j, k}, t_{ \bm j,k })\Big)\bigg),
    \nonumber
\end{align}
where $\mathcal L_I$ is the Lebesgue measure restricted on interval $I$.
Furthermore, for $B = \text{supp}(f)$, by part $(i)$ of Lemma~\ref{lemma: B bounded away under J1 infinity},
there exists some $T \in (0,\infty)$ such that the following claim holds:
for any piece-wise linear function $\xi \in \D[0,\infty)$ of the form
$
\xi(t) =
    \bm \mu_{\bm L}t
    + 
    \sum_{ \bm j \in \powersetTilde{d} }\sum_{ k = 1 }^{ \mathcal K_{\bm j} }\bm w_{\bm j, k} 
    \mathbbm{I}_{ [t_{\bm j, k},\infty) }(t)
$
with $\bm w_{\bm j,k} \in \R^d(\bm j)$ for each $\bm j \in \powersetTilde{d},\ k \in [\mathcal K_{\bm j}]$,
we have 
\begin{align}
    \xi \in B
    \quad\Longrightarrow\quad
    t_{\bm j,k} < T\ \forall \bm j \in \powersetTilde{d},\ k \in [\mathcal K_{\bm j}].
    \label{proof, intermediate goal 1, choice of T, corollary: LD for Levy with MRV}
\end{align}
Therefore,
provided that $\bm w_{\bm j,k} \in \R^d(\bm j)$ for each $\bm j \in \powersetTilde{d}$ and $k \in [\mathcal K_{\bm j}]$,
given any $t > T$ we have
\begin{align*}
    & f\Bigg(\phi^\text{inv}_t\bigg(
        \bm\mu_{\bm L}\bm 1  + \sum_{ \bm j \in \powersetTilde{d} } \sum_{ k \in [\mathcal K_{\bm j}] } \bm w_{\bm j,k}\mathbbm{I}_{ [t_{\bm j, k},t] } 
    \bigg)\Bigg)
    \cdot \mathbbm{I}\Big\{
        t_{\bm j,k} <t\ \forall \bm j \in \powersetTilde{d},\ k \in [\mathcal K_{\bm j}]
    \Big\}
    \\ 
    & = 
    f\Bigg(
        \bm\mu_{\bm L}\bm 1  + \sum_{ \bm j \in \powersetTilde{d} } \sum_{ k \in [\mathcal K_{\bm j}] } \bm w_{\bm j,k}\mathbbm{I}_{ [t_{\bm j, k},\infty) } 
    \Bigg)
    \cdot \mathbbm{I}\Big\{
        t_{\bm j,k} <t\ \forall \bm j \in \powersetTilde{d},\ k \in [\mathcal K_{\bm j}]
    \Big\}
    \quad 
    \text{by \eqref{proof, def phi t inverse, corollary: LD for Levy with MRV},}
    \\
    & = 
    f\Bigg(
        \bm\mu_{\bm L}\bm 1  + \sum_{ \bm j \in \powersetTilde{d} } \sum_{ k \in [\mathcal K_{\bm j}] } \bm w_{\bm j,k}\mathbbm{I}_{ [t_{\bm j, k},\infty) } 
    \Bigg)
    \qquad\text{due to \eqref{proof, intermediate goal 1, choice of T, corollary: LD for Levy with MRV} and $B = \text{supp}(f)$}.
\end{align*}
Also,
by the $\MRV$ condition of Assumption~\ref{assumption: ARMV in levy process},
for each $\bm j \in \powersetTilde{d}$ the measure $\mathbf C_{\bm j}(\cdot)$ is supported on $\R^d(\bm j)$ (see also Section~\ref{subsec: MRV}).
Then, in \eqref{proof, expression for measure breve C mathcal K, corollary: LD for Levy with MRV}, it holds for any $t > T$ that 
\begin{align*}
    & \breve{ \mathbf C }_{ \bm{\mathcal K};\bm\mu_{\bm L} }^{ \subZeroT{t} }(f_t)
    \\
    & = 
    \frac{1}{\prod_{ \bm j \in \powersetTilde{d} } \mathcal K_{\bm j}!} 
    \cdot
    \int 
    f\Bigg(
        \bm \mu_{\bm L}\bm 1  + \sum_{ \bm j \in \powersetTilde{d} } \sum_{ k \in [\mathcal K_{\bm j}] } \bm w_{\bm j,k}\mathbbm{I}_{ [t_{\bm j, k},\infty) } 
    \Bigg)
    \bigtimes_{\bm j \in \powersetTilde{d}}
    \bigtimes_{ k \in [\mathcal K_{\bm j}] }\bigg( (\mathbf C_{ \bm j } \times  \mathcal L_{(0,\infty)})\Big(d (\bm w_{\bm j, k}, t_{ \bm j,k })\Big)\bigg)
    \\ 
    & = 
    \breve{ \mathbf C }_{ \bm{\mathcal K};\bm\mu_{\bm L} }^{ \subInfty }(f).
\end{align*}
This concludes the proof of Claim \eqref{proof, intermediate goal 1,  corollary: LD for Levy with MRV}.

\medskip
\noindent
\textbf{Proof of Claim \eqref{proof, intermediate goal 2,  corollary: LD for Levy with MRV}}.
By the definition of the projection mapping $\phi_t$ in \eqref{def: projection mapping from D infty to D T},
for any $t^\prime > t > 0$ we have
$
\big\{ \bar{\bm L}_n^{ \subZeroT{t} } \in \phi_t(B)  \big\}
    \supseteq
    \big\{ \bar{\bm L}_n^{ \subZeroT{t^\prime} } \in \phi_{t^\prime}(B)  \big\}.
$
This monotonicity w.r.t.\ $t$
implies that, to prove Claim~\eqref{proof, intermediate goal 2,  corollary: LD for Levy with MRV},
it suffices to find $T \in (0,\infty)$ such that 
\begin{align}
    \limsup_{n \to \infty} \frac{\P(\bar{\bm L}_n^{ \subZeroT{T}} \in \phi_T(B) )
        }{\breve \lambda_{\bm k}(n)} & < \infty.
    \label{proof, intermediate goal 2, sub goal, corollary: LD for Levy with MRV}
\end{align}
Since $B = \text{supp}(f)$ is bounded away from $\barDxepskT{\bm\mu_{\bm L}}{\epsilon}{\leqslant \bm k}{[0,\infty)}$
under $\dj{\subInfty}$ for some $\epsilon > 0$,
Lemma~\ref{lemma: bounded away condition under J1, from infty to T} shows that $\phi_T(B)$ is bounded away from 
$\barDxepskT{\bm\mu_{\bm L}}{\epsilon}{\leqslant \bm k}{[0,T]}$ under $\dj{[0,T]}$ for any $T > 0$ large enough.
Claim~\eqref{proof, intermediate goal 2, sub goal, corollary: LD for Levy with MRV} then follows from the finite upper bound in Claim~\eqref{claim, finite index, theorem: LD for levy, MRV} of Theorem~\ref{theorem: LD for levy, MRV}.
\end{proof}

\subsection{Proof of Proposition~\ref{proposition, law of the levy process approximation, LD for Hawkes}}
\label{subsec: proof for propositions, LD for Hawkes, appendix}

To rigorously explain the coupling between the Hawkes process $\bm N(t)$ and the compound Poisson process $\bm L(t)$ in \eqref{def: Levy process, LD for Hawkes},
we first review the cluster representation of Hawkes processes,
which reveals the underlying branching structure in $\bm N(t)$.
Intuitively speaking, the cluster representation shows that the Hawkes process $\bm N(t)$ can be constructed by first generating a sequence of ``immigrants'' (i.e., points of the $0^\text{th}$ generation) arriving at Poisson rates $c^{\bm N}_i$
and then, iteratively, letting each generation of points to give birth to the next generation of descendants.
In particular, any type-$j$ point gives birth to type-$i$ points
according to an inhomogeneous Poisson processes with rates determined by 
$
\tilde B_{i \leftarrow j}f^{\bm N}_{i \leftarrow j}(\cdot).
$
This approach was first introduced in \cite{209288c5-6a29-3263-8490-c192a9031603} in the univariate setting. 
See also Chapter 4 of \cite{20.500.11850/151886} and Examples 6.3(c) and 6.4(c) of \cite{daley2003introduction} for a treatment in the multivariate setting.

Appealing to the more general framework of Poisson cluster processes,
we start by constructing a point process that identifies the \textit{centers}  of each
cluster, and then augment each center with a separate point process representing the arrival times and types of the offspring in this cluster.
First, we define a point process in the space $\bm N^{\mathcal C}(\cdot)$ in $[0,\infty) \times [d]$, where the superscript $\mathcal C$ denotes ``center''. 
Recall the positive immigration rates $c^{\bm N}_{i}$ in \eqref{def: conditional intensity, hawkes process}.
Independently for each $j \in [d]$, generate
$
0 < T^{\mathcal C}_{i;1} < T^{\mathcal C}_{i;2} < \ldots 
$
under a Poisson process with rate  $c^{\bm N}_{i}$.
Let 
\begin{align}
    \bm N^{\mathcal C}(\cdot)
    \delequal 
    \sum_{j \in [d]}\sum_{k \geq 0}\bm \delta_{ (T^\mathcal{C}_{j;k},j)  }(\cdot),
    \label{def: center process, Poisson cluster process, notation}
\end{align}
where $\notationdef{notation-dirac-measure}{\bm \delta_x}$ denotes the Dirac measure at $x$.
This is equivalent to the superposition of several independent Poisson processes, where each arrival $T^\mathcal{C}_{j;k}$ is augmented with a marker $j$ to denote its type.

Next, we associate each point $(T^\mathcal{C}_{j;k},j)$ in the center process with an \emph{offspring} process, which is a point process in the space $[0,\infty) \times [d]$
(we use the superscript $\mathcal O$ to denote ``offspring''):
\begin{align}
    \bm N^{\mathcal O}_{ (T^\mathcal{C}_{j;k},j) }(\cdot) 
    = 
    \sum_{ m = 0 }^{ K^{\mathcal O}{ (T^\mathcal{C}_{j;k},j) }  }
    \bm \delta_{ (T^{\mathcal O}_{ j;k }(m), A^{\mathcal O}_{ j;k }(m))  }(\cdot).
    \label{def: cluster process associated with each center, Poisson cluster process, notation}
\end{align}
In particular, each $\bm N^{\mathcal O}_{ (T^\mathcal{C}_{j;k},j) }(\cdot) $ is an independent copy of a point process  $C^{\mathcal O}_j$, whose law is described by the following iterative procedure.


\begin{definition}[Offspring Cluster Process $C^{\mathcal{ O}}_j$]
\label{def: offspring cluster process}
For each $j \in [d]$, the point process $C^{\mathcal{ O}}_j(\cdot)$ in the space $[0,\infty) \times [d]$ is defined as follows.
\begin{enumerate}[(i)]
    \item
        (\textbf{Ancestor Immigrant})
        In this procedure, we use $R^{(n);j}_i$ to denote the number of type-$i$ individuals in the $n^\text{th}$ generation of the cluster $C^{\mathcal O}_j$, 
        and $T^{(n);j}_i(k)$ for the time between the arrival of the ancestor of the cluster and the birth of the $k^\text{th}$ type-$i$ individual in the $n^\text{th}$ generation.
        Specifically,
        let $R^{(0);j}_j = 1$ and $R^{(0);j}_i = 0$ for each $i \in [d],\ i \neq j$,
        which represents that the type-$j$ ancestor is the only member in the $0^\text{th}$ generation of this cluster.
        Besides,
        we set $T^{(0);j}_j(1) = 0$.

    \item
        (\textbf{Offspring in the $n^\text{th}$ generation})
        Iteratively,
        do the following for $n = 1,2,\ldots$,
        until we arrive at some $n$ such that $R^{(n-1);j}_i = 0$ for any $i \in [d]$.
        For any $l \in [d]$ with $R^{(n-1);j}_l \geq 1$, and any $m \in [R^{(n-1);j}_l ]$,
        independently for each $i \in [d]$, generate the sequence
        $
        \big( T^{(n,m);j}_{i \leftarrow l}(k)  \big)_{k}
        $
        according to an inhomogeneous Poisson process with rate
        \begin{align}
            \tilde B_{i \leftarrow l}f^{\bm N}_{i \leftarrow l}\big( \ \bcdot\ - T^{(n-1);j}_{l}(m) \big).
            \label{def: non-homogeneous poisson for arrival times in n th generation, cluster representation}
        \end{align}
        Here, by sampling under the rate in \eqref{def: non-homogeneous poisson for arrival times in n th generation, cluster representation}, we mean the following.
        \begin{itemize}
            \item 
                First, generate an independent copy of $\tilde B_{i \leftarrow l}$.
                Conditioning on $\tilde B_{i \leftarrow l} = b$ for some $b \geq 0$,
                generate $B^{(n,m);j}_{i \leftarrow l}$---the count of type-$i$ children (in the $n^\text{th}$ generation) born by the $m^\text{th}$ type-$l$ individual in the $(n-1)^\text{th} $generation---by
                $
                B^{(n,m);j}_{i \leftarrow l} \sim 
                \text{Poisson}\big(b \cdot \mu^{\bm N}_{i \leftarrow l}\big),
                $
                with $\mu^{\bm N}_{i \leftarrow l}$ defined in \eqref{def mu i j for fertility function}.

            \item 
                Next, let $t_k$ be i.i.d.\ copies under the law with density $f^{\bm N}_{i\leftarrow l}(\cdot)/\mu^{\bm N}_{i \leftarrow l}$.
                Let $T^{(n,m);j}_{i \leftarrow l}(k) = t_k + T^{(n-1);j}_l(m)$ for each 
                $k \in [B^{(n,m);j}_{i \leftarrow j}]$.
        \end{itemize}
        Then, given $i \in [d]$, by
        ordering the arrival times $T^{(n,m);j}_{i \leftarrow l}(k)$ 
        across all $l \in [d]$, $m \in [R^{(n-1);j}_l ]$, and $k \in [ B^{(n,m);j}_{i \leftarrow l} ]$,
        we get
        $
        0 < T^{(n);j}_{i}(1) < T^{(n);j}_{i}(2) < \ldots < T^{(n);j}_{i}\big( R^{ (n);j }_{ i } \big),
        $
        where we use $R^{ (n);j}_{ i }$ to denote the count of type-$i$ individuals in the $n^\text{th}$ generation.
        Note that this strictly increasing sequence is almost surely well-defined (i.e., with no ties) 
        since the law of $T^{(n,m);j}_{i \leftarrow l}(k)$ 
        is absolutely continuous w.r.t.\ the Lebesgue measure.

    \item
        (\textbf{Definition of the cluster process})
        Let $K^{\mathcal O}_j \delequal \max\{n \geq 0:\ R^{(n);j}_i \geq 1\text{ for some }i \in [d]  \}$
        be the count of generations in this cluster.
        W.l.o.g.\ we set $R^{(n);j}_i = 0$ for any $n \geq K^{\mathcal O}_j + 1$ and $i \in [d]$.
        The cluster process is defined by
        \begin{align}
            C^\mathcal{O}_j(\cdot)
            \delequal 
            \sum_{n = 0}^{ K^\mathcal{O}_j }
            \sum_{i \in [d]}\sum_{  m = 1 }^{ R^{(n);j}_i }
            \bm \delta_{  (T^{(n);j}_{i}(m),i) }(\cdot).
            \label{def: cluster process, C mathcal O j}
        \end{align}
\end{enumerate}
\end{definition}

 In this paper, we work with conditions (see Section~\ref{subsec: tail of cluster S, statement of results}) which ensure that there are almost surely finitely many points in $C^{\mathcal O}_j$,
and hence the count of points $K^{\mathcal O}{ (T^\mathcal{C}_{j;k},j) }$ is almost surely finite in \eqref{def: cluster process associated with each center, Poisson cluster process, notation}.

Define
 the time-shift operator $\bm \theta_t$ by 
 $
 \bm \theta_t \mu(\cdot) \delequal \sum_{k \geq 1}\bm \delta_{ (t_k + t,a_k) }(\cdot)
 $
 for any $t \geq 0$ and point process $\mu(\cdot) = \sum_{k \geq 1}\bm \delta_{ (t_k,a_k) }$.
Augmenting the center process $\bm N^{\mathcal C}$ with each offspring cluster $\bm N^{\mathcal O}_{ (T^\mathcal{C}_{j;k},j) }$,
we obtain a point process
\begin{align}
    \bm N(\cdot) 
    \delequal
    \sum_{j \in [d]}\sum_{k \geq 0}
    \bm \theta_{  T^{\mathcal C}_{j;k}  } \bm N^{\mathcal O}_{ (T^\mathcal{C}_{j;k},j) }(\cdot) 
    =
    \sum_{j \in [d]}\sum_{k \geq 0}\sum_{ m = 0 }^{ K^{\mathcal O}{ (T^\mathcal{C}_{j;k},j) }  }
    \bm \delta_{ ( T^\mathcal{C}_{j;k} + T^{\mathcal O}_{ j;k }(m),  A^{\mathcal O}_{ j;k }(m))   }(\cdot).
    \label{def: branching process approach, N as a point process, notation section}
\end{align} 
For each $t \geq 0$ and $i \in [d]$, we define the counting process
\begin{align}
    N_i(t) \delequal 
    \bm N\big( [0,t] \times \{i\} \big)
    = 
    \sum_{j \in [d]}\sum_{k \geq 0}\sum_{ m = 0 }^{ K^{\mathcal O}{ (T^\mathcal{C}_{j;k},j) }  }
    \mathbbm{I}
    \bigg\{
        T^\mathcal{C}_{j;k} + T^{\mathcal O}_{ j;k }(m) \leq t,\ A^{\mathcal O}_{ j;k }(m) = i
    \bigg\}.
    \label{def: branching process approach, N i t, notation section}
\end{align}
Under the sub-criticality condition regarding the offspring distributions (i.e., Assumption~\ref{assumption: subcriticality}),
it has been shown in \cite{209288c5-6a29-3263-8490-c192a9031603,daley2003introduction} that the definitions in 
\eqref{def: center process, Poisson cluster process, notation}--\eqref{def: branching process approach, N i t, notation section}
using the cluster representation agree with Definition~\ref{def: hawkes process, conditional intensity}.

We formally define the size of the offspring cluster processes $C^{\mathcal O}_j$.
Using notations in Definition~\ref{def: offspring cluster process},
we define (for each $j \in [d]$)
\begin{align}
    \notationdef{notation-cluster-size-S-i-j}{S_{i \leftarrow j}} \delequal 
    \sum_{n \geq 0}R^{(n);j}_i,
    \qquad
    \bm S_j \delequal (S_{1\leftarrow j}, S_{2\leftarrow j}, \ldots, S_{d\leftarrow j})^\top.
    \label{def: cluster size vector S i, 2}
\end{align}
That is, 
$\bm S_j$ is the size vector of the offspring cluster process $\mathcal C^{\mathcal O}_j(\cdot)$ induced by a type-$j$ ancestor,
with each element $S_{i \leftarrow j}$ representing the count of type-$i$ points in the cluster.
Under the law of $C^{\mathcal O}_j(\cdot)$ specified above,
the vectors $\bm S_j$ solve the distributional fixed-point equations \eqref{def: fixed point equation for cluster S i},
under the offspring distributions $(B_{i \leftarrow j})_{i \in [d]}$ stated in \eqref{def: vector B i, offspring of type i individual}.

Meanwhile,
using notations in \eqref{def: cluster process associated with each center, Poisson cluster process, notation},
we denote the size vector of cluster $\bm N^{\mathcal O}_{ (T^\mathcal{C}_{j;k},j) }$ by
\begin{align}
    S^{(k)}_{i \leftarrow j} & \delequal 
    \sum_{m \geq 0}\mathbbm{I}\Big\{ A^\mathcal{O}_{j;k}(m) = i \Big\},
    \qquad
    \bm S^{(k)}_j \delequal \Big( S^{(k)}_{1 \leftarrow j}, S^{(k)}_{2 \leftarrow j}, \ldots, S^{(k)}_{d \leftarrow j}\Big)^\top. 
    \label{def: cluster size jth type k cluster, LD for Hawkes}
\end{align}
By definition, $\big(\bm S^{(k)}_j\big)_{k \geq 1}$ are i.i.d.\ copies of the $\bm S_j$ defined in \eqref{def: cluster size vector S i, 2},
and hence exhibit the $\MRV^*$ tail asymptotics characterized in Theorem~\ref{theorem: main result, cluster size}.

In the construction of $\bm L(t)$ in \eqref{def: Levy process, LD for Hawkes},
let $\bm S^{(k)}_j$ be defined as in \eqref{def: cluster size jth type k cluster, LD for Hawkes}.
Now,
Proposition~\ref{proposition, law of the levy process approximation, LD for Hawkes} is an immediate consequence of the construction of the compound Poisson process $\bm L(t)$ in \eqref{def: Levy process, LD for Hawkes} and the tail asymptotics of Hawkes process clusters established in Theorem~\ref{theorem: main result, cluster size}.

\begin{proof}[Proof of Proposition~\ref{proposition, law of the levy process approximation, LD for Hawkes}]
\linksinpf{proposition, law of the levy process approximation, LD for Hawkes}
Note that $\bm L(t)$ is the superposition of a sequence of independent compound Poisson processes.
That is, $\bm L(t) = \sum_{j \in [d]}\bm L_{\bcdot \leftarrow j}(t)$, where
\begin{align*}
    \bm L_{\bcdot \leftarrow j}(t) \delequal \sum_{k \geq 0}
    \bm S^{(k)}_j\mathbbm{I}_{ [T^\mathcal{C}_{j;k},\infty) }(t),
    \qquad\forall j \in [d],
\end{align*}
the $(T^\mathcal{C}_{j;k})_{k \geq 1}$ is a sequence generated by a Poisson process on $(0,\infty)$ with a constant rate $c_{j}^{\bm N}$, and $(\bm S^{(k)}_j)_{k \geq 1}$ are i.i.d.\ copies of $\bm S_j$.
In other words,
$\bm L_{\bcdot \leftarrow j}(t)$ is a L\'evy process with generating triplet $(\bm 0, \textbf 0, \nu_j)$,
where the L\'evy measure is 
$
\nu_j(\cdot) = c^{\bm N}_{j}\cdot \P(\bm S_j \in \ \cdot\ ).
$
Therefore, Claim~\eqref{claim, 1, proposition, law of the levy process approximation, LD for Hawkes} follows from $\bm L(t) = \sum_{j \in [d]}\bm L_{\bcdot \leftarrow j}(t)$.
Next, by Theorem~\ref{theorem: main result, cluster size},
\begin{align*}
    \nu_i \in \MRV^*
    \Bigg(
        (\bar{\bm s}_j)_{j \in [d]},\ 
        \big(\alpha(\bm j)\big)_{ \bm j \subseteq [d] },\ 
        (\lambda_{\bm j})_{\bm j \in \powersetTilde{d}},\ 
        \Big( c_{i}^{\bm N}\cdot  
        \mathbf C^{\bm j}_i
        \Big)_{\bm j \in \powersetTilde{d}}
    \Bigg),
    \qquad
    \forall i \in [d].
\end{align*}
Then, the $\MRV^*$ tail condition of $\nu = \sum_{i \in [d]}\nu_i$ claimed in \eqref{claim, 2, proposition, law of the levy process approximation, LD for Hawkes}
follows from  Definition~\ref{def: MRV} for $\MRV^*$,
as well as the definition of
$\mathbf C_{\bm j}=\sum_{i \in [d]}c^{\bm N}_{i} \cdot \mathbf C^{\bm j}_i$ in \eqref{def: measure C indices j, sample path LD for Hawkes}.
\end{proof}

\subsection{Proof of Proposition~\ref{proposition: asymptotic equivalence between Hawkes and Levy, LD for Hawkes}}
\label{subsec: proof for propositions, 2, LD for Hawkes, appendix}

To prove Proposition~\ref{proposition: asymptotic equivalence between Hawkes and Levy, LD for Hawkes}, we prepare a few lemmas.
First, we decompose $\bm L(t)$ in \eqref{def: Levy process, LD for Hawkes} into processes constructed by ``large'' clusters and ``small'' clusters separately. 
For each $\delta,T > 0$ and $n \geq 1$, let
\begin{align}
    \bar{\bm L}^{ >\delta}_n(t)
    & \delequal 
    \frac{1}{n}
    \sum_{ j \in [d] }\sum_{k \geq 0}
    \bm S^{(k)}_j
    \mathbbm{I}\bigg\{ \norm{\bm S^{(k)}_j } >n\delta  \bigg\}\cdot 
    \mathbbm{I}_{ [T^\mathcal{C}_{j;k},\infty) }(nt),
    \label{def: large jump process for Levy, LD for Hawkes}
    \\
    \bar{\bm L}^{ \leqslant\delta}_n(t)
    & \delequal 
    \bar{\bm L}_n(t) - \bar{\bm L}^{ >\delta}_n(t) = 
    \frac{1}{n}
    \sum_{ j \in [d] }\sum_{k \geq 0}
    \bm S^{(k)}_j
    \mathbbm{I}\bigg\{ \norm{\bm S^{(k)}_j } \leq n\delta  \bigg\}\cdot 
    \mathbbm{I}_{ [T^\mathcal{C}_{j;k},\infty) }(nt).
    \label{def: small jump process for Levy, LD for Hawkes}
\end{align}
We also define the (scaled) sample paths
\begin{align}
    \bar{\bm L}^{ >\delta;\subZeroT{T}}_n
    \delequal
    \big\{
        \bar{\bm L}^{ >\delta}_n(t):\ t \in [0,T]
    \big\},\qquad 
    \bar{\bm L}^{ >\delta;\subInfty}_n
    \delequal
    \big\{
        \bar{\bm L}^{ >\delta}_n(t):\ t \geq 0
    \big\},
    \label{def: sample path scaled process, Levy, LD for Hawkes}
\end{align}
and adopt notations $\bar{\bm L}^{ \leqslant\delta;\subZeroT{T}}_n$ and $\bar{\bm L}^{ \leqslant\delta;\subInfty}_n$ analogously.
Recall that 
$\mathbf 1(t) = t$ is the linear function with slope $1$, and that
$ \E \bm L(1) =  \bm \mu_{\bm N}$ (see \eqref{def: approximation, mean of N(t)}).
The next lemma establishes the asymptotic equivalence between $\bar{\bm L}_n^{\subInfty}$ and $\bar{\bm L}^{>\delta;\subInfty}_n + \bm \mu_{\bm N}\mathbf 1$ under $\dmp{\subInfty}$.

\begin{lemma}\label{lemma: asymptotic equivalence between L and large jump, LD for Hawkes}
\linksinthm{lemma: asymptotic equivalence between L and large jump, LD for Hawkes}
    Let Assumptions~\ref{assumption: subcriticality}--\ref{assumption: regularity condition 2, cluster size, July 2024} hold.
    For any $\Delta,\beta > 0$,
    \begin{align}
        \lim_{n \to \infty} n^\beta \cdot 
        \P\bigg(
            \dmp{\subInfty}\Big(
                \bar{\bm L}_n^{\subInfty},\ \bar{\bm L}^{>\delta;\subInfty}_n + \bm \mu_{\bm N}\mathbf 1
            \Big) > \Delta
        \bigg) = 0,
        \qquad
        \forall \delta > 0\text{ small enough.}
        \nonumber
    \end{align}
\end{lemma}

\begin{proof}\linksinpf{lemma: asymptotic equivalence between L and large jump, LD for Hawkes}
Fix $T > 0$ large enough such that $e^{-T} < \Delta/2$.
Note that
\begin{align*}
    \int_{t > T}e^{-t}
    \cdot \Big[ \dmp{[0,t]}\Big( \phi_t\big(\xi^{(1)}),\ \phi_t\big(\xi^{(2)}\big)\Big)  \wedge 1 \Big]dt
    \leq e^{-T} < \Delta/2,
    \qquad
    \forall \xi^{(1)},\xi^{(2)} \in \D[0,\infty).
\end{align*}
Besides , for each $t \in [0,T]$, the metric $\dmp{\subZeroT{t}}$ is bounded by the uniform metric on $\D[0,t]$.
Therefore,
it suffices to show that
\begin{align}
    \P\bigg(
        \sup_{t \in [0,T]}
        \norm{
            \bar{\bm L}_n(t)
            -
            \Big(  \bar{\bm L}_n^{>\delta}(t) + \bm \mu_{\bm N}t   \Big)
        } >\Delta/2   
    \bigg) = \lo(n^{-\beta}),
    \qquad
    \forall \delta > 0\text{ small enough.}
    \label{proof, goal, lemma: asymptotic equivalence between L and large jump, LD for Hawkes}
\end{align}
To proceed, let (for each $k \geq 1$)
\begin{align}
    \tau^{>\delta}_n(k)
    \delequal
    \big\{
        t > \tau^{>\delta}_n(k-1):\ \Delta\bar{\bm L}^{ >\delta }(t) \neq 0
    \big\}
    =
    \big\{
        t > \tau^{>\delta}_n(k-1):\ \norm{\Delta\bar{\bm L}(t)} >\delta
    \big\}
    \label{proof, def tau delta n k, large jump time, lemma: asymptotic equivalence between L and large jump, LD for Hawkes}
\end{align}
be the arrival time of the $k^\text{th}$ large jump in $\bar{\bm L}_n(t)$,
under the convention that $\tau^{>\delta}_n(0) = 0$.
Meanwhile,
for any non-negative integer $K$,
on the event $\big\{ \tau^{>\delta}_n(K+1) > T  \big\}$ (meaning that there are at most $K$ large jumps in $\bar{\bm L}_n(t)$ during $t \in [0,T]$),
note that
\begin{align*}
    & \sup_{t \in [0,T]}
        \norm{
            \bar{\bm L}_n(t)
            -
            \Big(  \bar{\bm L}_n^{>\delta}(t) + \bm\mu_{\bm N}t   \Big)
        } 
    \\ 
    & = 
    \sup_{t \in [0,T]}\norm{\bar{\bm L}^{\leqslant \delta}(t) - \mu_{\bm N}t}
    \quad \text{see \eqref{def: small jump process for Levy, LD for Hawkes}}
    \\ 
    & \leq 
    \sum_{ k = 1  }^{K+1}
    \underbrace{ \sup_{ t \in [ \tau^{>\delta}_n(k-1), \tau^{>\delta}_n(k) ) \cap [0,T]   }
    \norm{
    \bar{\bm L}^{\leqslant\delta}_n(t) - \bm \mu_{\bm N}\cdot\big(t - \tau^{>\delta}_n(k-1)\big) -  
    \bar{\bm L}^{\leqslant\delta}_n\big( \tau^{>\delta}_n(k-1) \big)
    } }_{ \delequal R_k(\delta)  }.
\end{align*}
In the last line of the display above, we applied $\tau^{>\delta}_n(K+1) > T$,
as well as the fact that $ \bar{\bm L}^{\leqslant\delta}_n(t)$ makes no jumps at any $t = \tau^{>\delta}_n(k)$; see the definitions in \eqref{def: small jump process for Levy, LD for Hawkes}.
Therefore, to prove Claim~\eqref{proof, goal, lemma: asymptotic equivalence between L and large jump, LD for Hawkes},
it suffices to show the existence of some $K \in \mathbb N$ such that
\begin{align}
    & \P\Big(
        \tau^{>\delta}_n(K+1) \leq T
    \Big) = \lo(n^{-\beta}),
    \qquad
    \forall \delta > 0\text{ sufficiently small},
    \label{proof, sub goal 1, lemma: asymptotic equivalence between L and large jump, LD for Hawkes}
\end{align}
and that for each $k \in [K+1]$,
\begin{align}
     & \P\Bigg(
    R_k(\delta)
    > \frac{\Delta}{2(K+1)}
    \Bigg) 
    = \lo(n^{-\beta}),
    \qquad
    \forall \delta > 0\text{ sufficiently small}.
    \label{proof, sub goal 2, lemma: asymptotic equivalence between L and large jump, LD for Hawkes}
\end{align}

\medskip
\noindent
\textbf{Proof of Claim \eqref{proof, sub goal 1, lemma: asymptotic equivalence between L and large jump, LD for Hawkes}}.
We prove that, for $K$ large enough, the claim holds for any $\delta > 0$.
In light of Proposition~\ref{proposition, law of the levy process approximation, LD for Hawkes} and the definitions in \eqref{def: large jump process for Levy, LD for Hawkes}--\eqref{def: small jump process for Levy, LD for Hawkes},
$\bar{\bm L}^{ > \delta}(t)$ is a L\'evy process with generating triplet $\Big(\bm 0, \textbf 0, n \cdot \nu_n\big(\ \cdot\ \cap \{ \bm x \in \R^d_+:\ \norm{\bm x} >\delta \}  \big) \Big)$,
where $\nu_n(A) = \nu\{ n\bm x:\ \bm x \in A  \}$, and $\nu$ is defined in \eqref{claim, 1, proposition, law of the levy process approximation, LD for Hawkes}, satisfying the $\MRV^*$ tail condition in \eqref{claim, 2, proposition, law of the levy process approximation, LD for Hawkes}.
Next, let $\alpha^*(\cdot)$ and $\alpha(\cdot)$ be defined as in \eqref{def: cluster size, alpha * l * j}--\eqref{def: cost function, cone, cluster},
and let $j^* \delequal \underset{ j \in [d] }{\arg\min}\ \bm c\big( \{j\} \big)
=
\underset{ j \in [d] }{\arg\min}\ \alpha^*(j).
$
Note that this argument minimum is unique and $\alpha^*(j^*) > 1$ under Assumptions~\ref{assumption: heavy tails in B i j} and \ref{assumption: regularity condition 2, cluster size, July 2024}.
As a result, $\bar\R^d_\leqslant(\{j^*\},\epsilon) = \{\bm 0\}$ for any $\epsilon > 0$.
By \eqref{cond, finite index, def: MRV},
\begin{align}
    \limsup_{n \to \infty}\frac{
        n\nu_n\{ \bm x \in \R^d_+:\ \norm{\bm x} > \delta  \}
    }{
        n\lambda_{\{ j^* \} }(n)
    }
    < \infty,
    \qquad\forall \delta > 0,
    \label{proof, bound 1, sub goal 1, lemma: asymptotic equivalence between L and large jump, LD for Hawkes}
\end{align}
where $\lambda_{\bm j}(n)$ is defined in \eqref{def: rate function lambda j n, cluster size}.
This implies 
$
n\nu_n\{ \bm x \in \R^d_+:\ \norm{\bm x} > \delta  \}
=
\bo\big( n\lambda_{\{ j^* \} }(n) \big)
$
as $n \to \infty$.
Then, for any $K \in \mathbb N$,
\begin{align*}
    \P\big( \tau^{>\delta}_n(K+1) \leq T \big)
    & = 
    \P\Big(
        \text{Poisson}\big( n\nu_n\{ \bm x \in \R^d_+:\ \norm{\bm x} > \delta  \}  \cdot T \big) \geq K+1
    \Big)
    \\ 
    & \leq
    T^{K+1} \cdot \big( n\nu_n\{ \bm x \in \R^d_+:\ \norm{\bm x} > \delta  \} \big)^{K+1}
    \quad 
    \text{due to }
    \P\big( \text{Poisson}(\lambda) \geq k \big) \leq \lambda^k
    \\ 
    & = 
    \bo\Big( \big( n \lambda_{ \{j^*\} }(n) \big)^{K+1}  \Big)
    \qquad\text{by \eqref{proof, bound 1, sub goal 1, lemma: asymptotic equivalence between L and large jump, LD for Hawkes}.}
\end{align*}
Lastly, due to $n \lambda_{ \{j^*\} }(n)  \in \RV_{ -(\alpha^*(j^*) - 1)  }(n)$ (see \eqref{def: rate function lambda j n, cluster size}) and $\alpha^*(j^*) > 1$,
it holds for all $K$ large enough that $(K+1)\cdot\big( \alpha^*(j^*) - 1  \big) > \beta$, and hence
$
\big( n \lambda_{ \{j^*\} }(n) \big)^{K+1}  = \lo(n^{-\beta}).
$
This concludes the proof of Claim~\eqref{proof, sub goal 1, lemma: asymptotic equivalence between L and large jump, LD for Hawkes}.

\medskip
\noindent
\textbf{Proof of Claim \eqref{proof, sub goal 2, lemma: asymptotic equivalence between L and large jump, LD for Hawkes}}.
By the Markov property at each stopping time $\tau^{>\delta}_n(k)$,
it suffices to show that (for any $\delta > 0$ small enough),
\begin{align*}
    & \P\Bigg(
        \sup_{ t \in [0,T]:\ t < \tau^{>\delta}_n(1)  }
        \norm{
            \bar{\bm L}^{\leqslant\delta}_n(t) - \bm \mu_{\bm N}t 
        } >
        \frac{\Delta}{2(K+1)}
    \Bigg)
    \\ 
    & 
    = 
    \P\Bigg(
        \sup_{ t \in [0,T]:\ t < \tau^{>\delta}_n(1)  }
        \norm{
            \bar{\bm L}_n(t) - \bm \mu_{\bm N}t 
        } >
        \frac{\Delta}{2(K+1)}
    \Bigg) = \lo(n^{-\beta}),
\end{align*}
where we applied the fact that $\bar{\bm L}^{\leqslant\delta}_n(t) = \bar{\bm L}_n(t)$ for any $t < \tau^{>\delta}_n(1)$; see \eqref{def: large jump process for Levy, LD for Hawkes}--\eqref{def: small jump process for Levy, LD for Hawkes}.
Besides, recall that $\E\bm L(1) = \bm\mu_{\bm N}$.
Applying Lemma~\ref{lemma: concentration of small jump process, Ld for Levy MRV}, we conclude the proof of claim~\eqref{proof, sub goal 2, lemma: asymptotic equivalence between L and large jump, LD for Hawkes}.
\end{proof}

Analogous to Lemma~\ref{lemma: asymptotic equivalence between L and large jump, LD for Hawkes}, we decompose the Hawkes process $\bm N(t)$ by considering the small or large clusters therein.
More precisely,
for the scaled process $\bar{\bm N}_n(t) = \bm N(nt)/n$, 
it follows from the cluster representation \eqref{def: branching process approach, N i t, notation section} that
\begin{align}
    \bar{N}_{n,i}(t)
    = 
    \frac{1}{n}
    \sum_{j \in [d]}\sum_{k \geq 0}\sum_{ m = 0 }^{ K^{\mathcal O}{ (T^\mathcal{C}_j(k),j) }  }
    \mathbbm{I}
    \bigg\{
        T^\mathcal{C}_{j;k} + T^{\mathcal O}_{ j;k }(m) \leq nt,\ A^{\mathcal O}_{ j;k }(m) = i
    \bigg\},
    \quad \forall t \geq 0,\ i \in [d],
    \nonumber
\end{align}
and $\bar{\bm N}_n(t) = \big(  \bar{N}_{n,1}(t), \bar{N}_{n,2}(t), \ldots, \bar{N}_{n,d}(t) \big)^\top$.
Meanwhile, recall that $\bm S^{(k)}_j$ is size vector of the cluster induced by the $k^\text{th}$ type-$j$ immigrant (see \eqref{def: cluster size jth type k cluster, LD for Hawkes})
and $T^{\mathcal C}_{j;k}$ is the arrival time of that immigrant.
By incorporating the cluster size, for each $n \geq 1$ and $\delta > 0$ we define
\begin{align}
    \bar N_{n,i}^{>\delta}(t)
    & \delequal
    \frac{1}{n}
    \sum_{j \in [d]}\sum_{k \geq 0}\sum_{ m = 0 }^{ K^{\mathcal O}{ (T^\mathcal{C}_j(k),j) }  }
    \mathbbm{I}\bigg\{
        \norm{\bm S^{(k)}_j} >n\delta
    \bigg\}\cdot 
    \mathbbm{I}
    \bigg\{
        T^\mathcal{C}_{j;k} + T^{\mathcal O}_{ j;k }(m) \leq nt,\ A^{\mathcal O}_{ j;k }(m) = i
    \bigg\},
    \label{def: large jump process for N, LD for Hawkes}
    \\ 
    \bar N_{n,i}^{\leqslant\delta}(t)
    & \delequal
    \bar N_{n,i}(t) - \bar N_{n,i}^{>\delta}(t)
    \label{def: small jump process for N, LD for Hawkes}
    \\ &
    =
    \frac{1}{n}
    \sum_{j \in [d]}\sum_{k \geq 0}\sum_{ m = 0 }^{ K^{\mathcal O}{ (T^\mathcal{C}_j(k),j) }  }
    \mathbbm{I}\bigg\{
        \norm{\bm S^{(k)}_j} \leq n\delta
    \bigg\}\cdot 
    \mathbbm{I}
    \bigg\{
        T^\mathcal{C}_{j;k} + T^{\mathcal O}_{ j;k }(m) \leq nt,\ A^{\mathcal O}_{ j;k }(m) = i
    \bigg\},
    \nonumber
\end{align}
and 
$
\bar{\bm N}^{>\delta}_n(t) \delequal \big( \bar N_{n,1}^{>\delta}(t),\ldots, \bar N_{n,d}^{>\delta}(t) \big)^\top,
$
$
\bar{\bm N}^{\leqslant\delta}_n(t) \delequal \big( \bar N_{n,1}^{\leqslant\delta}(t),\ldots, \bar N_{n,d}^{\leqslant\delta}(t) \big)^\top.
$
Analogous to \eqref{def: sample path scaled process, Levy, LD for Hawkes}, we define
\begin{align}
    \bar{\bm N}^{>\delta;\subZeroT{T}}_n
    \delequal
    \big\{
        \bar{\bm N}^{>\delta}_n(t):\ t \in [0,T]
    \big\},
    \qquad
    \bar{\bm N}^{>\delta;\subInfty}_n
    \delequal
    \big\{
        \bar{\bm N}^{>\delta}_n(t):\ t \geq 0
    \big\}.
    \label{def: sample path scaled process, N, LD for Hawkes}
\end{align}

The proof of the next lemma involves the lifetime of each cluster.
Adopting notations in \eqref{def: cluster process associated with each center, Poisson cluster process, notation}, we use
\begin{align}
    \notationdef{notation-life-time-for-each-cluster-in-N}{H^{(k)}_j}
    \delequal
    \max_{m \geq 0}T^{\mathcal O}_{ j;k }(m)
    \label{def: life time of the kth cluster with type j ancestor, LD for Hawkes}
\end{align}
to denote the lifetime  of the cluster process $\bm N^{\mathcal O}_{ (T^\mathcal{C}_j(k),j) }$ 
(i.e., the gap in the birth times between the ancestor and the last descendant in this cluster).
Similarly, in \eqref{def: cluster process, C mathcal O j} we use
\begin{align}
    H_j \delequal \max_{m \geq 0,\ i \in [d]} T^{(n);j}_i(m)
    \label{def: life time of the kth cluster with type j ancestor, indp copy, LD for Hawkes}
\end{align}
to denote the lifetime of the point process $C^\mathcal{O}_j$.
Since 
$\bm N^{\mathcal O}_{ (T^\mathcal{C}_j(k),j) }$
are i.i.d.\ copies of $C^\mathcal{O}_j$,
the sequence $H^{(k)}_j$ are also independent copies of $H_j$.
Lemma~\ref{lemma: asymptotic equivalence between N and large jump, LD for Hawkes} establishes the asymptotic equivalence between $\bar{\bm N}_n^{\subInfty}$ and $\bar{\bm N}^{>\delta;\subInfty}_n + \bm\mu_{\bm N}\mathbf 1$ under $\dmp{\subInfty}$.

\begin{lemma}\label{lemma: asymptotic equivalence between N and large jump, LD for Hawkes}
\linksinthm{lemma: asymptotic equivalence between N and large jump, LD for Hawkes}
    Let Assumptions~\ref{assumption: subcriticality}--\ref{assumption: regularity condition 2, cluster size, July 2024} hold.
    For any $\Delta,\beta > 0$,
    \begin{align}
        \lim_{n \to \infty} n^\beta \cdot 
        \P\bigg(
            \dmp{\subInfty}\Big(
                \bar{\bm N}_n^{\subInfty},\ \bar{\bm N}^{>\delta;\subInfty}_n + \bm\mu_{\bm N}\mathbf 1
            \Big) > \Delta
        \bigg) = 0,
        \qquad
        \forall \delta > 0\text{ small enough.}
        \nonumber
    \end{align}
\end{lemma}

\begin{proof}
\linksinpf{lemma: asymptotic equivalence between N and large jump, LD for Hawkes}
The proof is similar to that of Lemma~\ref{lemma: asymptotic equivalence between L and large jump, LD for Hawkes}.
In particular, by fixing $T$ large enough such that $e^{-T} < \Delta/2$, it suffices to prove that 
\begin{align}
    \P\bigg(
        \sup_{t \in [0,T]}
        \norm{
            \bar{\bm N}_n(t)
            -
            \Big(  \bar{\bm N}_n^{>\delta}(t) + \bm \mu_{\bm N}t   \Big)
        } >\Delta/2   
    \bigg) = \lo(n^{-\beta}),
    \qquad
    \forall \delta > 0\text{ small enough.}
    \label{proof, goal, lemma: asymptotic equivalence between N and large jump, LD for Hawkes}
\end{align}
First, by comparing the definitions in \eqref{def: small jump process for Levy, LD for Hawkes} and \eqref{def: small jump process for N, LD for Hawkes},
\begin{align}
    \bm{\bar N}^{\leqslant \delta}_n(t) \leq \bm{\bar L}^{\leqslant \delta}_n(t),
    \qquad\forall \delta > 0,\  n \geq 1,\ t \geq 0,
    \label{proof, ineq 1, lemma: asymptotic equivalence between N and large jump, LD for Hawkes}
\end{align}
since in $\bm{\bar L}^{\leqslant \delta}_n(t)$ we lump all descendants in the same cluster at the arrival time of the immigrant of the cluster.
Here, the order $\bm x \leq \bm y$ between two vectors in $\R^d$ means that $x_i \leq y_i$ for each $i \in [d]$.
Next, for each $M> 0$, let 
\begin{align}
    {\bm L}^{|M}(t)
    \delequal 
    \sum_{ j \in [d] }\sum_{k \geq 0}
    \bm S^{(k)}_j
    \mathbbm{I}\Big\{ H^{(k)}_j \leq M  \Big\}\cdot 
    \mathbbm{I}_{ [T^\mathcal{C}_{j;k},\infty) }(t).
    \nonumber
\end{align}
That is, ${\bm L}^{|M}(t)$ is a modification of $\bm L(t)$ in \eqref{def: Levy process, LD for Hawkes} by only considering clusters with lifetime below $M$.
Note that
\begin{align}
    \E\big[ \bm L^{|M}(1) \big]
    = \sum_{j \in [d]}c_{j}^{\bm N} \cdot 
    \E\big[
        \bm S_j \mathbbm{I}\{ H_j \leq M \}
    \big]
    \rightarrow
    \sum_{j \in [d]}c_{j}^{\bm N} \cdot 
    \E
        \bm S_j
    =\bm \mu_{\bm N},
    \qquad
    \text{as }M \to \infty,
    \label{proof, limit 1, lemma: asymptotic equivalence between N and large jump, LD for Hawkes}
\end{align}
by the monotone convergence theorem.
Besides, for each $n \geq 1$ and $\delta > 0$, we define processes
\begin{align}
    \bar{\bm L}^{|M}_n(t)
    & \delequal
    \frac{1}{n}
    \sum_{ j \in [d] }\sum_{k \geq 0}
    \bm S^{(k)}_j
    \mathbbm{I}\big\{  H^{(k)}_j \leq M   \big\}\cdot 
    \mathbbm{I}_{ [T^\mathcal{C}_{j;k},\infty) }(nt),
    \nonumber
    \\
    \bar{\bm L}^{\leqslant\delta|M}_n(t)
    & \delequal
    \frac{1}{n}
    \sum_{ j \in [d] }\sum_{k \geq 0}
    \bm S^{(k)}_j
    \mathbbm{I}\bigg\{ \norm{\bm S^{(k)}_j } \leqslant n\delta,\ H^{(k)}_j \leq M   \bigg\}\cdot 
    \mathbbm{I}_{ [T^\mathcal{C}_{j;k},\infty) }(nt),
    \nonumber
\end{align}
which scale both time and space by $n$.
For any $\epsilon,\delta > 0$, and any $n$ large enough such that $n\epsilon > M$,
\begin{align}
    \bm{\bar N}^{\leqslant \delta}_n(t) \geq \bm{\bar L}^{\leqslant \delta|M}_n\big( (t-\epsilon) \vee 0 \big),
    \qquad\forall \delta > 0,\  n \geq 1,\ t \geq 0.
    \label{proof, ineq 2, lemma: asymptotic equivalence between N and large jump, LD for Hawkes}
\end{align}
Indeed, for any cluster with its immigrant arriving at $T^{\mathcal C}_{j;k} \leq n(t - \epsilon)$ and lifetime $H^{(k)}_j \leq M$,
all the descendants in this cluster must have arrived by the time $n(t - \epsilon) + M < nt$.
Combining \eqref{proof, ineq 1, lemma: asymptotic equivalence between N and large jump, LD for Hawkes} and \eqref{proof, ineq 2, lemma: asymptotic equivalence between N and large jump, LD for Hawkes},
for any $\epsilon, M, \delta > 0$, it holds for all $n$ large enough with $n\epsilon > M$ that 
\begin{align*}
    & 
    \bm{\bar L}^{\leqslant \delta|M}_n\big( (t-\epsilon) \vee 0 \big)
    - \bm \mu_{\bm N} t
    \leq 
     \bm{\bar N}^{\leqslant \delta}_n(t) - \bm \mu_{\bm N} t 
     \leq 
     \bar{\bm L}^{\leqslant\delta}_n(t) - \bm \mu_{\bm N} t,
     \qquad \forall t \geq 0
     \\ 
     \Longrightarrow
     & 
     \sup_{t \in [0,T]}
     \norm{
        \bm{\bar N}^{\leqslant \delta}_n(t) - \bm\mu_{\bm N} t
     }
     \leq 
     \sup_{t \in [0,T]}
     \norm{
         \bar{\bm L}^{\leqslant\delta}_n(t) - \bm\mu_{\bm N} t
     }
     +
     \sup_{t \in [0,T]}
     \norm{
         \bm{\bar L}^{\leqslant \delta|M}_n\big( (t-\epsilon) \vee 0 \big)
    - \mu_{\bm N} t
     }.
\end{align*}
The last step follows from our choice of the $L_1$ norm and the preliminary bound that
$
y \leq x \leq z \ \Longrightarrow \ |x| \leq |y| + |z|.
$
Furthermore, note that 
\begin{align*}
    &
    \sup_{t \in [0,T]}
     \norm{
         \bm{\bar L}^{\leqslant \delta|M}_n\big( (t-\epsilon) \vee 0 \big)
    - \bm \mu_{\bm N} t
     }
     \\ 
     & \leq 
     \sup_{t \in [0,T]}
     \norm{
         \bm{\bar L}^{\leqslant \delta|M}_n(t)
    - t \cdot \E\big[\bm L^{|M}(1)\big]
     }
     + \epsilon \norm{\bm \mu_{\bm N}} +
     T \cdot \norm{
     \bm \mu_{\bm N} - \E\big[\bm L^{|M}(1)\big]
     }
     \\ 
     & \leq 
     \sup_{t \in [0,T]}
     \norm{
         \bm{\bar L}^{\leqslant \delta|M}_n(t)
    - t \cdot \E\big[\bm L^{|M}(1)\big]
     }
     +\Delta/6
     \qquad
     \text{for any $\epsilon$ small enough and $M$ large enough}.
\end{align*}
The last inequality follows from the limit in \eqref{proof, limit 1, lemma: asymptotic equivalence between N and large jump, LD for Hawkes}.
Therefore, to prove Claim \eqref{proof, goal, lemma: asymptotic equivalence between N and large jump, LD for Hawkes},
it suffices to show that given $M > 0$,
the bounds 
$
\P\Big(
    \sup_{t \in [0,T]}
     \norm{
         \bm{\bar L}^{\leqslant \delta|M}_n(t)
    - t \cdot \E\big[\bm L^{|M}(1)\big]
     } > \Delta/6
    \Big)
    =\lo(n^{-\beta})
$
and 
$
 \P\Big(
        \sup_{t \in [0,T]}
        \norm{
         \bm{\bar L}^{\leqslant \delta}_n(t)
    - \bm \mu_{\bm N} t
     } >\Delta/6  
    \Big) = \lo(n^{-\beta})
$
for any $\delta > 0$ small enough.
By repeating the arguments in Lemma~\ref{lemma: asymptotic equivalence between L and large jump, LD for Hawkes} for \eqref{proof, goal, lemma: asymptotic equivalence between L and large jump, LD for Hawkes}, one can establish both claims above.
\end{proof}

To proceed, we prepare
Lemma~\ref{lemma: lifetime of a cluster, LD for Hawkes} and characterize the tail asymptotics for $H_j$ in \eqref{def: life time of the kth cluster with type j ancestor, indp copy, LD for Hawkes} (i.e., the lifetime of clusters).

\begin{lemma}\label{lemma: lifetime of a cluster, LD for Hawkes}
\linksinthm{lemma: lifetime of a cluster, LD for Hawkes}
Let Assumptions~\ref{assumption: subcriticality} and \ref{assumption: heavy tails in B i j} hold.
Let $\beta > 0$
and $a:(0,\infty)\to (0,\infty)$ be a regularly varying function with $a(x) \in \RV_{-\beta}(x)$ as $x \to \infty$.
Let $f^{\bm N}_{p \leftarrow q}(\cdot)$ be the decay functions in \eqref{def: conditional intensity, hawkes process}.
Suppose that 
\begin{align}
   \int_{x/\log x}^\infty f^{\bm N}_{p \leftarrow q}(t)dt = \lo\big( a(x) \big)
    \ 
    \text{ as }x \to \infty,
    \qquad \forall p,q\in [d].
    \label{tail condition for fertility function, lemma: lifetime of a cluster, LD for Hawkes}
\end{align}
Then,
\begin{align}
    \P( H_j > n\epsilon ) = \lo\big( a(n) \big)\
    \text{ as }n \to \infty,
    \qquad \forall j \in [d],\  \epsilon > 0.
    \label{claim, lemma: lifetime of a cluster, LD for Hawkes}
\end{align}
\end{lemma}

\begin{proof}
\linksinpf{lemma: lifetime of a cluster, LD for Hawkes}
In this proof, we fix some $\epsilon > 0$ and $j \in [d]$.
Recall the cluster representation of Hawkes processes in \eqref{def: center process, Poisson cluster process, notation}--\eqref{def: branching process approach, N i t, notation section}.
The key of this proof is to establish a stochastic comparison of $H_j$, the lifetime of a cluster $\mathcal C^{\mathcal O}_j$,
which refines the bounds for the lifetime of clusters in \cite{10.36045/bbms/1170347811} by taking the maximum instead of the sum of the birth times within each generation.
To be more specific, we introduce a few notations.
Independently for each pair $(i,l) \in [d]^2$,
let $\big(B^{(n,m);j}_{i \leftarrow l}\big)_{ n,m \geq 1 }$ 
be 
the i.i.d.\ copies of $B_{i \leftarrow l}$ in \eqref{def: vector B i, offspring of type i individual}, i.e., with law
\begin{align*}
    B_{i \leftarrow l} \sim \text{Poisson}\bigg( \tilde B_{i \leftarrow l} \int_0^\infty f^{\bm N}_{i \leftarrow l}(t)dt \bigg).
\end{align*}
Here, recall that
by Poisson$(X)$ for a non-negative variable $X$, we mean the law 
$\P\big(\text{Poisson}(X) > y\big) = \int_0^\infty \P\big(\text{Poisson}(x) > y\big)\P(X \in dx)$.
This agrees with the notations in Definition~\ref{def: offspring cluster process},
where we use $B^{(n,m);j}_{i \leftarrow l}$ to denote, in the cluster process $C^\mathcal{O}_j(\cdot)$,
the count of type-$i$ children (in the $n^\text{th}$ generation) born by the $m^\text{th}$ type-$l$ individual in the $(n-1)^\text{th}$ generation.
Besides, independent from the sequence $\big(B^{(n,m);j}_{i \leftarrow l}\big)_{ n,m \geq 1 }$ ,
let $\big(X_{i \leftarrow l}^{(n,k,k^\prime)}\big)_{ n,k,k^\prime \geq 1  }$
be i.i.d.\ copies of $X_{i \leftarrow l}$ with law 
\begin{align}
    \P(X_{i \leftarrow l} \in E) = \frac{
        \int_{t \in E } f^{\bm N}_{i \leftarrow l}(t)dt
    }{
        \int_0^\infty f^{\bm N}_{i \leftarrow l}(t)dt
    },
    \qquad\forall\text{ Borel }E\subseteq(0,\infty).
    \label{proof, def law of X p q, lemma: lifetime of a cluster, LD for Hawkes}
\end{align}
By step (ii) in Definition~\ref{def: offspring cluster process},
we interpret
$X_{i \leftarrow l}^{(n,k,k^\prime)}$ as the time that the $k^\text{th}$ type-$l$ parent in the $(n-1)^\text{th}$ generation waited to give birth to its $(k^\prime)^\text{th}$ type-$i$ child (here, the children are not ordered by age).

As in Definition~\ref{def: offspring cluster process},
we use 
$
R^{(n);j}_i
$
to denote, in the cluster $\mathcal C^{\mathcal O}_j$,
the count of type-$i$ descendants in the $(n-1)^\text{th}$ generation,
and use the sequence $0 < T^{(n);j}_i(1) < T^{(n);j}_i(2) < \ldots$ to denote the birth times
of type-$i$ descendants in the $n^\text{th}$ generation (provided that the cluster does have at least $n$ generations and there are type-$i$ individuals in the $n^\text{th}$ generation).
In particular, note that each $T^{(n);j}_i(r)$ admits the expression
\begin{align*}
    T^{(n);j}_i(r) 
    = 
    X_{j_1 \leftarrow  j}^{(1,1,k^\prime_1)}
    + 
    X_{j_2 \leftarrow  j_1}^{(2,k_2,k^\prime_2)}
    + \ldots + 
    X_{ i \leftarrow  j_{n-1}}^{(n,k_n,k^\prime_n)},
\end{align*}
where, for the (random) indices, we have
$j_l \in [d]$, $k_l \in \big[R^{(l-1);j}_{ j_{l-1}  }\big]$, and 
$k^\prime_l \in \big[ B_{j_l \leftarrow j_{l-1}}^{(n,k_{l-1});j}  \big]$
for each $l$;
here, note that we must have $j_0 \equiv j$ and $k_1 \equiv 1$, since the only individual in the $0^\text{th}$ generation is the type-$j$ ancestor itself.
Next, note that 
\begin{align*}
    T^{(n);j}_i(r) 
    & \leq 
    X_{j_1 \leftarrow  j}^{(1,1,k^\prime_1)}
    + 
    X_{j_2 \leftarrow  j_1}^{(2,k_2,k^\prime_2)}
    + \ldots + 
    X_{ j_{n-1} \leftarrow  j_{n-2}}^{(n-1,k_{n -1 },k^\prime_{n - 1})}
    +
    \max_{p,q \in [d]}
    \max_{
        \substack{
            k \in  [R^{(n-1);j}_p] \\  k^\prime \in  [ B_{q \leftarrow p}^{(n,k);j}] 
            }
    }
    X_{q \leftarrow p}^{(n,k,k^\prime)}
    \\ 
    & \stleq
     X_{j_1 \leftarrow  j}^{(1,1,k^\prime_1)}
    + 
    X_{j_2 \leftarrow  j_1}^{(2,k_2,k^\prime_2)}
    + \ldots + 
    X_{ j_{n-1} \leftarrow  j_{n-2}}^{(n-1,k_{n -1 },k^\prime_{n - 1})}
    +
    \max_{p,q \in [d]}
    \max_{ m \in [ R^{(n);j}_q  ]  }X_{q \leftarrow p}^{(m)}
    \\ 
    & \leq 
   X_{j_1 \leftarrow  j}^{(1,1,k^\prime_1)}
    + 
    X_{j_2 \leftarrow  j_1}^{(2,k_2,k^\prime_2)}
    + \ldots + 
    X_{ j_{n-1} \leftarrow  j_{n-2}}^{(n-1,k_{n -1 },k^\prime_{n - 1})}
    +
    \max_{p,q \in [d]}
    \max_{ m \leq \sum_{ l \in [d] }R^{(n);j}_l   }X_{q \leftarrow p}^{(m)}.
\end{align*}
In the display above,
the stochastic comparison (i.e., the $\stleq$ step)
holds due to the independence between the $X^{(n,k,k^\prime)}_{q \leftarrow p}$'s and $\big( X^{(t,k,k^\prime)}_{u \leftarrow v} \big)_{ u,v \in [d], k \geq 1, k^\prime \geq 1, t \in [n-1]  }$,
where we use $X^{(m)}_{q \leftarrow p}$ to denote generic copies of $X_{q \leftarrow p}$ (see \eqref{proof, def law of X p q, lemma: lifetime of a cluster, LD for Hawkes}) that are independent of the $X^{(t,k,k^\prime)}_{u \leftarrow v}$'s.
We emphasize that this stochastic comparison holds due to the \emph{independence between the size (i.e., offspring counts) and the birth time distributions} in the cluster.
Repeating this argument inductively, we obtain
\begin{align}
   T^{(n);j}_i(r) 
    & \stleq
    n \cdot 
    \max_{p,q \in [d]}\ 
    \max_{1 \leq m \leq \norm{\bm S_j} }X_{q \leftarrow p}^{(m)}.
    \label{proof, intermediate, stochastic dominance for lifetime, lemma: lifetime of a cluster, LD for Hawkes}
\end{align}
We note that \eqref{proof, intermediate, stochastic dominance for lifetime, lemma: lifetime of a cluster, LD for Hawkes} refines the bounds for the lifetime of clusters in \cite{10.36045/bbms/1170347811}.

Adopting the notations in Definition~\ref{def: offspring cluster process},
we let $K^{\mathcal O}_j = \max\{n \geq 0:\ R^{(n);j}_i \geq 1\text{ for some }i \in [d]  \}$
        be the total count of generations in this cluster.
Besides, we use $R^{(n);j} \delequal \sum_{i \in [d]}R^{(n);j}_i$ to denote the offspring count in the $n^\text{th}$ generation of the cluster process $C^\mathcal{O}_j(\cdot)$.
By \eqref{proof, intermediate, stochastic dominance for lifetime, lemma: lifetime of a cluster, LD for Hawkes},
for any $c_1 > 0$ we have
\begin{align*}
    \P(H_j > n\epsilon) & \leq \P(K^{\mathcal O}_j \geq c_1\log n)
    + 
    \P\bigg(
        \max_{p,q \in [d]}\ 
    \max_{1 \leq m \leq \norm{\bm S_j} }X_{p \leftarrow q}^{(m)}
        >
        \frac{n\epsilon}{c_1\log n}
    \bigg)
    \\ 
    & \leq 
    \P(K^{\mathcal O}_j \geq c_1\log n)
    +
    \sum_{p,q \in [d]}
    \P\bigg(
    \max_{1 \leq m \leq \norm{\bm S_j} }X_{p \leftarrow q}^{(m)}
        >
        \frac{n\epsilon}{c_1\log n}
    \bigg).
\end{align*}
To prove \eqref{claim, lemma: lifetime of a cluster, LD for Hawkes}, it suffices to find some constant $c_1 > 0$ (depending only on $\beta$ and $j$) such that 
\begin{align}
    \P(K^{\mathcal O}_j  \geq c_1 \log n) = \lo\big( a(n) \big)\ \text{ as }n\to\infty, 
    \label{proof: claim 1, lemma: lifetime of a cluster, LD for Hawkes}
\end{align}
and that
\begin{align}
    \P\bigg( \max_{ 1 \leq m \leq \norm{\bm S_j}   }X^{(m)}_{p \leftarrow q} > \frac{ n\epsilon }{ c_1 \log n} \bigg)
    = \lo\big( a(n) \big)\ \text{ as }n\to\infty,
    \qquad\forall (p,q) \in [d]^2.
    \label{proof: claim 2, lemma: lifetime of a cluster, LD for Hawkes}
\end{align}

\medskip
\noindent
\textbf{Proof of Claim~\eqref{proof: claim 1, lemma: lifetime of a cluster, LD for Hawkes}}.
We fix some $\tilde \alpha \in (1, \min_{p,q}\alpha_{p \leftarrow q})$.
By Assumption~\ref{assumption: heavy tails in B i j},
we have
$
\E[\tilde B_{p\leftarrow  q}^{\tilde \alpha}] < \infty. 
$
Under this condition and
Assumption~\ref{assumption: subcriticality},
it has been established in Section~3  of 
\cite{KEVEI2021109067} (in particular, see Equation (6) in the paper and the discussion below)
that
\begin{align*}
    \E\big[ ( R^{(k);j}  )^{\tilde \alpha}\big] \leq c_0 v^k,
    \qquad\forall k \geq 1,
\end{align*}
for some $c_0 \in (0,\infty)$ and $v \in (0,1)$ whose values only depend on the law of the $\tilde B_{p,q}$'s.
Then, by Markov's inequality,
\begin{align*}
    \P(K^{\mathcal O}_j \geq  k)
    =
     \P\big( R^{(k);j}  > 0  \big)
     =
     \P\big( R^{(k);j}  \geq 1  \big)
    \leq 
    c_0 v^k,
    \qquad\forall k \geq 1.
\end{align*}
By picking $c_1 > 0$ large enough such that $c_1\log(v) < -\beta$, we get
\begin{align*}
    \P(K^{\mathcal O}_j \geq c_1 \log n)
    \leq 
    c_0 \cdot \nu^{ {c_1 \log(n)}   }
    =
    c_0 \cdot n^{ {c_1 \log(\nu)}   }
    =
    \lo\big(a(n)\big)
\end{align*}
due to $a(n) \in \RV_{-\beta}(n)$.
This verifies Claim~\eqref{proof: claim 1, lemma: lifetime of a cluster, LD for Hawkes}.

\medskip
\noindent
\textbf{Proof of Claim~\eqref{proof: claim 2, lemma: lifetime of a cluster, LD for Hawkes}}.
Since the $X^{(m)}_{p \leftarrow q}$'s are independent of the $X_{p\leftarrow q}^{(n,k,k^\prime)}$ and $R_{p}^{(n);j}$ (and hence the cluster $\mathcal C^{\mathcal O}_j$ and the cluster size vector $\bm S_j$),
\begin{align*}
    \P\bigg( \max_{ 1 \leq m \leq \norm{\bm S_j}   }X^{(m)}_{p \leftarrow q} > \frac{ n\epsilon }{c_1 \log n} \bigg)
    & = \sum_{k \geq 1}
    \P\bigg( \max_{ m \in [k]  }X^{(m)}_{p \leftarrow q} > \frac{ n\epsilon }{ c_1 \log n} 
    \bigg)
    \P(\norm{\bm S_j} = k)
    \\ 
    & \leq 
    \sum_{k \geq 1}
    k \cdot \P\bigg( X_{p \leftarrow q} > \frac{ n\epsilon }{ c_1 \log n} 
    \bigg) \cdot   \P(\norm{\bm S_j} = k)
    \\ 
    &
    =
    \E\big[\norm{\bm S_j}\big] 
        \cdot \P\bigg(X_{p \leftarrow q} > \frac{ n\epsilon }{ c_1 \log n} \bigg) 
    \\ 
    & = 
    \frac{\E\norm{\bm S_j}}{ \int_0^\infty f^{\bm N}_{p \leftarrow  q}(t)dt } \cdot 
    \int_{ t > { n\epsilon }/({c_1 \log n})  }f^{\bm N}_{p \leftarrow  q}(t)dt
    \qquad
    \text{by \eqref{proof, def law of X p q, lemma: lifetime of a cluster, LD for Hawkes}}
    \\ 
    & = \lo\big(a(n)\big)
    \qquad
    \text{due to \eqref{tail condition for fertility function, lemma: lifetime of a cluster, LD for Hawkes}}.
\end{align*}
This concludes the proof of Claim~\eqref{proof: claim 2, lemma: lifetime of a cluster, LD for Hawkes}.
\end{proof}

We prepare a lemma studying the distance between step functions under the metric $\dmp{\subZeroT{T}}$ in \eqref{def: product M1 metric}.

\begin{lemma}\label{lemma: step functions under M1 metric, LD for Hawkes}
\linksinthm{lemma: step functions under M1 metric, LD for Hawkes}
Let $T \in (0,\infty)$, $k, k^\prime \in \mathbb N$, $\epsilon > 0$, and $\xi,\xi^\prime \in \D[0,T] = \D\big([0,T],\R^d\big)$.
Suppose that 
\begin{itemize}
    \item 
        $\xi$ is a non-decreasing step function with $k$ jumps and vanishes at the origin; i.e., $\xi(t) = \sum_{i = 1}^k \bm w_i\mathbbm{I}_{[t_i,T]}(t)\ \forall t \in [0,T]$, 
        where $\bm w_i \in \R^d_+ \setminus \{\bm 0\}$ for each $i \in [k]$ and $0 < t_1 < \ldots < t_k = T$;

    \item 
        $\xi^\prime$ is a non-decreasing step function with $k^\prime$ jumps and vanishes at the origin;

    \item 
        $\xi\big( (t - \epsilon) \vee 0\big) \leq \xi^\prime(t) \leq \xi(t)$ for each $t \in [0,T]$.
\end{itemize}
Then, for any $t \in (0,T]$ such that $t_i \notin [t-\epsilon,t]\ \forall i \in [k]$ (i.e., $\xi(\cdot)$ remains constant over $[(t-\epsilon)\vee 0,t]$),
\begin{align}
    \dmp{\subZeroT{t}}(\xi,\xi^\prime) \leq \epsilon.
    \nonumber
\end{align}
\end{lemma}

\begin{proof}\linksinpf{lemma: step functions under M1 metric, LD for Hawkes}
We write $\xi(t) = (\xi_1(t),\ldots,\xi_d(t)\big)^\top$
and
$\xi^\prime(t) = (\xi^\prime_1(t),\ldots,\xi^\prime_d(t)\big)^\top$.
Let $I = \{ t \in (0,T]:\ t_i \notin [t,t+\epsilon]\ \forall i \in [k]  \}$.
By definition of the product $M_1$ metric in \eqref{def: product M1 metric},
it suffices to show that $\dm{\subZeroT{t}}(\xi_i,\xi^\prime_i) \leq \epsilon$ for each $i \in [d]$ and $t \in I$,
where $\dm{\subZeroT{t}}$ is the $M_1$ metric for $\R$-valued càdlàg paths; see \eqref{def: M1 metric 0 T in R 1}.
As a result, it suffices to prove the claim for  $\R$-valued càdlàg paths.
More specifically, 
let $T \in (0,\infty)$, $k,k^\prime \in \mathbb Z_+$, $\epsilon > 0$, and $x,y \in \D\big([0,T],\R\big)$, and suppose that 
\begin{itemize}
    \item
        $x$ is a non-decreasing step function with $k$ jumps and vanishes at the origin,

    \item 
        $y$ is a non-decreasing step function with $k^\prime$ jumps and vanishes at the origin,

    \item 
        and $x\big((t-\epsilon)\vee 0\big) \leq y(t) \leq x(t)$ for each $t \in [0,T]$.
\end{itemize}
It suffices to fix some $t \in (0,T]$ such that $\Delta x(u) = 0$ for each $u \in [ (t - \epsilon) \vee 0 ,t]$, and show that
\begin{align}
    \dm{\subZeroT{t}}(x,y) \leq \epsilon.
    \label{proof, goal, lemma: step functions under M1 metric, LD for Hawkes}
\end{align}
Since both $x$ and $y$ are non-decreasing step functions over $[0,t]$, there exist
some $k^* \in \mathbb Z_+$ and a sequence
$
0 = z_0 < z_1 < z_2 < \ldots < z_{k^*}
$
such that 
\begin{align}
    \big\{ z \in [0,\infty):\ z = x(u)\text{ or }y(u)\text{ for some }u \in [0,t]\big\}
    =
    \{z_0,z_1,\ldots,z_{k^*}\}.
    \nonumber
\end{align}
For each $l \in [k^*]$, let
\begin{align}
    t^x_l \delequal \min\{ u \geq 0:\ x(u) \geq z_l \},
    \quad 
    t^y_l \delequal \min\{ u \geq 0:\ y(u) \geq z_l \}
    \nonumber
\end{align}
be the first time $x(u)$ (resp., $y(u)$) crosses the level of $z_l$.
Due to $x\big((u-\epsilon)\vee 0\big) \leq y(u) \leq x(u)$ for each $u \in [0,t]$,
we have
\begin{align}
     t^y_l - \epsilon \leq t^x_l \leq t^y_l,
     \qquad\forall l \in [k^*].
    \label{proof, property, bound between jump times, lemma: step functions under M1 metric, LD for Hawkes}
\end{align}
Also, since $\Delta x(u) = 0$ for each $u \in [ (t - \epsilon) \vee 0 ,t]$,
we know that the value of the step function $x(\cdot)$ remains constant over $ [ (t - \epsilon) \vee 0 ,t]$.
Then by $x\big((t-\epsilon)\vee 0\big) \leq y(t) \leq x(t)$, at the right endpoint of the time interval $[0,t]$ we have
\begin{align}
    x(t) = y(t) = z_{k^*}.
    \label{proof, property, same value at the end point, lemma: step functions under M1 metric, LD for Hawkes}
\end{align}
Now, we prove Claim~\eqref{proof, goal, lemma: step functions under M1 metric, LD for Hawkes} by constructing suitable parametric representations for paths $x$ and $y$.
First, for the path $x$, we adopt the convention $t^x_0 = 0$, and define $\big(u^x(w), s^x(w)\big)$ such that
(for each $l \in [k^*]$)
over the interval $w \in \big[ \frac{ 2(l - 1)  }{ 2k^* + 1   }, \frac{ 2l }{ 2k^* + 1   }   \big)$,
the parametric representation $(u^x,s^x)$ spends half of the time moving uniformly and horizontally from $(t^x_{l-1},z_{l-1})$ to $(t^x_l,z_{l-1})$ over the connected graph of $x$,
and then spends the other half time moving uniformly and vertically from 
$(t^x_l,z_{l-1})$ to $(t^x_l,z_{l})$.
More precisely, for each $l \in [k^*]$, let
\begin{align*}
    u^x(w) \delequal
    \begin{cases}
        t^x_{l - 1} + (t^x_l - t^x_{l - 1}) \cdot 
        \Big[ (2k^* + 1) w - { 2(l - 1)  }\Big]
        & \text{if }w \in \big[ \frac{ 2(l - 1)  }{ 2k^* + 1   }, \frac{ 2(l - 1) + 1  }{ 2k^* + 1   }   \big)
        \\ 
        t^x_l 
        & \text{if }w \in \big[ \frac{ 2(l - 1) + 1  }{ 2k^* + 1   }, \frac{ 2l  }{ 2k^* + 1   }   \big)
    \end{cases},
\end{align*}
\begin{align*}
    s^x(w)
    \delequal 
    \begin{cases}
        x\big( u^x(w)\big) = z_{l-1}
        & \text{if }w \in \big[ \frac{ 2(l - 1)  }{ 2k^* + 1   }, \frac{ 2(l - 1) + 1  }{ 2k^* + 1   }   \big)
        \\ 
        z_{l - 1} + (z_l - z_{l-1}) \cdot  \Big[ (2k^* + 1 )w - \big( 2(l - 1) + 1\big) \Big]
        & \text{if }w \in \big[ \frac{ 2(l - 1) + 1  }{ 2k^* + 1   }, \frac{ 2l  }{ 2k^* + 1   }   \big)
    \end{cases}.
\end{align*}
By definition, $x(u)$ remains constant over $u \in [t_{k^*}^x,t]$, so we set
\begin{align*}
    u^x(w) \delequal t^x_{k^*} + (t - t^x_{k^*}) \cdot 
    \Big[
        (2k^* + 1)w - {2k^*}
    \Big],\ \ 
    s^x(w) \delequal z_{k^*},
    \qquad
    \forall w \in \bigg[ \frac{2k^*}{2k^* + 1}, 1  \bigg].
\end{align*}
Here, the choice of $s^x(w) = z_{k^*}$ is valid due to the property \eqref{proof, property, same value at the end point, lemma: step functions under M1 metric, LD for Hawkes}.
Analogously, for the path $y$ we define (for each $l \in [k^*]$)
\begin{align*}
    u^y(w) \delequal
    \begin{cases}
        t^y_{l - 1} + (t^y_l - t^y_{l - 1}) \cdot 
        \Big[ (2k^* + 1) w - { 2(l - 1)  }\Big]
        & \text{if }w \in \big[ \frac{ 2(l - 1)  }{ 2k^* + 1   }, \frac{ 2(l - 1) + 1  }{ 2k^* + 1   }   \big)
        \\ 
        t^y_l 
        & \text{if }w \in \big[ \frac{ 2(l - 1) + 1  }{ 2k^* + 1   }, \frac{ 2l  }{ 2k^* + 1   }   \big)
    \end{cases},
\end{align*}
\begin{align*}
    s^y(w)
    \delequal 
    \begin{cases}
        y\big( u^y(w)\big) = z_{l-1}
        & \text{if }w \in \big[ \frac{ 2(l - 1)  }{ 2k^* + 1   }, \frac{ 2(l - 1) + 1  }{ 2k^* + 1   }   \big)
        \\ 
        z_{l - 1} + (z_l - z_{l-1}) \cdot  \Big[ (2k^* + 1 )w - \big( 2(l - 1) + 1\big) \Big]
        & \text{if }w \in \big[ \frac{ 2(l - 1) + 1  }{ 2k^* + 1   }, \frac{ 2l  }{ 2k^* + 1   }   \big)
    \end{cases},
\end{align*}
and due to $y(u) \equiv z_{k^*}$ for any $u \in [t^y_{k^*},t]$ (see
\eqref{proof, property, same value at the end point, lemma: step functions under M1 metric, LD for Hawkes}),
we set
\begin{align*}
    u^y(w) \delequal t^y_{k^*} + (t - t^y_{k^*}) \cdot 
    \Big[
        (2k^* + 1)w - {2k^*}
    \Big],\ \ 
    s^y(w) \delequal z_{k^*},
    \qquad
    \forall w \in \bigg[ \frac{2k^*}{2k^* + 1}, 1  \bigg].
\end{align*}
Lastly, by our construction of the parametric representations $(u^x,s^x)$ and $(u^y,s^y)$, we have $s^y(w) = s^x(w)$ for any $w \in [0,1]$,
and $\sup_{w \in [0,1]}| u^x(w) - u^y(w) | \leq \epsilon$ due to \eqref{proof, property, bound between jump times, lemma: step functions under M1 metric, LD for Hawkes}.
This concludes the proof of Claim~\eqref{proof, goal, lemma: step functions under M1 metric, LD for Hawkes}.
\end{proof}

Equipped with Lemmas~\ref{lemma: lifetime of a cluster, LD for Hawkes} and \ref{lemma: step functions under M1 metric, LD for Hawkes},
we provide high probability bounds over the distance between $\bar{\bm L}^{>\delta;\subInfty}_n$ in \eqref{def: sample path scaled process, Levy, LD for Hawkes} and $\bar{\bm N}^{>\delta;\subInfty}_n$ in \eqref{def: sample path scaled process, N, LD for Hawkes}
under $\dmp{\subInfty}$.

\begin{lemma}\label{lemma: asymptotic equivalence between large jump N and L, LD for Hawkes}
\linksinthm{lemma: asymptotic equivalence between large jump N and L, LD for Hawkes}
Let Assumptions~\ref{assumption: subcriticality}--\ref{assumption: regularity condition 2, cluster size, July 2024} hold.
Let $\beta > 1$
and $a:(0,\infty)\to (0,\infty)$ be a regularly varying function with $a(x) \in \RV_{-\beta}(x)$ as $x \to \infty$.
Suppose that
\begin{align}
   \int_{x/\log x}^\infty f^{\bm N}_{p \leftarrow q}(t)dt = \lo\big( a(x)/x \big)
    \ 
    \text{ as }x \to \infty,
    \qquad \forall p,q\in [d].
    \label{tail condition for fertility function, lemma: asymptotic equivalence between large jump N and L, LD for Hawkes}
\end{align}
Then, for each $\Delta > 0$,
    \begin{align}
        \lim_{n \to \infty} \big(a(n)\big)^{-1} \cdot 
        \P\bigg(
            \dmp{\subInfty}\Big(
                \bar{\bm L}^{>\delta;\subInfty}_n,\ \bar{\bm N}^{>\delta;\subInfty}_n
            \Big) > \Delta
        \bigg) = 0,
        \qquad
        \forall \delta > 0\text{ small enough.}
        \nonumber
    \end{align}
\end{lemma}

\begin{proof}\linksinpf{lemma: asymptotic equivalence between large jump N and L, LD for Hawkes}
Analogous to the proof of Lemma~\ref{lemma: asymptotic equivalence between L and large jump, LD for Hawkes},
we fix some $T$ large enough such that $e^{-T} < \Delta/2$,
and note that it suffices to prove
\begin{align}
    \P\Bigg(
        \int_0^T
        e^{-t} \cdot 
        \bigg[
            \dmp{\subZeroT{t}}
            \Big(
                \bar{\bm L}^{>\delta;\subZeroT{t} }_n,\ \bar{\bm N}^{>\delta;\subZeroT{t}}_n
            \Big)
            \wedge 1
        \bigg]dt
        > \frac{\Delta}{2}
    \Bigg)
    =\lo\big(a(n)\big),
    \ \ 
    \text{$\forall \delta > 0$ small enough.}
    \label{proof, goal, lemma: asymptotic equivalence between large jump N and L, LD for Hawkes}
\end{align}
On the one hand, by comparing the definitions in \eqref{def: small jump process for Levy, LD for Hawkes} and \eqref{def: small jump process for N, LD for Hawkes}, we must have
\begin{align}
    \bar{\bm N}^{>\delta}_n(t) \leq \bar{\bm L}^{>\delta}_n(t),
    \qquad\forall \delta > 0,\  n \geq 1,\ t \geq 0.
    \label{proof, upper bound, lemma: asymptotic equivalence between large jump N and L, LD for Hawkes}
\end{align}
On the other hand,
recall that 
for some cluster induced by a type-$j$ immigrant, 
we use $\bm S^{(k)}_j$ to denote the cluster size,
$T^{\mathcal C}_{j;k}$ for the arrival time of the  immigrant inducing the cluster,
and
$H^{(k)}_j$ for the lifetime of the cluster (see \eqref{def: life time of the kth cluster with type j ancestor, LD for Hawkes}).
On the event
$
\big\{ H^{(k)}_{j} \leq n\epsilon\ \forall (k,j)\text{ with }T^{\mathcal C}_{j;k} \leq nT  \big\},
$
we must have
\begin{align}
     \bar{\bm N}^{>\delta}_n(t) \geq \bar{\bm L}^{>\delta}_n\big( (t-\epsilon)\vee 0 \big),
    \quad\forall \delta > 0,\  n \geq 1,\ t \in [0,T].
    \label{proof, lower bound, lemma: asymptotic equivalence between large jump N and L, LD for Hawkes}
\end{align}
Indeed, for any cluster induced by an immigrant arriving at time $T^{\mathcal C}_{j;k} \leq n(t - \epsilon)$ and with cluster lifetime $H^{(k)}_j \leq n\epsilon$,
all  descendants in this cluster must have arrived by the time $nt$.
Furthermore, let $\tau^{>\delta}_n(k)$ be defined as in \eqref{proof, def tau delta n k, large jump time, lemma: asymptotic equivalence between L and large jump, LD for Hawkes},
and note that the sequence of stopping times $\big(\tau^{>\delta}_n(k)\big)_{k \geq 1}$ corresponds to the arrival times of jumps in $\bar{\bm L}^{>\delta}_n(t)$.
Therefore,
for any $K \geq 1$, on the event
$$
A_n(K,T,\epsilon,\delta) \delequal 
\{ \tau^{>\delta}_n(K+1) > T  \}
\cap 
\big\{ H^{(k)}_{j} \leq n\epsilon\ \forall (k,j)\text{ with }T^{\mathcal C}_{j;k} \leq nT  \big\},
$$
we have 
$
\bar{\bm L}^{>\delta}_n\big((t-\epsilon)\vee 0\big)
\leq 
\bar{\bm N}^{>\delta}_n(t) 
\leq 
\bar{\bm L}^{>\delta}_n(t)
$
for any $t \in [0,T]$, and 
$
\bar{\bm L}^{>\delta;\subZeroT{T}}_n
$
is a step function that vanishes at the origin and has at most $K$ jumps (all belonging to $\R^d_+ \setminus \{\bm 0\}$).
By Lemma~\ref{lemma: step functions under M1 metric, LD for Hawkes},
there exists a Borel set $I \subseteq [0,T]$,
whose choice is random and depends on the paths $\bar{\bm L}^{>\delta;\subZeroT{t} }_n,\ \bar{\bm N}^{>\delta;\subZeroT{t}}_n$,
such that 
\begin{itemize}
    \item 
        $\mathcal L_\infty(I) \leq K\epsilon$, where $\mathcal L_\infty$ is the Lebesgue measure on $(0,\infty)$;

    \item 
$
\dmp{\subZeroT{t}}\big(
                \bar{\bm L}^{>\delta;\subZeroT{t} }_n,\ \bar{\bm N}^{>\delta;\subZeroT{t}}_n
            \big) \leq \epsilon
$
for any $t \in [0,T]\setminus I$.
\end{itemize}
Therefore, on the event $A_n(K,T,\epsilon,\delta)$, we have 
\begin{align*}
    & \int_0^T
        e^{-t} \cdot 
        \bigg[
            \dmp{\subZeroT{t}}
            \Big(
                \bar{\bm L}^{>\delta;\subZeroT{t} }_n,\ \bar{\bm N}^{>\delta;\subZeroT{t}}_n
            \Big)
            \wedge 1
        \bigg]dt
    \\ 
    & 
    \leq 
    \int_{ t \in [0,T]\setminus I  }
    \dmp{\subZeroT{t}}
            \Big(
                \bar{\bm L}^{>\delta;\subZeroT{t} }_n,\ \bar{\bm N}^{>\delta;\subZeroT{t}}_n
            \Big)dt 
    +
    \int_{t \in I}dt
    \leq 
    T\epsilon + K\epsilon.
\end{align*}
Furthermore, by picking $\epsilon > 0$ small enough we have $(T+K)\epsilon < \Delta/2$.
In summary, to prove Claim~\eqref{proof, goal, lemma: asymptotic equivalence between large jump N and L, LD for Hawkes},
it suffices to show the existence of some $K \geq 1$ such that 
\begin{align}
    \P\Big(
        \tau^{>\delta}_n(K+1) \leq T
    \Big) = \lo\big(a(n)\big),
    \qquad
    \text{$\forall \delta > 0$ small enough},
    \label{proof, sub goal 1, lemma: asymptotic equivalence between large jump N and L, LD for Hawkes}
\end{align}
and that for any $\epsilon > 0$,
\begin{align}
    \P\Big(
        H^{(k)}_{j} > n\epsilon\text{ for some }(k,j)\text{ with }T^{\mathcal C}_{j;k} \leq nT 
    \Big) 
    = \lo\big(a(n)\big).
    \label{proof, sub goal 2, lemma: asymptotic equivalence between large jump N and L, LD for Hawkes}
\end{align}

\medskip
\noindent
\textbf{Proof of Claim \eqref{proof, sub goal 1, lemma: asymptotic equivalence between large jump N and L, LD for Hawkes}}.
This has been established in the proof of Claim~\eqref{proof, sub goal 1, lemma: asymptotic equivalence between L and large jump, LD for Hawkes} in Lemma~\ref{lemma: asymptotic equivalence between L and large jump, LD for Hawkes}.

\medskip
\noindent
\textbf{Proof of Claim \eqref{proof, sub goal 2, lemma: asymptotic equivalence between large jump N and L, LD for Hawkes}}.
It suffices to fix some $j \in [d]$, $\epsilon > 0$ and show that
$$
     \P\Big(
        H^{(k)}_{j} > n\epsilon\text{ for some }k \geq 1\text{ such that }T^{\mathcal C}_{j;k} \leq nT 
    \Big) 
    = \lo\big(a(n)\big).
$$
Recall that $(T^{\mathcal C}_{j;k})_{k \geq 1}$ are generated by a Poisson process with constant rate $c^{\bm N}_{j}$.
Therefore, by picking $C$ large enough (whose value only depends on $T$ and the constants $c^{\bm N}_{j}$),
it follows from Cramer's Theorem that 
$
\P\big( T^{\mathcal C}_{j;\floor{nC}} \leq nT  \big) = \lo\big(a(n)\big).
$
Fixing such $C$, we only need to show that 
$
\P\big( H^{(k)}_{j} > n\epsilon\text{ for some }k \leq \floor{nC}  \big) = \lo\big(a(n)\big)
$
holds
for any $\delta > 0$ small enough.
However, since $H^{(k)}_j$ are i.i.d.\ copies of $H_j$ (see \eqref{def: life time of the kth cluster with type j ancestor, LD for Hawkes} and \eqref{def: life time of the kth cluster with type j ancestor, indp copy, LD for Hawkes}),
\begin{align*}
    \P\big( H^{(k)}_{j} > n\epsilon\text{ for some }k \leq \floor{nC}  \big)
    & \leq 
    nC \cdot \P(H_j > n\epsilon)
    =
    nC \cdot \lo\big( a(n)/n \big)= \lo\big(a(n)\big).
\end{align*}
The second to last equality follows from the condition~\eqref{tail condition for fertility function, lemma: asymptotic equivalence between large jump N and L, LD for Hawkes} and Lemma~\ref{lemma: lifetime of a cluster, LD for Hawkes}.
This concludes the proof of Claim~\eqref{proof, sub goal 2, lemma: asymptotic equivalence between large jump N and L, LD for Hawkes}.
\end{proof}

Lastly, we  prove Proposition~\ref{proposition: asymptotic equivalence between Hawkes and Levy, LD for Hawkes}.

\begin{proof}[Proof of Proposition~\ref{proposition: asymptotic equivalence between Hawkes and Levy, LD for Hawkes}]
Due to
\begin{align*}
     & \dmp{ \subInfty }\big(\bar{\bm N}^{\subInfty}_n, \bar{\bm L}^{\subInfty}_n\big) 
     \\
     & \leq 
      \dmp{ \subInfty }\big(\bar{\bm N}^{\subInfty}_n, \bar{\bm N}^{>\delta,\subInfty}_n + \mu_{\bm N}\mathbf 1\big)
      +
      \dmp{ \subInfty }\big(
        \bar{\bm N}^{>\delta,\subInfty}_n + \mu_{\bm N}\mathbf 1,
        \bar{\bm L}^{>\delta,\subInfty}_n + \mu_{\bm N}\mathbf 1
        \big)
    \\
    & \qquad
    +
    \dmp{ \subInfty }\big(\bar{\bm L}^{\subInfty}_n, \bar{\bm L}^{>\delta,\subInfty}_n + \mu_{\bm N}\mathbf 1\big)
    \\
    & =
    \dmp{ \subInfty }\big(\bar{\bm N}^{\subInfty}_n, \bar{\bm N}^{>\delta,\subInfty}_n + \mu_{\bm N}\mathbf 1\big)
      +
      \dmp{ \subInfty }\big(
        \bar{\bm N}^{>\delta,\subInfty}_n,
        \bar{\bm L}^{>\delta,\subInfty}_n
        \big)
    +
    \dmp{ \subInfty }\big(\bar{\bm L}^{\subInfty}_n, \bar{\bm L}^{>\delta,\subInfty}_n + \mu_{\bm N}\mathbf 1\big),
\end{align*}
Proposition~\ref{proposition: asymptotic equivalence between Hawkes and Levy, LD for Hawkes} 
follows directly from Lemmas~\ref{lemma: asymptotic equivalence between L and large jump, LD for Hawkes}, \ref{lemma: asymptotic equivalence between N and large jump, LD for Hawkes}, and \ref{lemma: asymptotic equivalence between large jump N and L, LD for Hawkes}.
\end{proof}

\subsection{Univariate Results for $M_1$ Topology of $\D[0,T]$}
\label{subsec: univariate results, LD for Hawkes, M1 of D0T}

We show that in the univariate setting,
Theorem~\ref{theorem: LD for Hawkes} can be uplifted to
the $M_1$ topology of $\D[0,T]$---the càdlàg space with compact domain.
Without loss of generality, throughout this section we fix $T = 1$ and adopt the notations
\begin{align*}
    \D = \D\big([0,1],\R\big),
    \quad 
    \dm{} = \dm{\subZeroT{1}}.
\end{align*}

We start by adapting the notions from the main paper to the univariate setting, as many of them now admit simpler or more explicit expressions.
First, the univariate Hawkes process $N(t)$ has initial value $N(0) = 0$ and conditional intensity 
\begin{align*}
    h^{N}(t) = c_N + \int_0^T \tilde B(s) f^N(t-s) dN(s).
\end{align*}
Here, the constant $c_N > 0$ is the arrival rate of immigrants, 
the decay function $f^N(\cdot)$ is such that $\norm{f^N}_1 < \infty$,
and the excitation rates $\big(\tilde B(s)\big)_{s > 0}$ are independent copies of $\tilde B$
with $\P(\tilde B > x) \in \RV_{-\alpha}(x)$ for some $\alpha > 1$.
Second, we adapt the tail asymptotics characterized in Section~\ref{subsec: tail of cluster S, statement of results} for Hawkes processes clusters to the unvariate setting.
Let $S$ be a generic copy for the size of a cluster.
For any $x > 0$, we have
\begin{align}
    \P(S > nx) \sim \frac{ (\E S)^{1+\alpha} }{x^\alpha} \cdot \P(B > n),
    \quad \text{ as }n \to \infty,
    \label{univariate result, cluster tail asymptotics}
\end{align}
where $B \distequal \text{Poisson}\big(\tilde B \cdot \norm{f^N}_1\big)$.
Also, note that the regularly varying law of $\tilde B$ implies
$$\P(B > x) \sim \P(\tilde B \cdot \norm{f^N}_1 > x) \sim \norm{f^N}^{\alpha}_1\P(\tilde B > x).$$
The asymptotics \eqref{univariate result, cluster tail asymptotics}
are well-documented in existing literature on branching processes (e.g., Theorem~1 of \cite{Asmussen_Foss_2018})
and are an immediate consequence of Theorem~\ref{theorem: main result, cluster size} in univariate settings.
In particular, using notations in Section~\ref{subsec: tail of cluster S, statement of results},
we have
$
 \P(S > nx) \sim \lambda(n) \cdot \mathbf C_S\big((x,\infty)\big)
$
as $n \to \infty$,
where
\begin{align}
    \lambda(n) = \P(B > n);\qquad 
     \mathbf C_S\big((x,\infty)\big) = { (\E S)^{1+\alpha} }\big/{x^\alpha}\ \forall x > 0.
     \label{def, lambda, measure C S, LD for Hawkes univariate case}
\end{align}

Next, we adapt the notions introduced in Section~\ref{sec: sample path large deviation}
 to the univariate setting.
By \eqref{def: scale function for Levy LD MRV}, for each $k \geq 1$ we set
\begin{align}
    \breve \lambda_k(n) = \big(n\P(B > n)\big)^k.
    \label{def breve lambda k, univariate case}
\end{align}
Analogous to \eqref{def: j jump path set with drft c, LD for Levy, MRV},
for each $k \geq 0$ and $x \in \R$, we define
\begin{align*}
    \breve \D_{k;x}
    \delequal 
    \Bigg\{
        x \bm 1 + \sum_{i = 1}^k w_i\mathbbm{I}_{ [t_i,1]  }:\ 
        w_i \geq 0\text{ and }t_i \in (0,1]\ \forall i \in [k]
    \Bigg\},
\end{align*}
where $\bm 1(t) = t$ is the identity mapping of $\R^1$.
That is, $\breve \D_{k;x}$ is the subset of $\D = \D\big([0,1],\R\big)$ containing any piece-wise linear function with slope $x$ that vanishes at the origin and has at most $k$ jumps (all upward).
Adapting \eqref{def: measure C indices j, sample path LD for Hawkes} and \eqref{def: approximation, mean of N(t)}, we define
\begin{align}
    \mathbf C = c_N \cdot \mathbf C_S,
    \qquad
    \mu_N = c_N \cdot \E S.
    \label{def, measure C, LD for Hawkes univariate case}
\end{align}
We also adapt \eqref{def: measure C type j L, LD for Levy, MRV} and define Borel measures on $\breve{\D}_{k;x}$ by (for each $k \geq 1$ and $x \in \R$)
\begin{align*}
    \breve{\mathbf C}_{k;x}(\ \cdot\ )
    & =
    \frac{1}{k!}\E
    \Bigg[
        \int_{ w_i > 0\ \forall i \in [k] } \mathbbm{I}\bigg\{  x \bm 1 + \sum_{i = 1}^k w_i\mathbbm{I}_{ [U_i,1]  } \in \ \cdot\  \bigg\} 
        \bigtimes_{i = 1}^k \mathbf C( d w_i)
    \Bigg]
    \\ 
    & = 
    \frac{
         ( \alpha c_N)^k (\E S)^{  k(1+\alpha)  }
    }{
        k!
    }
    \E
    \Bigg[
        \int_{ w_i > 0\ \forall i \in [k] } \mathbbm{I}\bigg\{  x \bm 1 + \sum_{i = 1}^k w_i\mathbbm{I}_{ [U_i,1]  } \in \ \cdot\  \bigg\} 
        \bigtimes_{i = 1}^k \frac{d w_i}{ (w_i)^{ 1 + \alpha }  }
    \Bigg],
\end{align*}
where the sequence $(U_i)_{i \geq 1}$ are independent copies of Unif$(0,1)$, and the last line of the display above follows from the definition of $\mathbf C_S(\cdot)$ in \eqref{def, lambda, measure C S, LD for Hawkes univariate case}.
Now, let
\begin{align*}
    \bar N_n(t) = N(nt)/n,
    \qquad
    \bar{\bm N}_n = \{ \bar N_n(t):\ t \in [0,1]  \}.
\end{align*}
We are ready to state the strengthened univariate results.

\begin{theorem}
\label{theorem: LD for Hawkes, univariate results wrt M1 D0T}
    Let $k \in \mathbb N$.
    Let $A$ be a Borel set of $\D$ equipped with the Skorokhod $M_1$ topology.
    Suppose that
    \begin{align}
        \int_{x/\log x}^\infty f^{\bm N}(t)dt = \lo\big( \breve \lambda_{k}(x)\big/x \big),
        \label{tail condition on decay function, theorem: LD for Hawkes, univariate results wrt M1 D0T}
    \end{align}
    and the set $A$ is bounded away from $\breve \D_{k-1;\mu_N}$ under $\dm{}$.
    Then,
    \begin{align*}
        \breve{ \mathbf C }_{ k;\mu_N  }(A^\circ)
        \leq 
        \liminf_{n \to \infty}
        \frac{
            \P(\bar{\bm N}_n \in A)
        }{
            \breve \lambda_k(n)
        }
        \leq 
        \limsup_{n \to \infty}
        \frac{
            \P(\bar{\bm N}_n \in A)
        }{
            \breve \lambda_k(n)
        }
        \leq 
        \breve{ \mathbf C }_{ k;\mu_N  }(A^-),
    \end{align*}
    where the interior and exterior sets $A^\circ$ and $A^-$ are determined under $\dm{}$
    (i.e., w.r.t.\ the $M_1$ topology of $\D = \D\big([0,1],\R\big)$).
\end{theorem}

Since Theorem~\ref{theorem: LD for Hawkes, univariate results wrt M1 D0T} is established w.r.t.\ the $M_1$ topology of $\D[0,T]$ (i.e., the càdlàg space with compact domain) instead of $\D[0,\infty)$,
our results now address a broader class of events or functional of $\bar{\bm N}_n$
compared to only applying Theorem~\ref{theorem: LD for Hawkes} directly to univariate Hawkes processes.
We conclude this subsection with the proof of Theorem~\ref{theorem: LD for Hawkes, univariate results wrt M1 D0T}.

\begin{proof}[Proof of Theorem~\ref{theorem: LD for Hawkes, univariate results wrt M1 D0T}]
The proof is in the same spirit as Theorem~\ref{theorem: LD for Hawkes}.
In particular, like in \eqref{def: Levy process, LD for Hawkes} and \eqref{def: scaled Levy process, LD for Hawkes},
we set $L(t)$ as the compound Poisson process with $c_N$ as arrival rates and $S$ (the cluster size) as increments,
and let $\bar{\bm L}_n = \{ L(nt)/n:\ t \in [0,1]  \}$ be the scaled version of the sample path of $L(t)$ embedded in $\D$.
By Proposition~\ref{proposition, law of the levy process approximation, LD for Hawkes} and Lemma~\ref{lemma: asymptotic equivalence when bounded away, equivalence of M convergence},
it suffices to fix some positive integer $k$ and verify the following:
Let $A$ be a Borel set of $\D$ equipped with the $M_1$ topology such that $A$ is bounded away from $\breve \D_{k-1;\mu_N}$ under $\dm{}$;
for each $\Delta > 0$ we have
\begin{align}
    \P\Big( \dm{}\big( \bar{\bm N}_n, \bar{\bm L}_n  \big)  \mathbbm{I}\big\{\bar{\bm N}_n\text{ or }\bar{\bm L}_n \in A\big\} > \Delta  \Big) = \lo\big(\breve \lambda_k(n)\big),\quad \text{as }n \to \infty.
    \label{proof, goal, theorem: LD for Hawkes, univariate results wrt M1 D0T}
\end{align}

To this end, we consider the definitions in \eqref{def: large jump process for Levy, LD for Hawkes}--\eqref{def: sample path scaled process, Levy, LD for Hawkes} and \eqref{def: large jump process for N, LD for Hawkes}--\eqref{def: sample path scaled process, N, LD for Hawkes} in the unvariate setting.
That is, we use $\bar{\bm N}^{>\delta}_n$ (resp., $\bar{\bm  N}^{\leqslant\delta}_n$) to denote a modified version of $\bar{\bm N}_n$ where we only include clusters in $N(t)$ whose sizes are larger than (resp., bounded by) $n\delta$.
We adopt similar notations $\bar{\bm L}^{>\delta}_n$ and $\bar{\bm L}^{\leqslant\delta}_n$ for $\bar{\bm L}_n$.
First, the proof of Claim~\eqref{proof, goal, lemma: asymptotic equivalence between L and large jump, LD for Hawkes} in Lemma~\ref{lemma: asymptotic equivalence between L and large jump, LD for Hawkes} implies that for any $\Delta > 0$,
\begin{align}
    \P\bigg(
        \sup_{t \in [0,1]}
        \bigg|
            \bar{L}_n(t)
            -
            \Big(  \bar{ L}_n^{>\delta}(t) + \mu_{\bm N}t   \Big)
        \bigg| >\Delta  
    \bigg) = \lo\big(\breve \lambda_k(n)\big),
    \qquad
    \forall \delta > 0\text{ small enough.}
    \label{proof, AE for small jump process, L, theorem: LD for Hawkes, univariate results wrt M1 D0T}
\end{align}
Similarly, the proof of Claim~\eqref{proof, goal, lemma: asymptotic equivalence between N and large jump, LD for Hawkes} in Lemma~\ref{lemma: asymptotic equivalence between N and large jump, LD for Hawkes} implies that for any $\Delta > 0$,
\begin{align}
    \P\bigg(
        \sup_{t \in [0,1]}
        \bigg|
            \bar{N}_n(t)
            -
            \Big(  \bar{N}_n^{>\delta}(t) + \mu_{\bm N}t   \Big)
        \bigg| >\Delta  
    \bigg) = \lo\big(\breve \lambda_k(n)\big),
    \qquad
    \forall \delta > 0\text{ small enough.}
     \label{proof, AE for small jump process, N, theorem: LD for Hawkes, univariate results wrt M1 D0T}
\end{align}
Meanwhile, in \eqref{proof, goal, theorem: LD for Hawkes, univariate results wrt M1 D0T} note that:
(i) $\dm{}$ is upper bounded by the uniform metric;
(ii) 
whenever $\bar{\bm L}_n \in A$ and $\dm{}\big( \bar{\bm L}_n, \bar{\bm L}^{>\delta}_n + \mu_N \mathbf 1  \big) \leq \Delta$,
we must have $\bar{\bm L}^{>\delta}_n + \mu_N \mathbf 1 \in A^\Delta$; the same can be argued for $\bar{\bm N}_n$;
and 
(iii) for any $\Delta > 0$ small enough the set $A^\Delta$ is still bounded away from $\breve \D_{k-1;\mu_N}$.
Therefore, due to \eqref{proof, AE for small jump process, L, theorem: LD for Hawkes, univariate results wrt M1 D0T}--\eqref{proof, AE for small jump process, N, theorem: LD for Hawkes, univariate results wrt M1 D0T} and the arbitrariness of the choice of $A$ and $\Delta> 0$ in \eqref{proof, goal, theorem: LD for Hawkes, univariate results wrt M1 D0T},
it suffices to show the following:
Let $A$ be a Borel set of  $(\D, \dm{} )$ that is bounded away from $\breve \D_{k-1;\mu_N}$ under $\dm{}$ and let $\Delta > 0$;
for any $\delta > 0$ small enough we have
\begin{align}
    \P\Big( \dm{}\big( \bar{\bm N}^{>\delta}_n, \bar{\bm L}^{>\delta}_n  \big)  
    \mathbbm{I}\underbrace{ \big\{\bar{\bm N}^{>\delta}_n+\mu_N \mathbf 1 \in A\text{ or }\bar{\bm L}^{>\delta}_n + \mu_N \mathbf 1\in A\big\}
    }_{  \delequal A_n(\delta) }
    > \Delta  \Big) = \lo\big(\breve \lambda_k(n)\big).
    \label{proof, goal, 2, theorem: LD for Hawkes, univariate results wrt M1 D0T}
\end{align}

Without loss of generality, in \eqref{proof, goal, 2, theorem: LD for Hawkes, univariate results wrt M1 D0T} we fix $\delta > 0$, and take some $\Delta > 0$ small enough such that $A^\Delta$ is still bounded away from $\breve \D_{k-1;\mu_N}$ under $\dm{}$.
To proceed, we define a few events. 
Let $H$ be the lifetime of a generic copy of the cluster $S$, and the sequence $(S_j,H_j)$ be independent copies of $(S,H)$.
Generate the sequence $0 < T_1 < T_2 < \ldots$ by a Poisson process with rate $c_N$, which denote the arrival times of each immigrant in $N(t)$.
By definition, the associated compound Poisson process admits the form $L(t) = \sum_{j \geq 1}S_j\mathbbm{I}\{ T_j \leq t \}$.
For each $\epsilon > 0$, define events
\begin{align}
    F_n(\epsilon)\delequal 
    \bigg\{
        H_j \leq n\epsilon\text{ for any }j \geq 1\text{ with }T_j \leq n
    \bigg\}.
    \nonumber
\end{align}
That is, on the event $F_n(\epsilon)$, any immigrant that arrives by time $n$ induces a cluster with lifetime shorter than $n\epsilon$.
Next, let $K^{>\delta}_n$ be the count of jumps in $\bar{\bm L}^{>\delta}_n$.
Due to the $\bo(n)$ time scaling in $\bar{\bm L}^{>\delta}_n$,
we have $K^{>\delta}_n =\#\{ j \geq 1:\ T_j \leq n,\ S_j > n\delta  \}$.
Now, define events
\begin{align}
    E^{(1)}_n(\delta) \delequal 
    \Big\{
        K^{>\delta}_n = k
    \Big\},
    \qquad
    E^{(2)}_n(\delta,\epsilon) \delequal 
    \Big\{
     \bar{\bm L}^{>\delta}_n(t) - \bar{\bm L}^{>\delta}_n(t-) = 0\ \forall t \in [1-\epsilon,1]
    \Big\}.
    \nonumber
\end{align}
Also, recall the definition of events $A_n(\delta)$ in \eqref{proof, goal, 2, theorem: LD for Hawkes, univariate results wrt M1 D0T}.
Suppose that (i) for any $\epsilon > 0$, we have
\begin{align}
    \P\Big(  \big( F_n(\epsilon) \big)^\complement \Big) & = o\big(\breve \lambda_k(n)\big);
    \label{proof, claim 1, theorem: LD for Hawkes, univariate results wrt M1 D0T}
\end{align}
(ii)  next,
\begin{align}
    \P\bigg(  \Big( A_n(\delta) \cap F_n(\epsilon) \Big) \setminus E^{(1)}_n(\delta)     \bigg) & = \lo\big(\breve \lambda_k(n)\big),
    \quad \forall \epsilon \in (0,\Delta),
    \label{proof, claim 2, theorem: LD for Hawkes, univariate results wrt M1 D0T}
    \\ 
    \lim_{\epsilon \downarrow 0}\lim_{n \to \infty}
        \P\bigg(  \Big( A_n(\delta) \cap F_n(\epsilon) \cap E^{(1)}_n(\delta) \Big) \setminus E^{(2)}_n(\delta,\epsilon)     \bigg)\bigg/& \breve \lambda_k(n) = 0;
        \label{proof, claim 3, theorem: LD for Hawkes, univariate results wrt M1 D0T}
    \end{align}
(iii) furthermore, for any $\epsilon \in (0,\Delta)$,
\begin{align}
    \text{it holds on the event }A_n(\delta) \cap F_n(\epsilon) \cap E^{(1)}_n(\delta) \cap E^{(2)}_n(\delta,\epsilon)&\text{ that } \dm{}\big( \bar{\bm N}^{>\delta}_n, \bar{\bm L}^{>\delta}_n  \big)  \leq \Delta.
    \label{proof, claim 4, theorem: LD for Hawkes, univariate results wrt M1 D0T}
\end{align}
Then, the bound \eqref{proof, goal, 2, theorem: LD for Hawkes, univariate results wrt M1 D0T} follows directly from Claims~\eqref{proof, claim 1, theorem: LD for Hawkes, univariate results wrt M1 D0T}--\eqref{proof, claim 4, theorem: LD for Hawkes, univariate results wrt M1 D0T}.
Now, it only remains to verify these claims.

\medskip
\noindent
\textbf{Proof of Claim \eqref{proof, claim 1, theorem: LD for Hawkes, univariate results wrt M1 D0T}}.
This is an immediate consequence of Lemma~\ref{lemma: lifetime of a cluster, LD for Hawkes} under condition~\eqref{tail condition on decay function, theorem: LD for Hawkes, univariate results wrt M1 D0T}.
See also
the proof of Claim~\eqref{proof, sub goal 2, lemma: asymptotic equivalence between large jump N and L, LD for Hawkes} in Lemma~\ref{lemma: asymptotic equivalence between large jump N and L, LD for Hawkes}.

\medskip
\noindent
\textbf{Proof of Claim \eqref{proof, claim 2, theorem: LD for Hawkes, univariate results wrt M1 D0T}}.
By the law of the compound Poisson process $L(t)$,
\begin{align*}
    \P\big( K^{>\delta}_n \geq k + 1 \big)
    & = 
    \P\Big( \text{Poisson}\big( n\P( S > n\delta )  \big) \geq k + 1  \Big)
    \\ 
    & \leq \big( n\P( S > n\delta )  \big)^{k+1}\qquad\text{due to }\P\big(\text{Poisson}(\lambda) \geq m\big)\leq \lambda^m
    \\ 
    & = \bo\Big( \big( n\P( B > n )  \big)^{k+1} \Big) = \lo\big(\breve \lambda_k(n)\big)
    \qquad
    \text{by \eqref{univariate result, cluster tail asymptotics} and \eqref{def breve lambda k, univariate case}.}
\end{align*}
Therefore, it suffices to show that for any $\epsilon \in (0,\Delta)$,
\begin{align}
     \Big\{\bar{\bm N}^{>\delta}_n+\mu_N \mathbf 1 \in A\text{ or }\bar{\bm L}^{>\delta}_n + \mu_N \mathbf 1\in A\Big\}
     \cap F_n(\epsilon) \cap \Big\{ K^{>\delta}_n \leq k - 1  \Big\}   
     = \emptyset.
    \nonumber
\end{align}
First, on $\big\{ K^{>\delta}_n \leq k - 1  \big\}$, we have 
$
\bar{\bm L}^{>\delta}_n + \mu_N \mathbf 1 \in \breve \D_{k-1;\mu_N}.
$
Since  $A$ is bounded away from $\breve \D_{k-1;\mu_N}$, this confirms that 
$
\big\{ \bar{\bm L}^{>\delta}_n + \mu_N \mathbf 1\in A \big\} \cap \big\{ K^{>\delta}_n \leq k - 1  \big\} = \emptyset.
$
Next, on the event $F_n(\epsilon)$, the lifetime of any cluster arriving by time $n$ is upper bounded by $n\epsilon$.
Note that in $\bar N^{>\delta}_n$ and $\bar L^{>\delta}_n$, due to the $\bo(n)$ time scaling,
this translates to an $\epsilon$ upper bound time-wise.
Also, recall that we only consider $\epsilon \in (0,\Delta)$.
By the same arguments in Lemma~\ref{lemma: step functions under M1 metric, LD for Hawkes},
one can see that, on the event $F_n(\epsilon) \cap \big\{ K^{>\delta}_n \leq k - 1  \big\}$,
by setting $\xi(t) = \bar L^{>\delta}_n(t) \wedge \bar N^{>\delta}_n(1)\ \forall t \in [0,1]$,
the path $\xi\in \D$ admits the form
\begin{align*}
    \xi(t) = \mu_N t + \sum_{ i = 1 }^{k-1} w_i \mathbbm{I}_{ [t_i,1] }(t),
    \qquad\text{where }w_i \geq 0\text{ and }t_i \in (0,1]\ \forall i \in [k-1],
\end{align*}
and satisfies $\dm{}( \xi, \bar{\bm N}^{>\delta}_n ) < \epsilon < \Delta$.
This confirms that, on the event $F_n(\epsilon) \cap \big\{ K^{>\delta}_n \leq k - 1  \big\}$, we must have
$
 \bar{\bm N}^{>\delta}_n + \mu_N \mathbf 1 \in (\breve \D_{k-1;\mu_N})^\Delta.
$
However, this implies that on the event $F_n(\epsilon) \cap \big\{ K^{>\delta}_n \leq k - 1  \big\}$, we must have 
$
\bar{\bm N}^{>\delta}_n + \mu_N \mathbf 1 \notin A.
$
Indeed, note that
\begin{align*}
    & F_n(\epsilon) \cap \big\{ K^{>\delta}_n \leq k - 1  \big\} \cap \big\{ \bar{\bm N}^{>\delta}_n + \mu_N \mathbf 1 \in A \big\}
    \subseteq 
    \big\{  \bar{\bm N}^{>\delta}_n + \mu_N \mathbf 1 \in (\breve \D_{k-1;\mu_N})^\Delta \cap A  \big\}.
\end{align*}
Meanwhile, recall the running assumption that $A^\Delta$ is bounded away from  $\breve \D_{k-1;\mu_N}$ under $\dm{}$,
which implies $(\D_{k-1;\mu_N})^\Delta \cap A  = \emptyset$,
and hence 
$
 F_n(\epsilon) \cap \big\{ K^{>\delta}_n \leq k - 1  \big\} \cap \big\{ \bar{\bm N}^{>\delta}_n + \mu_N \mathbf 1 \in A \big\} = \emptyset.
$
This concludes the proof of Claim~\eqref{proof, claim 2, theorem: LD for Hawkes, univariate results wrt M1 D0T}.

\medskip
\noindent
\textbf{Proof of Claim \eqref{proof, claim 3, theorem: LD for Hawkes, univariate results wrt M1 D0T}}.
It suffices to show that 
\begin{align*}
    \lim_{\epsilon \downarrow 0}
        \P\bigg( 
        \underbrace{ K^{>\delta}_n = k;\ \exists t \in [1-\epsilon,1]\text{ such that }\bar L^{>\delta}_n(t) - \bar L^{>\delta}_n(t-) \neq 0 
        }_{  = E^{(1)}_n(\delta) \setminus E^{(2)}_n(\delta,\epsilon) }
        \bigg)\bigg/& \breve \lambda_k(n) = 0.
\end{align*}
By the law of compound Poisson processes, we have (where the $U_j$'s are i.i.d.\ copies of Unif$(0,1)$)
\begin{align*}
    & \P\big(E^{(1)}_n(\delta) \setminus E^{(2)}_n(\delta,\epsilon)\big)
    \\ 
    & = 
    \P\Big(
        \text{Poisson}\big( n\P(S > n\delta)  \big) = k
    \Big)
    \cdot \P\Big(  U_j \in [1-\epsilon,1]\text{ for some }j \in [k] \Big)
    \\ 
    & \leq 
    \big( n\P(S > n\delta)  \big)^k \cdot k\epsilon
    \\ 
    & \sim 
    k\epsilon \cdot \Bigg( \frac{ (\E S)^{1 + \alpha}  }{  \delta^\alpha }  \Bigg)^k \cdot \breve \lambda_k(n)
    \qquad
    \text{by \eqref{univariate result, cluster tail asymptotics} and \eqref{def breve lambda k, univariate case}.}
\end{align*}
Sending $\epsilon \downarrow 0$, we conclude the proof of Claim \eqref{proof, claim 3, theorem: LD for Hawkes, univariate results wrt M1 D0T}.

\medskip
\noindent
\textbf{Proof of Claim \eqref{proof, claim 4, theorem: LD for Hawkes, univariate results wrt M1 D0T}}.
This follows form Lemma~\ref{lemma: step functions under M1 metric, LD for Hawkes} and our choice of $\epsilon \in (0,\Delta)$.
In particular, note that in the one-dimensional case, the product $M_1$ metric $\dmp{ \subZeroT{1} }$ degenerates to the standard $M_1$ metric $\dm{}$ on $\D = \D\big([0,1],\R\big)$.
\end{proof}

\section{Sample Path Large Deviations for L\'evy Processes with $\MRV$ Increments (General Case)}
\label{sec: appendix, LD for Levy, General Case}

In this section, we provide an analog of Theorem~\ref{theorem: LD for levy, MRV}---the sample path LDPs for L\'evy Processes with $\MRV$ Increments---under relaxed conditions.
Such developments allow us to extend Theorem~\ref{theorem: LD for levy, MRV} beyond the context of Hawkes processes considered in the main paper and to cover a broader class of L\'evy processes whose increments exhibit a more flexible form of multivariate hidden regular variation.

Specifically, we work with the following condition regarding the L\'evy measure $\nu(\cdot)$ of the $d$-dimensional L\'evy process $\bm L(t)$.

\begin{assumption}[$\MRV$ increments in L\'evy Processes (General Case)]
\label{assumption: ARMV in levy process, general case}
The L\'evy measure $\nu(\cdot)$ of $\bm L(t)$ is supported on $\R^d_+$
and satisfies
$
\nu \in
    \MRV
    \Big(
(\bar{\bm s}_j)_{j \in [m]},
\big(\alpha(\bm j)\big)_{ \bm j \subseteq [m] },
(\lambda_{\bm j})_{ \bm j \in \powersetTilde{m} }, (\mathbf C_{\bm j})_{ \bm j \in \powersetTilde{m} }
\Big)
$
under
\begin{itemize}
    \item some positive integer $m \in [d]$,
    \item a collection of vectors $(\bar{\bm s}_j)_{j \in [m]}$ in $\R^d_+$ that are linearly independent,
    \item a collection of real numbers $(\alpha(\bm j))_{\bm j \subseteq [m]}$ that are strictly monotone w.r.t.\ $\bm j$ (see \eqref{def: monoton sequence of tail indices in MHRV}) with $\alpha(\emptyset) = 0$,
    $\alpha(\bm j) >1\ \forall \bm j \in \powersetTilde{m}$, and $\alpha(\bm j) \neq \alpha(\bm j^\prime)\ \forall \bm j,\bm j^\prime \in \powersetTilde{m}$ with $\bm j\neq \bm j^\prime$;
    \item regularly varying functions $\lambda_{\bm j}(n) \in \RV_{ -\alpha(\bm j) }(n)$ (for each $\bm j \in \powersetTilde{m}$),
    \item Borel measures $(\mathbf C_{\bm j})_{\bm j \subseteq [m]}$ where, for each $\bm j \in \powersetTilde{m}$,
    the measure $\mathbf C_{\bm j}$ is
    supported on $\R^d(\bm j)$,
        such that 
        $
        \mathbf C_{\bm j}(A) < \infty
        $
        holds for any $\epsilon > 0$ and Borel set $A \subseteq \R^d_+$ bounded away from $\bar\R^d(\bm j,\epsilon)$ (see \eqref{def: cone R d i basis S index alpha}).
\end{itemize}
        
\end{assumption}

Throughout this section, we denote the basis by $\bar{\textbf S} = \{ \bar{\bm s}_j:\ j \in [m] \}$ and the collection of tail indices by $\bm \alpha = \big(\alpha(\bm j)\big)_{ \bm j \subseteq [m] }$.
Besides, we adopt the simplified notations
$
{\R^d(\bm j)} = {\R^d(\bm j;\bar{\textbf S})},
$
$
{ \bar \R^d(\bm j, \epsilon) } = \bar \R^d(\bm j, \epsilon;\bar{\textbf S}),
$
and
$
{\bar\R^{d}_\leqslant(\bm j,\epsilon)}=
{\bar\R^{d}_\leqslant(\bm j,\epsilon; \bar{\textbf S},\bm \alpha)}
$
in Section~\ref{subsec: MRV}.
Compared to Assumption~\ref{assumption: ARMV in levy process}, the key differences in Assumption~\ref{assumption: ARMV in levy process, general case} are that
\begin{itemize}
    \item 
        we drop the condition~\eqref{cond, additivity of rate functions for MRV, Ld for Levy} on the rate functions $\lambda_{\bm j}(\cdot)$;

    \item 
        we now consider a weaker $\MRV$ condition (instead of the $\MRV^*$ condition in Assumption~\ref{assumption: ARMV in levy process}); also, the number of base vectors $\bar s_j$ considered is not fixed as $d$.
\end{itemize}
When imposing the more flexible Assumption~\ref{assumption: ARMV in levy process, general case}, one potential issue is that we remain agnostic
about
the tail behaviors outside of the cone $\R^d([m])$.
This potential lack of knowledge about the hidden regular variation in certain regions
is reflected in Theorem~\ref{theorem: LD for levy, MRV, general case}---the main result of this section---through the following definition:
\begin{align}
    \notationdef{notation-breve-R-d-type-j-modified-cones}{\hat{\R}^d(\bm j,\epsilon;\bar{\textbf S})}
    \delequal
    \begin{cases}
        \bar{\R}(\bm j,\epsilon;\bar{\textbf S}) & \text{if }\bm j \in \powersetTilde{m},\ \bm j \neq [m]
        \\
        \R^d_+ & \text{if }\bm j = [m]
    \end{cases}.
    \label{def: breve R d type j modified cones}
\end{align}
In other words,
we only modify the definition $\bar{\R}^d(\bm j,\epsilon;\bar{\textbf S})$ to include the entire domain $\R^d_+$.
We also write $\notationdef{notation-breve-R-d-type-j-modified-cones-without-basis}{\hat{\R}^d(\bm j,\epsilon)} \delequal \hat{\R}^d(\bm j,\epsilon;\bar{\textbf S})$ as there will be no ambiguity about $\bar{\textbf S}$ in this section.

We introduce a few notions that are required for the statement of Theorem~\ref{theorem: LD for levy, MRV, general case}.
Recall that 
we use notations like $\bm k = (k_{i})_{ i \in \mathcal J  } \in A^{ \mathcal J }$
to denote vectors of length $|\mathcal J|$ with each coordinate taking values in $A$ and indexed by elements in $\mathcal J$.
Besides, recall that we use $\powersetTilde{m}$ to denote the collection of all \emph{non-empty} subsets of $[m]$.
Analogous to \eqref{def: j jump path set with drft c, LD for Levy, MRV},
given $\bm x \in \R^d$, $\bm{\mathcal K} = (\mathcal K_{\bm j})_{ \bm j \in \powersetTilde{m} } \in \mathbb Z_+^{ \powersetTilde{m} } \setminus \{\bm 0\}$, 
we define
\begin{equation}    \label{def: j jump path set with drft c, LD for Levy, MRV, general case}
    \begin{aligned}
        \hat \D^\epsilon_{ \bm{\mathcal K};\bm x  }[0,T]
    & \delequal 
    \Bigg\{
        \bm x\bm 1 + \sum_{ \bm j \in \powersetTilde{m} }\sum_{k = 1}^{ \mathcal K_{\bm j} }\bm w_{\bm j,k} \mathbbm{I}_{ [t_{\bm j, k}, T] }:
        \\
    &
        t_{\bm j,k} \in (0,T]\text{ and }\bm w_{\bm j,k} \in \hat\R^d(\bm j, \epsilon)\ 
        \forall \bm j \in \powersetTilde{m},\ k \in [\mathcal K_{\bm j}];
        \ t_{\bm j,k} \neq t_{\bm j^\prime,k^\prime}\ \forall (\bm j,k)\neq (\bm j^\prime,k^\prime)
    \Bigg\},
    \end{aligned}
\end{equation}
i.e., the set of all paths $\xi$ of the form
\begin{align}
    \xi(t) = \bm x t + \sum_{ \bm j \in \powersetTilde{m} }\sum_{ k = 1 }^{ \mathcal K_{\bm j} }\bm w_{\bm j, k} \mathbbm{I}_{ [t_{\bm j, k},T] }(t),
    \qquad\forall t \in [0,T],
    \label{expression for xi in breve D set, general case}
\end{align}
where
each $\bm w_{\bm j,k}$ belongs to the cone $\hat{\R}^d(\bm j,\epsilon)$,
and any two numbers in the sequence $(t_{\bm j,k})_{ \bm j \in \powersetTilde{m},\ k \in [\mathcal K_{\bm j}]  }$ are not equal (i.e., the arrival times of jumps would not coincide).
Next, define 
\begin{align}
    {\hat{c}(\bm{\mathcal K})} \delequal 
    \sum_{ \bm j \in \powersetTilde{m} }\mathcal K_{\bm j} \cdot \Big( \alpha(\bm j) - 1 \Big),
    \qquad \forall 
    \bm{\mathcal K} = (\mathcal K_{\bm j})_{ \bm j \in \powersetTilde{m} } \in \mathbb Z_+^{ \powersetTilde{m} }.
    \label{def: cost alpha j, LD for Levy MRV, general case}
\end{align}
Note that this agrees with  $\breve{c}(\cdot)$ in \eqref{def: cost alpha j, LD for Levy MRV},
and we repeat the definition here so that ${\hat{c}(\bm{\mathcal K})}$ is consistent with other notations introduced in this section.
Given $\bm x \in \R^d$, $\epsilon \geq 0$, $\bm{\mathcal K} = (\mathcal K_{\bm j})_{ \bm j \in \powersetTilde{m} } \in \mathbb Z_+^{ \powersetTilde{m} } \setminus \{\bm 0\}$,
let
\begin{align}
    \hat \D^\epsilon_{ \leqslant {\bm{\mathcal K}};\bm x }[0,T] 
    \delequal 
    \bigcup_{
        \substack{
            \bm{\mathcal K}^\prime \in  \mathbb Z_+^{ \powersetTilde{m} }:
            \\
            \bm{\mathcal K}^\prime\neq \bm{\mathcal K},\ \hat c(\bm{\mathcal K}^\prime) \leq \hat c(\bm{\mathcal K})
        }
    }
    \hat \D^\epsilon_{ \bm{\mathcal K}^\prime;\bm x  }[0,T].
    \label{def: path with costs less than j jump set, drift x, LD for Levy MRV, general case}
\end{align}
Meanwhile, given  $\bm{\mathcal K} = (\mathcal K_{\bm j})_{ \bm j \in \powersetTilde{m} } \in \mathbb Z_+^{ \powersetTilde{m} } \setminus \{\bm 0\}$, let
\begin{align}
    {\hat \lambda_{\bm{\mathcal K}}(n)}
    \delequal 
    \prod_{ \bm j \in \powersetTilde{m} } \Big( n\lambda_{ \bm j }(n)\Big)^{ \mathcal K_{\bm j}},
    \label{def: scale function for Levy LD MRVk, general case}
\end{align}
where $\lambda_{\bm j}(n) \in \RV_{ -\alpha(\bm j)  }(n)$ is the rate function of the $\MRV$ condition in Assumption~\ref{assumption: ARMV in levy process, general case}.
Then, by the definition of $\hat c(\cdot)$ in \eqref{def: cost alpha j, LD for Levy MRV, general case}, we have 
$
\hat\lambda_{\bm{\mathcal K}}(n) \in \RV_{ -\hat c(\bm{\mathcal K})  }(n).
$
Lastly, 
analogous to \eqref{def: measure C type j L, LD for Levy, MRV},
given
$\bm x\in \R^d$ and $\bm{\mathcal K} = (\mathcal K_{\bm j})_{ \bm j \in \powersetTilde{m} } \in \mathbb Z_+^{ \powersetTilde{m} } \setminus \{\bm 0\}$,
we define the Borel measure 
\begin{equation}  \label{def: measure C type j L, LD for Levy, MRV, general case}        
    \begin{aligned}
       {\hat{\mathbf C}^{ \subZeroT{T} }_{\bm{\mathcal K};\bm x}(\ \cdot\ )}
    & \delequal
    \frac{1}{\prod_{ \bm j \in \powersetTilde{m} } \mathcal K_{\bm j}!} 
    \\
    & \quad\cdot
    \int 
    \mathbbm{I}\Bigg\{
        \bm x \bm 1  + \sum_{ \bm j \in \powersetTilde{m} } \sum_{ k \in [\mathcal K_{\bm j}] } \bm w_{\bm j,k}\mathbbm{I}_{ [t_{\bm j, k},T] } \in\ \cdot\ 
    \Bigg\}
    \bigtimes_{\bm j \in \powersetTilde{m}}
    \bigtimes_{ k \in [k_{\bm j}] }\bigg( (\mathbf C_{ \bm j } \times  \mathcal L_{(0,T)})\Big(d (\bm w_{\bm j, k}, t_{ \bm j,k })\Big)\bigg),
    \end{aligned}
\end{equation}
where the $\mathbf C_{\bm j}$'s are the limiting measures for the $\MRV$ condition of Assumption~\ref{assumption: ARMV in levy process, general case},
$\mathcal L_{(0,T)}$ is the Lebesgue measure on the interval $(0,T)$,
and we use $\nu_1 \times \nu_2$ to denote the product measure of $\nu_1$, $\nu_2$.

Let $\bm \mu_{\bm L}$ and $\bar{\bm L}^{ \subZeroT{T} }_n$ be defined as in \eqref{def: expectation of levy process}--\eqref{def: scale levy process bar L t}.
We are now ready to state Theorem~\ref{theorem: LD for levy, MRV, general case}.
\begin{theorem}\label{theorem: LD for levy, MRV, general case}
    Let Assumption~\ref{assumption: ARMV in levy process, general case} hold.
    Let $T \in (0,\infty)$, 
    $\bm{\mathcal K} = (\mathcal K_{\bm j})_{ \bm j \in \powersetTilde{m} } \in \mathbb Z_+^{ \powersetTilde{m} } \setminus \{\bm 0\}$,
    and let $B$ be a Borel set of $\D[0,T]$ equipped with the Skorokhod $J_1$ topology.
    Suppose that $B$  is bounded away from 
    $\hat \D^\epsilon_{ \leqslant \bm{\mathcal K};\bm \mu_{\bm L} }[0,T]$
    under $\dj{[0,T]}$ for some (and hence all) $\epsilon > 0$ small enough.
    Then,
    \begin{align}
        \hat{\mathbf C}_{\bm{\mathcal K};\bm \mu_{\bm L}}^{ _{[0,T]} }(B^\circ)
        \leq 
        \liminf_{n \to \infty}
        \frac{
        \P\big(\bar{\bm L}^{ _{[0,T]} }_n \in B\big)
        }{
            \hat \lambda_{\bm{\mathcal K} }(n)
        }
        \leq 
        \limsup_{n \to \infty}
        \frac{
        \P\big(\bar{\bm L}^{ _{[0,T]} }_n \in B\big)
        }{
            \hat \lambda_{\bm{\mathcal K} }(n)
        }
        \leq 
        \hat{\mathbf C}_{\bm{\mathcal K};\bm \mu_{\bm L}}^{ _{[0,T]} }(B^-) < \infty.
        \nonumber
    \end{align}
\end{theorem}

In other words, 
Theorem~\ref{theorem: LD for levy, MRV, general case} addresses the general case where a more flexible (and hence less detailed) description of the multivariate hidden regular variation (as in Assumption~\ref{assumption: ARMV in levy process, general case}) for the L\'evy measure $\nu(\cdot)$ is available, at the cost of slightly weaker characterizations of sample path large deviations when compared to Theorem~\ref{theorem: LD for levy, MRV}.
In particular, 
without conditions like \eqref{cond, additivity of rate functions for MRV, Ld for Levy},
we do not know if the rate functions $\hat\lambda_{\bm{\mathcal K}}(n)$ are comparable (as $n \to \infty$) for different $\bm{\mathcal K}$ if their tail indices $\hat c(\bm{\mathcal K})$ are equal;
as a result,
Theorem~\ref{theorem: LD for levy, MRV, general case} 
does not allow the merging of multiple cases corresponding to different $\bm{\mathcal K}$ with equal tail indices $\hat c(\bm{\mathcal K})$.
Besides, 
the modification from $\bar\R^d(\bm j,\epsilon)$ to $\hat\R^d(\bm j,\epsilon)$ in \eqref{def: breve R d type j modified cones}
enlarges
the set $\hat \D^\epsilon_{ \bm{\mathcal K};\bm \mu_{\bm L} }[0,T]$ 
in \eqref{def: j jump path set with drft c, LD for Levy, MRV, general case}
(when compared to $\breve \D^\epsilon_{ \bm{\mathcal K};\bm \mu_{\bm L} }[0,T]$ in \eqref{def: j jump path set with drft c, LD for Levy, MRV}).
This generally leads to fewer sets $B$ that are bounded away from $\hat \D^\epsilon_{ \leqslant\bm{\mathcal K};\bm \mu_{\bm L} }[0,T]$,
and thus fewer combinations of rare events and the rates $\bm{\mathcal K}$ to which Theorem~\ref{theorem: LD for levy, MRV, general case} can be applied.

The proof of Theorem~\ref{theorem: LD for levy, MRV, general case} is almost identical to that of Theorem~\ref{theorem: LD for levy, MRV} stated in Section~\ref{subsubsec: proof of LD with MRV},
and we provide a proof sketch below for the sake of completeness.

\begin{proof}[Proof Sketch for Theorem~\ref{theorem: LD for levy, MRV, general case}]
Analogous to Theorem~\ref{theorem: LD for levy, MRV},
 w.l.o.g.\ we consider the case where $T = 1$ and $\bm \mu_{\bm L} = \bm 0$,
and adopt the notations $\bar{\bm L}_n = \bar{\bm L}_n^{\subZeroT{1}}$, $\D = \D\big([0,1],\R^d\big)$, and $\dj{} = \dj{\subZeroT{1}}$.
Again, the proof relies on the notion of asymptotic equivalence for $\mathbb M$-convergence.
By Lemma~\ref{lemma: asymptotic equivalence when bounded away, equivalence of M convergence},
it suffices to verify the following two conditions:
\begin{enumerate}[(i)]
    \item
        For each $\bm{\mathcal K} = (\mathcal K_{\bm j})_{ \bm j \in \powersetTilde{m} } \in \mathbb Z_+^{ \powersetTilde{m} } \setminus \{\bm 0\}$
        and $\Delta > 0$,
        \begin{align}
            \lim_{n \to \infty}
    \P\Big(
        \dj{}\big( \hat{\bm L}_n^{>\delta}, \bar{\bm L}_n \big)
        > \Delta 
    \Big) 
    \Big/ \hat\lambda_{\bm{\mathcal K}}(n)
    = 0,
    \qquad
    \forall \delta > 0\text{ small enough,}
            \nonumber
        \end{align}
        where $\hat{\bm L}_n^{>\delta}$ is defined in \eqref{def: large jump time tau for Levy process L}--\eqref{def: large jump approximation, LD for Levy MRV}.

    \item
        Let $\bm{\mathcal K} = (\mathcal K_{\bm j})_{ \bm j \in \powersetTilde{m} } \in \mathbb Z_+^{ \powersetTilde{m} } \setminus \{\bm 0\}$ and $\epsilon > 0$.
        Let $f: \D \to [0,\infty)$ be bounded and continuous (w.r.t.\ the $J_1$ topology of $\D$).
        Suppose 
        that
        $B = \text{supp}(f)$
        is bounded away from $\hat\D^\epsilon_{ \leqslant \bm{\mathcal K}; \bm 0  }[0,1]$
            under $\dj{}$.
        Then, there exists $\delta_0 > 0$ such that for all but countably many $\delta \in (0,\delta_0)$,
        \begin{align*}
              \lim_{n \to \infty}
            {
                \E\big[ f(\hat{\bm L}^{>\delta}_n)\big]
         }\big/{
             \hat \lambda_{\bm{\mathcal K} }(n)
            }
            =
            \hat{\mathbf C}^{ \subZeroT{1} }_{\bm{\mathcal K};\bm 0}(f) < \infty.
        \end{align*}
\end{enumerate}
One can see that these two conditions are similar to the statements in Propositions~\ref{proposition: asymptotic equivalence, LD for Levy MRV} and \ref{proposition: weak convergence, LD for Levy MRV}.
Likewise, these two conditions can be verified by repeating the arguments in Section~\ref{subsec: proof, LD of Levy, proof of technical results}.
The only key difference stems from the definition of the modified cones $\hat\R^d(\bm j,\epsilon)$ in \eqref{def: breve R d type j modified cones}.
In particular,
any vector in $\R^d_+ \setminus \R^d([m])$ would now fall into $\hat\R^d([m],\epsilon)$ (under any $\epsilon \geq 0$).
Therefore, when keeping track of the large jumps in $\bm L(t)$, there is no need to specifically count those that lie outside of $\R^d([m])$ (i.e., the notions of $\tilde\tau^{>\delta}_n(k)$ and $\widetilde K^{>\delta}_n$ introduced in \eqref{def: tilde tau n delta k, big jump outside of full cone}--\eqref{def: tilde K n delta k, big jump outside of full cone} are not needed here).
Instead, 
for each $\delta > 0$, $\bm j \in \powersetTilde{m}$, let
\begin{align*}
    \hat\R^{>\delta}(\bm j) \delequal 
    \Bigg\{
        \bm x \in \R^d_+:
        \norm{\bm x} > \delta,\ \bm x \in \hat{\R}^d(\bm j,\delta)
        \setminus 
        \bigg( 
            \bigcup_{ \bm j^\prime \subseteq [m]:\ \bm j^\prime \neq \bm j,\ \alpha(\bm j^\prime) \leq \alpha(\bm j)  } \hat{\R}^d(\bm j^\prime,\delta)
        \bigg)
    \Bigg\},
\end{align*}
which adapts the definition of $\bar\R^{>\delta}(\bm j)$ in \eqref{def: tilde R geq delta, LD for Levy} to our current setting.
Next,
for each $\delta > 0$, $\bm j \in \powersetTilde{m}$, and $k \geq 1$, we define
(under the convention that $ { \hat\tau^{>\delta}_{ n }(0; \bm j)} \equiv 0$)
\begin{align*}
    { \hat\tau^{>\delta}_{ n }(k; \bm j)}
    & \delequal
    \inf\Big\{ 
        t >  \hat\tau^{>\delta}_{ n }(k - 1; \bm j):\ 
         \Delta \bar{\bm L}_n(t) \in \hat\R^{>\delta}(\bm j)
    \Big\},
    \qquad
    {\hat{\bm W}^{>\delta}_{ n }(k; \bm j)}
    \delequal 
    \Delta \bar{\bm L}_n\Big( \hat\tau^{>\delta}_{ n }(k; \bm j) \Big).
\end{align*}
Besides, for each $\delta > 0$, let
\begin{align*}
    {\hat K^{>\delta}_n(\bm j)} \delequal \max\{ k \geq 0:\ \hat\tau^{>\delta}_n(k;\bm j) \leq 1 \},
    \quad
    \forall \bm j \in \powersetTilde{m};
    \qquad
    {\hat{\bm K}^{>\delta}_n}
    \delequal
    \Big( {\hat K^{>\delta}_n(\bm j)} \Big)_{ \bm j \in \powersetTilde{m} }.
\end{align*}
In short, we adapt the definitions in \eqref{def: large jump in cone i, tau > delta n k i, LD for Levy MRV}--\eqref{def: K delta n type j, count vector, LD for Levy}
to count each type of large jumps in $\bm L(t)$ based on their locations w.r.t.\ the (modified) cones $\hat \R^d(\bm j)$ in \eqref{def: breve R d type j modified cones}, which are generated by the basis described in Assumption~\ref{assumption: ARMV in levy process, general case}.
Then, to prove conditions (i) and (ii) listed above,
it only remains to repeat the proofs in Section~\ref{subsubsec: proof of LD with MRV}.
The changes are minimal and, in fact, the arguments are now even simpler.
For instance, when adapting Lemma~\ref{lemma: asymptotic law for large jumps, LD for Levy MRV},
as noted above there is no need to consider $\tilde K^{>\delta}_n$ anymore.
Besides, in Lemmas~\ref{lemma: choice of bar epsilon and bar delta, LD for Levy MRV} and \ref{lemma: finiteness of measure breve C k},
the notion of allocation $\mathbb A(\bm k)$ is not needed 
(as noted right after Theorem~\ref{theorem: LD for levy, MRV, general case}, we do not consider merging the cases corresponding to different $\bm{\mathcal K}$ here),
so it suffices to fix some 
$\bm{\mathcal K} = (\mathcal K_{\bm j})_{ \bm j \in \powersetTilde{m} } \in \mathbb Z_+^{ \powersetTilde{m} } \setminus \{\bm 0\}$ in the proof.
The arguments are almost identical, and we omit the details here to avoid repetition.
\end{proof}

\ifshowtheoremtree
\newpage

\section{Theorem Tree}\label{sec: appendix, theorem tree}
\textbf{Theorem Tree of Theorem~\ref{theorem: LD for levy, MRV}}

\begin{thmdependence}[leftmargin=*]
    \thmtreenode{-}{Theorem}{theorem: LD for levy, MRV}{}
    \begin{thmdependence}
        \thmtreenode{-}{Lemma}{lemma: asymptotic equivalence when bounded away, equivalence of M convergence}{}

        \thmtreenode{-}{Proposition}{proposition: asymptotic equivalence, LD for Levy MRV}{}
        \begin{thmdependence}
            \thmtreenode{-}{Lemma}{lemma: concentration of small jump process, Ld for Levy MRV}{}

            \thmtreenode{-}{Lemma}{lemma: large jump approximation, LD for Levy MRV}{}

            \thmtreenode{-}{Lemma}{lemma: asymptotic law for large jumps, LD for Levy MRV}{}
            
        \end{thmdependence}

        \thmtreenode{-}{Proposition}{proposition: weak convergence, LD for Levy MRV}{}
        \begin{thmdependence}
            \thmtreenode{-}{Lemma}{lemma: zero mass for boundary sets, LD for Levy}{}

            \thmtreenode{-}{Lemma}{lemma: asymptotic law for large jumps, LD for Levy MRV}{}

            \thmtreenode{-}{Lemma}{lemma: choice of bar epsilon and bar delta, LD for Levy MRV}{}

            \thmtreenode{-}{Lemma}{lemma: finiteness of measure breve C k}{}
                \begin{thmdependence}
                    \thmtreenode{-}{Lemma}{lemma: choice of bar epsilon and bar delta, LD for Levy MRV}{}
                \end{thmdependence}
        \end{thmdependence}
        
    \end{thmdependence}
\end{thmdependence}

\bigskip
\noindent
\textbf{Theorem Tree of Theorem~\ref{theorem: LD for Hawkes}}
\begin{thmdependence}[leftmargin=*]
    \thmtreenode{-}{Theorem}{theorem: LD for Hawkes}{}
    \begin{thmdependence}
        \thmtreenode{-}{Lemma}{lemma: asymptotic equivalence when bounded away, equivalence of M convergence}{}

        \thmtreenode{-}{Theorem}{corollary: LD for Levy with MRV}{}
        \begin{thmdependence}
            \thmtreenode{-}{Theorem}{theorem: LD for levy, MRV}{}

            \thmtreenode{-}{Theorem}{portmanteau theorem M convergence}{}

            \thmtreenode{-}{Lemma}{lemma: B bounded away under J1 infinity}{}
            \begin{thmdependence}
                \thmtreenode{-}{Lemma}{lemma: finiteness of measure breve C k}{}
            \end{thmdependence}

            \thmtreenode{-}{Lemma}{lemma: bounded away condition under J1, from infty to T}{}
        \end{thmdependence}

        \thmtreenode{-}{Proposition}{proposition, law of the levy process approximation, LD for Hawkes}{}
        \begin{thmdependence}
            \thmtreenode{-}{Theorem}{theorem: main result, cluster size}{}
        \end{thmdependence}

        \thmtreenode{-}{Proposition}{proposition: asymptotic equivalence between Hawkes and Levy, LD for Hawkes}{}
        \begin{thmdependence}
            \thmtreenode{-}{Lemma}{lemma: asymptotic equivalence between L and large jump, LD for Hawkes}{}
            \begin{thmdependence}
                \thmtreenode{-}{Lemma}{lemma: concentration of small jump process, Ld for Levy MRV}{}

                \thmtreenode{-}{Proposition}{proposition, law of the levy process approximation, LD for Hawkes}{}
            \end{thmdependence}

        \thmtreenode{-}{Lemma}{lemma: asymptotic equivalence between N and large jump, LD for Hawkes}{}

        \thmtreenode{-}{Lemma}{lemma: asymptotic equivalence between large jump N and L, LD for Hawkes}{}
        \begin{thmdependence}
            \thmtreenode{-}{Lemma}{lemma: lifetime of a cluster, LD for Hawkes}{}

            \thmtreenode{-}{Lemma}{lemma: step functions under M1 metric, LD for Hawkes}{}
        \end{thmdependence}

        \end{thmdependence}
        
    \end{thmdependence}
\end{thmdependence}

\ifshowtheoremlist
\newpage
\footnotesize
\newgeometry{left=1cm,right=1cm,top=0.5cm,bottom=1.5cm}
\linkdest{location of theorem list}
\listoftheorems
\fi

\ifshowequationlist
\newpage
\linkdest{location of equation number list}
\section*{Numbered Equations}

\fi

\ifshownavigationpage
\newpage
\normalsize
\tableofcontents

\section*{Navigation Links}

\ifshowtheoremlist
    \noindent
    \hyperlink{location of theorem list}{List of Theorems}
    \bigskip
\fi

\ifshowtheoremtree
    \noindent
    \hyperlink{location of theorem tree}{Theorem Tree}
    \bigskip
\fi

\ifshowtheoremlist
    \noindent
    \hyperlink{location of equation number list}{Numbered Equations}
    \bigskip
\fi

\ifshownotationindex
    \noindent
    \hyperlink{location of notation index}{Notation Index}
    \begin{itemize}
    \item[] 
        \hyperlink{location, notation index A}{A},
        \hyperlink{location, notation index B}{B},
        \hyperlink{location, notation index C}{C},
        \hyperlink{location, notation index D}{D},
        \hyperlink{location, notation index E}{E},
        \hyperlink{location, notation index F}{F},
        \hyperlink{location, notation index G}{G},
        \hyperlink{location, notation index H}{H},
        \hyperlink{location, notation index I}{I},
        \hyperlink{location, notation index J}{J},
        \hyperlink{location, notation index K}{K},
        \hyperlink{location, notation index L}{L},
        \hyperlink{location, notation index M}{M},
        \hyperlink{location, notation index N}{N},
        \hyperlink{location, notation index O}{O},
        \hyperlink{location, notation index P}{P},
        \hyperlink{location, notation index Q}{Q},
        \hyperlink{location, notation index R}{R},
        \hyperlink{location, notation index S}{S},
        \hyperlink{location, notation index T}{T},
        \hyperlink{location, notation index U}{U},
        \hyperlink{location, notation index V}{V},
        \hyperlink{location, notation index W}{W},
        \hyperlink{location, notation index X}{X},
        \hyperlink{location, notation index Y}{Y},
        \hyperlink{location, notation index Z}{Z}
    \end{itemize}
\fi

\fi

\end{document}